\documentclass[11pt,letterpaper,reqno]{amsart}
\usepackage[foot]{amsaddr}
\usepackage[margin=1in]{geometry}
\usepackage{amsmath}
\usepackage{amssymb}
\usepackage{amsthm}
\usepackage{xcolor}
\usepackage{enumerate}
\usepackage{hyperref}

\allowdisplaybreaks

\def\eps{\varepsilon}
\def\E{\mathbb{E}}
\def\P{\mathbb{P}}
\def\R{\mathbb{R}}
\def\CC{\mathbb{C}}
\def\DD{\mathbb{D}}
\def\N{\mathbb{N}}
\def\OO{\mathbb{O}}
\def\d{\mathrm{d}}
\def\btheta{\boldsymbol{\theta}}
\def\bphi{\boldsymbol{\varphi}}
\def\bxi{\boldsymbol{\xi}}
\def\bLambda{\boldsymbol{\Lambda}}
\def\bPi{\boldsymbol{\Pi}}
\def\X{\mathbf{X}}
\def\W{\mathbf{W}}
\def\M{\mathbf{M}}
\def\H{\mathbf{H}}
\def\A{\mathbf{A}}
\def\B{\mathbf{B}}
\def\S{\mathbf{S}}

\def\Id{\mathrm{Id}}
\def\J{\mathbf{J}}
\def\O{\mathbf{O}}
\def\a{\mathbf{a}}
\def\b{\mathbf{b}}
\def\e{\mathbf{e}}
\def\g{\mathbf{g}}
\def\u{\mathbf{u}}

\def\x{\mathbf{x}}
\def\y{\mathbf{y}}
\def\z{\mathbf{z}}

\def\cE{\mathcal{E}}
\def\cG{\mathcal{G}}
\def\cH{\mathcal{H}}
\def\cR{\mathcal{R}}
\def\cF{\mathcal{F}}

\def\cN{\mathcal{N}}
\def\cO{\mathcal{O}}
\def\cP{\mathcal{P}}
\def\cS{\mathcal{S}}
\def\cT{\mathcal{T}}
\def\cZ{\mathcal{Z}}
\def\GP{\operatorname{GP}}
\def\GOE{\operatorname{GOE}}

\def\diag{\operatorname{diag}}
\def\sign{\operatorname{sign}}
\def\sp{\operatorname{span}}
\def\Haar{\operatorname{Haar}}

\def\Tr{\operatorname{Tr}}
\def\Gibbs{\mathrm{Gibbs}}
\def\Ising{\mathrm{Ising}}
\def\Law{\operatorname{Law}}
\def\Var{\operatorname{Var}}
\def\Cov{\operatorname{Cov}}
\def\op{\mathrm{op}}
\def\Lip{\mathrm{Lip}}

\def\1{\mathbf{1}}
\def\I{\mathrm{I}}
\def\II{\mathrm{II}}
\def\III{\mathrm{III}}

\def\id{\operatorname{id}}
\def\DKL{D_{\mathrm{KL}}}

\def\v{\mathbf{v}}

\def\C{\mathbf{C}}
\def\bR{\mathbf{R}}
\def\bE{\mathbf{E}}
\def\bPhi{\boldsymbol{\Phi}}
\def\bDelta{\boldsymbol{\Delta}}
\def\bSigma{\boldsymbol{\Sigma}}

\newtheorem{theorem}{Theorem}[section]
\newtheorem{lemma}[theorem]{Lemma}
\newtheorem{proposition}[theorem]{Proposition}
\newtheorem{corollary}[theorem]{Corollary}

\theoremstyle{definition}
\newtheorem{assumption}[theorem]{Assumption}
\newtheorem{definition}[theorem]{Definition}
\newtheorem{remark}[theorem]{Remark}

\numberwithin{equation}{section}
\begin{document}

\title[DMFT for the orthogonally-invariant SK model]{Dynamical mean-field limit and replica-symmetric free energy for the orthogonally-invariant SK model}

\author{Zhou Fan}
\author{Theodor Misiakiewicz}
\author{Leda Wang}
\author{Garrett G.\ Wen}
\address{Yale University, Department of Statistics and Data Science}
\email{zhou.fan@yale.edu}
\email{theodor.misiakiewicz@yale.edu} \email{leda.wang@yale.edu} \email{gang.wen@yale.edu}

\begin{abstract}
We study a class of diffusion processes on $\R^n$ interacting through a
symmetric matrix $\X \in \R^{n \times n}$. When eigenvectors of
$\X$ are random and uniformly distributed on the orthogonal group, we derive a
dynamical mean-field limit for the empirical law of sample paths,
extending a classical characterization due to Sompolinsky and Zippelius for
$\X \sim \GOE$. This limit takes the form of a generalized Langevin equation
with correlated Gaussian noise and memory, whose correlation and response
kernels relate to those of the original diffusion process through a system of
convolution equations depending on $\X$ via the free cumulants of its eigenvalue
distribution.

Specializing to the overdamped Langevin diffusion $\{\btheta^t\}_{t \geq 0}$
for a Gibbs measure
\[\mu(\btheta) \propto \exp\Big(\frac{1}{2}\btheta^\top \X\btheta\Big)
\prod_{i=1}^n \nu(\d\theta_i),\]
we analyze this mean-field limit under an assumption that the dynamics are rapidly
mixing. In this regime, the correlation/response kernels admit
time-translation-invariant approximants which satisfy a
fluctuation-dissipation relation. The generalized Langevin equation admits a
Markovian approximation with coupling to an auxiliary multivariate OU-process,
and converges in time to a replica-symmetric prediction for the
empirical law of coordinates under $\mu$. The correlation structure of
the auxiliary process may be understood through the infinitesimal generator of a
Markov semigroup for the lifted process of path histories
$\bxi^t=\{\btheta^s\}_{s:s \leq t}$.

Consequently, we show that the free energy of the Gibbs measure
converges to a replica-symmetric limit under an explicit
high-temperature condition, which for an Ising model is
\[\|\X\|_\op<1/2.\]
By recent dynamical universality results, this implies that the same characterization of the free energy holds in fully deterministic models without random disorder, as long as
$\X$ satisfies a set of deterministic delocalization conditions.
\end{abstract}

\maketitle

\tableofcontents

\section{Introduction}

Let $\X \in \R^{n \times n}$ be a (deterministic or random) symmetric matrix. 
Consider a diffusion process or gradient flow in $\R^n$ interacting through $\X$,
\begin{equation}\label{eq:dynamicsintro}
\d\btheta^t=[f(\btheta^t)+\X\btheta^t]\d t+\sqrt{2\gamma}\,\d\b^t
\end{equation}
where $f$ is a scalar function applied coordinate-wise. Under suitable mean-field
assumptions for $\X$ to be specified, the first goal of our work is to obtain
an asymptotic characterization of the empirical distribution of coordinates of
\eqref{eq:dynamicsintro} in the limit $n \to \infty$. Applying this
characterization to study Langevin dynamics for the Gibbs measure
\begin{equation}\label{eq:gibbsmeasureintro}
\mu(\d\btheta)=\frac{1}{\cZ}\exp\Big(\frac{1}{2}\btheta^\top \X\btheta\Big)
\prod_{i=1}^n \nu(\d\theta_i),
\end{equation}
the second goal of our work is to obtain an asymptotic characterization of the
free energy and replica overlaps of \eqref{eq:gibbsmeasureintro}
in a regime where these dynamics are rapidly mixing.

For models where $\X$ has
independent Gaussian entries up to transposition symmetry, such
questions are richly studied in the physics and mathematics literatures on spin
glasses and disordered systems. A dynamical mean-field limit for a prototypical
example of \eqref{eq:dynamicsintro} was first derived by Sompolinsky and
Zippelius in \cite{sompolinsky1982relaxational},
and is given by the law of a scalar generalized
Langevin equation
\begin{equation}\label{eq:sompolinskydmft}
\d\theta^t=\left[f(\theta^t)+\int_0^t R_\theta(t,s)\theta^s \d s+g^t\right]\d t
+\sqrt{2\gamma}\,\d b^t,
\qquad \{g^t\} \sim \GP(0,C_\theta),
\end{equation}
with correlated Gaussian noise $\{g^t\}$ and self-consistently defined
correlation and response kernels $(C_\theta,R_\theta)$.
Convergence in empirical law of the sample paths of \eqref{eq:dynamicsintro} to
\eqref{eq:sompolinskydmft}, and an associated annealed large deviations
principle, was first proved mathematically by Ben Arous and Guionnet
in \cite{arous1997symmetric,guionnet1997averaged}. The Ising case of the Gibbs measure \eqref{eq:gibbsmeasureintro}
with $\nu(\d\theta)$ supported on $\{\pm 1\}$ corresponds to the classical
Sherrington-Kirkpatrick (S-K) model \cite{sherrington1975solvable}, for which the form of the asymptotic
free energy was famously conjectured by Parisi in \cite{parisi1979infinite}
and proved by
Guerra and Talagrand \cite{guerra2003broken,talagrand2006parisi}; we refer to
\cite{talagrand2010mean,talagrand2011mean,panchenko2013sherrington} for textbook treatments
of this model and its rich history. Universality over classes of random 
matrices $\X$ having independent non-Gaussian entries (up to symmetry)
has been shown for the free energy of the S-K model in \cite{carmona2006universality} and for
dynamics similar to \eqref{eq:dynamicsintro} in
\cite{dembo2021diffusions,dembo2021universality}.

Arguably, this universality class of disorder matrices with independent entries remains quite limited, raising a question of whether analogous results hold under more general mean-field structure. In this work, we aim to extend some of these developments to a class of disorder
matrices $\X \in \R^{n \times n}$ that are orthogonally-invariant in law ---
i.e.\ whose matrix of eigenvectors is uniformly distributed on the orthogonal
group --- and to an associated universality class that encompasses examples of
deterministic matrices $\X$. Our contributions are threefold:
\begin{enumerate}[1.]
\item For the dynamics \eqref{eq:dynamicsintro} with 
orthogonally-invariant $\X \in \R^{n \times n}$ having a compactly supported
asymptotic spectral law $\Lambda$, we develop an extension of the
mean-field limit \eqref{eq:sompolinskydmft}.

This extension has the same structural form as \eqref{eq:sompolinskydmft},
\begin{equation}\label{eq:dmftintro}
\d\theta^t=\left[f(\theta^t)+\int_0^t R_g(t,s)\theta^s \d s+g^t\right]\d t
+\sqrt{2\gamma}\,\d b^t,
\qquad \{g^t\} \sim \GP(0,C_g),
\end{equation}
with correlation and response kernels $(C_g,R_g)$ now defined from the kernels $(C_\theta,R_\theta)$ which characterize the correlation and response of $\{\btheta^t\}_{t \geq 0}$ via convolution equations and the free cumulants of
$\Lambda$ (cf.\ Definition \ref{def:DMFT}). We establish almost-sure
convergence of the empirical law of sample paths of \eqref{eq:dynamicsintro} to this limit, when $f$  has bounded derivatives up to $3^\text{rd}$ order.  \\

\item We carry out an analysis of this mean-field limit \eqref{eq:dmftintro}
for the overdamped Langevin dynamics of the Gibbs measure
\eqref{eq:gibbsmeasureintro} when $\nu(\d\theta)=e^{-U(\theta)}\d\theta$
has a smooth potential,
under an assumption that these dynamics converge at an
exponential dimension-free rate from both an equilibrium and an independent
random initialization (cf.\ Assumption \ref{assump:convergence}).

In this dynamical regime, we prove that the preceding kernels $(C_g,R_g)$ have
time-translation-invariant limits $(c_g,r_g)$ as $t \to \infty$,
which satisfy a fluctuation-dissipation relation
\begin{equation}\label{eq:rgfdtintro}
r_g(t)={-}c_g'(t).
\end{equation}
In contrast to the mean-field limit of Sompolinsky-Zippelius in
\eqref{eq:sompolinskydmft}, the correlation kernel $c_g(\cdot)$ may not have a
mixture-of-exponentials form. We instead establish a sequence of approximations
$c_g(t) \approx u_m^\top e^{-tA_m}u_m$ with non-diagonal, asymmetric matrices
$A_m \in \R^{m \times m}$, leading to a sequence of Markovian
approximations of \eqref{eq:dmftintro},
\begin{equation}\label{eq:markovapproxintro}
\begin{aligned}
\d\theta^t&=[f(\theta^t)+u_m^\top x^t+z]\d t+\sqrt{2}\,\d b^t\\
\d x^t&={-}A_m[x^t-u_m\theta^t]\d t+\Sigma_m\,\d b_x^t.
\end{aligned}
\end{equation} 
Via these approximations, we show that
$\{\theta^t\}_{t \geq 0}$ in \eqref{eq:dmftintro} converges to an equilibrium
law that coincides with a replica-symmetric prediction for the empirical distribution of coordinates under the Gibbs measure
\eqref{eq:gibbsmeasureintro}.\\

\item A sufficient condition ensuring Assumption \ref{assump:convergence}
is a log-Sobolev inequality for the Gibbs measure, which
may be verified under an explicit high-temperature condition
\cite{bauerschmidt2019very}. From this, we prove replica-symmetric
limits for the overlaps and free energy of \eqref{eq:gibbsmeasureintro}.

Concretely, for the Ising model with $\nu(\d\theta) \propto e^h
\delta_{\theta=1}+e^{-h}\delta_{\theta=-1}$ where $h \in \R$ is an external field,
we establish this result under the high-temperature assumption  
\begin{equation}\label{eq:hightempintro}
\|\X\|_\op<1/2.
\end{equation}
(In the setting of the classical S-K model where $\X=\beta\J$ and $\J \sim \GOE(n)$, this corresponds to the assumption $\beta<1/4$, which is a subset of the full replica-symmetric regime \cite{lopatto2026replica}.)
Leveraging the dynamical universality results of
\cite{dudeja2023universality,wang2024universality}, we deduce
that the same conclusion holds when $h=0$ for a general class of
purely deterministic matrices $\X \in \R^{n \times n}$ satisfying delocalization
conditions (cf.\ Assumption \ref{assump:X-universal}), and for arbitrary $h \in
\R$ when $\X$ is a ``semi-random'' matrix conjugated by a diagonal matrix of
random $\{\pm 1\}$ signs.
\end{enumerate}
These results are presented in Section \ref{sec:results}.

\subsection{Existing literature} \emph{Dynamics.}
The form of the mean-field limit \eqref{eq:dmftintro} for the
diffusion equation \eqref{eq:dynamicsintro} with orthogonally-invariant $\X$
may be anticipated from the analysis of Opper,
\c{C}akmak, and Winther \cite{opper2016theory}, who computed the dynamical mean-field
equations for an analogous class of discrete-time dynamics.
The mean-field limit \eqref{eq:dmftintro} and convolutional structure of its
correlation and response kernels may be understood as a passage of this
result to continuous time.

A particular class of discrete-time dynamics --- analogous
to Approximate Message Passing (AMP) algorithms \cite{kabashima2003cdma,donoho2009message} and TAP iterations
\cite{bolthausen2014iterative} in models with Gaussian disorder $\X$ --- was highlighted in
\cite{opper2016theory} as having purely Gaussian limiting dynamics. For this AMP class,
convergence to the mean-field limit was proved mathematically in
\cite{fan2022approximate},
extending an approach of iterative conditioning of Haar-orthogonal matrices
previously developed in \cite{ma2017orthogonal,takeuchi2019rigorous,rangan2019vector} to analyze a simpler family of divergence-free
``vector-AMP'' algorithms. Alternative proofs and extensions of
the result of \cite{fan2022approximate} have since been obtained in \cite{liu2024unifying,gorini2026universality}, and universality
over semi-random and deterministic disorder matrices $\X$ has been shown
in increasingly general contexts in
\cite{dudeja2023universality,wang2024universality,gorini2026universality}.

It was shown by Celentano, Cheng, and Montanari \cite{celentano2021high}, in an analogous
statistical model with i.i.d.\ Gaussian disorder, that the
mean-field limit of  general discrete-time iterative algorithms may be
deduced from that of AMP, and also that the mean-field limit of
continuous-time gradient flow dynamics may be proved using an approach different
from \cite{arous1997symmetric,guionnet1997averaged}. This approach relies on a reduction to AMP dynamics via
time-discretization, and a passage to the continuous-time limit through
verifying contractivity of the dynamical
fixed-point equations in a suitably defined Banach metric.
The approach has been extended to derive mean-field limits for
broad families of discrete-time iterative algorithms in
\cite{gerbelot2024rigorous,han2025entrywise,celentano2025state,dandi2025sequential}
and for continuous-time Langevin diffusions and stochastic gradient dynamics 
in \cite{fan2025dynamicalI,fan2026high,nishiyama2026high}. We apply a similar proof strategy in our work, to pass the
characterization of AMP in \cite{fan2022approximate} to continuous time.

Mathematical analysis of aging in the dynamical mean-field limit of a
spherical 2-spin model was carried out in \cite{arous2001aging},
of convergence in time of the Cugliandolo-Kurchan/Crisanti-Horner-Sommers
equations of the spherical p-spin model to an FDT equilibrium in \cite{dembo2007limiting},
and of convergence in time of these equations under disorder-dependent initial
conditions in \cite{dembo2019dynamics,dembo2025dynamics}. For non-spherical models, aging dynamics and their connection to the static replica-symmetry-breaking structure of the Gibbs mesaures was analyzed using (non-rigorous) statistical physics techniques in \cite{altieri2020dynamical}. A mathematical analysis for a non-spherical model with Gaussian disorder
was carried out in \cite{fan2025dynamicalII}
in a dynamical rapid-mixing regime similar to our work; we build upon several ideas
of \cite{fan2025dynamicalII}, discussed in Section~\ref{sec:proof}.\\

\emph{Free energy.} The Ising case of the model \eqref{eq:gibbsmeasureintro}
was introduced by Marinari, Parisi, and Ritort \cite{marinari1994replica} as a model which may
exhibit similar mean-field phenomena as models with deterministic disorder.
Replica-symmetric and 1-RSB predictions for the asymptotic
free energy of this model were computed in
\cite{marinari1994replica,cherrier2003role}.
TAP equations for the magnetization were derived in \cite{parisi1995mean} using a
high-temperature expansion and in \cite{opper2001adaptive} using an alternative
cavity method. \cite{maillard2019high} extended the former analysis to a broad class of
statistical models with orthogonally-invariant random disorder, and discussed its
connection to dynamical algorithms.

Mathematical proofs of such predictions remain scarce in the
literature, even in high temperature. For a
high-temperature regime of the Ising model without external field, where
quenched and annealed free energies coincide, a replica-symmetric limit for the free
energy was proved in \cite{bhattacharya2016high} using the second-moment method. This was extended to an Ising
model with external field $h \in \R$ in \cite{fan2024replica} via a conditional
second-moment method analogous to that of \cite{bolthausen2018morita} for the S-K model and
\cite{ding2019capacity} for the Ising perceptron, conditioning on the filtration of a
TAP iteration for computing the magnetization. This result of
\cite{fan2024replica}
showed, for $\X=\beta\J$ where the asymptotic spectral law of $\J$ converges to
$\bar \Lambda$, that the free energy converges to its replica-symmetric prediction
under a high-temperature condition of the form
$\beta \in (0,\beta_0)$ for some
\[\beta_0 \equiv \beta_0(\Law(\bar\Lambda)).\]
As the analysis of the second-moment variational problem in \cite{fan2024replica} was
perturbative around $\beta=0$, it is difficult to glean a more explicit bound for
$\beta_0$, and we believe that this bound would be substantially smaller than
our explicit guarantee in \eqref{eq:hightempintro}. In a related statistical generalized linear model with orthogonally-invariant design, replica-symmetric predictions for the free energy were derived in \cite{kabashima2008inference,takahashi2022macroscopic}. This was proved mathematically for a linear model with design factorizable through Gaussian randomness in \cite{barbier2018mutual}, via a Gaussian interpolation approach. 
The replica-symmetric ground state energy for convex potentials was also proved in
\cite{gerbelot2022asymptotic}
via analysis of a vector AMP algorithm for convex optimization.

One perspective on the challenge of using a dynamical approach to prove
properties of the Gibbs measure is in certifying the optimality of the dynamics.
For example, in the
context of \eqref{eq:gibbsmeasureintro} and related statistical models,
AMP algorithms/TAP iterations that conjecturally converge to the magnetization
are known \cite{opper2016theory,rangan2019vector,ccakmak2019memory}, but this convergence has thus far only been
certified in high-temperature regimes analogous to that of
\cite{fan2024replica}, relying on a-priori understanding of the asymptotic free
energy and structure of the Gibbs measure \cite{fan2022tap,li2023random}.
Likewise, while mean-field characterizations
of the ground state energy in convex landscapes may be established
by certifying optimality of optimization algorithms through a
first-order condition \cite{bayati2011lasso,donoho2016high,gerbelot2022asymptotic}, certifying global optimality of
such algorithms in non-convex landscapes remains challenging. In this work, we
take an approach of analyzing the standard overdamped Langevin dynamics, which
has a more complex dynamical mean-field limit than AMP-type algorithms, but
whose convergence to the Gibbs measure may be certified through
entropy-decay arguments external to the mean-field theory.

As a benefit of our dynamical approach to analyzing the free energy, our results inherit the aforementioned universality over semi-random and deterministic matrices $\X$ \cite{dudeja2023universality,wang2024universality,gorini2026universality}, which has been shown for algorithm dynamics but not yet for static properties of the Gibbs measure. Our results confirm that disordered-systems characterizations of the Gibbs measure can hold for systems without random disorder, as originally anticipated by \cite{marinari1994replica}.

\subsection{Proof ideas}\label{sec:proof}
We describe briefly the main ideas of our analyses.

As discussed in the preceding section, \eqref{eq:dmftintro} is obtained by
passing a characterization of a discrete-time AMP algorithm in
\cite{opper2016theory,fan2022approximate} to
continuous time. An important distinction between these discrete-time
and continuous-time characterizations is that the correlation/response kernels
of the former are defined iteratively/inductively over time. The first contribution of our work is to formalize a sense in which the continuous-time kernels
$C_\theta,R_\theta,C_g,R_g$ may instead be understood as the fixed points of a
contractive map.

An obstacle to this formalization is that the convolutional
equations mapping
\begin{equation}\label{eq:convolution}
(C_\theta,R_\theta) \mapsto (C_g,R_g)
\end{equation}
do not, in general, ensure that $C_g$ remains a positive-semidefinite (p.s.d.)
kernel, which is needed to define the process \eqref{eq:dmftintro}; thus, it
is a bit unclear on what domain such a contractive map should be defined.
A related issue is that even when $C_g$ is p.s.d., it seems challenging to
analyze contractivity of \eqref{eq:convolution} under a Wasserstein-coupling
metric for $C_\theta,C_g$ as used in the arguments of
\cite{celentano2021high,fan2025dynamicalI,fan2026high}. We resolve these issues by
formalizing the fixed-point interpretation of the dynamical mean-field limit
via a sequence of mappings
\[(C_\theta,R_\theta) \mapsto (C_g,R_g) \mapsto \cP(C_g,R_g) \mapsto (C_\theta,R_\theta),\]
where $\cP$ projects $C_g$ under an alternative weighted $L^2$ metric onto a subset of the p.s.d.\ cone. We then
develop different arguments to
establish contractivity in this metric, and leverage results of \cite{fan2022approximate} to show that at the fixed point of these mappings, $C_g$ is in fact p.s.d., and thus is also a fixed point without projection by $\cP$.

Turning to the analysis of the overdamped
Langevin equation, from the
assumed convergence condition in Assumption \ref{assump:convergence} and an
identification of $(C_\theta,R_\theta)$ as the limits of
coordinate-wise correlation and response functions of
$\{\btheta^t\}_{t \geq 0}$, one may deduce the existence of
time-translation-invariant limits $c_\theta,r_\theta$ for these kernels at
large times, together with their fluctuation-dissipation relation
\begin{equation}\label{eq:crthetaFDTintro}
r_\theta(t)={-}c_\theta'(t)
\end{equation}
(Lemma \ref{lemma:tti}). 
These arguments are similar to \cite{fan2025dynamicalII} which analyzed a statistical model with Gaussian disorder.
Interestingly, we are able to verify that \eqref{eq:crthetaFDTintro}
implies also the fluctuation-dissipation relation
$r_g(t)={-}c_g'(t)$ corresponding to the kernels
$C_g,R_g$ appearing in the generalized Langevin
equation \eqref{eq:dmftintro}. This verification is a short
calculation in the Fourier domain (Lemma \ref{lemma:ttig}).

We show that $c_g(\cdot)$ may be interpreted as
the large-$n$ limit of the correlation function of coordinates of
\[\g^t=\X \btheta^t - \int_{-\infty}^t r_g(t-s)\btheta^s\d s\]
at equilibrium, which are observables of the full path history
$\bxi^t:=\{\btheta^s\}_{s:s \leq t}$.
As this ``lifted'' process $\{\bxi^t\}_{t \in \R}$ remains Markovian over this
space of path histories, writing $\nu$ for its stationary measure
and $\{P_t\}_{t \geq 0}$ for its
$L^2(\nu)$-semigroup, this implies a representation
\[c_g(t)=\lim_{n \to \infty} \frac{1}{n}\,\E_\nu[f^\top P_t f]\]
where $f$ is the function for which $\g^t=f(\bxi^t)$.
An important distinction with the setting of Gaussian disorder is that here
$\{P_t\}_{t \geq 0}$ is irreversible, and its infinitesimal generator is not
self-adjoint on $L^2(\nu)$. Thus, we do not obtain from this representation a
mixture-of-exponentials form for $c_g(\cdot)$, and in general, we do not expect
such a form to hold true. Instead, we establish an approximation
\[c_g(t) \approx u_m^\top e^{-tA_m}u_m\]
where ${-}A_m \in \R^{m \times m}$ is a (non-diagonal, asymmetric)
matrix that constitutes a finite-dimensional 
approximation for the generator of $\{P_t\}_{t \geq 0}$.
Together with the fluctuation-dissipation
relation \eqref{eq:rgfdtintro}, this leads to the previously stated Markov
approximation \eqref{eq:markovapproxintro} for \eqref{eq:dmftintro}.

Finally, we show convergence of \eqref{eq:dmftintro} to an equilibrium law
characterized by replica-symmetric fixed point equations, by analyzing 
convergence of the Markovian system \eqref{eq:markovapproxintro} together with
its approximation error for \eqref{eq:dmftintro}. The primary technical
challenge is that the generator ${-}A$ for the above semigroup
$\{P_t\}_{t \geq 0}$ has non-trivial (in fact, dense)
null space in $L^2(\nu)$, and thus the spectral gap of its approximant ${-}A_m$ must vanish with the
approximation error. Thus, it seems challenging to prove convergence to
equilibrium of the Markovian
system \eqref{eq:markovapproxintro} over a fixed time horizon that is
independent of this error and the dimension $m$; on the other hand,
the accumulated approximation error over a $m$-dependent time horizon may not be
small under a naive coupling and Gr\"onwall argument.
We circumvent this issue by using a careful
construction of ${-}A_m$ that instead approximates a slightly damped version of
${-}A$, and that allows for a stable coupling of
\eqref{eq:dmftintro} and \eqref{eq:markovapproxintro} over an arbitrarily long
time horizon on which \eqref{eq:markovapproxintro} is guaranteed to converge
(Corollary \ref{cor:semigroupapprox}, Lemma \ref{lemma:thetacoupling}). 
This implies the replica-symmetric characterization of the Gibbs measure
\eqref{eq:gibbsmeasureintro} under Assumption \ref{assump:convergence}, for
laws $\nu$ having a smooth potential. The concrete implication for the Ising
model is then deduced by verifying Assumption \ref{assump:convergence} using the
log-Sobolev inequality of \cite{bauerschmidt2019very} for a
sequence of double-well potentials that converge to the Ising law.

\section{Main results}\label{sec:results}

\subsection{Dynamical mean-field limit}

Let $\{\btheta^t\}_{t \geq 0}$ be a diffusion process on $\R^n$ taking the form
\begin{equation}\label{eq:dynamics}
\d\btheta^t=[f(\btheta^t)+\X\btheta^t]\d t+\sqrt{2\gamma}\,\d \b^t
\end{equation}
where $\X \in \R^{n \times n}$ is a symmetric matrix, $f:\R \to \R$ is a
Lipschitz-continuous function applied coordinate-wise, and $\{\b^t\}_{t \geq 0}$
is a standard Brownian motion on $\R^n$. The case $\gamma=0$ corresponds to gradient flow.  

Our first main result establishes a deterministic limit, as $n \to \infty$,
of the empirical distribution of coordinate sample paths
\[\{(\theta_i^t)_{t \in [0,T]}\}_{i=1}^n\]
for any fixed time horizon $T>0$ independent of $n$. For expositional clarity, let us first state
this result for the following class of random disorder matrices $\X$ that 
are orthogonally invariant in law; we discuss universality over the disorder in
Section \ref{subsec:universality}.

\begin{assumption}\label{assump:X}
$\X=\O\bLambda\O^\top$, where $\O \sim \Haar(\OO(n))$ is a Haar-uniform
random matrix on the orthogonal group, and
$\bLambda=\diag(\lambda_1,\ldots,\lambda_n)$ is a (deterministic or random)
diagonal matrix independent of $\O$. There exists a constant $\Lambda_{\max}>0$
such that almost surely for all large $n$,
\[\|\X\|_\op \leq \Lambda_{\max}.\]
Furthermore, there exists a random variable $\Lambda$ (supported on
$[-\Lambda_{\max},\Lambda_{\max}]$) with $\E[\Lambda]=0$ and $\Var[\Lambda]>0$
such that almost surely as $n \to \infty$,
\[\frac{1}{n}\sum_{i=1}^n \delta_{\lambda_i} \to \Law(\Lambda)
\text{ weakly.}\]
Throughout, we denote by
$\{\kappa_p\}_{p \geq 2}$ the free cumulants of this random variable $\Lambda$.
\end{assumption}

We impose, in addition, the following assumptions on the
dynamics \eqref{eq:dynamics}.

\begin{assumption}\label{assump:dynamics}
\begin{enumerate}[(a)]
\item $f:\R \to \R$ is fixed for all $n$, and three-times continuously-differentiable
with uniformly bounded first, second, and third derivatives.  
\item $\gamma \geq 0$ is a fixed scalar parameter, and
the Brownian motion $\{\b^t\}_{t \geq 0}$ is independent of $\X$.
\item The initial state
$\btheta^0=(\theta_1^0,\ldots,\theta_n^0)$ is independent of $\X$ and
$\{\b^t\}_{t \geq 0}$. There exists a random variable $\theta^0$ with
$\Var[\theta^0]>0$ and finite moments of all orders,
such that for each fixed $p \geq 1$, almost surely as $n \to \infty$,
\[\frac{1}{n}\sum_{i=1}^n \delta_{\theta_i^0} \to \Law(\theta^0) \text{ weakly
and in Wasserstein-$p$.}\]
\end{enumerate}
\end{assumption}

Under these assumptions, we will show that the limit as $n \to \infty$ of
the dynamics \eqref{eq:dynamics} admits the following mean-field
characterization:

\begin{definition}\label{def:DMFT}
Let $\theta^0$, $\{b^t\}_{t \geq 0}$, and $\{g^t\}_{t \geq 0}$
be mutually independent scalar variables/processes on a filtered
probability space $(\Omega,\cF,\{\cF_t\}_{t \geq 0},\P)$, where
$\Law(\theta^0)$ is given by Assumption \ref{assump:dynamics}(c),
$\{b^t\}_{t \geq 0}$ is a standard $\cF_t$-adapted Brownian motion, and
\[\{g^t\}_{t \geq 0} \sim \GP(0,C_g)\]
is a continuous, mean-zero, $\cF_t$-adapted Gaussian process with a symmetric
positive-semidefinite covariance kernel $C_g:[0,\infty) \times [0,\infty)
\to \R$. Let $\{\theta^t\}_{t \geq 0}$ be the solution of the SDE
\begin{equation}\label{eq:dmft-theta}
\d\theta^t=\left[f(\theta^t)+\int_0^t R_g(t,s)\theta^s ds+g^t\right]
\d t+\sqrt{2 \gamma}\,\d b^t, \qquad \theta^t|_{t=0}=\theta^0,
\end{equation}
where $R_g:[0,\infty) \times [0,\infty) \to \R$
is a response kernel satisfying $R_g(t,s)=0$ for $s>t$. Let
$\{\frac{\d\theta^t}{\d g^s}\}_{t \geq s \geq 0}$ be the solution of the
response equation
\begin{equation}\label{eq:dmft-response}
\frac{\d}{\d t} \frac{\d\theta^t}{\d g^s}=\left[f'(\theta^t)\frac{\d\theta^t}{\d g^s}
+\int_s^t R_g(t,r)\frac{\d \theta^r}{\d g^s}\d r\right],
\qquad \frac{\d\theta^t}{\d g^s}\bigg|_{t=s}=1.
\end{equation}
These kernels $C_g,R_g$, in conjunction with a pair of auxiliary kernels
$C_\theta,R_\theta:[0,\infty) \times [0,\infty) \to \R$,
are self-consistently defined from the laws of the above processes by the
fixed point conditions
\begin{align}
C_\theta(t,s)&=\E[\theta^t \theta^s]\label{eq:dmft-Ctheta}\\
R_\theta(t,s)&=\E\left[\frac{\d\theta^t}{\d g^s}\right]
\text{ for } s \in [0,t],
\quad R_\theta(t,s)=0 \text{ for } s>t,
\label{eq:dmft-Rtheta}\\
C_g(t,s)&=\sum_{p,q \geq 0} \kappa_{p+q+2}\,
R_\theta^{\ast p} * C_\theta * \bar R_\theta^{\ast q}(t,s),
\quad
\text{ where } \bar R_\theta(t,s)=R_\theta(s,t) \text{ for all } s,t \geq 0,\label{eq:dmft-Cg}\\
R_g(t,s)&=\sum_{p \geq 1} \kappa_{p+1}\,R_\theta^{\ast p}(t,s).\label{eq:dmft-Rg}
\end{align}
Here, $A*B$ denotes the standard convolution for kernels defined over
$[0,\infty)$,
\[[A \ast B](t,s)=\int_0^\infty A(t,r)B(r,s)\d r,
\quad A^{\ast p}=\underbrace{A \ast \ldots \ast A}_{p \text{
times}}, \quad A^{\ast 0}*B=B*A^{\ast 0}=B.\]
\end{definition}

\begin{theorem}\label{thm:dmft-approx}
Suppose Assumptions \ref{assump:X} and \ref{assump:dynamics} hold. Then:
\begin{enumerate}[(a)]
\item There exist domains $\cS_\theta^+$ for $(C_\theta,R_\theta)$ and 
$\cS_g^+$ for $(C_g,R_g)$ for which the kernels and processes of
Definition \ref{def:DMFT} are well-defined and unique.
That is, (\ref{eq:dmft-theta}--\ref{eq:dmft-response}) have unique
$\cF_t$-adapted solutions for any $(C_g,R_g) \in \cS_g^+$,
the series (\ref{eq:dmft-Cg}--\ref{eq:dmft-Rg}) are absolutely convergent for
any $(C_\theta,R_\theta) \in \cS_\theta^+$, and there exists a
unique fixed point $(C_\theta,R_\theta,C_g,R_g) \in \cS_\theta^+
\times \cS_g^+$ satisfying (\ref{eq:dmft-Ctheta}--\ref{eq:dmft-Rg}).
\item Fix any $T>0$ not depending on $n$. Define
\[\g^t=\X\btheta^t-\int_0^t R_g(t,s)\btheta^s\d s.\]
Let $\frac{1}{n}\sum_{i=1}^n \delta_{\{(\theta_i^t,g_i^t,b_i^t)\}_{t \in
[0,T]}}$ be the empirical measure of sample paths
$(\theta_i^t,g_i^t,b_i^t)_{t \in [0,T]}$ and let
$\Law(\{(\theta^t,g^t,b^t)\}_{t \in [0,T]})$ be the joint law of these 
processes in Definition \ref{def:DMFT}, as probability measures on
$C([0,T],\R^3)$. Then almost surely
as $n \to \infty$,
\[\frac{1}{n}\sum_{i=1}^n \delta_{\{(\theta_i^t,g_i^t,b_i^t)\}_{t \in [0,T]}}
\to \Law(\{(\theta^t,g^t,b^t)\}_{t \in [0,T]}) \text{ weakly and in
Wasserstein-2.}\]
\item Write $\langle \cdot \rangle$ for the expectation over $\{\b^t\}_{t \geq
0}$, conditional on $\X$ and $\btheta^0$. Then for any fixed $s,t \geq 0$
with $s \leq t$, almost surely as $n \to \infty$,
\[\frac{1}{n}\sum_{i=1}^n \langle \theta_i^t\theta_i^s \rangle
\to C_\theta(t,s), \,
\frac{1}{n}\sum_{i=1}^n \langle g_i^t g_i^s \rangle
\to C_g(t,s), \,
\frac{1}{n}\sum_{i=1}^n \langle \theta_i^tb_i^s
\rangle \to \sqrt{2\gamma}\int_0^s R_\theta(t,\tau) \d \tau.\]
\end{enumerate}
\end{theorem}

An explicit definition of the domains $\cS_\theta^+,\cS_g^+$ may be found in Appendix \ref{sec:dmft-approx} of the Supplementary Appendices. We remark that the third claim of Theorem \ref{thm:dmft-approx}(c) will lead to an interpretation of $R_\theta$ as an averaged coordinate-wise response function for the dynamics $\{\btheta^t\}_{t \geq 0}$ (cf.\ \eqref{eq:response} and \eqref{eq:Rthetarepr} in Section \ref{sec:corrresponse}).

\subsection{Langevin dynamics and replica-symmetric free energy}

Now let $U:\R \to \R$ be a scalar confining potential, and consider the Gibbs
measure having Lebesgue density on $\R^n$ given by
\begin{equation}\label{eq:gibbsmeasure}
\mu_\Gibbs(\btheta)=\frac{1}{\cZ}\exp\left(\frac{1}{2}\btheta^\top \X\btheta
-\sum_{i=1}^n U(\theta_i)\right)
\end{equation}
where $\cZ$ is the partition function
\begin{equation}\label{eq:partition}
\cZ=\int_{\R^n} \exp\left(\frac{1}{2}\btheta^\top \X\btheta
-\sum_{i=1}^n U(\theta_i)\right)\d\btheta.
\end{equation}
We will assume
\begin{equation}\label{eq:Uconvexity}
\liminf_{|x| \to \infty} U''(x) \geq \alpha
\end{equation}
for a sufficiently large constant $\alpha>0$ so that $\cZ<\infty$.
%We will assume also (without loss of generality, by absorbing an additive
%perturbation $c\,\Id$ of $\X$ into a quadratic term of $U(\cdot)$), that
%\[\kappa_1=\E[\Lambda]=0.\]
We write
\[\langle f(\btheta) \rangle=\E_{\btheta \sim \mu_\Gibbs}
[f(\btheta) \mid \X], \qquad
\langle f(\btheta,\btheta') \rangle=\E_{\btheta,\btheta' \overset{iid}{\sim}
\mu_\Gibbs}[f(\btheta,\btheta') \mid \X]\]
for the average over independent replicas from this Gibbs measure.

We will analyze the preceding dynamical mean-field limit for the overdamped
Langevin diffusion
\begin{equation}\label{eq:langevin}
\d\btheta^t=\big[\X\btheta^t-U'(\btheta^t)\big]\d t+\sqrt{2}\,\d\b^t
\end{equation}
having $\mu_\Gibbs$ as stationary law. We carry out this analysis under
a condition that these dynamics
\eqref{eq:langevin} exhibit dimension-free convergence in the following sense,
both from the given initial state $\btheta^0$ (independent of $\X$) satisfying
Assumption \ref{assump:dynamics} as well as from an equilibrium initialization.

\begin{assumption}\label{assump:convergence}
Let $\{\btheta^t\}_{t \geq 0}$ be the solution of \eqref{eq:langevin}
with the given initial state $\btheta^0$. Let $\{\tilde\btheta^t\}_{t \in \R}$
denote a stationary process also satisfying \eqref{eq:langevin} with
equilibrium initial state
\[\tilde\btheta^0 \sim \mu_\Gibbs.\]
Let
$\Law(\btheta^t \mid \btheta^s)$ denote the law of $\btheta^t$
conditional on $(\X,\btheta^s)$, and define similarly $\Law(\tilde \btheta^t \mid \tilde \btheta^s)$.
Then for some constants $C,c>0$, almost surely (over $(\X,\btheta^0)$)
for all large $n$,
\[\begin{aligned}
\E\left[\frac{1}{n}W_2(\Law(\btheta^t \mid \btheta^s),\,\mu_\Gibbs)^2 \;\bigg|\; \X,\btheta^0\right] &\leq
Ce^{-c(t-s)}
 \text{ for all } t \geq s \geq 0,\\
\E\left[\frac{1}{n}W_2(\Law(\tilde \btheta^\tau \mid \tilde \btheta^0),\,\mu_\Gibbs)^2 \;\bigg|\; \X\right]
&\leq Ce^{-c\tau}
\text{ for all } \tau \geq 0.
\end{aligned}\]
\end{assumption}

Under this condition, assuming that the potential $U(\cdot)$ is
suitably confining, our second main result is the following theorem establishing
deterministic limits for the free energy and
replica overlaps of the above Gibbs measure.

\begin{theorem}\label{thm:replica}
Suppose there exists a (deterministic or random) initial state $\btheta^0$ for
which Assumptions \ref{assump:X} and \ref{assump:dynamics} hold with $f=-U'$.
Suppose that the bound $\alpha>0$ in
\eqref{eq:Uconvexity} satisfies
\begin{equation}\label{eq:alphacondition}
\alpha>\Lambda_{\max},
\qquad \alpha>\sum_{p \geq 1}|\kappa_{p+1}|
\left(\frac{1}{n}\langle \|\btheta\|_2^2 \rangle
-\frac{1}{n}\langle \btheta^\top \btheta' \rangle+\eps\right)^p
\end{equation}
almost surely for all large $n$ and some constant $\eps>0$.
Suppose that Assumption \ref{assump:convergence} holds for this initial state
$\btheta^0$.

For any $v \geq q \geq 0$ such that the R-transform series
\begin{equation}\label{eq:Rtransform}
\cR_\Lambda(x)=\sum_{p \geq 1} \kappa_{p+1}x^p
\end{equation}
is absolutely convergent in a neighborhood of $v-q$ and $q\cR_\Lambda'(v-q) \geq 0$, define
the probability density on $\R$
\begin{equation}\label{eq:replicalaw}
\mu(\theta \mid v,q,\xi)
\propto
\exp\left(-U(\theta)+\frac{1}{2}\cR_\Lambda(v-q)\theta^2+\sqrt{q\cR_\Lambda'(v-q)}\,\xi\theta\right).
\end{equation}
Then:
\begin{enumerate}[(a)]
\item There exist deterministic almost-sure limits
\begin{equation}\label{eq:vstardef}
v_*=\lim_{n \to \infty}\frac{1}{n}\langle \|\btheta\|_2^2 \rangle,
\qquad q_*=\lim_{n \to \infty}\frac{1}{n}\langle \btheta^\top \btheta' \rangle
\end{equation}
satisfying the above conditions.
These solve the fixed-point equations
\begin{equation}\label{eq:replicafixedpoint}
v_*=\E_{\xi \sim \cN(0,1)}
\langle \theta^2 \rangle_{v_*,q_*,\xi},
\qquad q_*=\E_{\xi \sim \cN(0,1)} \langle \theta \rangle_{v_*,q_*,\xi}^2
\end{equation}
where $\langle \cdot \rangle_{v,q,\xi}$ denotes the expectation under
$\mu(\cdot \mid v,q,\xi)$. Furthermore, almost surely
\begin{equation}\label{eq:thetaXthetalimit}
\lim_{n \to \infty} \frac{1}{n}\langle \btheta^\top \X\btheta\rangle
=v_*\cR_\Lambda(v_*-q_*)+q_*(v_*-q_*)\cR_\Lambda'(v_*-q_*).
\end{equation}
For any function $f:\R \to \R$ satisfying
$|f(x)-f(y)| \leq C(1+|x|+|y|)|x-y|$ for some $C>0$, almost surely
\begin{equation}\label{eq:pseudoliptest}
\lim_{n \to \infty} \frac{1}{n}\sum_{i=1}^n \langle f(\theta_i) \rangle
=\E_{\xi \sim \cN(0,1)} \langle f(\theta) \rangle_{v_*,q_*,\xi}.
\end{equation}
\item Suppose, in addition, that almost surely over $(\X,\btheta^0)$ for all
large $n$, for every $\beta \in [0,1]$,
the condition \eqref{eq:alphacondition} and Assumption
\ref{assump:convergence} hold for the Gibbs measure, dynamics, and free
cumulants associated to $\beta \X$, with some constants $C,c,\eps>0$
that are uniform over $\beta \in [0,1]$.

Let $\cZ$ be
the partition function \eqref{eq:partition} (for $\beta=1$). Then almost surely
\begin{align}
&\lim_{n \to \infty} \frac{1}{n}\log \cZ\notag\\
&=\E_{\xi \sim \cN(0,1)} \log \int \exp\Big({-}U(\theta)
+\frac{1}{2}\cR_\Lambda(v_*-q_*)\theta^2+\sqrt{q_*\cR_\Lambda'(v_*-q_*)}\,\xi\theta\Big)\notag\\
&+\frac{1}{2}\int_0^{v_*-q_*} \cR_\Lambda(s)\d s
-\frac{1}{2}(v_*-q_*)\cR_\Lambda(v_*-q_*)-\frac{1}{2}q_*(v_*-q_*)\cR_\Lambda'(v_*-q_*).\label{eq:freeenergylimit}
\end{align}
\end{enumerate}
\end{theorem}

It may be checked that the free energy limit \eqref{eq:freeenergylimit} coincides with the formula predicted by the replica method under an ansatz of replica symmetry, and that \eqref{eq:replicafixedpoint} are the fixed-point equations arising from this replica calculation. We note that the first condition of \eqref{eq:alphacondition} ensures that the Gibbs measure \eqref{eq:gibbsmeasure} is well-defined, and the second condition of \eqref{eq:alphacondition} ensures that $U''(x) \geq \cR_\Lambda(v_*-q_*)$ for all large $|x|$, hence the scalar law \eqref{eq:replicalaw} is well-defined at $(v_*,q_*)$.

All conditions of Theorem \ref{thm:replica} may be checked in a
``high-temperature'' regime of sufficiently small $\Lambda_{\max}$ depending on
$U(\cdot)$, where Assumption \ref{assump:convergence} (in a stronger form,
uniformly over all deterministic initial states $\btheta^0 \in \R^n$) is
implied by a log-Sobolev inequality for $\mu_\Gibbs$
\cite{bauerschmidt2019very}. The following corollary gives a concrete
statement of such a result with explicit bounds for $\Lambda_{\max}$,
restricting to a class of potentials $U(x)$ where $e^{-U(x)}$
is a convolution of a Gaussian law with a compactly supported probability
measure.

\begin{corollary}\label{cor:hightemp}
Suppose $U:\R \to \R$ is any potential taking the form
\[e^{-U(x)}=\int_{-M}^M \sqrt{\frac{\alpha}{2\pi}}
e^{-\frac{\alpha}{2}(x-y)^2} \d\nu(y)\]
for some $\alpha,M>0$ and probability measure $\nu$ supported
on $[-M,M]$. Suppose $\X$ satisfies Assumption \ref{assump:X} where,
%$\kappa_1=\E[\Lambda]=0$ and,
for some $\eps>0$,
\begin{equation}\label{eq:sufficientalphacondition}
\begin{gathered}
\Lambda_{\max}<\alpha, \qquad
2\Lambda_{\max}\left(M^2+\frac{1}{\alpha+\Lambda_{\max}}\right)<1,\\
\sum_{p \geq 1} |\kappa_{p+1}|
\left[\frac{1}{\alpha-\Lambda_{\max}}
+\frac{\alpha^2M^2}{(\alpha-\Lambda_{\max})^2}+\eps\right]^p<\alpha.
\end{gathered}
\end{equation}
Then all conclusions of Theorem \ref{thm:replica} hold.
\end{corollary}

Specializing to $\nu$ supported on $\{\pm 1\}$, applying a growth estimate
$\limsup_{p \to \infty} |\kappa_p|^{1/p} \leq 2\Lambda_{\max}$ for the free
cumulants (Proposition \ref{prop:freecumulants}),
and taking the limit $\alpha \to \infty$,
we obtain the following consequence for the classical Ising model
of \cite{marinari1994replica} with an external field.

\begin{corollary}\label{cor:ising}
Fix a constant $h \in \R$, and consider the Gibbs measure
\begin{equation}\label{eq:ising}
\mu_\Ising=\frac{1}{\cZ}\exp\left(\frac{1}{2}\btheta^\top\X\btheta
+h\sum_{i=1}^n \theta_i\right) \text{ for } \btheta \in \{\pm 1\}^n,
\end{equation}
with partition function
\[\cZ_\Ising=\sum_{\btheta \in \{\pm 1\}^n}
\exp\left(\frac{1}{2}\btheta^\top\X\btheta
+h\sum_{i=1}^n \theta_i\right)\]
and associated Gibbs average $\langle \cdot \rangle_{\Ising}$.
Suppose $\X$ satisfies Assumption \ref{assump:X} with 
%$\kappa_1=\E[\Lambda]=0$ and
$\Lambda_{\max}<1/2$. Then there exists a deterministic almost sure limit
\[q_*=\lim_{n \to \infty} \frac{1}{n}\langle \btheta^\top \btheta'
\rangle_\Ising,\]
which satisfies the fixed point equation
\[q_*=\E_{\xi \sim \cN(0,1)}
\tanh\left(h+\sqrt{q_*\cR_\Lambda'(1-q_*)}\,\xi\right)^2.\]
Furthermore,
\begin{align}
&\lim_{n \to \infty} \frac{1}{n}\log \cZ_\Ising
=\E_{\xi \sim \cN(0,1)} \log 2\cosh\Big(
h+\sqrt{q_*\cR_\Lambda'(1-q_*)}\,\xi\Big)\notag\\
&+\frac{1}{2}\int_0^{1-q_*} \cR_\Lambda(s)\d s
+\frac{1}{2}q_*\cR_\Lambda(1-q_*)-\frac{1}{2}q_*(1-q_*)\cR_\Lambda'(1-q_*).\label{eq:freeenergylimitising}
\end{align}
\end{corollary}

This limit \eqref{eq:freeenergylimitising} coincides with the replica formula computed in \cite{opper2001adaptive,cherrier2003role}, and proved previously under a less explicit high temperature assumption in \cite{fan2024replica}.

\subsection{Universality over the disorder}\label{subsec:universality}

Orthogonal invariance of the disorder matrix $\X$ in Assumption
\ref{assump:X} underlies a mean-field characterization of a discrete-time
AMP algorithm in
\cite{fan2022approximate}, from which the dynamical mean-field limit in
Theorem \ref{thm:dmft-approx} and results for the Gibbs measure in Theorem
\ref{thm:replica} and Corollaries \ref{cor:hightemp} and \ref{cor:ising}
are derived.

It has been shown, in increasing generality in
\cite{dudeja2023universality,wang2024universality,gorini2026universality}, that 
this characterization holds more broadly over universality classes of
(possibly deterministic) matrices $\X \in \R^{n \times n}$.
This implies, as a consequence, that our preceding results carry over to such
universality classes.

We provide here a concrete statement of universality for one such class of
deterministic disorder matrices $\X$ studied in
\cite{dudeja2023universality,wang2024universality}.

\begin{assumption}\label{assump:X-universal}
Let $\X \in \R^{n \times n}$ be a deterministic symmetric matrix such that
$\|\X\|_\op \leq \Lambda_{\max}$
for a constant $\Lambda_{\max}>0$, and
there exists a random variable $\Lambda$ with $\E[\Lambda]=0$ and $\Var[\Lambda]>0$
such that as $n \to \infty$, the empirical eigenvalue distribution of $\X$
satisfies
\[\frac{1}{n}\sum_{i=1}^n \delta_{\lambda_i(\X)}\to \Law(\Lambda)
\text{ weakly}.\]
Furthermore, for each fixed integer $k \geq 1$, any
constant $\eps>0$, and all large $n \geq n_0(\eps,k)$,
\[\max_{i=1}^n \left|(\X^k)_{ii}-\frac{1}{n}\Tr \X^k\right|<n^{-1/2+\eps},
\qquad \max_{i \neq j} \left|(\X^k)_{ij}\right|<n^{-1/2+\eps}.\]
\end{assumption}

For example, one may take $\X=\beta\M$ where $\beta \in \R$ and $\M$
is a Hadamard matrix (with $n$ a power of 2) or the sine transform
\[\M=\left(\sqrt{\frac{2}{n+1}}\sin\frac{\pi ij}{n+1}\right)_{i,j=1}^n.\]
In either case, $\M$ is orthogonal, and Assumption \ref{assump:X-universal}
holds with $\Lambda_{\max}=|\beta|$ and
$\Lambda=\frac{1}{2}\delta_{\beta}+\frac{1}{2}\delta_{-\beta}$.

\begin{corollary}\label{cor:universality}
Suppose $\X$ satisfies Assumption \ref{assump:X-universal}.
% with $\kappa_1=\E[\Lambda]=0$.
\begin{enumerate}[(a)]
\item If Assumption \ref{assump:dynamics} holds,
$f:\R \to \R$ is odd, and $\theta^0$ is
sign invariant in law (i.e.\ $f(x)={-}f(-x)$ for all $x \in \R$, and
$\Law(\theta^0)=\Law({-}\theta^0)$),
then Theorem \ref{thm:dmft-approx} holds for the dynamics
$\{\btheta^t\}_{t \geq 0}$.
\item If $U(\cdot)$ and $\Lambda$ satisfy all conditions of
Corollary \ref{cor:hightemp}, and $U:\R \to \R$ is even
(i.e.\ $U(x)=U(-x)$ for all $x \in \R$),
then all statements of Theorem \ref{thm:replica} hold for $\mu_\Gibbs$.
\item If $h=0$ in the Ising model \eqref{eq:ising} and $\Lambda_{\max}<1/2$,
then all statements of Corollary \ref{cor:ising} hold.
\item If $f(\cdot)$ is not odd, $\theta^0$ is not sign-invariant in law,
$U(\cdot)$ is not even, or $h \neq 0$, then the
same conclusions hold for the dynamics and Gibbs measures defined with
$\S\X\S$ in place of $\X$, where $\S$ is a
diagonal matrix of i.i.d.\ Rademacher $\pm 1$ signs
independent of $\btheta^0$ and $\{\b^t\}_{t \geq 0}$.
\end{enumerate}
\end{corollary}

\begin{remark}
Let $\X=(\Id-\frac{1}{n}\1\1^\top)\H(\Id-\frac{1}{n}\1\1^\top)$ be a 
``puncturing'' of a deterministic matrix $\H$ satisfying the conditions of
Assumption \ref{assump:X-universal}, and let
$\tilde \X=(\Id-\frac{1}{n}\1\1^\top)\tilde \H(\Id-\frac{1}{n}\1\1^\top)$
be a similar puncturing of an orthogonally-invariant random matrix $\tilde\H$
satisfying Assumption \ref{assump:X} with the same spectral limit $\Law(\Lambda)$.
One may also verify using \cite[Proposition 2.7, Lemma
B.1]{wang2024universality} and the recent result of
\cite[Theorem 5.3]{gorini2026universality} that, even if
$f(\cdot)$ is not odd / $U(\cdot)$ is not even,
the mean-field limits for the dynamics
$\{\btheta^t\}_{t \geq 0}$ and Gibbs measure $\mu_\Gibbs$ coincide for
$\X$ and $\tilde\X$.

However, without further assumptions for $f(\cdot)$ / $U(\cdot)$
such as the sign symmetry conditions in Corollary \ref{cor:universality} above,
these limits would, in general, be different from
those stated in Theorems \ref{thm:dmft-approx} and \ref{thm:replica}
for the unpunctured orthogonally-invariant matrix $\tilde \H$. It is possible to derive analogues of
our results for this puncturing of $\tilde \H$, but for brevity we will not do
so in the current work.
\end{remark}

The remainder of this paper proves the preceding results. Section \ref{sec:dmft-approx} proves Theorem \ref{thm:dmft-approx} on the dynamical mean-field limit, Section \ref{sec:corrresponse} analyzes the kernels $C_\theta,R_\theta,C_g,R_g$ for the overdamped Langevin diffusion at equilibrium, Section \ref{sec:replica} proves Theorem \ref{thm:replica} and Corollaries \ref{cor:hightemp} and \ref{cor:ising} on the Gibbs measure, and Section \ref{sec:universality} proves Corollary \ref{cor:universality} on universality over the disorder.

\section{Dynamical mean-field limit}\label{sec:dmft-approx}

In this section, we prove Theorem \ref{thm:dmft-approx}. We will omit arguments
that are similar to ones in \cite{fan2025dynamicalI} for brevity, and focus on
the differences in our setting of a non-Gaussian disorder matrix $\X$.

\subsection{Kernel spaces and kernel mappings}\label{subsec:spaces}

Fixing two constants $K_0,K_1>0$, we define ``envelope'' functions 
$\Phi_{R_\theta},\Phi_{R_g},\Phi_{C_\theta},\Phi_{C_g}$ on $[0,\infty)$ via
\begin{align}
\frac{\d}{\d t}\Phi_{R_\theta}(t) &= K_0\left(\Phi_{R_\theta}(t) + \int_0^t
\Phi_{R_g}(t-s)\Phi_{R_\theta}(s)\d s\right),\qquad \Phi_{R_\theta}(0)=K_0,\label{eq:envelope-R-theta}\\
\Phi_{R_g}(t) &= \sum_{p \geq 1} K_1^{p+1} \Phi_{R_\theta}^{\ast p}(t), \label{eq:envelope-R-g}\\
\frac{\d}{\d t}\Phi_{C_\theta}(t) &= K_0\left(\Phi_{C_\theta}(t)
+\int_0^t \Phi_{R_g}(t-s)[\Phi_{C_\theta}(s)+\Phi_{C_\theta}(t)]\d s
+ \Phi_{C_g}(t)+1\right),\qquad \Phi_{C_\theta}(0)=K_0,\label{eq:envelope-C-theta}\\
\sqrt{\Phi_{C_g}(t)} &= \sum_{p \geq 0} K_1^{p+1} \Phi_{R_\theta}^{\ast p} \ast
\sqrt{\Phi_{C_\theta}}(t),\label{eq:envelope-C-g}
\end{align}
where, for $a,b:[0,\infty) \to \R$, $[a*b](t)=\int_0^t a(t-s)b(s)\d s$.
It may be verified through a Laplace transform argument that 
this system has a unique solution, satisfying 
$\Phi_{R_\theta},\allowbreak \Phi_{R_g},\allowbreak
\Phi_{C_\theta},\allowbreak \Phi_{C_g} \in S(\lambda)$ 
for all sufficiently large $\lambda>0$ where
\[
  S(\lambda)=
  \left\{
    \Phi:[0,\infty)\to \R,\;\Phi(t),\frac{\d}{\d t}\Phi(t) \geq 0 \text{ for all } t \geq 0,\;
    \int_0^\infty e^{-\lambda t}\Phi(t)\d t<\infty
  \right\}.
\]
We omit details of this verification, and refer to
\cite[Lemma 3.1]{fan2025dynamicalI} for an analogous argument.

We define from these envelope functions the following kernel spaces.

\begin{definition}\label{def:S}
Let $T>0$, let $D=\{d_1,\ldots,d_m\} \subset [0,T]$ be a finite subset,
and call \[[0,d_1),[d_1,d_2),\ldots,[d_{m-1},d_m),[d_m,T]\] 
the maximal intervals of $[0,T] \setminus D$. 

Fixing $T$-dependent
constants $K_0,K_1,K_2,K_3>0$ where $K_0,K_1$ define the above envelopes,
let $\mathcal{S}_\theta^+(T,D)$ be the space of all
pairs $(C_\theta,R_\theta)$ such that:
    \begin{enumerate}
    \item $C_\theta:[0,T] \times [0,T] \to \R$ is a positive-semidefinite
covariance kernel satisfying
\[C_\theta(t,t) \leq \Phi_{C_\theta}(t) \text{ for all } 0\leq t\leq T,\]
and for any maximal interval $I$ of $[0,T]\setminus D$,
        \begin{equation}\label{eq:Ctheta-cont}
            |C_\theta(t,t)+C_\theta(s,s)-2C_\theta(t,s)| \leq  K_2|t-s|
\text{ for all } t,s \in I.
        \end{equation}
        \item $R_\theta:[0,T] \times [0,T] \to \R$ is a function satisfying
\[|R_\theta(t,s)| \leq \Phi_{R_\theta}(t-s) \text{ for all } 0 \leq s \leq t
\leq T, \qquad R_\theta(t,s)=0 \text{ for all } s>t,\]
and for any two maximal intervals $I,I'$ of $[0,T]\setminus D$,
\begin{equation}\label{eq:Rtheta-cont}
\begin{aligned}
|R_\theta(t,s)-R_\theta(t',s')| &\leq K_2(|t-t'|+|s-s'|)\\
&\quad \text{ for all } t,t'\in
I,~s,s'\in I' \text{ with } s\le t \text{ and } s'\le t'.
\end{aligned}
\end{equation}
\end{enumerate}

Let $\mathcal{S}_g(T,D)$ be the space of all pairs
$(C_g,R_g)$ such that:
\begin{enumerate}
    \item $C_g:[0,T] \times [0,T] \to \R$ is a symmetric function (not necessarily
positive-semidefinite) satisfying
\[C_g(t,t) \leq \Phi_{C_g}(t) \text{ for all } 0\leq t\leq T,\]
and for any two maximal intervals $I,I'$ of $[0,T]\setminus D$,
        \begin{equation}\label{eq:Cg-cont1}
|C_g(t,t)+C_g(s,s)-2C_g(t,s)| \leq K_2K_3|t-s|
\text{ for all } t,s \in I,
\end{equation}
\begin{equation}\label{eq:Cg-cont2}
            |C_g(t,s)-C_g(t',s')| \leq K_2K_3(\sqrt{|t-t'|}+\sqrt{|s-s'|})
\text{ for all } t,t' \in I,~s,s' \in I'.
        \end{equation}
      \item $R_g:[0,T] \times [0,T] \to \R$ is a function satisfying
\[|R_g(t,s)| \leq \Phi_{R_g}(t-s) \text{ for all } 0 \leq s \leq t \leq T,
\qquad R_g(t,s)=0 \text{ for all } s>t,\]
and for any two maximal intervals $I,I'$ of $[0,T]\setminus D$,
\begin{equation}\label{eq:Rg-cont}
\begin{aligned}
|R_g(t,s)-R_g(t',s')| &\leq K_2K_3(|t-t'|+|s-s'|)\\
&\quad \text{ for all } t,t'\in
I,~s,s'\in I' \text{ with } s\le t \text{ and } s'\le t'.
\end{aligned}
\end{equation}
\end{enumerate}
We denote by $\cS_g^+(T,D) \subset \cS_g(T,D)$
the subset of pairs $(C_g,R_g)$ where $C_g$ is positive-semidefinite.
\end{definition}

These spaces depend on the constants $K_0,K_1,K_2,K_3>0$, although we make this
dependence implicit in the notation. The final spaces $\cS_\theta^+,\cS_g^+$ in
Theorem \ref{thm:dmft-approx}(a) will be the spaces of kernels
$C_\theta,R_\theta,C_g,R_g$ on $[0,\infty) \times [0,\infty)$ whose restrictions
to $[0,T] \times [0,T]$
belong to $\cS_\theta^+(T,\emptyset)$ and $\cS_g^+(T,\emptyset)$ with
$D=\emptyset$ for every $T>0$.

We define a mapping
\[\cT_{\theta \to g}:\cS_\theta^+(T,D) \to \cS_g(T,D),
\qquad (C_\theta,R_\theta) \mapsto (C_g,R_g)\]
by (\ref{eq:dmft-Cg}--\ref{eq:dmft-Rg}):
\[C_g=\sum_{p,q \geq 0} \kappa_{p+q+2}\,
R_\theta^{\ast p} * C_\theta * \bar R_\theta^{\ast q},
\qquad
R_g=\sum_{p \geq 1} \kappa_{p+1}\,R_\theta^{\ast p}.\]
Note that $C_g$ is not necessarily positive-semidefinite.
Restricting to positive-semidefinite $C_g$, we define a reverse mapping
\[\cT_{g \to \theta}:\cS_g^+(T,D) \to \cS_\theta^+(T,\emptyset),
\qquad (C_g,R_g) \mapsto (C_\theta,R_\theta)\]
by (\ref{eq:dmft-Ctheta}--\ref{eq:dmft-Rtheta}):
\[C_\theta(t,s)=\E[\theta^t \theta^s],
\qquad R_\theta(t,s)=\E\left[\frac{\d\theta^t}{\d g^s}\right]\]
where $\{\theta^t\}_{t \geq 0}$ and $\{\frac{\partial \theta^t}{\partial
g^s}\}_{t \geq s \geq 0}$ solve (\ref{eq:dmft-theta}--\ref{eq:dmft-response})
with $R_g$ and $\{g^t\}_{t \geq 0} \sim \GP(0,C_g)$. The following lemma
verifies that these mappings are well-defined.

\begin{lemma}\label{lemma:well_defined}
Fix any $T>0$. Then for all sufficiently large ($T$-dependent)
constants $K_0,K_1,K_2,K_3>0$ defining the spaces of Definition \ref{def:S},
the following holds:

For any finite subset $D \subset [0,T]$, given any $(C_g,R_g) \in
\cS_g^+(T,D)$, the system (\ref{eq:dmft-theta}--\ref{eq:dmft-response}) has a
unique solution, and $\cT_{g\to \theta}(C_g,R_g) \in \cS_\theta^+(T,\emptyset)$.
Given any $(C_\theta,R_\theta) \in \cS_\theta^+(T,D)$, 
the series (\ref{eq:dmft-Cg}--\ref{eq:dmft-Rg}) are pointwise convergent, and
$\cT_{\theta\to g}(C_\theta,R_\theta) \in \cS_g(T,D)$.
\end{lemma}
\begin{proof}
Given $(C_g,\allowbreak R_g) \in \cS_g^+(T,\allowbreak D)$, existence and uniqueness of $\{\theta^t\}_{t
\in [0,T]}$ and $\{\frac{\partial \theta^t}{\partial g^s}\}_{0 \leq s \leq t
\leq T}$ solving (\ref{eq:dmft-theta}--\ref{eq:dmft-response}) follows from the
same argument as \cite[Lemma~3.3]{fan2025dynamicalI}, which we omit here for brevity.
Therefore, $\cT_{g\to \theta}$ is well-defined.

Let $(C_\theta,R_\theta)=\cT_{g \to \theta}(C_g,R_g)$.
We check that $(C_\theta,R_\theta) \in \cS_\theta^+(T,\emptyset)$:
First, according to \eqref{eq:dmft-response},
\begin{align*}
 & \left| \frac{\d }{\d t} \frac{\partial \theta^t}{\partial g^s}\right|
 \leq \|f\|_\Lip \left| \frac{\partial \theta^t}{\partial g^s}\right| + \int_s^t
\Phi_{R_g}(t-r) \left|\frac{\partial\theta^r}{\partial g^s}\right|\d r
\end{align*} 
where $\|f\|_\Lip<\infty$ is the Lipschitz constant of $f$.
Noting that $\frac{\partial \theta^s}{\partial g^s}=1$ and
comparing with \eqref{eq:envelope-R-theta}, this implies for $K_0>0$ large enough
that
\begin{align}\label{eq:bound-response-process} 
\left|\frac{\partial \theta^t}{\partial g^s}\right| \leq \Phi_{R_\theta}(t-s).
\end{align}
Then by definition of $R_\theta$ in \eqref{eq:dmft-Rtheta},
$|R_\theta(t,s)| \le \Phi_{R_\theta}(t-s)$ for all $0\le s\le t\le T$.
Similarly, for any $0 \le s,t,t' \le T$ with $s \leq t' \leq t$, we have
\begin{align}
    |R_\theta(t,s) - R_\theta(t',s)| &\le \E\left|\frac{\partial
\theta^t}{\partial g^s} - \frac{\partial \theta^{t'}}{\partial g^s}\right| \le
\E \int_{t'}^t \left| \frac{\d }{\d r} \frac{\partial \theta^r}{\partial g^s}
\right|\d r\notag\\
    &\le \E \int_{t'}^t \left(\|f\|_\Lip \left| \frac{\partial
\theta^r}{\partial g^s}\right| +  \int_s^r \Phi_{R_g}(r-r')
\left|\frac{\partial\theta^{r'}}{\partial g^s}\right|\d r'\right)\d r
    \le K_2|t-t'|,\label{eq:Rthetacontt}
\end{align}
the last inequality holding by \eqref{eq:bound-response-process} for some
large enough ($T$-dependent) $K_2>0$.
For continuity in \(s\), by the response equation \eqref{eq:dmft-response}, for
\(t \ge s\),
\[
\frac{\partial \theta^t}{\partial g^s}
=
1+
\int_s^t
\left[
f'(\theta^u)\frac{\partial \theta^u}{\partial g^s}
+
\int_s^u R_g(u,r)\frac{\partial \theta^r}{\partial g^s}\,\d r
\right]\d u.
\]
Consider any $t \geq s' \geq s$. Then
\[
\begin{aligned}
\left|
\frac{\partial \theta^{s'}}{\partial g^s}-1
\right|
&\le
\int_s^{s'}
\|f\|_\Lip \left|\frac{\partial \theta^u}{\partial g^s}\right|\,\d u
+
\int_s^{s'}\int_s^u |R_g(u,r)|
\left|\frac{\partial \theta^r}{\partial g^s}\right|
\,\d r\,\d u  \le C |s'-s|,
\end{aligned}
\]
the last holding by
\eqref{eq:bound-response-process} and $|R_g(u,r)| \le \Phi_{R_g}(T)$
for some large enough ($T$-dependent) constant $C>0$.
Then subtracting the two response equations for
\(\frac{\partial \theta^t}{\partial g^s}\) and
\(\frac{\partial \theta^t}{\partial g^{s'}}\) gives
\begin{align*}
&\left|
\frac{\partial \theta^t}{\partial g^s}
-\frac{\partial \theta^t}{\partial g^{s'}}
\right|
\le
\left|
\frac{\partial \theta^{s'}}{\partial g^s}
-1 \right| +
\int_{s'}^t
\|f\|_\Lip
\left|
\frac{\partial \theta^u}{\partial g^s}
-
\frac{\partial \theta^u}{\partial g^{s'}}
\right|
\,\d u \\
&\quad+
\int_{s'}^t\int_{s'}^u
|R_g(u,r)|
\left|
\frac{\partial \theta^r}{\partial g^s}
-
\frac{\partial \theta^r}{\partial g^{s'}}
\right|
\,\d r\,\d u +
\int_{s'}^t\int_s^{s'}
|R_g(u,r)|
\left|
\frac{\partial \theta^r}{\partial g^s}
\right|
\,\d r\,\d u.
\end{align*}
The first and last terms are bounded by \(C|s'-s|\) for some $C>0$, and the
second and third terms by $C \sup_{u \in [s',t]} |
\frac{\partial \theta^u}{\partial g^s}
-\frac{\partial \theta^u}{\partial g^{s'}}|$.
Then by Gr\"onwall's inequality,
\[
\left|
\frac{\partial \theta^t}{\partial g^s}
-
\frac{\partial \theta^t}{\partial g^{s'}}
\right|
\le
K_2|s'-s|
\]
for some sufficiently large $K_2>0$. Taking expectations gives
\begin{equation}\label{eq:Rthetaconts}
\left|R_\theta(t,s)-R_\theta(t,s')\right| \le K_2|s'-s|.
\end{equation}
Combining \eqref{eq:Rthetacontt} and \eqref{eq:Rthetaconts} shows the continuity property \eqref{eq:Rtheta-cont}.
This verifies all properties for $R_\theta$ in Definition \ref{def:S}.

For the covariance kernel $C_\theta$, applying It\^o's formula to 
the equation \eqref{eq:dmft-theta} gives
  \begin{align*}
      \frac{\d}{\d t} \E (\theta^t)^2 &= 2\E[\theta^t f(\theta^t)] + 2\int_0^t
R_g(t,s)\E[\theta^t\theta^s]\d s + 2\E[\theta^t g^t] + 2\gamma\\
      &\le K_0\left(\E(\theta^t)^2+\int_0^t \Phi_{R_g}(t-s) [\E (\theta^s)^2
+ \E (\theta^t)^2] \d s + C_g(t,t)+1\right),
  \end{align*}
the second inequality holding for large enough $K_0>0$
since $\{g^t\}_{t \geq 0} \sim \GP(0,C_g)$, $f$ is Lipschitz, and $|R_g(s,r)|\le
\Phi_{R_g}(s-r)$.
Comparing with \eqref{eq:envelope-C-theta}, this implies 
$C_\theta(t,t)=\E(\theta^t)^2 \le \Phi_{C_\theta}(t)$.
Similarly, for all $s,s',t,t' \in [0,T]$, one can show
\[
|C_\theta(t,t)+C_\theta(s,s)-2C_\theta(t,s)| = \E (\theta^t - \theta^s)^2
\le K_2|t-s|
\]
for large enough $K_2>0$. This verifies all properties for $C_\theta$ in
Definition \ref{def:S} with $D=\emptyset$, so $(C_\theta,R_\theta) \in \cS_\theta^+(T,\emptyset)$.

For the map $\cT_{\theta\to g}$, given $(C_\theta,R_\theta)\in
\cS_\theta^+(T,D)$, note that Proposition \ref{prop:freecumulants}
implies $|\kappa_p| \leq K_1^p$ for any sufficiently large $K_1>0$ and all
$p \geq 1$, and $|R_\theta(t,s)| \leq \Phi_{R_\theta}(t-s)$. Then
\[\left|\sum_{p \geq 1} \kappa_{p+1} R_\theta^{*p}(t,s)\right|
\leq \sum_{p \geq 1} K_1^{p+1} \Phi_{R_\theta}^{*p}(t-s)
\leq \sum_{p \geq 1} K_1^{p+1}\Phi_{R_\theta}(T)^p
\frac{(t-s)^{p-1}}{(p-1)!}<\infty,\]
so the series \eqref{eq:dmft-Rg} is pointwise convergent. Similarly
\eqref{eq:dmft-Cg} is pointwise convergent, so
$\cT_{\theta\to g}$ is well-defined.

Let $(C_g,R_g)=\cT_{\theta\to g}(C_\theta,R_\theta)$. 
Applying $|\kappa_p| \leq K_1^p$, $|R_\theta(t,s)| \leq \Phi_{R_\theta}(t-s)$,
and
\[
|C_\theta(t,s)| \le \sqrt{C_\theta(t,t)}\sqrt{C_\theta(s,s)} \le \sqrt{\Phi_{C_\theta}(t)}\sqrt{\Phi_{C_\theta}(s)},
\]
we can verify that
\begin{align*}
    |R_g(t,s)| &\le \sum_{p \geq 1} K_1^{p+1}\,\Phi_{R_\theta}^{\ast p}(t-s) =
\Phi_{R_g}(t-s),\\
    |C_g(t,t)| &\le \sum_{p,q \geq 0} K_1^{p+q+2}\,
(\Phi_{R_\theta}^{\ast p} \ast \sqrt{\Phi_{C_\theta}})(t) (\Phi_{R_\theta}^{\ast
q} \ast \sqrt{\Phi_{C_\theta}})(t)\\
&=\left(\sum_{p\ge 0} K_1^{p+1}\,
(\Phi_{R_\theta}^{\ast p} \ast \sqrt{\Phi_{C_\theta}})(t)\right)^2 = \Phi_{C_g}(t).
\end{align*}

For any maximal interval $I$ of $[0,T] \setminus D$ and any $s,t \in I$
with $t \geq s$, note that
\begin{align}
&|C_g(t,t)+C_g(s,s)-2C_g(t,s)|\notag\\
&=\Bigg|\kappa_2\Big(C_\theta(t,t)+C_\theta(s,s)-2C_\theta(t,s)\Big)
+\sum_{p\ge 1}
  \kappa_{p+2}
  \int_0^t (R_\theta^{\ast p}(t,r)-R_\theta^{\ast p}(s,r))(C_\theta(r,t)
  -C_\theta(r,s))\d r \notag\\
&\hspace{0.5in}+ \sum_{q\ge 1}
  \kappa_{q+2}
  \int_0^t \d u\,
  (C_\theta(t,u)-C_\theta(s,u))
(R_\theta^{\ast q}(t,u)-R_\theta^{\ast q}(s,u))\notag\\
&\hspace{0.5in}+\sum_{p,q\ge 1}
  \kappa_{p+q+2}
  \int_0^t \d r \int_0^t \d u\,
    (R_\theta^{\ast p}(t,r)-R_\theta^{\ast p}(s,r))C_\theta(r,u)
    (R_\theta^{\ast q}(t,u)-R_\theta^{\ast q}(s,u))\Bigg|.\label{eq:Cgdiffexpr}
\end{align}
For any \(p\ge 1\) and $t \ge s\ge r$,
\[
\begin{aligned}
R_\theta^{\ast p}(t,r)-R_\theta^{\ast p}(s,r)
&=
\int_r^s
\bigl(R_\theta(t,u)-R_\theta(s,u)\bigr)
R_\theta^{\ast(p-1)}(u,r)\,\d u+
\int_s^t
R_\theta(t,u)R_\theta^{\ast(p-1)}(u,r)\,\d u.
\end{aligned}
\]
Therefore, combined with 
\[
\begin{aligned}
|R_\theta(t,u) - R_\theta(s,u)| &\leq K_2 |t-s|,\\
|R_\theta^{\ast p}(t,s)| &\le \Phi_{R_\theta}^{\ast p}(t-s)
\le \frac{\Phi_{R_\theta}(T)^p (t-s)^{p-1}}{(p-1)!},
\end{aligned}
\]
we have that
\begin{align*}
|R_\theta^{\ast p}(t,r)-R_\theta^{\ast p}(s,r)|
&\le \int_r^s
K_2|t-s|
\Phi_{R_\theta}^{\ast (p-1)}(u-r)\,\d u 
+
\int_s^t \Phi_{R_\theta}(t-u)\Phi_{R_\theta}^{\ast (p-1)}(u-r)
 \d u\\
 &\le (K_2T+\Phi_{R_\theta}(T)) \frac{\Phi_{R_\theta}(T)^{p-1}T^{p-1}}{(p-1)!} |t-s|. 
\end{align*}
This bound holds also for $r \in [s,t]$, noting in this case that
\[R_\theta^{\ast p}(t,r)-R_\theta^{\ast p}(s,r)
=R_\theta^{\ast p}(t,r)=\int_r^t R_\theta(t,u)R_\theta^{*(p-1)}(u,r)\,\d u\]
and applying the same argument. Applying this bound for summands of
\eqref{eq:Cgdiffexpr} where either $p \geq 1$ or $q \geq 1$, and applying
$|C_\theta(t,t)+C_\theta(s,s)-2C_\theta(t,s)| \leq K_2|t-s|$
for $p=q=0$, we obtain
\[|C_g(t,t)+C_g(s,s)-2C_g(t,s)| \leq K_2K_3|t-s|\]
for all large enough $K_2,K_3>0$, which is the continuity property
\eqref{eq:Cg-cont1}.

Since $C_\theta$ is positive-semidefinite, writing $u \sim \GP(0,C_\theta)$,
for any two maximal intervals $I,I'$ of $[0,T] \setminus D$ and any
$t,t' \in I$ and $s,s' \in I'$, we have also
\[\begin{aligned}
|C_\theta(t,s)-C_\theta(t',s')|
&\leq |\E[(u^t-u^{t'})u^s]|+|\E[u^{t'}(u^s-u^{s'})]|\\
&\leq \sqrt{\Phi_{C_\theta}(T)} \cdot K_2(\sqrt{|t-t'|}+\sqrt{|s-s'|}),
\end{aligned}\]
by Cauchy-Schwarz. Then applying a similar argument as above shows
\[|C_g(t,s)-C_g(t',s')| \leq K_2K_3(\sqrt{|t-t'|}+\sqrt{|s-s'|})\]
for all large enough $K_2,K_3>0$,
which is the continuity property \eqref{eq:Cg-cont2}.
Finally, for any maximal intervals $I,I'$ of
$[0,T] \setminus D$ and $s,s' \in I$, $t,t' \in I'$ with $s \leq t$ and $s' \leq
t'$, a similar argument as above using $|R_\theta(t,s)-R_\theta(t',s')| \leq
K_2(|t-t'|+|s-s'|)$ checks that
\[|R_g(t,s)-R_g(t',s')| \leq K_2K_3(|t-t'|+|s-s'|)\]
for large enough $K_2,K_3>0$, which is \eqref{eq:Rg-cont}.
Thus $(C_g,R_g)\in \cS_g(T,D)$.
\end{proof}

Henceforth, given $T>0$, we will fix $T$-dependent constants
$K_0,K_1,K_2,K_3>0$ large enough such that Lemma \ref{lemma:well_defined}
holds.

\subsection{Contractivity and uniqueness of fixed point with projection}

Next, we define a metric on $\cS_\theta^+$ and $\cS_g$. Fix $\lambda>0$.
For a pair of response kernels $(R_1,R_2) \equiv (R_{\theta,1},R_{\theta,2})$ 
or $(R_1,R_2) \equiv (R_{g,1},R_{g,2})$ satisfying the conditions of
Definition \ref{def:S}, we define
\begin{align}\label{eq:metric-response}
  d_\lambda(R_1, R_2) = \sup_{0 \leq s \leq t \leq T} e^{-\lambda
t}|R_1(t,s)-R_2(t,s)|.
\end{align} 
For a pair of (not necessarily positive-semidefinite) correlation kernels
$(C_1,\allowbreak C_2) \equiv (C_{\theta,1},\allowbreak C_{\theta,2})$ 
or $(C_1,\allowbreak C_2) \equiv (C_{g,1},\allowbreak C_{g,2})$ satisfying the conditions of
$\cS_\theta^+$/$\cS_g$ in Definition \ref{def:S}, we define
\begin{align}\label{eq:hilbertinnerproduct}
\langle C_1, C_2 \rangle_{\lambda}=\int_0^T \int_0^T e^{-2\lambda
(t\vee s)}C_1(t,s) C_2(t,s)\,\d t\,\d s,
\qquad \|C\|_\lambda=\langle C,C \rangle_{\lambda}^{1/2}
 \end{align} 
and the metric
\begin{align}
\label{eq:metric-covariance}
  d_\lambda(C_1, C_2)=\|C_1-C_2\|_\lambda= \sqrt{\int_0^T \int_0^T e^{-2\lambda (t\vee s)}(C_1(t,s) - C_2(t,s))^2 \,\d t \,\d s}.
\end{align}
Finally, for either
$(C,R) \equiv (C_\theta,R_\theta)$ or $(C,R) \equiv (C_g,R_g)$, we set
\[d_\lambda((C_1,R_1),(C_2,R_2))
=d_\lambda(R_1,R_2)+d_\lambda(C_1,C_2).\]
By the uniform continuity properties of the kernels on $[0,T] \setminus D$
in Definition \ref{def:S},
this defines a valid metric $d_\lambda$ on both $\cS_\theta^+$ and
$\cS_g$.

\begin{lemma}[Contractivity of $\cT_{g \to \theta}$]\label{lem:contraction-Rtheta}
Let $X_i=(C_{g,i},R_{g,i})\in \mathcal{S}_{g}^+(T,D)$ for $i=1,2$, and let
$Y_i=(C_{\theta,i},R_{\theta,i})=\cT_{g \to \theta}(C_{g,i},R_{g,i})$.
Then for any fixed $0<\alpha<1$ and all sufficiently large $\lambda>0$
(depending on $\alpha,T$),
\begin{align}\label{equation:contraction_Rfix}
d_\lambda(Y_1,Y_2) \leq \alpha\cdot d_\lambda(X_1,X_2).
\end{align}
\end{lemma}
\begin{proof}
We first consider $C_{g,1}=C_{g,2} \equiv C_g$ and different $R_{g,1},R_{g,2}$.
Coupling the processes $\{\theta_i^t\}_{t \geq 0}$ 
and $\{\frac{\partial \theta_i^t}{\partial g^s}\}_{t \geq s \geq 0}$ 
solving (\ref{eq:dmft-theta}--\ref{eq:dmft-response}) with $R_{g,i}$
by the same realizations of $\{g^t\}_{t \in [0,T]} \sim \GP(0,C_g)$ and
$\{b^t\}_{t \in [0,T]}$, a Gr\"onwall
argument (c.f.\ \cite[Lemma 3.7]{fan2025dynamicalI}) shows,
for all sufficiently large $\lambda>0$ and a constant $C>0$ not depending on
$\lambda$,
\begin{align*}
\sup_{0 \leq t \leq T} e^{-2\lambda t}\E(\theta_1^t-\theta_2^t)^2 \leq
\frac{C}{\lambda}\,d_\lambda(R_{g,1},R_{g,2})^2,\\
\sup_{0 \leq s \leq t \leq T}
e^{-\lambda t}
\E \left|\frac{\partial \theta_1^t}{\partial g^s} - \frac{\partial
\theta_2^t}{\partial g^s}\right| \leq \frac{C}{\lambda}
\,d_\lambda(R_{g,1},R_{g,2}).
\end{align*}
The second inequality implies
$d_\lambda(R_{\theta,1},R_{\theta,2}) \leq \frac{C}{\lambda}
d_\lambda (R_{g,1},R_{g,2})$. By Cauchy-Schwarz,
\begin{align*}
|C_{\theta,1}(t,s)-C_{\theta,2}(t,s)|^2 &\leq 2\E(\theta_1^t-\theta_2^t)^2\E(\theta_1^s)^2 + 2\E(\theta_2^t)^2\E(\theta_1^s-\theta_2^s)^2\\
&\leq 2\Phi_{C_\theta}(T)\left( \E(\theta_1^t-\theta_2^t)^2 + \E(\theta_1^s-\theta_2^s)^2 \right). \end{align*}
Integrating both sides, we have
\[\begin{aligned}
d_{\lambda}(C_{\theta,1}, C_{\theta,2})^2
&=\int_0^T\int_0^T e^{-2\lambda(t \vee s)}
|C_{\theta,1}(t,s)-C_{\theta,2}(t,s)|^2\,\d t\, \d s
\\
&\leq 4\Phi_{C_\theta}(T)T^2
\sup_{0 \leq t \leq T} e^{-2\lambda t}\E(\theta_1^t-\theta_2^t)^2,
\end{aligned}\]
so the first inequality implies 
$d_{\lambda}(C_{\theta,1},C_{\theta,2}) \leq \frac{C'}{\sqrt{\lambda}}
d_\lambda (R_{g, 1}, R_{g, 2})$. Thus, for any $\alpha \in (0,1)$ and all large
enough $\lambda>0$,
\begin{equation}\label{eq:contractionsameCg}
d_\lambda(Y_1,Y_2) \leq \alpha \cdot d_\lambda(R_{g,1},R_{g,2}).
\end{equation}

We next consider $R_{g,1}=R_{g,2} \equiv R_g$ and different $C_{g,1},C_{g,2}$.
For $u \in [0,1]$, define the interpolation $C_u=(1-u)C_{g,1}+uC_{g,2}$,
and let $g_u \sim \GP(0,C_u)$.
Note that $C_u$ induces a positive-semidefinite, trace-class covariance operator
$Q_u$ on $\mathcal{H}:=L^2([0,T])$, given by
\[
    (Q_u\varphi)(r)
    :=
    \int_0^T C_u(r,r')\varphi(r')\,\d r',
    \qquad
    \varphi\in \mathcal{H}.
\]
Let $\{e_k\}_{k\ge 1}$ be an orthonormal basis of
$\mathcal{H}$, and let $\{\xi_k\}_{k\ge 1}$ be i.i.d.~standard Gaussian random
variables. Then we may realize $g_u \sim \GP(0,C_u)$ as
$g_u=\sum_{k\ge 1}Q_u^{1/2}e_k \cdot \xi_k$.

Fixing an integer $N \geq 1$, define $\Pi_N : {\mathcal H}\to {\sp(e_1,\ldots, e_N)}$ as the orthogonal projection onto the subspace spanned by $e_1,\ldots,e_N$. Consider the truncation
$g_{u,N}:=\Pi_N g_u$, 
and define $\theta_{g_{u,N}}$ by
\[
    \d \theta_{g_{u,N}}^t
    =
        \left[f(\theta_{g_{u,N}}^t)
        +
        \int_0^t R_g(t,s)\theta_{g_{u,N}}^s \d s
        +
        g_{u,N}^t \right]
    \d t + \sqrt{2\gamma}\,\d b^t, \qquad \theta_{g_{u,N}}^0=\theta^0.
\] 
Let us compute
$\frac{\d}{\d u} \E[\theta_{g_{u,N}}^t\theta_{g_{u,N}}^s]$:
For any $h\in \mathcal{H}$, define $\{\theta_h^t\}_{t \geq 0}$ by
\[
    \theta_{h}^t
    =
    \theta^0
    +
    \int_0^t
    \left[
        f(\theta_{h}^s)
        +
        \int_0^s R_g(s,r)\theta_{h}^r \d r
        +
        {h}^s
    \right]
    \d s + \sqrt{2\gamma}\,b^t.
\]
By \cite[Theorem 2.5]{wang2017stochastic}, for any $h,k,l \in \cH$,
the above equation has a unique solution $\{\theta_h^t\}_{t \geq 0}$, and
there exist directional derivative processes
$D_k\theta_h$ and $D_{kl}^2 \theta_h$ on $[0,T]$
for which
\[\lim_{\eps \to 0}
\E_{\{b^t\}}\left(\frac{\theta_{h+\eps k}^t-\theta_h^t}{\eps}
-D_k\theta_h^t\right)^2=0,
\qquad \lim_{\eps \to 0}
\E_{\{b^t\}}\left(\frac{D_k\theta_{h+\eps l}^t-D_k\theta_h^t}{\eps}
-D_{kl}^2\theta_h^t\right)^2=0\]
where $\E_{\{b^t\}}$ is the expectation over the Brownian motion
$\{b^t\}_{t \geq 0}$. Then 
$\E_{\{b^t\}}[\theta_{h+\eps k+\eps'l}^t\allowbreak
\theta_{h+\eps k+\eps'l}^s]$ is twice continuously-differentiable in
$(\eps,\eps')$, and 
\begin{equation}\label{eq:thetatssecondder}
\begin{aligned}
\partial_\eps\partial_{\eps'}
\E_{\{b^t\}}[\theta_{h+\eps k+\eps'l}^t
\theta_{h+\eps k+\eps'l}^s]\big|_{\eps=\eps'=0}
&=\E_{\{b^t\}}\Big[(D_{kl}^2\theta_h^t)\theta_h^s
+\theta_h^t (D_{kl}^2\theta_h^s)+(D_k\theta_h^t)(D_l\theta_h^s)
\\
&\qquad +(D_l\theta_h^t)(D_k\theta_h^s)\Big].
\end{aligned}
\end{equation}
Furthermore by \cite[Theorem 2.5]{wang2017stochastic}, $D_k\theta_h$ and
$D_{kl}^2\theta_h$ are the unique solutions of
\begin{align*}
    D_k\theta_h^t
    &=
    \int_0^t
    \left[
        f'(\theta_h^s)D_k\theta_h^s
        +
        \int_0^s R_g(s,r) D_k\theta_h^r \d r
        +
        k^s
    \right]
    \d s,\\
    D_{kl}^2\theta^t
    &=
    \int_0^t
    \left[
        f''(\theta_h^s)
        (D_k\theta_h^s)(D_l\theta_h^s)
        +f'(\theta_h^s)D_{kl}^2 \theta_h^s
        + \int_0^s R_g(s,r) D_{kl}^2\theta_h^r \d r
    \right]
    \d s.
\end{align*}
Equivalently, define first-order and second-order response kernels $J_h,K_h$ by
\begin{align*}
    J_h(t,r)
    &=
    1
    +
    \int_r^t
    \bigg[
        f'(\theta_h^s)J_h(s,r)
        +
        \int_r^s R_g(s,r')J_h(r',r)\d r
    \bigg]
    \d s\\
    K_h(t,r,r')
    &=
    \int_{r \vee r'}^t
    \bigg[
        f''(\theta_h^s)J_h(s,r)J_h(s,r')
        +
        f'(\theta_h^s)K_h(s,r,r')
        +
        \int_{r \vee r'}^s R_g(s,s') K_h(s',r,r') \d s' 
    \bigg]
    \d s
\end{align*}
for $t \geq \max(r,r')$, and $J_h(t,r)=0$ and $K_h(t,r,r')=0$ if $t<r$ and/or
$t<r'$. Then it may be checked that the above solutions are given by
\begin{equation}\label{eq:second-deriv-representation}
D_k\theta_h^t=\int_0^T J_h(t,r) k^r\,\d r,
\qquad
D_{kl}^2\theta_h^t=\int_0^T \int_0^T K_h(t,r,r')k^r l^{r'}\,\d r\,\d r'.
\end{equation}
Hence \eqref{eq:thetatssecondder} takes the form
\begin{align*}
\partial_\eps\partial_{\eps'}
\E_{\{b^t\}}[\theta_{h+\eps k+\eps'l}^t
\theta_{h+\eps k+\eps'l}^s]\big|_{\eps=\eps'=0}
&=\int_0^T \int_0^T H_h(t,s;r,r')k^r l^{r'}\,\d r\,\d r',
\end{align*}
where
\begin{align*}
    H_h(t,s;r,r')
    &:=
    K_h(t,r,r')\theta_h^s
    +
    K_h(s,r,r')\theta_h^t
    +
    J_h(t,r)J_h(s,r')
    +
    J_h(t,r')J_h(s,r).
\end{align*}
Then, by the finite-dimensional Gaussian interpolation formula
(see e.g.\ \cite[Lemma 7.2.7]{vershynin2018high}),
\begin{align}
    \frac{\d}{\d u}\E[\theta_{g_{u,N}}^t\theta_{g_{u,N}}^s]
&=\frac{\d}{\d u}\E[\E_{\{b^t\}}[\theta_{g_{u,N}}^t\theta_{g_{u,N}}^s]]\notag\\
    &=\hspace{-2pt}
    \frac12
    \sum_{i,j=1}^N
    \bigl\langle (Q_1-Q_0)e_i,e_j\bigr\rangle_{\mathcal{H}}
    \,
    \E
\left[\partial_\eps\partial_{\eps'}
\E_{\{b^t\}}[\theta_{g_{u,N}+\eps e_i+\eps'e_j}^t
\theta_{g_{u,N}+\eps e_i+\eps'e_j}^s]\big|_{\eps=\eps'=0}\right]
    \notag\\
    &=
    \frac12
    \sum_{i,j=1}^N
    \bigl\langle (Q_1-Q_0)e_i,e_j\bigr\rangle_{\mathcal{H}}
    \,
    \E\left[
        \int_0^T\int_0^T
        H_{g_{u,N}}(t,s;r,r')
        e_i^re_j^{r'}
        \,\d r\,\d r'
    \right].\label{eq:duCg}
\end{align}

By the Parseval identity, as $N \to \infty$,
\[\sum_{i,j=1}^N
    \bigl\langle (Q_1-Q_0)e_i,e_j\bigr\rangle_{\mathcal{H}}
    e_i^re_j^{r'} \to (C_{g,2}-C_{g,1})(r,r')
\]
in $L^2([0,T] \times [0,T])$.
Also as $N \to \infty$, $g_{u,N} \to g_u$ in $L^2([0,T])$. Then
applying a Gr\"onwall argument to the definitions of the above processes
shows that $\theta_{g_{u,N}} \to \theta_{g_u}$ and
$J_{g_{u,N}}(t,\cdot) \to J_{g_u}(t,\cdot)$
in $L^2([0,T])$, and $K_{g_{u,N}}(t,\cdot,\cdot) \to
K_{g_u}(t,\cdot,\cdot)$ and
$H_{g_{u,N}}(t,s;\cdot,\cdot) \to H_{g_u}(t,s;\cdot,\cdot)$
in $L^2([0,T] \times [0,T])$. Then,
integrating \eqref{eq:duCg} over $u \in [0,1]$ and taking the limit $N \to
\infty$,
\[\begin{aligned}
(C_{\theta,2}-C_{\theta,1})(t,s)
&=\E[\theta_{g_1}^t\theta_{g_1}^s]
-\E[\theta_{g_0}^t\theta_{g_0}^s]
    \\
&=\int_0^1 \d u\left(
    \frac12
    \int_0^T\int_0^T
    (C_{g,2}-C_{g,1})(r,r')
    \E\left[
        H_{g_u}(t,s;r,r')
    \right]
    \,\d r\,\d r'\right).
\end{aligned}\]
As $f',f''$ are bounded and $R_g$ is bounded on $[0,T]$,
a Gr\"onwall argument checks that $J_h(t,r)$ and $K_h(t,r,r')$ are both uniformly
bounded over $t,r,r' \in [0,T]$ and $h =g_u$ for all $u \in [0,1]$. Then
$\E[H_{g_u}(t,s;r,r')]$ is also uniformly bounded over $t,s,r,r' \in [0,T]$
and $u \in [0,1]$.  Furthermore, we have $\E[H_{g_u}(t,s;r,r')]=0$ unless $r,r' \leq
\max(t,s)$. Thus, there exists a $T$-dependent constant $C>0$ for which
\begin{align*}
|(C_{\theta,2}-C_{\theta,1})(t,s)|
&\leq C\int_0^{t \vee s}\int_0^{t \vee s}
    |(C_{g,2}-C_{g,1})(r,r')|\,\d r\,\d r'
\end{align*}
Then by Cauchy-Schwarz, for $T$-dependent constants $C',C''>0$ (not depending on
$\lambda$),
\begin{align*}
&d_\lambda(C_{\theta,1},C_{\theta,2})^2
=\int_0^T \int_0^T e^{-2\lambda(t \vee s)}
(C_{\theta,2}(t,s)-C_{\theta,1}(t,s))^2\d t\,\d s\\
&\leq C'\int_0^T\d t\int_0^T \d s\,
e^{-2\lambda(t \vee s)}
\int_0^{t \vee s}\d r\int_0^{t \vee s} \d r'
    (C_{g,2}(r,r')-C_{g,1}(r,r'))^2\\
&=C'\int_0^T \d r\int_0^T \d r'
e^{-2\lambda(r \vee r')}
(C_{g,2}(r,r')-C_{g,1}(r,r'))^2
\mathop{\int_0^T \d t\int_0^T\d s} \1_{t \vee s \geq r \vee r'}
e^{-2\lambda(t \vee s)+2\lambda(r \vee r')}\\
&\leq \frac{C''}{\lambda}
\int_0^T \d r\int_0^T \d r'
e^{-2\lambda(r \vee r')}
(C_{g,2}(r,r')-C_{g,1}(r,r'))^2
=\frac{C''}{\lambda}d_\lambda(C_{g,1},C_{g,2})^2.
\end{align*}
Thus for any $\alpha \in (0,1)$ and all sufficiently large $\lambda$,
$d_\lambda(C_{\theta,1},C_{\theta,2}) \leq
\alpha \cdot d_\lambda(C_{g,1},C_{g,2})$.

The analysis of the response $R_\theta$ is similar: Fixing $s \in [0,T]$,
there exist directional derivative processes
$D_kJ_h(\cdot,s)$
and $D_{kl}^2J_h(\cdot,s)$ over $t \in [s,T]$ for which
\begin{align*}
\lim_{\eps \to 0}
\E_{\{b^t\}}\left(\frac{J_{h+\eps k}(t,s)-J_h(t,s)}{\eps}
-D_kJ_h(t,s)\right)^2&=0,\\
\qquad \lim_{\eps \to 0}
\E_{\{b^t\}}\left(\frac{D_kJ_{h+\eps l}(t,s)-D_kJ_h(t,s)}{\eps}
-D_{kl}^2J_h(t,s)\right)^2&=0.
\end{align*}
These are the unique solutions of the equations
\begin{align*}
D_kJ_h(t,s)
    &=
    \int_s^t
    \bigg[
        f''(\theta_h^a)(D_k\theta_h^a)J_h(a,s)
        +
        f'(\theta_h^a)D_kJ_h(a,s)
        +
        \int_s^a R_g(a,b)D_kJ_h(b,s)\d b
    \bigg]\d a,\\
    D_{kl}^2J_h(t,s)
    &=
    \int_s^t
    \bigg[
        \Bigl(
            f'''(\theta_h^a)(D_k\theta_h^a)(D_l\theta_h^a)
            +
            f''(\theta_h^a)D_{kl}^2\theta_h^a
        \Bigr)J_h(a,s)
        +
        f''(\theta_h^a)(D_k\theta_h^a)D_lJ_h(a,s)
        \\
        &\qquad
        +
        f''(\theta_h^a)(D_l\theta_h^a)D_kJ_h(a,s)
        +
        f'(\theta_h^a)D_{kl}^2J_h(a,s)
        +
        \int_s^a R_g(a,b) D_{kl}^2J_h(b,s) \d b
    \bigg]\d a .
\end{align*}
Analogously to \eqref{eq:second-deriv-representation},
we may represent these solutions as
\[
    D_kJ_h(t,s)
    =
    \int_0^T M_h(t,s;r)k^r\,\d r,\qquad
    D_{kl}^2J_h(t,s)
    =
    \int_0^T\int_0^T
    L_h(t,s;r,r')k^rl^{r'}\,\d r\,\d r',
\]
where $M_h,L_h$ are the response processes
\begin{align*}
M_h(t,s;r)
    &=
    \int_s^t
    \bigg[
        f''(\theta_h^a)J_h(a,r)J_h(a,s)
        +
        f'(\theta_h^a)M_h(a,s;r)
        +
        \int_s^a R_g(a,b)M_h(b,s;r)\d b
    \bigg]\d a,\\
    L_h(t,s;r,r')
    &=
    \int_s^t
    \bigg[
        \Bigl(
            f'''(\theta_h^a)J_h(a,r)J_h(a,r')
            +
            f''(\theta_h^a)K_h(a,r,r')
        \Bigr)J_h(a,s)
        +
        f''(\theta_h^a)J_h(a,r)M_h(a,s;r')
        \\
        &\hspace{0.4in}
        +
        f''(\theta_h^a)J_h(a,r')M_h(a,s;r)
        +
        f'(\theta_h^a)L_h(a,s;r,r')
        +
        \int_s^a R_g(a,b)L_h(b,s;r,r')\d b
    \bigg]\d a.
\end{align*}
Applying the finite-dimensional Gaussian interpolation formula, we have
\begin{align*}
    \frac{\d}{\d u}\E J_{g_{u,N}}(t,s)
    &=
    \frac12
    \sum_{i,j=1}^N
    \bigl\langle (Q_1-Q_0)e_i,e_j\bigr\rangle_{\cH}
    \,
    \E\left[
        \partial_\eps \partial_{\eps'}J_{g_{u,N}+\eps e_i+\eps'
e_j}(t,s)\big|_{\eps=\eps'=0}
    \right]
    \\
    &=
    \frac12
    \sum_{i,j=1}^N
    \bigl\langle (Q_1-Q_0)e_i,e_j\bigr\rangle_{\cH}
    \,
    \E\left[
        \int_0^T\int_0^T
        L_{g_{u,N}}(t,s;r,r')
        e_i^re_j^{r'}
        \,\d r\,\d r'
    \right].
\end{align*}
Integrating over $u \in [0,1]$, taking the limit $N \to \infty$, and applying
that $L_{g_u}(t,s;r,r')$ is uniformly bounded and is equal to 0 unless
$r,r' \leq t$, we get
\[\begin{aligned}
|(R_{\theta,2}-R_{\theta,1})(t,s)|
&=\left|\lim_{N \to \infty} \E J_{g_{1,N}}(t,s)-\E J_{g_{0,N}}(t,s)\right|\\
&\leq C \int_0^t\int_0^t
|(C_{g,2}-C_{g,1})(r,r')|\d r\,\d r'.
\end{aligned}\]
Then, for constants $C,C'>0$ depending on $T$ but not $\lambda$,
\begin{align*}
d_\lambda(R_{\theta,1},R_{\theta,2})
&=\sup_{0 \leq s \leq t \leq T}
e^{-\lambda t}|R_{\theta,1}(t,s)-R_{\theta,2}(t,s)|\\
&\leq C\sup_{t \in [0,T]}
\int_0^t\int_0^t
e^{-\lambda t+\lambda(r \vee r')}
e^{-\lambda (r \vee r')}|C_{g,2}(r,r')-C_{g,1}(r,r')|\d r\,\d r'\\
&\leq C\sup_{t \in [0,T]}
\left(\int_0^t \int_0^t
e^{-2\lambda t+2\lambda(r \vee r')} \d r\,\d r'\right)^{1/2}
d_\lambda(C_{g,1},C_{g,2}) \leq
\frac{C'}{\sqrt{\lambda}}d_\lambda(C_{g,1},C_{g,2}).
\end{align*}
Combining these bounds for $C_\theta$ and $R_\theta$,
for any $\alpha \in (0,1)$ and all large enough $\lambda>0$,
\begin{equation}\label{eq:contractionsameRg}
d_\lambda(Y_1,Y_2) \leq \alpha \cdot d_\lambda(C_{g,1},C_{g,2})
\end{equation}

The lemma then follows from applying \eqref{eq:contractionsameCg} to bound
$d_\lambda(\cT_{g \to \theta}(C_{g,1},R_{g,1}),\allowbreak
\cT_{g \to \theta}(C_{g,1},R_{g,2}))$, and then
\eqref{eq:contractionsameRg} to bound
$d_\lambda(\cT_{g \to \theta}(C_{g,1},R_{g,2}),\allowbreak
\cT_{g \to \theta}(C_{g,2},R_{g,2}))$.
\end{proof}

\begin{lemma}[Lipschitz continuity of $\cT_{\theta \to g}$]
\label{lem:contraction-Lipschitz}
    Let $Y_i=(C_{\theta,i},R_{\theta,i}) \in \cS_\theta^+(T,D)$ for $i=1,2$, and
let $X_i=(C_{g,i},R_{g,i})=\cT_{\theta\to g}(Y_i)$. Then there exists a constant
$C>0$ not depending on $\lambda$ such that for all sufficiently large
$\lambda>0$, 
    \[
    d_\lambda(X_1,X_2) \leq C\cdot d_\lambda(Y_1,Y_2).
    \]
\end{lemma}

\begin{proof}
We have
\[
\begin{aligned}
  C_{g,1}(t,s)-C_{g,2}(t,s)
  &=
  \sum_{p,q\ge 0}
  \kappa_{p+q+2}
  \bigl(
    R_{\theta,1}^{\ast p}\ast (C_{\theta,1}-C_{\theta,2})
    \ast \bar R_{\theta,1}^{\ast q}
  \bigr)(t,s) \\
  &\hspace{0.5in}
  +
  \sum_{p,q\ge 0}
  \kappa_{p+q+2}
  \bigl[
    \bigl(R_{\theta,1}^{\ast p}\ast C_{\theta,2}\ast \bar R_{\theta,1}^{\ast q}\bigr)(t,s)
    -
    \bigl(R_{\theta,2}^{\ast p}\ast C_{\theta,2}\ast \bar R_{\theta,2}^{\ast q}\bigr)(t,s)
  \bigr].
\end{aligned}
\]
Consequently, 
\(
  d_\lambda(C_{g,1},C_{g,2}) \leq I_C+I_R,
\)
where
\begin{align*}
  I_C
  &:=\sum_{p,q\ge 0}
  |\kappa_{p+q+2}| \cdot
  \big\|
    R_{\theta,1}^{\ast p}\ast (C_{\theta,1}-C_{\theta,2})
    \ast \bar R_{\theta,1}^{\ast q}\big\|_\lambda,\\
  I_R
  &:=\sum_{p,q\ge 0}
  |\kappa_{p+q+2}| \cdot
  \big\|\bigl(R_{\theta,1}^{\ast p}\ast C_{\theta,2}\ast \bar R_{\theta,1}^{\ast
q}\bigr)(t,s)-\bigl(R_{\theta,2}^{\ast p}\ast C_{\theta,2}\ast \bar
R_{\theta,2}^{\ast q}\bigr)\big\|_\lambda.
\end{align*}
We first bound \(I_C\). Note that
$|R_{\theta,1}(t,s)|\le \Phi_{R_\theta}(t-s)$ for all $0\le s\le t\le T$, and
\[\|\Phi_{R_\theta}^{\ast p}\|_{L^1([0,T])}
\le \frac{(\Phi_{R_\theta}(T) T)^{p}}{p!}.\]
For $p,q \geq 1$,
\begin{align*}
&\big\|
    R_{\theta,1}^{\ast p}\ast (C_{\theta,1}-C_{\theta,2})
    \ast \bar R_{\theta,1}^{\ast q}\big\|_\lambda^2\\
    &=\int_0^T \int_0^T e^{-2\lambda(t \vee s)}
  \left(\int_0^t \d t'\int_0^s \d s'
  R_{\theta,1}^{*p}(t,t')
  (C_{\theta,1}-C_{\theta,2})(t',s')
  R_{\theta,1}^{*q}(s,s')\right)^2\d t\, \d s\\
  &\leq C \int_0^T \int_0^T e^{-2\lambda(t \vee s)}
  \int_0^t \d t'\int_0^s \d s'
  R_{\theta,1}^{*p}(t,t')^2
  (C_{\theta,1}-C_{\theta,2})(t',s')^2
  R_{\theta,1}^{*q}(s,s')^2\d t\, \d s\\
  &\leq C\int_0^T \d t'\int_0^T \d s'
  e^{-2\lambda(t' \vee s')}
  (C_{\theta,1}-C_{\theta,2})(t',s')^2
  \int_{t'}^T  \d t \int_{s'}^T \d s
  R_{\theta,1}^{*p}(t,t')^2
  R_{\theta,1}^{*q}(s,s')^2\\
  &\leq C'\|C_{\theta,1}-C_{\theta,2}\|_\lambda^2 \|\Phi_{R_\theta}^{*p}\|_{L^1([0,T])}^2
  \|\Phi_{R_\theta}^{*q}\|_{L^1([0,T])}^2,
\end{align*}
where $C,C'>0$ are $T$-dependent constants independent of $p,q$. The same bound holds for $p=0$ and/or $q=0$. Thus
\begin{align*}
  I_C 
  &\le C\sum_{p,q\ge 0} |\kappa_{p+q+2}| \frac{(\Phi_{R_\theta}(T) T)^p}{p!} \frac{(\Phi_{R_\theta}(T) T)^q}{q!} \|C_{\theta,1}-C_{\theta,2}\|_\lambda
\leq C'd_\lambda(C_{\theta,1},C_{\theta,2})
\end{align*}
for constant $C,C'>0$ not depending on $\lambda$.

We now bound \(I_R\). First,
we have $|C_{\theta,2}(t,s)|
  \le C_{\theta,2}(t,t)^{1/2} C_{\theta,2}(s,s)^{1/2} \le 
  \Phi_{C_\theta}(T)$.
Next, for the product of response kernels, writing $t_0=t$ and $s_0=s$,
\begin{align*}
&\left|
  \prod_{a=0}^{p-1}R_{\theta,1}(t_a,t_{a+1})
  \prod_{b=0}^{q-1}R_{\theta,1}(s_b,s_{b+1})
  -
  \prod_{a=0}^{p-1}R_{\theta,2}(t_a,t_{a+1})
  \prod_{b=0}^{q-1}R_{\theta,2}(s_b,s_{b+1})
\right|\\
&\le e^{\lambda(t\vee s)}
\left(
    \sup_{0\le r\le q\le T}
    e^{-\lambda q}|R_{\theta,1}(q,r)-R_{\theta,2}(q,r)|
  \right)
\sum_{\ell=1}^{p+q}
\prod_{m\ne \ell} \Phi_{R_\theta}(\xi_m),
\end{align*}
where \(\{\xi_m\}_{m=1}^{p+q}\) denotes the collection of increments
  $t_0-t_1,\allowbreak\ldots,\allowbreak t_{p-1}-t_p,\allowbreak
  s_0-s_1,\allowbreak\ldots,\allowbreak s_{q-1}-s_q$.
(The empty product is interpreted as \(1\).) 
Hence, integrating over $t_1,\ldots,t_p$ and $s_1,\ldots,s_q$,
\[
  I_R
  \le
  Cd_\lambda(R_{\theta,1},R_{\theta,2})
  \int_0^T\int_0^T
    \sum_{p,q\ge 0}
    |\kappa_{p+q+2}|(p+q)\Phi_{R_\theta}(T)^{p+q-1}
    \frac{t^p}{p!}\frac{s^q}{q!}\,
  \d t \d s.
\]
This shows
$I_R \le C d_\lambda(R_{\theta,1}, R_{\theta,2})$ for a constant $C>0$ not
depending on $\lambda$. Therefore,
\begin{align*}
   d_{\lambda}(C_{g,1},C_{g,2}) \le C d_\lambda(Y_1,Y_2).
  \end{align*}
  By \eqref{eq:dmft-Rg} and a similar argument omitted for brevity,
  $d_{\lambda}(R_{g,1}, R_{g,2}) \le Cd_{\lambda}(R_{\theta,1}, R_{\theta,2})$.
Thus
    $d_\lambda(X_1,X_2) \leq Cd_\lambda(Y_1,Y_2)$.
\end{proof}

Fixing $T>0$, define the projection
\[\cP_D:\cS_g(T,D) \to \cS_g^+(T,D),
\qquad (C_g,R_g) \mapsto (C_g^+,R_g)\]
that preserves $R_g$ and projects $C_g$ onto the subset $\cS_g^+(T,D)$ of the positive-semidefinite cone
\begin{equation}\label{eq:hilbertprojection}
C_g^+=\operatorname{arg\,min}_{C \in \cS_g^+(T,D)} \|C-C_g\|_\lambda
\end{equation}
in the weighted $L^2$-metric $d_\lambda$. We note that if $(C_g^{(m)},R_g^{(m)}) \in \cS_g^+(T,D)$
for $m=1,2,\ldots$ is a Cauchy sequence under $d_\lambda$,
then by completeness of the weighted $L^\infty$-metric
$d_\lambda(R_{g,1},R_{g,2})$, there exists $R_g:[0,T] \times [0,T] \to \R$ for
which
\begin{equation}\label{eq:Rgcauchy}
\lim_{m \to \infty} \sup_{s,t \in [0,T]} |R_g^{(m)}(t,s)-R_g(t,s)|=0.
\end{equation}
By completeness of the weighted $L^2$-metric $d_\lambda(C_{g,1},C_{g,2})$,
there also exists $C_g:[0,T] \times [0,T] \to \R$ for which
$\|C_g\|_\lambda<\infty$
and $\lim_{m \to \infty} d_\lambda(C_g^{(m)},C_g)=0$. 
Then the continuity condition \eqref{eq:Cg-cont2} for each kernel $C_g^{(m)}$
implies there exists
a modification $\tilde C_g:[0,T] \times [0,T] \to \R$ of $C_g$ (equal to $C_g$
Lebesgue-a.e.) for which
\begin{equation}\label{eq:Cgcauchy}
\lim_{m \to \infty} \sup_{s,t \in [0,T]} |C_g^{(m)}(t,s)-\tilde C_g(t,s)|=0.
\end{equation}
The statements \eqref{eq:Rgcauchy} and \eqref{eq:Cgcauchy} imply that $R_g$ and $\tilde C_g$ satisfy all conditions of
$\cS_g^+(T,D)$, so $(\tilde C_g,R_g) \in \cS_g^+(T,D)$,
and $d_\lambda((C_g^{(m)},R_g^{(m)}),(\tilde C_g,R_g)) \to 0$.
Then $\cS_g^+(T,D)$ is a
closed, convex subset of $\cS_g(T,D)$ under $d_\lambda$, so the
projection \eqref{eq:hilbertprojection} is well-defined. Furthermore, by the
non-expansive property of projections in a Hilbert space,
\begin{equation}\label{eq:projectioncontractive}
d_\lambda(C_{g,1}^+,C_{g,2}^+) \leq d_\lambda(C_{g,1},C_{g,2}).
\end{equation}

Lemma \ref{lemma:well_defined} then implies that
$\cT_{g \to g}^D:=\cP_D \circ \cT_{\theta \to g} \circ \cT_{g \to \theta}$
is a well-defined mapping from $\cS_g^+(T,D)$ to itself. The following corollary
collects the preceding results to show that this mapping is contractive in
$d_\lambda$ and has a unique fixed point in $\cS_g^+(T,D)$.
 
\begin{corollary}\label{cor:continuous-dmft-existunique-with-projection}
Fix any $T>0$ and finite subset $D \subset [0,T]$. Then
$\cT_{g \to g}^D:=\cP_D \circ \cT_{\theta \to g} \circ \cT_{g \to \theta}$
has a unique fixed point in $\cS_g^+(T,D)$.
\end{corollary}
\begin{proof}
The preceding argument verifies that 
$\cS_g^+(T,D)$ is complete under the metric $d_\lambda$. Lemmas
\ref{lem:contraction-Rtheta} and \ref{lem:contraction-Lipschitz}, together with
\eqref{eq:projectioncontractive}, show that
$\cT_{g\to g}^D$ is a contraction on $\cS_g^+(T,D)$ with respect to $d_\lambda$,
for any sufficiently large $\lambda>0$. Thus it has a unique fixed point by the
Banach fixed point theorem.
\end{proof}

\subsection{Discrete-time dynamical mean-field limit}

We next prove a dynamical mean-field limit for a discretized version of the
dynamics, and identify continuous-time embeddings of its correlation and response
kernels as fixed points of a discretized version of the preceding mappings.

Fix a discretization step size $\eta>0$. Given $\X$, $\btheta^0$, and the
Brownian motion $\{\b^t\}_{t \geq 0}$ defining \eqref{eq:dynamics},
set $\b_\eta^k=\b^{k\eta}$ and define the discretized dynamics 
\begin{equation}\label{eq:disc-Langevin}
\btheta^{k+1}_\eta-\btheta^{k}_\eta=\eta[f(\btheta^{k}_\eta)+\X\btheta^{k}_\eta]
+ \sqrt{2\gamma} (\b^{k+1}_\eta - \b^{k}_\eta),
\qquad \btheta_\eta^0=\btheta^0.
\end{equation}
For its mean-field limit, let $\theta^0$ and $\{b^t\}_{t \geq 0}$ be as in
Definition \ref{def:DMFT}, and set
\[\theta_\eta^0=\theta^0, \qquad b_\eta^k=b^{k\eta}.\]
Given the processes $\{\theta_\eta^i\}_{0 \leq i \leq k}$
and $\{\frac{\partial \theta_\eta^i}{\partial g_\eta^j}\}_{0 \leq j<i \leq k}$
up to step $k$, define correlation and response matrices 
$\C_\theta,\bR_\theta,\C_g,\bR_g \in \R^{(k+1) \times (k+1)}$ by
\begin{equation}\label{eq:discreteCR}
\begin{gathered}
\C_\theta=\Big(\E[\theta^{i}_\eta\theta^{j}_\eta]\Big)_{0 \leq i,j \leq k},
\qquad \bR_\theta=\Big(\1_{i>j}\E\Big[\frac{\partial
\theta^i_\eta}{\partial g^j_\eta}\Big]\Big)_{0 \leq i,j \leq k},\\
\C_g=\sum_{p,q\ge 0} \kappa_{p+q+2}
\bR_\theta^p\C_\theta\bR_\theta^{q\top},
\qquad \bR_g=\sum_{p \geq 1} \kappa_{p+1}\bR_\theta^p.
\end{gathered}
\end{equation}
Here, the matrices $\bR_\theta,\bR_g$ are strictly lower-triangular.
We will verify in Theorem \ref{thm:discreteDMFT} below that
$\C_g$ is positive-semidefinite. Then, letting
\begin{equation}\label{eq:discreteg}
(g^0_\eta,\ldots,g^k_\eta) \sim \cN(0,\C_g)
\end{equation}
be independent of $\theta^0$ and $\{b^t\}_{t \geq 0}$, and writing
$\bR_g=(r_{ij})_{0 \leq i,j \leq k}$ for the entries of $\bR_g$, define
the processes up to step $k+1$ by
\begin{align}
\theta^{k+1}_\eta
&=\theta^{k}_\eta + \eta\Big[f(\theta^{k}_\eta)+\sum_{q=0}^k r_{kq}\theta^{q}_\eta+g^{k}_\eta\Big] + \sqrt{2\gamma} (b^{k+1}_\eta - b^{k}_\eta),\label{eq:DMFTdiscrete2}\\ 
\frac{\partial \theta^{k+1}_\eta}{\partial g^{j}_\eta}  
&=\frac{\partial \theta^{k}_\eta}{\partial g^{j}_\eta}+\eta\Big[f'(\theta^{k}_\eta)\frac{\partial \theta^{k}_\eta}{\partial g^{j}_\eta}
+ \sum_{q=j+1}^k r_{kq}\frac{\partial \theta^{q}_\eta}{\partial
g^{j}_\eta}\Big] \text{ for } 0 \leq j < k,
\qquad \frac{\partial \theta^{k+1}_\eta}{\partial
g^{k}_\eta}=\eta.\label{eq:DMFTdiscreteresponse}
\end{align}
These definitions (\ref{eq:discreteCR}--\ref{eq:DMFTdiscreteresponse})
are understood iteratively in time for $k=0,1,2,\ldots$, and
\eqref{eq:DMFTdiscreteresponse} is the usual derivative
of \eqref{eq:DMFTdiscrete2} in $g_\eta^j$ (by the chain rule). 
Note that $\C_\theta,\bR_\theta,\C_g,\bR_g$ up to each step $k \geq 0$
are the upper-left corners of those up to step $k+1$; thus we may write
unambiguously $(\C_\theta)_{ij},(\bR_\theta)_{ij}$ etc.\ for their entries.

We denote also
\begin{equation}\label{eq:discretegk}
\g^{k}_\eta = \X \btheta^{k}_\eta - \sum_{q=0}^k r_{kq} \btheta^{q}_\eta.
\end{equation}

\begin{theorem}\label{thm:discreteDMFT}
The matrix $\C_g$ is positive-semidefinite. 
For any fixed integer $K \ge 0$, almost surely as $n \to \infty$,
\begin{align}\label{eq:AMP-SE}
&\frac{1}{n}\sum_{i=1}^n
\delta_{(\theta^{0}_{\eta,i},\ldots,\theta^{K}_{\eta,i},g^{0}_{\eta,i},\ldots,g^{K}_{\eta,i},b^{0}_{\eta,i},\ldots,b^{K}_{\eta,i})}
\to \Law(\theta^{0}_\eta,\ldots,\theta^{K}_\eta, g^{0}_\eta,\ldots,g^{K}_\eta,
b^{0}_\eta,\ldots,b^{K}_\eta)
\end{align}
weakly and in Wasserstein-2, where the right side denotes the joint law under
(\ref{eq:discreteg}--\ref{eq:DMFTdiscrete2}). Furthermore,
for any fixed $0 \leq j \leq k$, almost surely
\begin{equation}\label{eq:discretekernelconvergence}
\frac{1}{n}\langle \btheta_\eta^{k\top}\btheta_\eta^j \rangle
\to (\C_\theta)_{k,j},
\quad \frac{1}{n}\langle \g_\eta^{k\top}\g_\eta^j \rangle
\to (\C_g)_{k,j},
\quad
\frac{1}{n}\langle \btheta_\eta^{k\top}
\b_\eta^j \rangle \to \sqrt{2\gamma}\sum_{q=0}^{j-1}(\bR_\theta)_{k,q}.
\end{equation}
\end{theorem}

\begin{proof}
Set $\v^0=\btheta^0$ and $\v^k=\btheta^{k}_\eta-\btheta^{k-1}_\eta$, so that
$\btheta^{k}_\eta=\v^0+\ldots+\v^k$. 
It is readily checked that for any deterministic scalar
values $\{a_{jk}\}_{0 \leq j \leq k}$,
the dynamics \eqref{eq:disc-Langevin} may be equivalently written as
\begin{equation}\label{eq:discreteAMP}
\begin{aligned}
\z^k&=\X\v^k-\sum_{j=0}^k a_{jk}\v^j,\\
\v^{k+1}&=\underbrace{\eta\left[f\Big(\sum_{j=0}^k \v^j\Big)+\sum_{j=0}^k
\Big(\sum_{i=0}^j a_{ij}\v^i+\z^j\Big)\right] +
\sqrt{2\gamma}(\b_\eta^{k+1}-\b_\eta^k)}_{:=v_{k+1}(\z^0,\ldots,\z^k,\v^0,\b_\eta^1,\ldots,\b_\eta^{k+1})},
\end{aligned}
\end{equation}
where $v_{k+1}:\R^{2k+3} \to \R$ is a scalar function defined recursively for
$k=0,1,2,\ldots$ so that
$\v^{k+1}=v_{k+1}(\z^0,\ldots,\z^k,\v^0,\b_\eta^1,\ldots,\b_\eta^{k+1})$ holds.
Since $f:\R \to \R$ is Lipschitz and continuously-differentiable,
each function $v_{k+1}$ is also Lipschitz and continuously-differentiable
in all arguments.

We define iteratively in time a corresponding AMP state evolution:
Let $v^0=\theta^0$. Given the law of Gaussian variables $(z^0,\ldots,z^{k-1})$
independent of $\theta^0$ and $\{b^t\}_{t \geq 0}$, define
\begin{align*}
\bPhi
&=\Big(\1_{i<j}\E\Big[
\frac{\partial}{\partial z^i}
v_j(z^0,\ldots,z^{j-1},v^0,b_\eta^1,\ldots,b_\eta^j)
\Big]
\Big)_{0 \leq i,j \leq k},\\
\bDelta
&=\Big(
\E[
v_i(z^0,\ldots,z^{i-1},v^0,b_\eta^1,\ldots,b_\eta^i)
v_j(z^0,\ldots,z^{j-1},v^0,b_\eta^1,\ldots,b_\eta^j)]
\Big)_{0 \leq i,j \leq k}
\end{align*}
where $v_i(z^0,\ldots,z^{i-1},v^0,b_\eta^1,\ldots,b_\eta^i) \equiv v^0$ for
$i=0$. Then define the AMP Onsager matrix and state evolution covariance by
\[
\A=\sum_{p \geq 1} \kappa_{p+1}\bPhi^p,
\qquad
\bSigma=\sum_{p,q \geq 0} \kappa_{p+q+2}
\bPhi^{p\top}\bDelta\bPhi^q,
\]
let $(z^0,\ldots,z^k) \sim \cN(0,\bSigma)$ (independent of $\theta^0$ and
$\{b^t\}_{t \geq 0}$), and let
\begin{equation}\label{eq:vkdef}
v^k=v_k(z^0,\ldots,z^{k-1},v^0,b_\eta^1,\ldots,b_\eta^k).
\end{equation}
Note that $\A$ is strictly upper-triangular, and the matrices
$\bPhi,\bDelta,\A,\bSigma$ up to step $k$ are the upper-left corners of those up
to step $k+1$.
Writing $\A=(a_{ij})_{0 \leq i,j \leq k}$ and choosing the
coefficients $a_{ij}$ in \eqref{eq:discreteAMP} to be the entries of $\A$,
\cite[Theorem 4.3]{fan2022approximate} then implies that $\bSigma$
above is positive-semidefinite (hence the above law of $(z^0,\ldots,z^k)$
is well-defined), and for any fixed $K$, almost surely in Wasserstein-2,
\begin{align}\label{eq:AMPconvergence}
\lim_{n \to \infty}
\frac{1}{n}\sum_{i=1}^n
\delta_{(v_i^0,\ldots,v_i^K,z_i^0,\ldots,z_i^K,b_{\eta,i}^0,\ldots,b_{\eta,i}^K)}
\to 
\Law\left(v^0,\ldots,v^K,z^0,\ldots,z^K,b_\eta^0,\ldots,b_\eta^K\right).
\end{align}

To check that this implies \eqref{eq:AMP-SE},
recall $\btheta_\eta^k=\v^0+\ldots+\v^k$, define
$\g_\eta^k=\z^0+\ldots+\z^k$, and define also
from the above AMP state evolution variables
\begin{align*}
\theta^{j}_\eta:=v^0+\ldots+v^j,
\qquad g^{j}_\eta:=z^0+\cdots+z^j.
\end{align*}
Denote the upper-triangular
matrix $\bE=(\1_{i\leq j})_{0\leq i,j\leq k}$. Comparing the above definitions
of $\bPhi,\bDelta,\A,\bSigma$ with the
definitions (\ref{eq:discreteCR}), we then have
\[\bR_\theta^\top=\bE^{-1}\bPhi\bE,
\qquad \C_\theta=\bE^\top\bDelta\bE,
\qquad \bR_g^\top=\bE^{-1}\A\bE,
\qquad \C_g=\bE^\top\bSigma\bE\]
and $(g_\eta^0,\ldots,g_\eta^k) \sim \cN(0,\C_g)$. In particular, $\C_g$ is
positive-semidefinite as claimed. Noting that $\A\bE=\bE\bR_g^\top$
and writing $\bR_g=(r_{ij})_{0 \leq i,j \leq k}$,
the recursions for $\z^k$ and $\v^{k+1}$ in \eqref{eq:discreteAMP} then imply
\begin{align*}
\g_\eta^k&=\X(\v^0+\ldots+\v^k)-\sum_{j=0}^k r_{kj}(\v^0+\ldots+\v^j)
=\X\btheta_\eta^k-\sum_{j=0}^k r_{kj}\btheta_\eta^j,\\
\btheta_\eta^{k+1}-\btheta_\eta^k
=\v^{k+1}&=\eta\Big[
 f(\v^0+\cdots+\v^k)
 +\sum_{j=0}^k r_{kj}(\v^0+\cdots+\v^j)
 +\g_\eta^k \Big]+ \sqrt{2\gamma}(\b_\eta^{k+1}-\b_\eta^k)\\
&=\eta\Big[f(\btheta_\eta^k)+\sum_{j=0}^k r_{kj}\btheta_\eta^j+\g_\eta^k
 \Big] + \sqrt{2\gamma}(\b_\eta^{k+1}-\b_\eta^k).
\end{align*}
Thus $\g_\eta^k$ coincides with the definition \eqref{eq:discretegk}, and we
have analogously from the definition of $v^{k+1}$ in \eqref{eq:vkdef} that
\begin{align*}
\theta_\eta^{k+1}-\theta_\eta^k
=v^{k+1}=\eta\Bigg[f(\theta_\eta^k)+\sum_{j=0}^k r_{kj}\theta_\eta^j
+g_\eta^k\Bigg] + \sqrt{2\gamma}(b_\eta^{k+1}-b_\eta^k),
\end{align*}
coinciding with \eqref{eq:DMFTdiscrete2}. Thus
$\{\theta_\eta^k,g_\eta^k,b_\eta^k\}_{k \geq 0}$ have the same joint law as defined in
(\ref{eq:discreteg}--\ref{eq:DMFTdiscrete2}), and the Wasserstein-2
convergence \eqref{eq:AMPconvergence} implies \eqref{eq:AMP-SE} as desired.

The first two statements of \eqref{eq:discretekernelconvergence} follow
immediately from \eqref{eq:AMP-SE} and dominated convergence to take
the expectation over $\{\b^t\}_{t \geq 0}$. Also from \eqref{eq:AMP-SE} and
dominated convergence, for any $j \leq k$,
\[
\frac{1}{n}\langle \btheta_\eta^{k\top}\b_\eta^j \rangle
\to \E [\theta^k_\eta b^j_\eta].\]
Denoting $w^q=b_\eta^{q+1}-b_\eta^q \sim \cN(0,\eta)$,
Gaussian integration-by-parts gives
\[
\E [\theta^k_\eta (b^{q+1}_\eta-b^q_\eta)] = \E [\theta^k_\eta w^q] = \eta\,\E
\Big[\frac{\partial \theta^k_\eta}{\partial w^q}\Big]
=\sqrt{2\gamma}\,\E\Big[\frac{\partial \theta^k_\eta}{\partial g^q_\eta}\Big]\]
where the last step holds by the form of the equation \eqref{eq:DMFTdiscrete2}.
Noting that $\E[\theta_\eta^k b_\eta^0]=0$ and summing over $q=0,\ldots,j-1$
shows the last statement of \eqref{eq:discretekernelconvergence}.
\end{proof}

%An important fact that deserves attention 
%from the above proof is that \begin{align}
%  \C_k &=\bE_k^\top\bSigma_k\bE_k
%=\sum_{p,q\ge 0} \kappa_{p+q+2}
%[(\bE_k^{-1}\bPhi_k\bE_k)^p]^\top(\bE_k^\top\bDelta_k\bE_k)
%(\bE_k^{-1}\bPhi_k\bE_k)^{q} \succeq 0,
%\end{align}
%as $\bSigma_k$ is positive semidefinite from the AMP state evolution, which is
%not a priori clear from the definition of $\C_k$ in \eqref{eq:RkCkupdate2}.
 
We next identify piecewise-constant continuous-time embeddings of the
discrete-time mean-field limit as fixed points of a discretized system of
mappings. Still fixing the discretization step size $\eta>0$,
for $t\geq 0$, define
\begin{align*}
  \lfloor t\rfloor
  = \max\{k\eta: k\in \mathbb Z,\ k\eta\leq t\},
\qquad \lceil t\rceil
  = \lfloor t\rfloor+\eta,
\qquad
  [t]
  = \lfloor t\rfloor/\eta \in \mathbb{Z}.
\end{align*}
We use the convention that $\int_s^t f(r)\d r=0$ if $t<s$.
Fixing $T>0$, define the discontinuity set
\[D_\eta=[0,T] \cap \eta\mathbb{Z}.\]
Let $\cS_\theta^{\eta,+} \subset \cS_\theta^+(T,D_\eta)$ and $\cS_g^\eta \subset
\cS_g(T,D_\eta)$ be the subsets of the spaces
$\cS_\theta^+(T,D_\eta)$, $\cS_g(T,D_\eta)$ defined in Section \ref{subsec:spaces},
for which the kernels have the piecewise-constant structure
\begin{equation}\label{eq:piecewiseconstant}
C(t,s)=C(\lfloor t \rfloor,\lfloor s \rfloor),
\qquad R(t,s)=\1_{t \geq s}\,R(\lfloor t \rfloor,\lfloor s \rfloor)
\end{equation}
for $(C,R)=(C_\theta,R_\theta)$ or $(C,R)=(C_g,R_g)$, respectively. We define a map
\begin{align}\label{eq:Tthetageta}
\cT_{\theta \to g}^\eta:\cS_\theta^{\eta,+} \to \cS_g^\eta,
\qquad (C_\theta,R_\theta) \mapsto (C_g,R_g)
\end{align}
by
\begin{align*}
R_g(t,s)&=\kappa_2R_\theta(t,s)
+\sum_{p \geq 2}\kappa_{p+1}\int_{\lceil s \rceil}^{\lfloor t \rfloor} \d t_1
\int_{\lceil s \rceil}^{\lfloor t_1 \rfloor} \d t_2
\ldots \int_{\lceil s \rceil}^{\lfloor t_{p-2} \rfloor} \d t_{p-1}\\
&\hspace{2in}R_\theta(t,t_1)R_\theta(t_1,t_2)\ldots R_\theta(t_{p-1},s),\\
C_g(t,s)&=\sum_{p,q \geq 0} \kappa_{p+q+2}\int_0^{\lfloor t \rfloor} \d t_1
\ldots \int_0^{\lfloor t_{p-1} \rfloor} \d t_p
\int_0^{\lfloor s \rfloor} \d s_1
\ldots \int_0^{\lfloor s_{q-1} \rfloor} \d s_q\\
&\hspace{2in}R_\theta(t,t_1)\ldots R_\theta(t_{p-1},t_p)
R_\theta(s,s_1)\ldots R_\theta(s_{q-1},s_q) C_\theta(t_p,s_q),
\end{align*}
where $t_0=t$ and $s_0=s$. Let
\[\cS_g^{\eta,+}=\big\{\cP_{D_\eta}(C_g,R_g):(C_g,R_g) \in
\cS_g^\eta\big\} \subset \cS_g^+(T,D_\eta)\]
be the image of $\cS_g^\eta$ under the preceding projection $\cP_{D_\eta}$.
We denote
\[\cT_{\theta \to g}^{\eta,+}:\cS_\theta^{\eta,+} \to \cS_g^{\eta,+},
\qquad \cT_{\theta \to g}^{\eta,+}=\cP_{D_\eta} \circ \cT_{\theta \to g}^\eta.\]

Next, we define a map
\begin{align}\label{eq:Tgthetaeta}
  \cT_{g \to \theta}^\eta:\cS_g^{\eta,+} \to \cS_\theta^{\eta,+},
\qquad (C_g,R_g) \mapsto (C_\theta,R_\theta)
\end{align}
by
\[C_\theta(t,s)=\E[\bar\theta^t\bar\theta^s],
\qquad R_\theta(t,s)=\E\left[\frac{\partial \bar\theta^t}{\partial g^s}\right]\]
where $\{\bar\theta^t\}_{t \geq 0}$ and $\{\frac{\partial \bar\theta^t}{\partial
g^s}\}_{t \geq s \geq 0}$ solve
\begin{align}
\bar\theta^t&=\bar\theta^0+\int_0^{\lfloor t\rfloor}
  \bigg[f(\bar\theta^r)
    +\int_0^{\lfloor r\rfloor} R_g(r,q)
    \bar\theta^q\,\d q+g^{\lfloor r \rfloor}\bigg]\d r+\sqrt{2\gamma}\,b^{\lfloor
t\rfloor}, \qquad \bar\theta^0=\theta^0,
\label{eq:piecewise_discrete_dmft_theta}\\
\frac{\partial \bar\theta^t}{\partial g^s}
  &=1+\int_{\lceil s\rceil}^{\lfloor t\rfloor}
  \bigg[
    f'(\bar\theta^r)
    \frac{\partial \bar\theta^r}{\partial g^s}
    +\int_{\lceil s\rceil}^{\lfloor r\rfloor}
    R_g(r,q)\frac{\partial \bar\theta^q}{\partial g^s}\,\d q
  \bigg]\d r \text{ for } t \geq s,
\quad \frac{\partial \bar\theta^t}{\partial g^s}=0 \text{ for } t<s,\label{eq:piecewise_discrete_dmft_response}
\end{align}
and $\{g^t\}_{t \in [0,T]} \sim \GP(0,C_g)$ in
\eqref{eq:piecewise_discrete_dmft_theta} which is independent of $\theta^0$ and
$\{b^t\}_{t \geq 0}$.

\begin{lemma}\label{lemma:discretemappings}
For any fixed $T>0$ and all sufficiently small $\eta>0$, the following hold:
\begin{enumerate}[(a)]
\item For any $(C_g,R_g) \in \cS_g^{\eta,+}$, we have $\cT_{g \to
\theta}^\eta(C_g,R_g) \in \cS_\theta^{\eta,+}$.
\item For any $(C_\theta,R_\theta) \in \cS_\theta^{\eta,+}$, we have $\cT_{\theta
\to g}^\eta(C_\theta,R_\theta) \in \cS_g^\eta$
and $\cT_{\theta \to g}^{\eta,+}(C_\theta,R_\theta) \in \cS_g^{\eta,+}$.
\item Let $\C_\theta,\bR_\theta,\C_g,\bR_g$ be the matrices
\eqref{eq:discreteCR} defined up to a step
$k>T/\eta$, and define their continuous-time embeddings
\[C_\theta^\eta(t,s)=(\C_\theta)_{[t],[s]},
\qquad C_g^\eta(t,s)=(\C_g)_{[t],[s]},\]
\[\begin{aligned}
R_\theta^\eta(t,s)&=\begin{cases}
0 & \text{ if } s>t\\
1 & \text{ if } s \leq t,\,[s]=[t] \\
\frac{1}{\eta}\,(\bR_\theta)_{[t],[s]} & \text{ if } [s]<[t]
\end{cases},\\
R_g^\eta(t,s)&=\begin{cases}
0 & \text{ if } s>t\\
\kappa_2 & \text{ if } s \leq t,\,[s]=[t] \\
\frac{1}{\eta}\,(\bR_g)_{[t],[s]} & \text{ if } [s]<[t].
\end{cases}
\end{aligned}
\]
Then $(C_g^\eta,R_g^\eta) \in \cS_g^{\eta,+} \cap \cS_g^\eta$ and
$(C_\theta^\eta,R_\theta^\eta) \in \cS_\theta^{\eta,+}$. Furthermore,
$\cT_{g \to \theta}^\eta(C_g^\eta,R_g^\eta)=(C_\theta^\eta,R_\theta^\eta)$
and $\cT_{\theta \to g}^\eta(C_\theta^\eta,R_\theta^\eta)=
\cT_{\theta \to g}^{\eta,+}(C_\theta^\eta,R_\theta^\eta)=(C_g^\eta,R_g^\eta)$.
\end{enumerate}
\end{lemma}
\begin{proof}
The proof is similar to \cite[Lemma 4.3]{fan2025dynamicalI}, so we will omit
some details.

Similar arguments as in Lemma \ref{lemma:well_defined}
show that $(C_g,R_g) \in \cS_g^{\eta,+}$
implies $\cT_{g \to \theta}^\eta(C_g,R_g) \in \cS_\theta^\eta(T,D_\eta)$,
and $(C_\theta,R_\theta) \in \cS_\theta^{\eta,+}$ implies
$\cT_{\theta \to g}^\eta(C_\theta,R_\theta) \in \cS_g(T,D_\eta)$. 
An induction in time checks that the processes solving
\eqref{eq:piecewise_discrete_dmft_theta}
and \eqref{eq:piecewise_discrete_dmft_response} must satisfy
\begin{equation}\label{eq:thetapiecewiseconstant}
\theta^t=\theta^{\lfloor t \rfloor},
\qquad \frac{\partial \theta^t}{\partial g^s}
=\1_{t \geq s}\frac{\partial \theta^{\lfloor t \rfloor}}{\partial g^{\lfloor s
\rfloor}}.
\end{equation}
Then $\cT_{g \to \theta}^\eta(C_g,R_g)$ satisfies \eqref{eq:piecewiseconstant},
so $\cT_{g \to \theta}^\eta(C_g,R_g)
\in \cS_\theta^{\eta,+}$. If $(C_\theta,R_\theta) \in \cS_\theta^{\eta,+}$
satisfies \eqref{eq:piecewiseconstant}, then the definition of
\eqref{eq:Tthetageta} also implies that
$\cT_{\theta \to g}^\eta(C_\theta,R_\theta)$ satisfies
\eqref{eq:piecewiseconstant}, so
$\cT_{\theta \to g}^\eta(C_\theta,R_\theta) \in \cS_g^\eta$
and hence (by definition) $\cT_{\theta \to g}^{\eta,+}(C_\theta,R_\theta) \in \cS_g^{\eta,+}$.
Thus parts (a) and (b) hold.

Letting $(C_\theta^\eta,R_\theta^\eta,C_g^\eta,R_g^\eta)$ be the embeddings of part (c),
an induction in time using the piecewise-constant properties
\eqref{eq:piecewiseconstant} and
\eqref{eq:thetapiecewiseconstant} also checks that
$\cT_{g \to \theta}^\eta(C_g^\eta,R_g^\eta)=(C_\theta^\eta,R_\theta^\eta)$
and $\cT_{\theta \to g}^\eta(C_\theta^\eta,R_\theta^\eta)=(C_g^\eta,R_g^\eta)$.
Theorem \ref{thm:discreteDMFT} shows that $\C_g$ is
positive-semidefinite; hence $C_g^\eta$ is also a
positive-semidefinite kernel on $[0,T] \times [0,T]$, so
$\cP_{D_\eta}(C_g^\eta)=C_g^\eta$. Then also
$\cT_{\theta \to g}^{\eta,+}(C_\theta^\eta,R_\theta^\eta)=(C_g^\eta,R_g^\eta)$.

Finally, similar arguments as in Lemmas \ref{lem:contraction-Rtheta}
and \ref{lem:contraction-Lipschitz} show that
$\cT_{g \to \theta}^\eta \circ \cT_{\theta \to g}^{\eta,+}:\cS_\theta^{\eta,+} \to \cS_\theta^{\eta,+}$ is contractive in
the metric $d_\lambda$ for all sufficiently large $\lambda$. Then, since
$\cS_\theta^{\eta,+}$ is also closed under $d_\lambda$,
this map has a unique fixed point $(\bar C_\theta^\eta,\bar R_\theta^\eta) \in
\cS_\theta^{\eta,+}$. Letting $(\bar C_g^\eta,\bar R_g^\eta)=
\cT_{\theta \to g}^{\eta,+}(\bar C_\theta^\eta,\bar R_\theta^\eta)
\in \cS_g^{\eta,+}$, an induction in time using \eqref{eq:Tthetageta} and
\eqref{eq:Tgthetaeta} shows that we must have
$(\bar C_\theta^\eta,\bar R_\theta^\eta)=(C_\theta^\eta,R_\theta^\eta)$
and $(\bar C_g^\eta,\bar R_g^\eta)=(C_g^\eta,R_g^\eta)$. Thus 
$(C_\theta^\eta,R_\theta^\eta) \in \cS_\theta^{\eta,+}$ and
$(C_g^\eta,R_g^\eta)\in \cS_g^{\eta,+} \cap \cS_g^{\eta}$,
showing part (c).
\end{proof}

\subsection{Proof of Theorem \ref{thm:dmft-approx}}

We proceed to show Theorem \ref{thm:dmft-approx}. The following lemma bounding
the difference between the continuous-time and discretized mappings will be
used to show part (a).

\begin{lemma}\label{lem:DMFT-mapping-discretization-distance}
There exists a constant $C>0$ depending on $T$ but not on $\lambda,\eta$
such that for all sufficiently large $\lambda>0$ and
sufficiently small $\eta>0$,
\begin{enumerate}[(a)]
\item For any $(C_g,R_g) \in \cS_g^{\eta,+}$,
\[d_\lambda\big(\cT_{g\to \theta}(C_g,R_g),\cT_{g\to \theta}^\eta(C_g,R_g)\big)
\le C\sqrt{\eta}.\]
\item For any $(C_\theta,R_\theta) \in \cS_\theta^{\eta,+}$,
\[d_\lambda\big(\cT_{\theta\to g}(C_\theta,R_\theta),\cT_{\theta\to
g}^{\eta}(C_\theta,R_\theta)\big) \le C\eta.\]
\end{enumerate}
\end{lemma}
\begin{proof}
Given $(C_g,\allowbreak R_g) \in \cS_g^{\eta,+}$, let
$(C_\theta,\allowbreak R_\theta)=\cT_{g \to
\theta}(C_g,R_g)$ and $(\bar C_\theta,\bar R_\theta)=\cT_{g \to
\theta}^\eta(C_g,R_g)$.
By arguments similar to \cite[Lemma 4.4]{fan2025dynamicalI}, there exists a
coupling of $(\{\theta^t\}_{t \geq 0},\allowbreak\{\frac{\partial \theta^t}{\partial
g^s}\}_{t \geq s \geq 0})$ solving (\ref{eq:dmft-theta}--\ref{eq:dmft-response})
and $(\{\bar\theta^t\}_{t \geq 0},\allowbreak\{\frac{\partial \bar\theta^t}{\partial
g^s}\}_{t \geq s \geq 0})$ solving
(\ref{eq:piecewise_discrete_dmft_theta}--\ref{eq:piecewise_discrete_dmft_response})
for which, for all large $\lambda>0$ and small $\eta>0$,
\[\sup_{t \in [0,T]} e^{-\lambda t} \sqrt{\E(\theta^t-\bar \theta^t)^2}
\leq C\sqrt{\eta},
\qquad \sup_{0 \leq s \leq t \leq T}
e^{-\lambda t}\E\left|\frac{\partial \theta^t}{\partial g^s}
-\frac{\partial \bar\theta^t}{\partial g^s}\right| \leq C\sqrt{\eta}.\]
The second inequality implies $d_\lambda(R_\theta,\bar R_\theta) \leq
C\sqrt{\eta}$, while the first inequality and $\E(\theta^t)^2 \leq \Phi_{C_\theta}(T)$
imply
\begin{align*}
&\sup_{s,t \in [0,T]} e^{-\lambda (s \vee t)}|C_\theta(s,t)-\bar C_\theta(s,t)|
=\sup_{s,t \in [0,T]} e^{-\lambda (s \vee
t)}|\E(\theta^s\theta^t)-\E(\bar\theta^s\bar \theta^t)|\\
&\leq \sup_{s,t \in [0,T]} e^{-\lambda s}\sqrt{\E(\theta^s-\bar \theta^s)^2}
\sqrt{\E(\theta^t)^2}
+e^{-\lambda t}\sqrt{\E(\theta^t-\bar \theta^t)^2}
\sqrt{\E(\bar\theta^s)^2}
\leq C'\sqrt{\eta}.
\end{align*}
Thus part (a) follows.

For part (b), given $(C_\theta,R_\theta) \in \cS_\theta^{\eta,+}$, let
$(C_g,R_g)=\cT_{\theta\to g}(C_\theta,R_\theta)$
and $(\bar C_g,\bar R_g)=\cT_{\theta\to g}^\eta(C_\theta,R_\theta)$. According
to \eqref{eq:dmft-Rg} and \eqref{eq:Tthetageta}, we know that
\[R_g-\bar R_g=\sum_{p \geq 2} \kappa_{p+1}\Big[R_\theta^{\ast p}-
R_\theta^{\lfloor\ast\rfloor p}\Big]\]
where, denoting $t=s_0$ and $s=s_p$, we have
\begin{align*}
R^{\ast p}(t,s)&= \int_{D_p(t,s)} \prod_{a=0}^{p-1}R(s_a,s_{a+1})\, \d
s_1\cdots\d s_{p-1},\\
&\hspace{1in}D_p(t,s)=\left\{(s_1,\ldots,s_{p-1}):s \leq s_{p-1}\leq\cdots\leq s_1\leq t
\right\},\\
R^{\lfloor\ast\rfloor p}(t,s)&=\int_{D_{p,\eta}(t,s)}\prod_{a=0}^{p-1}R(s_a,s_{a+1})\, \d s_1\cdots\d s_{p-1},
\hspace{0.3in}D_{p,\eta}(t,s)=
\left\{\begin{aligned}
&(s_1,\ldots,s_{p-1}):s_{a+1} \leq \lfloor s_a\rfloor\\
&\text{ for } a=0,\ldots,p-1
\end{aligned}\right\}.
\end{align*}
Since $D_{p,\eta}(t,s)\subseteq D_p(t,s)$ and
$|R_\theta(t,s)|\le \Phi_{R_\theta}(T)$, we have the entrywise bound
\[
\left|R_\theta^{\ast p}(t,s) - R_\theta^{\lfloor \ast \rfloor p}(t,s) \right|
\leq \Phi_{R_\theta}(T)^p\, \left|D_p(t,s)\setminus D_{p,\eta}(t,s)\right|.
\]
If a point belongs to $D_p(t,s)\setminus D_{p,\eta}(t,s)$, then for at least one $a\in\{0,\ldots,p-1\}$,
$s_{a+1}>\lfloor s_a\rfloor$,
and therefore
$0\leq s_a-s_{a+1}<\eta$.
Using a union bound over the $p$ possible gaps gives
\[
\left|D_p(t,s)\setminus D_{p,\eta}(t,s)\right| \leq \frac{p(t-s)^{p-2}\eta}{(p-2)!}.
\]
Therefore, 
\begin{align*}
d_\lambda(R_g, \bar R_g)
\le \sup_{0 \leq s \leq t \leq T} |R_g(t,s)-\bar R_g(t,s)|
\leq
%\sum_{p \geq 0} |\kappa_{p+1}|\, d_\lambda((\bar R_\theta^\eta)^{\ast p}, (\bar R_\theta^\eta)^{\lfloor\ast\rfloor p}) \le |\kappa_{2}|\Phi_{R_\theta}(T)\sqrt{T\eta} +
\sum_{p \geq 2} |\kappa_{p+1}|\frac{p\Phi_{R_\theta}(T)^pT^{p-2}}{(p-2)!}\eta
\le C\eta.
\end{align*}
A similar argument shows $d_\lambda(C_g, \bar C_g) \le C\eta$.
\end{proof}

\begin{proof}[Proof of Theorem \ref{thm:dmft-approx}(a)]
Fix any $\eta>0$ small enough.
Let $X^\eta \in \cS_g^+(T,D_\eta)$ be the unique fixed
point of the projected continuous-time mapping
\[\cT_{g \to g}^\eta:=\cP_{D_\eta} \circ
\cT_{\theta \to g} \circ \cT_{g \to \theta},\]
as guaranteed by Corollary
\ref{cor:continuous-dmft-existunique-with-projection}.
This definition depends on $\eta$
through $\cP_{D_\eta}$.

Let $\bar X^\eta$ be the fixed point of the projected
discrete-time mapping
\[\bar \cT_{g \to g}^\eta:=
\cT_{\theta \to g}^{\eta,+} \circ \cT_{g \to \theta}^\eta
=\cP_{D_\eta} \circ \cT_{\theta \to g}^\eta \circ
\cT_{g \to \theta}^\eta\]
that is constructed in Lemma \ref{lemma:discretemappings}.
By the triangle inequality, we have
\begin{align*}
    d_{\lambda}(X^\eta, \bar X^\eta) &= d_{\lambda}(\cT_{g\to g}^\eta(X^\eta),
\bar\cT_{g\to g}^\eta(\bar X^\eta))\\
&\leq d_{\lambda}(\cT_{g\to g}^\eta(X^\eta),
\cT_{g\to g}^\eta(\bar X^\eta))+ d_{\lambda}(\cT_{g\to g}^\eta(\bar X^\eta),
\bar\cT^{\eta}_{g\to g}(\bar X^\eta)).
  \end{align*}
Lemmas
\ref{lem:contraction-Rtheta} and
\ref{lem:contraction-Lipschitz}
ensure for the first term that
$d_{\lambda}(\cT_{g\to g}^\eta(X^\eta),\allowbreak\cT_{g\to g}^\eta(\bar X^\eta))
<\frac{1}{2}d_\lambda(X^\eta,\bar X^\eta)$, for sufficiently large $\lambda>0$.
Applying contractivity of $\cP_{D_\eta}$ with respect to $d_\lambda(\cdot)$,
the triangle inequality, and Lemmas
\ref{lem:contraction-Lipschitz} and
\ref{lem:DMFT-mapping-discretization-distance}, the second term is bounded as
\begin{align}
&d_{\lambda}(\cT^{\eta}_{g\to g}(\bar X^\eta),\bar \cT_{g\to g}^\eta(\bar
X^\eta))\notag\\
&\leq d_{\lambda}(\cT_{\theta \to g} \circ \cT_{g \to \theta}(\bar X^\eta),
\cT_{\theta \to g}^\eta \circ \cT_{g \to \theta}^\eta(\bar X^\eta))\notag\\
&\leq d_{\lambda}(\cT_{\theta \to g} \circ \cT_{g \to \theta}(\bar X^\eta),
\cT_{\theta \to g} \circ \cT_{g \to \theta}^\eta(\bar X^\eta))
+d_{\lambda}(\cT_{\theta \to g} \circ \cT_{g \to \theta}^\eta(\bar X^\eta),
\cT_{\theta \to g}^\eta \circ \cT_{g \to \theta}^\eta(\bar X^\eta))\notag\\
&\leq C\sqrt{\eta}+C\eta \leq C'\sqrt{\eta}.\label{eq:TTetabound}
\end{align}
Thus, rearranging the above shows, for some $C>0$ not depending on $\eta$,
\begin{equation}\label{eq:XXetabound}
d_{\lambda}(X^\eta, \bar X^\eta) \leq C\sqrt{\eta}.
\end{equation}

By \eqref{eq:XXetabound} and Lemmas \ref{lem:contraction-Rtheta} and \ref{lem:contraction-Lipschitz},
\begin{align*}
    &d_{\lambda}(X^\eta,\cT_{\theta \to g} \circ \cT_{g \to \theta}(X^\eta))\\
    &\leq d_{\lambda}(X^\eta,\bar X^\eta) +
d_{\lambda}(\bar X^\eta,\cT_{\theta \to g} \circ \cT_{g \to \theta}(\bar X^\eta)) +
d_{\lambda}(\cT_{\theta \to g} \circ \cT_{g \to \theta}(\bar X^\eta), \cT_{\theta \to g} \circ \cT_{g \to
\theta}(X^\eta)) \\ 
    &\le C''\sqrt{\eta}+
d_{\lambda}(\bar X^\eta,\cT_{\theta \to g} \circ \cT_{g \to \theta}(\bar X^\eta))
  \end{align*}
Also by the identity $\bar X^\eta=\cT_{\theta \to g}^\eta \circ \cT_{g \to
\theta}^\eta(\bar X^\eta)$ (without projection) from 
Lemma \ref{lemma:discretemappings} and the bound for the second line of
\eqref{eq:TTetabound},
\[d_{\lambda}(\bar X^\eta,\cT_{\theta \to g} \circ \cT_{g \to \theta}(\bar
X^\eta))=d_{\lambda}(\cT_{\theta \to g}^\eta \circ \cT_{g \to \theta}^\eta(\bar
X^\eta),\cT_{\theta \to g} \circ \cT_{g \to \theta}(\bar X^\eta))
\leq C'\sqrt{\eta}.\]
So for some $C>0$ not depending on $\eta$,
\begin{equation}\label{eq:Xfixedpoint}
d_{\lambda}(X^\eta,\cT_{\theta \to g} \circ \cT_{g \to \theta}(X^\eta))
\leq C\sqrt{\eta}.
\end{equation}

 For any $\delta>0$, let $D_{\eta \cup \delta} = D_{\eta} \cup D_{\delta}$, and define $X^{\delta} \in \mathcal{S}^+(T,D_{\delta})$ analogously to $X^{\eta}$. Then $X^{\eta}, X^{\delta} \in \mathcal{S}_g^+(T,D_{\eta \cup \delta})$, and 
\begin{align}
   \nonumber d_{\lambda}(X^\eta,X^\delta) &\leq d_{\lambda}(X^\eta, \mathcal{T}_{\theta \to g} \circ \mathcal{T}_{g \to \theta}(X^\eta))+ d_{\lambda}(\mathcal{T}_{\theta \to g} \circ \mathcal{T}_{g \to \theta}(X^\delta),X^\delta)\\ \nonumber& \qquad \qquad \qquad+ d_{\lambda}(\mathcal{T}_{\theta \to g} \circ \mathcal{T}_{g \to \theta}(X^\eta),\mathcal{T}_{\theta \to g} \circ \mathcal{T}_{g \to \theta}(X^\delta))\\ \nonumber& \leq C(\sqrt{\eta}+\sqrt{\delta})+ \frac{1}{2}d_{\lambda}(X^\eta,X^{\delta}).
\end{align}
Consequently, $d_{\lambda}(X^{\eta},X^{\delta}) \leq C'(\sqrt{\eta}+\sqrt{\delta})$. This shows along any sequence of values $\eta>0$ converging to 0 that
$\{X^{\eta}\}$ is a Cauchy sequence under $d_\lambda$. Then by completeness of the ambient weighted $L^\infty$ and $L^2$ spaces over $[0,T] \times [0,T]$ as in \eqref{eq:Rgcauchy} and \eqref{eq:Cgcauchy}, there
exist $R_g,C_g:[0,T] \times [0,T] \to \R$ such that $X=(C_g,R_g)$ satisfies $d_{\lambda}(X,X^\eta) \leq C\sqrt{\eta}$. Together with \eqref{eq:Xfixedpoint}, we have $d_{\lambda}(X,\mathcal{T}_{\theta \to g}\circ \mathcal{T}_{g \to \theta}(X^\eta)) \leq C\sqrt{\eta}.$   By Lemma \ref{lemma:well_defined},
the image of $\cT_{\theta \to g} \circ \cT_{g \to \theta}$ is contained in
$\cS_g(T,\emptyset)$ with discontinuity set $D=\emptyset$. Then as in \eqref{eq:Cgcauchy}, there exists a continuous modification of $C_g$ so that $X=(C_g,R_g) \in \cS_g(T,\emptyset)$. Writing $X^\eta=(C_g^\eta,R_g^\eta)$, since also $\lim_{\eta \to 0} d_\lambda(X,X^\eta)=0$ and $C_g^\eta$ is positive-semidefinite, we have that $C_g$ is positive-semidefinite, so $X \in \mathcal{S}_g^+(T,\emptyset)$. 

Finally,
\begin{align*}
   \nonumber &d_{\lambda}(X, \mathcal{T}_{\theta \to g}\circ \mathcal{T}_{g \to \theta}(X))\\
    &\leq d_{\lambda}(X, X^\eta)
    +d_{\lambda}(X^{\eta}, \mathcal{T}_{\theta \to g}\circ \mathcal{T}_{g \to \theta}(X^\eta))+d_{\lambda}(\mathcal{T}_{\theta \to g}\circ \mathcal{T}_{g \to \theta}(X^\eta), \mathcal{T}_{\theta \to g}\circ \mathcal{T}_{g \to \theta}(X))\leq C\sqrt{\eta}.
\end{align*}
Since $\eta>0$ is arbitrary, this shows that $X \in \cS_g^+(T,\emptyset)$ is a fixed point of $\mathcal{T}_{\theta \to g}\circ \mathcal{T}_{g \to \theta}$. Uniqueness of such a fixed point in $\cS_g^+(T,\emptyset)$ is immediate as $\mathcal{T}_{\theta \to g}\circ \mathcal{T}_{g \to \theta}$ is contractive under $d_\lambda$.

This shows that the components of Theorem \ref{thm:dmft-approx}(a) are 
well-defined up to any fixed time $T>0$, where the fixed points
$(C_\theta,R_\theta)$ and $(C_g,R_g)$ are unique in
$\cS_\theta^+(T,\emptyset)$ and
$\cS_g^+(T,\emptyset)$. Then Theorem \ref{thm:dmft-approx}(a) holds upon setting
$\cS_\theta^+$ as the space of kernels
$(C_\theta,R_\theta)$ on $[0,\infty) \times [0,\infty)$ whose
restrictions to $[0,T] \times [0,T]$ belong to $\cS_\theta^+(T,\emptyset)$ for
each $T>0$, and similarly for $\cS_g^+$.
\end{proof}

In the remainder of the proof, we fix $T>0$,
$X=(C_g,R_g) \in \cS_g^+(T,\emptyset)$ as
the preceding unique fixed point of $\cT_{\theta \to g} \circ \cT_{g \to
\theta}$, and $\bar X^\eta=(C_g^\eta,R_g^\eta) \in \cS_g^+(T,D_\eta)$ as the
fixed point of $\cT_{\theta \to g}^\eta \circ \cT_{g \to
\theta}^\eta$. To show Theorem \ref{thm:dmft-approx}(b), we record the following
continuity properties of $C_g,C_g^\eta$.
 
\begin{lemma}\label{lem:discrete_dmft_approx_continue}
For all sufficiently small \(\eta>0\),
there is a constant $C>0$ depending on $T$ but not $\eta$, such that 
\begin{align}
  |C_g(t,s)-C_g(t',s')|
  &\le
  C\left(\sqrt{|t-t'|}+\sqrt{|s-s'|}\right),
  \label{eq:Cg-holder-cont}\\
  |C_g^\eta(t,s)-C_g^\eta(t',s')|
  &\le
  C\left(\sqrt{|t-t'|+\eta}+\sqrt{|s-s'|+\eta}\right)
  \label{eq:Cg-holder-discrete}
\end{align}
for all \(s,t,s',t'\in[0,T]\).
\end{lemma}
\begin{proof}
Since $(C_g,R_g) \in \cS_g^+(T,\emptyset)$, we have from \eqref{eq:Cg-cont2} that
\[
  |C_g(t,s)-C_g(t',s')|
  \leq
  C\left(\sqrt{|t-t'|}+\sqrt{|s-s'|}\right).
\]
For \eqref{eq:Cg-holder-discrete}, let $\{\bar \theta^t,\frac{\partial
\bar \theta^t}{\partial g^s}\}$ be the processes solving the discretized
equations
(\ref{eq:piecewise_discrete_dmft_theta}--\ref{eq:piecewise_discrete_dmft_response})
for the kernels $(C_g^\eta,R_g^\eta)$. By arguments analogous to those of
Lemma \ref{lemma:well_defined},
for some $C>0$ depending on $T$ but not $\eta$, we have  
\begin{align*}
|C_\theta^\eta(t,s)-C_\theta^\eta(t',s')|
&\leq C(\sqrt{|t-t'|+\eta}+\sqrt{|s-s'|+\eta}) \text{ for all } s,t \in
[0,T],\\
|R_\theta^\eta(t,s)-R_\theta^\eta(t',s)| &\le C(|t-t'|+\eta)
\text{ for all } s \leq t,t' \leq T,
\end{align*}
and hence also  
\[|C_g^\eta(t,s)-C_g^\eta(t',s')|
  \leq
  C\left(\sqrt{|t-t'|+\eta}+\sqrt{|s-s'|+\eta}\right).\]
\end{proof}

\begin{proof}[Proof of Theorem \ref{thm:dmft-approx}(b)]
Fix $\eta>0$, and consider a piecewise-constant embedding 
of the discretized Langevin process
\eqref{eq:disc-Langevin} defined as $\bar\btheta^t = \btheta_\eta^{[t]}$.
An induction in time checks that this solves
\begin{equation}\label{eq:langevin_discret}
\bar\btheta^t=\btheta^0+\int_0^{\lfloor{t}\rfloor} [f(\bar\btheta^s)
+\X\bar\btheta^s]\d s +\sqrt{2\gamma}\,\b^{\lfloor{t}\rfloor}
\end{equation}
where $\{\b^t\}_{t \geq 0}$ is the same Brownian motion as defining the
continuous process $\{\btheta^t\}_{t \geq 0}$.
Define also a piecewise constant embedding $\bar \g^t=\g_\eta^{[t]}$ of
\eqref{eq:discretegk}, and $\bar \b^t=\b_\eta^{[t]}=\b^{[t]\eta}$.
By \cite[Lemma 4.7]{fan2025dynamicalI}, for a constant $C>0$ independent of
$\eta$, a.s.\ for all large $n$,
\begin{align*}
\sup_{t\in[0,T]} \frac{1}{\sqrt{n}}\|\b^t\|_2&\leq C,
\\
\sup_{t\in[0,T]}
\frac{1}{\sqrt{n}}\big(\|{\b^t - \b^{\lfloor{t}\rfloor}}\|_2 + \|{\b^t -
\b^{\lceil{t}\rceil}}\|_2\big) &\leq C\sqrt{\eta\max(\log(1/\eta),1)}.
\end{align*}
Then by a Gr\"onwall argument (see \cite[Lemma 4.8]{fan2025dynamicalI}), a.s.\ for
all large $n$,
\begin{align}
\sup_{t\in[0,T]}
\frac{1}{\sqrt{n}}\|{\btheta^t}\|_2 \leq C,\qquad \sup_{t\in[0,T]}
\frac{1}{\sqrt{n}}\|{\btheta^t - \bar\btheta^{t}}\|_2 \leq
C\sqrt{\eta\max(\log(1/\eta),1)}\label{eq:theta-discret-norm}
\end{align}
For $\bar \g^t$, we have
\begin{align*}
    \|\g^t - \bar \g^t\|_2 &= \left\|\X(\btheta^t-\bar\btheta^t)-\left(\int_0^t
R_g(t,s)\btheta^s\d s - \int_0^{\lfloor t\rfloor} R_g(t,s)\bar\btheta^s \d
s\right)\right\|_2\\
    & \le \|\X\|_\op\|\btheta^t-\bar\btheta^t\|_2 + \Phi_{R_g}(T)\int_0^{\lfloor
t\rfloor} \|\btheta^s-\bar\btheta^s\|_2 \d s +  \Phi_{R_g}(T)\int_{\lfloor
t\rfloor}^t \|\btheta^s\|_2\d s,
\end{align*}
so also a.s.\ for all large $n$,
\begin{align*}
\sup_{t\in[0,T]}
\frac{1}{\sqrt{n}}\|{\g^t - \bar\g^{t}}\|_2 \leq C'\sqrt{\eta\max(\log(1/\eta),1)}.
\end{align*}
This implies that for any fixed $m \geq 1$ and $t_1,\ldots, t_m \in [0,T]$,
\begin{align}
&\limsup_{n\rightarrow\infty} W_2\Bigg( \frac{1}{n}\sum_{j=1}^n
\delta_{\left(\theta^{t_1}_j, \ldots, \theta^{t_m}_j,g^{t_1}_j, \ldots,
g^{t_m}_j,b^{t_1}_j, \ldots, b^{t_m}_j\right)},\notag\\
&\hspace{3.4cm} \frac{1}{n}\sum_{j=1}^n
\delta_{\left(\bar\theta^{t_1}_j, \ldots,
\bar\theta_j^{t_m},\bar g^{t_1}_j, \ldots,
\bar g_j^{t_m},\bar b^{t_1}_{j}, \ldots,
\bar b_{j}^{t_m}\right)}\Bigg)<C\sqrt{\eta\max(\log(1/\eta),1)}\text{ a.s.}\label{eq:discretelangevinW2}
\end{align}

Now let $\{\theta^t,g^t,b^t\}_{t \geq 0}$ be the components of the
continuous-time dynamical mean-field limit in Definition
\ref{def:DMFT}, defined from the fixed point $X=(C_g,R_g)$.
Let $\{\bar\theta^t\}_{t \geq 0}$ be the solution of
\eqref{eq:piecewise_discrete_dmft_theta} with the same Brownian motion
$\{b^t\}_{t \geq 0}$, defined from the discrete-time fixed
point $\bar X^\eta=(C_g^\eta,R_g^\eta)$ and $\{\bar g^t\}_{t \geq 0} \sim
\GP(0,C_g^\eta)$. Let $\bar b^t=b^{\lfloor t \rfloor}$. The bound
\eqref{eq:XXetabound} implies that for any sufficiently large $\lambda>0$,
\begin{align*}
\lim_{\eta \to 0} \sup_{0\leq s \leq t \leq T} |R_g(t,s)-R_g^\eta(t,s)|
&\leq \lim_{\eta \to 0} e^{\lambda T}d_\lambda(R_g,R_g^\eta)=0,\\
\lim_{\eta \to 0} \int_0^T \d t \int_0^T \d s\,|C_g(t,s)-C_g^\eta(t,s)|^2
&\leq \lim_{\eta \to 0} e^{2\lambda T}d_\lambda(C_g,C_g^\eta)=0.
\end{align*}
The continuity properties for $C_g,C_g^\eta$ in Lemma
\ref{lem:discrete_dmft_approx_continue} then imply that also
\[\lim_{\eta \to 0} \sup_{s,t \in [0,T]} |C_g(t,s)-C_g^\eta(t,s)|=0.\]
Then (c.f.\ \cite[Lemma~D.1]{fan2025dynamicalII}) there exists a coupling of
$\{g^t\}_{t \in [0,T]}$ and $\{\bar g^t\}_{t \in [0,T]}$ for which
\[\lim_{\eta \to 0} \sup_{t \in [0,T]} \E(g^t-\bar g^t)^2=0.\]
Under this coupling of $\{g^t\}_{t \in [0,T]}$ and $\{\bar g^t\}_{t \in [0,T]}$,
a Gr\"onwall argument (see \cite[Eq.\ (100)]{fan2025dynamicalI})
shows that the resulting coupling
of $\{\theta^t\}_{t \in [0,T]}$ and $\{\bar \theta^t\}_{t \in [0,T]}$ satisfies
\[\lim_{\eta \to 0} \sup_{t \in [0,T]} \E(\theta^t-\bar\theta^t)^2=0.\]
Hence, for any fixed $m \geq 1$ and $t_1,\ldots,t_m \in [0,T]$,
\begin{align}
&\lim_{\eta \to 0} W_2\Big(
\Law(\theta^{t_1},\ldots,
\theta^{t_m}, g^{t_1},\ldots,
g^{t_m}, b^{t_1},\ldots, b^{t_m}),\notag\\
&\hspace{1in}\Law(\bar\theta^{t_1},\ldots,
\bar\theta^{t_m},\bar g^{t_1},\ldots,
\bar g^{t_m}, \bar b^{t_1},\ldots,\bar b^{t_m})\Big)=0.\label{eq:discretedmftW2}
\end{align}
Combining \eqref{eq:discretelangevinW2}, \eqref{eq:discretedmftW2}, and Theorem
\ref{thm:discreteDMFT} which shows
\[\begin{aligned}
\lim_{n \to \infty} W_2\Bigg(&\frac{1}{n}\sum_{j=1}^n
\delta_{\left(\bar\theta^{t_1}_j, \ldots,
\bar\theta_j^{t_m},\bar g^{t_1}_j, \ldots,
\bar g_j^{t_m},\bar b^{t_1}_{j}, \ldots,
\bar b_{j}^{t_m}\right)},\Law(\bar\theta^{t_1},\ldots,
\bar\theta^{t_m},\bar g^{t_1},\ldots,
\bar g^{t_m}, \bar b^{t_1},\ldots,\bar b^{t_m})\Bigg)=0 \text{ a.s.},
\end{aligned}\]
and taking $\eta \to 0$, we have that for any fixed $m \geq 1$
and $t_1,\ldots,t_m \in [0,T]$,
\begin{align}\label{eq:pointwise-W2-convergence}
&\begin{aligned}
\lim_{n \to \infty}
W_2\Bigg(\frac{1}{n}\sum_{j=1}^n \delta_{\left(\theta^{t_1}_j, \ldots,
\theta^{t_m}_j,g^{t_1}_j, \ldots, g^{t_m}_j,b^{t_1}_j, \ldots,
b^{t_m}_j\right)},\,
\Law(\theta^{t_1},\ldots,
\theta^{t_m}, g^{t_1},\ldots,g^{t_m}, b^{t_1},\ldots,b^{t_m})\Bigg)=0
\text{ a.s.}
\end{aligned}
\end{align}

We may strengthen this to Wasserstein-2 convergence over $C([0,T],\R^3)$ using
a similar tightness argument as in \cite[Theorem 2.5(b)]{fan2025dynamicalI}:
Let
\[\x=\{\x^t\}_{t \in [0,T]}
=\{(\btheta^t,\g^t,\b^t)\}_{t \in [0,T]},
\qquad
x=\{x^t\}_{t \in [0,T]}=\{(\theta^t,g^t,b^t)\}_{t \in [0,T]}.\]
Denote
\[\|x\|_\infty=\sup_{t \in [0,T]} |\theta^t|+
\sup_{t \in [0,T]} |g^t|+\sup_{t \in [0,T]} |b^t|,\]
and write the coordinates of $\x$ as $x_j=(\theta_j,g_j,b_j) \in C([0,T],\R^3)$ for $j=1,\ldots,n$.
Fix any test function $F:C([0,T],\R^3) \to \R$ satisfying
\begin{equation}\label{eq:Fpseudolipschitz}
\big|F(x)-F(x')\big|
\leq C\|x-x'\|_\infty(1+\|x\|_\infty+\|x'\|_\infty).
\end{equation}
Fix any $\eta>0$, and let $\{\tilde\x^t\}_{t \in [0,T]}$ and $\{\tilde x^t\}_{t
\in [0,T]}$ be the piecewise-linear
interpolations of $\{\tilde \x^t\}_{t \in [0,T] \cap \eta \mathbb{Z}}$
and $\{\tilde x^t\}_{t \in [0,T] \cap \eta \mathbb{Z}}$, respectively, with
knots at $[0,T] \cap \eta \mathbb{Z}$. Then
\eqref{eq:pointwise-W2-convergence} implies
\begin{align}\label{eq:pointwise-W2-convergence-joint}
\lim_{n \to \infty} \frac{1}{n}\sum_{j=1}^n F(\{\tilde x_j^t\}_{t \in [0,T]})
=\E F(\{\tilde x^t\}_{t \in [0,T]}) \text{ a.s.}
\end{align}

By the property \eqref{eq:Fpseudolipschitz} and Cauchy-Schwarz, we have
\[\begin{aligned}
\left|\frac{1}{n}\sum_{j=1}^n F(\{x_j^t\}_{t \in [0,T]})
-F(\{\tilde x_j^t\}_{t \in [0,T]})\right|
\leq C\left(\frac{1}{n}\sum_{j=1}^n \|x_j-\tilde x_j\|_\infty^2\right)^{1/2}
\left(1+\frac{1}{n}\sum_{j=1}^n \|x_j\|_\infty^2\right)^{1/2}.
\end{aligned}\]
By a modulus-of-continuity bound for Brownian motion,
the bound \eqref{eq:theta-discret-norm}, and a Gr\"onwall argument
(see \cite[Theorem 2.5(b)]{fan2025dynamicalI}), for some $C>0$ and
$\alpha \in (0,1/2)$, a.s.\ for all large $n$,
\[\frac{1}{n}\sum_{j=1}^n \|b_j\|_\infty^2 \leq C,
\quad \frac{1}{n}\sum_{j=1}^n \|\theta_j\|_\infty^2 \leq C,
\quad \frac{1}{n}\sum_{j=1}^n \|g_j\|_\infty^2 \leq C,\]
\[\frac{1}{n}\sum_{j=1}^n \|b_j-\tilde b_j\|_\infty^2 \leq C\eta^{2\alpha},
\quad \frac{1}{n}\sum_{j=1}^n \|\theta_j-\tilde\theta_j\|_\infty^2 \leq
C\eta^{2\alpha},
\quad \frac{1}{n}\sum_{j=1}^n \|g_j-\tilde g_j\|_\infty^2 \leq
C\eta^{2\alpha}.\]
Applying this above and taking $\eta \to 0$,
\begin{equation}\label{eq:W2tightnesslangevin}
\lim_{\eta \to 0}
\left|\frac{1}{n}\sum_{j=1}^n F(\{x_j^t\}_{t \in [0,T]})
-F(\{\tilde x_j^t\}_{t \in [0,T]})\right|=0 \text{ a.s.}
\end{equation}
For the components of the dynamical mean-field limit,
by standard Gaussian process bounds and a Gr\"onwall argument, also
\[\E\|b\|_\infty^2 \leq C,
\quad \E\|g\|_\infty^2 \leq C,
\quad \E\|\theta\|_\infty^2 \leq C,\]
\[\E\|b-\tilde b\|_\infty^2 \leq C\eta^{2\alpha},
\quad \E\|g-\tilde g\|_\infty^2 \leq C\eta^{2\alpha},
\quad \E\|\theta-\tilde\theta\|_\infty^2 \leq C\eta^{2\alpha},\]
and hence
\begin{equation}\label{eq:W2tightnessdmft}
\lim_{\eta \to 0} \big|\E F(\{x^t\}_{t \in [0,T]})
-\E F(\{\tilde x^t\}_{t \in [0,T]})\big|=0.
\end{equation}
Combining \eqref{eq:pointwise-W2-convergence-joint}, \eqref{eq:W2tightnesslangevin},
and \eqref{eq:W2tightnessdmft} and taking $\eta \to 0$
shows
\begin{align}\label{eq:pointwise-W2-convergence-2}
\frac{1}{n}\sum_{j=1}^n F(\{x_j^t\}_{t \in [0,T]})
\rightarrow \E F(\{x^t\}_{t \in [0,T]}) \text{ a.s.}
\end{align}
This holds for any function $F:C([0,T],\R^3) \to \R$
satisfying \eqref{eq:Fpseudolipschitz}, implying the claimed Wasserstein-2
convergence in Theorem \ref{thm:dmft-approx}(b).
\end{proof}

\begin{proof}[Proof of Theorem \ref{thm:dmft-approx}(c)]
Theorem \ref{thm:dmft-approx}(b) implies, almost surely as $n \to \infty$,
\[\frac{1}{n}\sum_{i=1}^n \theta_i^t\theta_i^s \to
\E[\theta^t \theta^s]=C_\theta(t,s),
\quad \frac{1}{n}\sum_{i=1}^n g_i^tg_i^s \to
\E[g^tg^s]=C_g(t,s).\]
The first two claims of Theorem \ref{thm:dmft-approx}(c) follow from this and
dominated convergence to take an expectation over $\{\b^t\}_{t \in [0,T]}$. For
the third claim, \eqref{eq:discretelangevinW2} implies
\[\lim_{\eta \to 0}
\limsup_{n \to \infty} \left|\frac{1}{n}\sum_{i=1}^n \theta_i^t b_i^s
-\bar\theta_i^t \bar b_i^s\right|=0 \text{ a.s.}\]
Theorem \ref{thm:discreteDMFT} shows
\[\frac{1}{n}\sum_{i=1}^n \bar\theta_i^t \bar b_i^s
\to \sqrt{2\gamma}\int_0^{\lfloor s \rfloor} R_\theta^\eta(t,s) \d s \text{ a.s.}\]
The bound \eqref{eq:XXetabound} and $|R_\theta^\eta(t,s)| \leq
\Phi_{R_\theta}(T)$ imply
\[\lim_{\eta \to 0}\left|\int_0^s R_\theta(t,s) \d s 
-\int_0^{\lfloor s \rfloor} R_\theta^\eta(t,s) \d s\right|=0.\]
Then the third claim of Theorem \ref{thm:dmft-approx}(c) follows from combining
these three statements, taking $\eta \to 0$, and again applying dominated
convergence to take an expectation over $\{\b^t\}_{t \in [0,T]}$.
\end{proof}

\section{Analyses of correlation and response}\label{sec:corrresponse}

In this section, we specialize the  dynamical mean-field limit of Theorem \ref{thm:dmft-approx} to the overdamped
Langevin equation \eqref{eq:langevin}. Our analysis relies on several
important structural
properties of the kernels $C_\theta,R_\theta,C_g,R_g$ at large times.
We collect these properties in the following two lemmas.

First, under Assumption \ref{assump:convergence}, we have the following
properties for $C_\theta,R_\theta$, including the existence of
time-translation-invariant limits $c_\theta,r_\theta:\R \to \R$ for these
kernels as $t \to \infty$, and a fluctuation-dissipation relation between
these limits.

\begin{lemma}\label{lemma:tti}
Suppose Assumptions \ref{assump:X} and \ref{assump:dynamics} hold with $f=-U'$,
the lower bound $\alpha>0$ in \eqref{eq:Uconvexity} satisfies
$\alpha>\Lambda_{\max}$, and Assumption \ref{assump:convergence} holds. Then:
\begin{enumerate}[(a)]
\item The Gibbs measure $\mu_\Gibbs$ is well-defined, and there exist deterministic almost sure limits
\begin{equation}\label{eq:vstardef-thm}
v_*=\lim_{n \to \infty}\frac{1}{n}\langle \|\btheta\|_2^2 \rangle,
\qquad q_*=\lim_{n \to \infty}\frac{1}{n}\langle \btheta^\top \btheta' \rangle.
\end{equation}
\item For each fixed $\tau \in \R$, there exists a limit
$c_\theta(\tau)=\lim_{s \to \infty} C_\theta(s+\tau,s)$. This defines a
symmetric positive-semidefinite function
$c_\theta:\R \to \R$ satisfying, for some $C,c>0$ and
all $s,t \geq 0$ and $\tau \in \R$,
\begin{equation}\label{eq:cthetabound}
\begin{gathered}
c_\theta(0)=v_*, \qquad \lim_{\tau \to \infty} c_\theta(\tau)=q_*,\\
|C_\theta(t,s)-c_\theta(t-s)|\leq Ce^{-c\min(s,t)},
\qquad |c_\theta(\tau)-q_*| \leq Ce^{-c|\tau|}.
\end{gathered}
\end{equation}
\item For each fixed $\tau \geq 0$, there exists a limit
$r_\theta(\tau)=\lim_{s \to \infty} R_\theta(s+\tau,s)$. This defines a
function $r_\theta:[0,\infty) \to \R$ satisfying, for some $C,c>0$
and all $t \geq s \geq 0$ and $\tau \geq 0$,
\begin{equation}\label{eq:rthetabound}
\begin{gathered}
|R_\theta(t,s)| \leq Ce^{-c(t-s)}, \qquad
|r_\theta(\tau)| \leq Ce^{-c\tau},\\
|R_\theta(t,s)-r_\theta(t-s)| \leq Ce^{-cs}, \qquad
\int_0^t |R_\theta(t,s)-r_\theta(t-s)| \d s \leq Ce^{-ct}.
\end{gathered}
\end{equation}
\item $c_\theta$ is continuous on $\R$, smooth on $(0,\infty)$,
and has a right derivative $c_\theta'(0^+)=\lim_{\tau \to 0^+} c_\theta'(\tau)$
at $\tau=0$. We have the fluctuation-dissipation relation
\begin{equation}\label{eq:crthetaFDT}
r_\theta(\tau)={-}c_\theta'(\tau) \text{ for } \tau>0,
\qquad r_\theta(0)={-}c_\theta'(0^+)=1.
\end{equation}
Furthermore, for all $\tau>0$, we have
\[c_\theta(\tau) \geq 0,
\qquad r_\theta(\tau) \geq 0, \qquad r_\theta'(\tau) \leq 0.\]
\end{enumerate}
\end{lemma}

We may then define time-translation-invariant limits for $C_g,R_g$ as $t \to
\infty$ by
\begin{align*}
r_g(\tau)&=\sum_{p \geq 1} \kappa_{p+1}\,r_\theta^{\ast p}(\tau) \text{ for } \tau \geq 0,
\qquad \text{ where } r_\theta(\tau)=0 \text{ for } \tau<0,\\
c_g(\tau)&=\sum_{p,q \geq 0} \kappa_{p+q+2}\,
r_\theta^{\ast p} * c_\theta * \bar r_\theta^{\ast q},
\qquad \text{ where } \bar r_\theta(\tau)=r_\theta(-\tau) \text{ for all } \tau
\in \R.
\end{align*}
Here, for integrable functions $a,b:\R \to \R$, we use the convolution notation
\[[a*b](\tau)=\int_\R a(\tau-s)b(s)\d s,
\qquad a^{*p}=\underbrace{a*\ldots*a}_{p \text{ times}},
\qquad a^{*0}*b=b*a^{*0}=b.\]
These definitions of $r_g,c_g$ are alternatively expressed in the Fourier
domain (c.f.\ \eqref{eq:Fouriercalculations}) as
\begin{equation}\label{eq:Fourierrgcg}
\cF[r_g](\omega)=\cR_\Lambda(\cF[r_\theta](\omega)),
\,
\cF[c_g^0](\omega)=\cF[c_\theta^0](\omega)
\frac{\cR_\Lambda(\cF[r_\theta](\omega))
-\overline{\cR_\Lambda(\cF[r_\theta](\omega))}}
{\cF[r_\theta](\omega)-\overline{\cF[r_\theta](\omega)}},
\end{equation}
where $\cF[f](\omega)=\int_\R f(\tau)e^{-i\omega\tau}\d\tau$ is the Fourier
transform of $f$, $\cR_\Lambda$ is the R-transform \eqref{eq:Rtransform}, and $c_\theta^0(\tau)=c_\theta(\tau)-\lim_{s \to \infty} c_\theta(s)$ and $c_g^0(\tau)=c_g(\tau)-\lim_{s \to \infty} c_g(s)$.

When $v_*-q_*$ of Lemma \ref{lemma:tti}
is within the radius of convergence of \eqref{eq:Rtransform},
we have the following properties for $C_g,R_g$, including
importantly a fluctuation-dissipation relation between $c_g$ and $r_g$.

\begin{lemma}\label{lemma:ttig}
In addition to the conditions of Lemma \ref{lemma:tti}, suppose that
\begin{equation}\label{eq:Rdomain}
\sum_{p \geq 1} |\kappa_{p+1}|(v_*-q_*+\eps)^p<\alpha
\end{equation}
for some constant $\eps>0$ and the limits $v_*,q_*$ defined in Lemma
\ref{lemma:tti}. Then:
\begin{enumerate}[(a)]
\item $r_g:[0,\infty) \to \R$ is continuous on $[0,\infty)$ and
continuously-differentiable on $(0,\infty)$, with
\begin{equation}\label{eq:intrgbound}
\int_0^\infty |r_g(\tau)| \d \tau<\alpha
\end{equation}
and $\int_0^\infty |r_g'(\tau)|\d\tau<\infty$.
Furthermore, for some $C,c>0$ and all $t \geq s \geq 0$ and $\tau \geq 0$,
\begin{equation}\label{eq:rgbounds}
\begin{gathered}
|R_g(t,s)| \leq Ce^{-c(t-s)}, \qquad |r_g(\tau)| \leq Ce^{-c\tau},\\
|R_g(t,s)-r_g(t-s)| \leq Ce^{-cs}, \qquad
\int_0^t |R_g(t,s)-r_g(t-s)| \d s \leq Ce^{-ct}.
\end{gathered}
\end{equation}
\item $c_g:\R \to \R$ is continuous, symmetric, and positive-semidefinite on
$\R$, twice continuously-differentiable on $(0,\infty)$, and has
a right derivative $c_g'(0^+)=\lim_{\tau \to 0^+} c_g'(\tau)$ at $\tau=0$.
We have the fluctuation-dissipation relation
\begin{equation}\label{eq:cgrgFDT}
r_g(\tau)={-}c_g'(\tau) \text{ for all } \tau>0,
\qquad r_g(0)={-}c_g'(0^+).
\end{equation}
Furthermore $\cR_\Lambda(v_*-q_*) \geq 0$ and
$q_*\cR_\Lambda'(v_*-q_*) \geq 0$, and we have
\begin{equation}\label{eq:cgbounds}
\begin{gathered}
c_g(0)=\cR_\Lambda(v_*-q_*)+q_*\cR_\Lambda'(v_*-q_*),
\qquad
\lim_{\tau \to \infty} c_g(\tau)=q_*\cR_\Lambda'(v_*-q_*),\\
\qquad |C_g(t,s)-c_g(t-s)| \leq Ce^{-c\min(s,t)},
\qquad |c_g(\tau)-q_*\cR_\Lambda'(v_*-q_*)| \leq Ce^{-c\tau}.
\end{gathered}
\end{equation}
\end{enumerate}
\end{lemma}

The remainder of this section proves Lemmas \ref{lemma:tti}
and \ref{lemma:ttig}. In these proofs and in our subsequent proof of Theorem
\ref{thm:replica}, to ease the exposition of statements that hold almost
surely for all large $n$, we assume the following additional condition.

\begin{assumption}\label{assump:as}
For some constants $C,c>0$, with probability 1 over $(\X,\btheta^0)$
(for every $n \geq 1$), we have
$\|\X\|_\op \leq \Lambda_{\max}$, $\frac{1}{n}\|\btheta^0\|_2^2 \leq C$,
and the statements of Assumption \ref{assump:convergence} hold.
\end{assumption}
This is without loss of generality: Indeed, letting $\cE_n$ be the
$(\X,\btheta^0)$-dependent event where the statements of Assumption
\ref{assump:as} hold, Assumptions \ref{assump:X}, \ref{assump:dynamics}, and
\ref{assump:convergence} ensure that $\cE_n$ holds a.s.\ for all large $n$.
We may define $(\tilde \X,\tilde \btheta^0)=(\X,\btheta^0)$ on $\cE_n$ and
define $(\tilde \X,\tilde \btheta^0)$ arbitrarily on $\cE_n^c$ (e.g., setting $\tilde \X=0$ and $\tilde\btheta^0=0$) to satisfy
Assumption \ref{assump:as}. Then Lemmas \ref{lemma:tti},
\ref{lemma:ttig}, and Theorem \ref{thm:replica} for $(\tilde \X,\tilde
\btheta^0)$ imply the corresponding results for $(\X,\btheta^0)$, since
$(\tilde \X,\tilde \btheta^0)=(\X,\btheta^0)$ a.s.\ for all large $n$.

\begin{proof}[Proof of Lemma \ref{lemma:tti}]
Throughout, $\{\tilde\btheta^t\}_{t \in \R}$ denotes a stationary
process satisfying \eqref{eq:langevin} with equilibrium initial state
$\tilde\btheta^0 \sim \mu_\Gibbs$.
To ease notation, we will use consistently the convention that
\[\langle f(\btheta^t) \rangle=\E[f(\btheta^t) \mid \X,\btheta^0]\]
denotes an expectation conditional on $\btheta^0$
(the given initial condition satisfying Assumptions \ref{assump:dynamics} and
\ref{assump:convergence}), while
\[\langle f(\tilde \btheta^t) \rangle=\E[f(\tilde \btheta^t) \mid \X]\]
denotes an expectation marginalizing over $\tilde\btheta^0 \sim \mu_\Gibbs$.
The notation
$\langle f(\btheta,\btheta') \rangle$ without time index continues to denote
``static'' expectations over $\btheta,\btheta' \overset{iid}{\sim} \mu_\Gibbs$.
(We will write these expectations more explicitly under couplings of these
components.) We assume without loss of generality that Assumption
\ref{assump:as} holds.
Constants $C,C',c,c'>0$ depend on $\Lambda_{\max}$ and the potential $U(\cdot)$
but do not depend on $n$, and may change from instance to instance.\\

{\bf Parts (a) and (b):} Since $f=-U'$ satisfies Assumption
\ref{assump:dynamics} and $U(\cdot)$ also satisfies \eqref{eq:Uconvexity} with
$\alpha>\Lambda_{\max}$, there are constants $C,\eps>0$ such that
\[Cx^2+C \geq U(x) \geq \frac{\Lambda_{\max}+\eps}{2}\,x^2-C \text{ for all } x \in \R.\]
Then
\begin{align*}
\cZ&=\int \exp\left(\frac{1}{2}\btheta^\top \X\btheta-\sum_{i=1}^n
U(\theta_i)\right)\d \btheta
\geq \int \exp\left({-}\Big(\frac{\Lambda_{\max}}{2}+C\Big)\|\btheta\|_2^2
-Cn\right)\d\btheta \geq e^{-C'n},\\
\cZ&\leq \int\exp\left({-}\Big(\frac{\Lambda_{\max}+\eps}{2}-\frac{\Lambda_{\max}}{2}\Big)\|\btheta\|_2^2
+Cn\right)\d\btheta \leq e^{C'n},
\end{align*}
so $\mu_\Gibbs$ is well-defined.
Also for any $M>0$, by a chi-squared tail bound,
\begin{align*}
\int_{\|\btheta\|_2^2 \geq Mn}
\|\btheta\|_2^2 \exp\left(\frac{1}{2}\btheta^\top \X\btheta-\sum_{i=1}^n
U(\theta_i)\right)\d \btheta
&\leq \int_{\|\btheta\|_2^2 \geq Mn}
\|\btheta\|_2^2
\exp\left({-}\Big(\frac{\eps}{2}\Big)\|\btheta\|_2^2
+Cn\right)\d\btheta\\
&\leq e^{C'n} \cdot e^{-cMn}
\end{align*}
Thus for $M>0$ large enough, we have
$\langle \1\{\|\btheta\|_2^2 \geq Mn\} \cdot \frac{1}{n}\|\btheta\|_2^2 \rangle
\leq 1$, implying for $C=M+1$ that
\begin{equation}\label{eq:thetanormbound}
\frac{1}{n}\langle \|\btheta\|_2^2 \rangle \leq C.
\end{equation}

Assumption \ref{assump:convergence} implies that for every $s \geq 0$, 
there exists a coupling of $(\btheta^0,\btheta^s)$
with $\tilde \btheta^0$ having law $\mu_\Gibbs$ conditional on $\btheta^0$
(i.e.\ $\tilde \btheta^0$ is independent of $\btheta^0$), such that
\begin{equation}\label{eq:thetatthetacouple}
\frac{1}{n}\,\E[\|\btheta^s-\tilde \btheta^0\|_2^2 \mid \X,\btheta^0]
\leq Ce^{-cs}.
\end{equation}
Combined with \eqref{eq:thetanormbound} to bound $\frac{1}{n}\langle
\|\tilde \btheta^0\|_2^2 \rangle$, this implies also
\begin{equation}\label{eq:thetatnormbound}
\frac{1}{n}\langle \|\btheta^s\|_2^2 \rangle \leq C \text{ for all } s \geq 0.
\end{equation}
Denote
\[P_t f(\btheta)=\E[f(\btheta^t) \mid \X,\,\btheta^0=\btheta]\]
for any $f:\R^n \to \R^m$ and $m \geq 1$ for which this expectation exists.
Let $\id:\R^n \to \R^n$ be the identity map. Then for any fixed
$s,\tau \geq 0$,
\[\frac{1}{n} \sum_{i=1}^n \langle \theta_i^{s+\tau}\theta_i^s \rangle
=\frac{1}{n}\E[\btheta^{s\top} \E[\btheta^{s+\tau} \mid \X,\btheta^s] \mid
\X,\btheta^0]=\frac{1}{n}\langle \btheta^{s\top} P_\tau\id(\btheta^s) \rangle,\]
so Theorem \ref{thm:dmft-approx}(c) implies that
\begin{equation}\label{eq:Cthetainterp}
C_\theta(s+\tau,s)=\lim_{n \to \infty}
\frac{1}{n}\langle \btheta^{s\top} P_\tau\id(\btheta^s) \rangle \text{ a.s.}
\end{equation}
Let us bound $P_t(\btheta)-P_t(\btheta')$ for any $\btheta,\btheta' \in \R^n$:
For two diffusions
$\{\btheta^t\}_{t \geq 0}$ and $\{{\btheta'}^t\}_{t \geq 0}$ 
satisfying \eqref{eq:langevin} with
initial conditions $\btheta$ and $\btheta'$, coupling these by
the same Brownian motion, we have
\[\frac{\d}{\d t} \|\btheta^t-{\btheta'}^t\|_2^2
=2(\btheta^t-{\btheta'}^t)^\top[\X\btheta^t-\X{\btheta'}^t-U'(\btheta^t)+U'({\btheta'}^t)] \leq C\|\btheta^t-{\btheta'}^t\|_2^2,\]
and thus for all $t \geq 0$,
\begin{equation}\label{eq:thetainitcoupling}
\|\btheta^t-{\btheta'}^t\|_2 \leq e^{C't}\|\btheta-{\btheta'}\|_2.
\end{equation}
Then also by Jensen's inequality, for all $t \geq 0$,
\begin{equation}\label{eq:Ptaudiff}
\|P_t\id(\btheta)-P_t\id({\btheta'})\|_2
\leq e^{C't} \|\btheta-{\btheta'}\|_2.
\end{equation}
Let $(\btheta^0,\btheta^s)$ and $\tilde\btheta^0$ be coupled as in \eqref{eq:thetatthetacouple}.
Then by this bound \eqref{eq:Ptaudiff}
applied to $\btheta=\btheta^s$ and $\btheta'=\tilde \btheta^0$ and
Cauchy-Schwarz,
\begin{equation}\label{eq:Cthetacompare}
\left|\frac{1}{n}\langle \btheta^{s\top} P_\tau\id(\btheta^s) \rangle
-\frac{1}{n}\langle \tilde \btheta^{0\top} P_\tau\id(\tilde \btheta^0)
\rangle\right| \leq Ce^{C\tau} \cdot e^{-cs}.
\end{equation}
Then by the existence a.s.\ of the limit \eqref{eq:Cthetainterp} for any fixed
$s,\tau \geq 0$,
\[\limsup_{n \to \infty}
\frac{1}{n}\langle \tilde \btheta^{0\top} P_\tau\id(\tilde \btheta^0)
\rangle-\liminf_{n \to \infty}
\frac{1}{n}\langle \tilde \btheta^{0\top} P_\tau\id(\tilde \btheta^0)
\rangle \leq 2Ce^{C\tau} \cdot e^{-cs} \text{ a.s.}\]
Since $s \geq 0$ is arbitrary, this implies that for each fixed $\tau \geq 0$,
there exists a (possibly random) almost sure limit
\begin{equation}\label{eq:cthetadef}
c_\theta(\tau):=\lim_{n \to \infty} 
\frac{1}{n}\langle \tilde \btheta^{0\top} \tilde \btheta^\tau
\rangle=\lim_{n \to \infty} 
\frac{1}{n}\langle \tilde \btheta^{0\top} P_\tau\id(\tilde
\btheta^0)\rangle \text{ a.s.}
\end{equation}
Taking $n \to \infty$ followed by $s \to \infty$
in \eqref{eq:Cthetacompare} shows that $c_\theta(\tau)$ is deterministic and
given by
\begin{equation}\label{eq:cthetaqualitative}
c_\theta(\tau)=\lim_{s \to \infty} C_\theta(s+\tau,s).
\end{equation}
Defining $c_\theta(\tau)=c_\theta(-\tau)$ for $\tau \leq 0$, and noting that
$\langle \tilde \btheta^{0\top} \tilde \btheta^\tau \rangle
=\langle \tilde \btheta^{0\top} \tilde \btheta^{-\tau} \rangle$ by
reversibility, \eqref{eq:cthetaqualitative} and
the first equality of \eqref{eq:cthetadef} then also hold for all $\tau \in \R$.
Since $\tau \mapsto \frac{1}{n}\langle \tilde \btheta^{0\top} \tilde \btheta^\tau \rangle$ is symmetric positive-semidefinite for each $n$,
this implies by \eqref{eq:cthetadef}
that $c_\theta(\tau)$ is also symmetric positive-semidefinite.

To show part (a), define $v_*=c_\theta(0)$.
Then by \eqref{eq:cthetadef},
\[v_*=c_\theta(0)=\lim_{n \to \infty}
\frac{1}{n} \langle \|\tilde \btheta^0\|_2^2 \rangle
=\lim_{n \to \infty} \frac{1}{n} \langle \|\btheta\|_2^2 \rangle \text{ a.s.}\]
For any fixed $\tau \geq 0$, Assumption \ref{assump:convergence} implies
there exists a coupling of
$(\tilde\btheta^0,\tilde \btheta^\tau)$ with $\btheta$ having law
$\mu_\Gibbs$ conditional on $\tilde \btheta^0$ (i.e.\ $\btheta$ is independent
of $\tilde \btheta^0$), such that
\begin{equation}\label{eq:thetaequilibriumcoupling}
\frac{1}{n}\,\E[\|\tilde \btheta^\tau-\btheta\|_2^2 \mid \X]
\leq Ce^{-c\tau}.
\end{equation}
Then by \eqref{eq:thetanormbound} and Cauchy-Schwarz,
\begin{equation}\label{eq:cthetacompare}
\left|\frac{1}{n}\langle \tilde \btheta^{0\top}\tilde \btheta^\tau
\rangle-\frac{1}{n} \langle \btheta^\top \btheta' \rangle\right|
=\left|\frac{1}{n}\E[\tilde \btheta^{0\top}\tilde \btheta^\tau\mid \X]
-\frac{1}{n}\E[\tilde \btheta^{0\top}\btheta \mid \X]\right|
\leq Ce^{-c\tau}.
\end{equation}
Thus, since the limit \eqref{eq:cthetadef} exists a.s.,
\[\limsup_{n \to \infty} \frac{1}{n} \langle \btheta^\top \btheta' \rangle
-\liminf_{n \to \infty} \frac{1}{n} \langle \btheta^\top \btheta' \rangle
\leq 2Ce^{-c\tau} \text{ a.s.}\]
Here $\tau \geq 0$ is arbitrary, so there exists an almost sure limit
\[q_*=\lim_{n \to \infty} \frac{1}{n} \langle \btheta^\top \btheta' \rangle
\text{ a.s.}\]
Taking the limit $n \to \infty$ followed by $\tau \to \infty$
in \eqref{eq:cthetacompare} shows that
$q_*$ is deterministic and given by $q_*=\lim_{\tau \to \infty} c_\theta(\tau)$.
This proves part (a).

To check the bounds \eqref{eq:cthetabound} in part (b), similarly by Assumption
\ref{assump:convergence}, there exists a coupling of
$(\btheta^s,\btheta^{s+\tau})$ with $\btheta$ having law $\mu_\Gibbs$
conditional on $\btheta^s$ (i.e.\ $\btheta$ is independent of $\btheta^s$)
such that
\begin{equation}\label{eq:thetascoupling}
\E\left[\frac{1}{n}\,\E[\|\btheta^{s+\tau}-\btheta\|_2^2 \mid \X,\btheta^s] \;\bigg|\;
\X,\btheta^0\right]
\leq Ce^{-c\tau}.
\end{equation}
By this, \eqref{eq:thetatnormbound},
and Cauchy-Schwarz,
\[\E\left[\left|\frac{1}{n}\E[\btheta^{s\top}\btheta^{s+\tau} \mid \X,\btheta^s]
-\frac{1}{n}\,\E[\btheta^{s\top} \btheta \mid \X,\btheta^s]\right| \;\bigg|\; \X,\btheta^0\right]
\leq Ce^{-c\tau}.\]
Then coupling $(\btheta^0,\btheta^s)$ with $\tilde \btheta^0$ via
\eqref{eq:thetatthetacouple} (where $\tilde\btheta^0$ is independent of
$\btheta^0$, and all $(\btheta^0,\btheta^s,\tilde \btheta^0)$ are
independent of the above $\btheta \sim \mu_\Gibbs$)
and applying Cauchy-Schwarz again,
\begin{equation}\label{eq:Cthetastauqstar}
\begin{aligned}
\left|\frac{1}{n} \langle \btheta^{s\top}\btheta^{s+\tau} \rangle
-\frac{1}{n}\langle \btheta^\top \btheta' \rangle\right|
&=\left|\frac{1}{n}\E[\btheta^{s\top}\btheta^{s+\tau} \mid \X,\btheta^0]
-\frac{1}{n}\E[\tilde \btheta^{0\top} \btheta \mid \X,\btheta^0]\right|
\\
&\leq Ce^{-c\tau}+Ce^{-cs}.
\end{aligned}
\end{equation}
Taking the limit $n \to \infty$ in \eqref{eq:Cthetacompare},
\eqref{eq:cthetacompare}, and \eqref{eq:Cthetastauqstar} shows
\[\begin{gathered}
|C_\theta(s+\tau,s)-c_\theta(\tau)| \leq Ce^{C\tau} \cdot e^{-cs},
\quad
|c_\theta(\tau)-q_*| \leq Ce^{-c\tau},\\
|C_\theta(s+\tau,s)-q_*| \leq Ce^{-c\tau}+Ce^{-cs}.
\end{gathered}\]
Choosing a constant $\iota>0$ small enough, the first bound is at most
$C'e^{-c's}$ when $\tau \leq \iota s$. Applying this bound for $\tau
\leq \iota s$ and the latter two bounds for $\tau>\iota s$, also
\[|C_\theta(s+\tau,s)-c_\theta(\tau)| \leq C''e^{-c''\iota}.\]
This shows all statements of part (b).\\

{\bf Part (c):} In light of the last statement of Theorem
\ref{thm:dmft-approx}(c), we first compute a simple form for
$\frac{\d}{\d s}\frac{1}{n}\langle\btheta^{t\top} \b^s \rangle$
over $s \in (0,t)$. Since
$\{\btheta^t\}_{t \geq 0}$ solves \eqref{eq:langevin}, we have
\[\sqrt{2}\,\b^s=\btheta^s-\btheta^0-\int_0^s [\X\btheta^r-U'(\btheta^r)]\d r.\]
Then
\begin{equation}\label{eq:thetabderiv}
\sqrt{2}\,\frac{\d}{\d s}
\left[\frac{1}{n}\langle\btheta^{t\top} \b^s \rangle\right]
=\frac{\d}{\d s}\left[\frac{1}{n}\langle \btheta^{t\top} \btheta^s
\rangle\right]-\frac{1}{n}\langle \btheta^{t\top}[\X\btheta^s-U'(\btheta^s)]
\rangle.
\end{equation}
Define $e_i:\R^n \to \R$ by $e_i(\btheta)=\theta_i$. Then
\[\frac{1}{n}\langle \btheta^{t\top}\btheta^s \rangle
=\frac{1}{n}\sum_{i=1}^n P_s(e_i \cdot P_{t-s} e_i)(\btheta^0).\]
Letting
$-A f(\btheta)=\nabla f(\btheta)^\top [\X\btheta-U'(\btheta)]
+\Tr \nabla^2 f(\btheta)$
be the generator of the diffusion, we have
\[\partial_\tau P_\tau e_i={-}AP_\tau e_i={-}P_\tau Ae_i,
\qquad \partial_s P_s(e_i \cdot P_\tau e_i)
={-}P_sA(e_i \cdot P_\tau e_i)\]
as equalities pointwise on $\R^n$.
Then by the chain rule and dominated convergence,
\begin{equation}\label{eq:semigroupderiv}
\frac{\d}{\d s}P_s(e_i \cdot P_{t-s} e_i)
={-}P_s A(e_i \cdot P_{t-s} e_i)
+P_s(e_i \cdot P_{t-s} Ae_i).
\end{equation}
Here, using $\nabla e_i(\btheta)=\e_i$ (the $i^\text{th}$ standard basis
vector in $\R^n$) and $\nabla^2 e_i(\btheta)=0$, we have
\begin{align*}
&{-}A(e_i \cdot P_{t-s}e_i)(\btheta)\\
&=\bigl[P_{t-s}e_i(\btheta) \cdot \nabla e_i(\btheta)
+e_i(\btheta) \cdot \nabla P_{t-s}e_i(\btheta)\bigr]^\top
[\X\btheta-U'(\btheta)]\\
&\qquad+\Tr\bigl[\nabla P_{t-s}e_i(\btheta) \cdot \nabla e_i(\btheta)^\top
+\nabla e_i(\btheta) \cdot \nabla P_{t-s}e_i(\btheta)^\top
+e_i(\btheta) \cdot \nabla^2 P_{t-s}e_i(\btheta)\bigr]\\
&=P_{t-s}e_i(\btheta) \cdot \e_i^\top[\X\btheta-U'(\btheta)]
+2\e_i^\top \nabla P_{t-s}e_i(\btheta)-(e_i \cdot AP_{t-s}e_i)(\btheta)
\end{align*}
Applying $P_s$ to this function,
the last term cancels the second term of \eqref{eq:semigroupderiv}, and thus
\[\frac{\d}{\d s}P_s(e_i \cdot P_{t-s} e_i)(\btheta^0)
=\langle P_{t-s}e_i(\btheta^s) \cdot
\e_i^\top[\X\btheta^s-U'(\btheta^s)] \rangle
+2\langle \e_i^\top \nabla P_{t-s}e_i(\btheta^s) \rangle.\]
Then
\begin{align*}
\frac{\d}{\d s}\left[\frac{1}{n}\langle \btheta^{t\top}\btheta^s
\rangle\right]
&=\frac{1}{n}\sum_{i=1}^n \frac{\d}{\d s}P_s(e_i \cdot P_{t-s} e_i)(\btheta^0)\\
&=\frac{1}{n} \langle \btheta^{t\top}[\X\btheta^s-U'(\btheta^s)] \rangle
+\frac{2}{n}\sum_{i=1}^n \langle \e_i^\top \nabla P_{t-s}e_i(\btheta^s) \rangle.
\end{align*}
Applying this to \eqref{eq:thetabderiv}, for all $s \in (0,t)$ we have
\begin{equation}\label{eq:response}
\frac{1}{\sqrt{2}}\,\frac{\d}{\d s}
\left[\frac{1}{n}\langle\btheta^{t\top} \b^s \rangle\right]
=\frac{1}{n}\sum_{i=1}^n \langle \e_i^\top \nabla P_{t-s}e_i(\btheta^s)\rangle.
\end{equation}
We note that $\e_i^\top \nabla P_{t-s}e_i(\btheta^s)$ is a standard form of a
response function for the Markov diffusion, see e.g.\ \cite{dembo2010markovian}.

We now check that for any fixed $t \geq 0$, the right side of
\eqref{eq:response} as a function
of $s \in [0,t]$ is uniformly bounded and equicontinuous
for all $n$. Given $\{\btheta^t\}_{t \geq 0}$,
define a process $\{\J^{t,s}\}_{t \geq s \geq 0}$ by
\[\frac{\d}{\d t}\J^{t,s}=[\X-\diag(U''(\btheta^t))]\J^{t,s} \text{ for } t \geq
s, \qquad \J^{s,s}=\Id.\]
By Assumptions \ref{assump:as} and \ref{assump:dynamics} for $f=-U'$,
we have $\|\X-\diag(U''(\btheta))\|_\op \leq C$ for a constant $C>0$ and any
$\btheta \in \R^n$, and thus
\begin{equation}\label{eq:Jopbound}
\|\J^{t,s}\|_\op \leq e^{C(t-s)} \text{ for all } t \geq s \geq 0.
\end{equation}
For any $s \geq 0$, conditional on $\btheta^s$, there
exists a modification of $\{\btheta^t\}_{t \geq s}$ 
and $\{\J^{t,s}\}_{t:\,t \geq s}$ where $\btheta^t$ is
continuously-differentiable in $\btheta^s$ 
with Jacobian $\frac{\d \btheta^t}{\d \btheta^s}=\J^{t,s}$
\cite[Theorem II.3.1]{kunita2006stochastic}.
Then by dominated convergence,
\[\begin{aligned}
\frac{1}{n}\sum_{i=1}^n \e_i^\top \nabla P_{t-s} e_i(\btheta^s)
&=\frac{1}{n}\sum_{i=1}^n \frac{\partial}{\partial \theta_i^s}
\E[\theta_i^t \mid \X,\btheta^s]\\
&=\frac{1}{n}\sum_{i=1}^n 
\E\left[\frac{\partial \theta_i^t}{\partial \theta_i^s}\,\bigg|\,
\X,\btheta^s\right]
=\E\left[\frac{1}{n}\Tr \J^{t,s}\,\bigg|\,\X,\btheta^s\right]
\end{aligned}\]
Thus, \eqref{eq:response} is equivalently 
$\frac{1}{n} \langle \Tr \J^{t,s} \rangle$, where $\langle \cdot \rangle$
here and below denotes the expectation over both
$\{\btheta^t\}_{t \geq 0}$ and $\{\J^{t,s}\}_{t \geq s \geq 0}$
conditional on $(\X,\btheta^0)$. For any $s,s' \in [0,t]$ with
$s \leq s'$, we have
\[\|\J^{s',s}-\J^{s',s'}\|_\op
=\|\J^{s',s}-\J^{s,s}\|_\op
\leq \int_s^{s'}
\left\|\frac{\d}{\d r}\J^{r,s}\right\|_\op \d r
\leq \int_s^{s'} Ce^{C(r-s)}\d r \leq C(t)|s'-s|\]
for some constant $C(t)>0$ not depending on $n$. Then by Gr\"onwall's
inequality, for a (different) constant $C(t)>0$, also
\[\left|\frac{1}{n}\Tr \J^{t,s}-\frac{1}{n}\Tr \J^{t,s'}\right|
\leq \|\J^{t,s}-\J^{t,s'}\|_\op \leq C(t)|s'-s|.\]
This and \eqref{eq:Jopbound} imply
the uniform boundedness and equicontinuity of \eqref{eq:response} in the form
$\frac{1}{n} \langle \Tr \J^{t,s} \rangle$ over
$s \in [0,t]$, as claimed. This uniform boundedness, Arzel\`a-Ascoli, and
Theorem \ref{thm:dmft-approx}(c) imply that
\[\lim_{n \to \infty}
\sup_{s \in [0,t]}
\left|\frac{1}{\sqrt{2}} \cdot \frac{1}{n}\langle \btheta^{t\top}\b^s \rangle
-\int_0^s R_\theta(t,\tau)\d \tau\right|=0 \text{ a.s.}\]
Then this uniform equicontinuity, the
continuity of $s \mapsto R_\theta(t,s)$ over $s \in [0,t]$ ensured by the domain $\cS_\theta^+$ of
Theorem \ref{thm:dmft-approx}, and
\eqref{eq:response} and Arzel\`a-Ascoli imply that for any $t \geq 0$,
\begin{equation}\label{eq:Rthetarepr}
\lim_{n \to \infty} \sup_{s \in [0,t]}
\left|R_\theta(t,s)-\frac{1}{n}\sum_{i=1}^n \langle \e_i^\top \nabla
P_{t-s}e_i(\btheta^s) \rangle\right|=0 \text{ a.s.}
\end{equation}

We now apply arguments analogous to part (b).
To define $r_\theta(\tau)$, let us establish the bound
\begin{align}
\left|\frac{1}{n}\sum_{i=1}^n \e_i^\top \nabla P_t e_i(\btheta)
-\frac{1}{n}\sum_{i=1}^n \e_i^\top \nabla P_t e_i({\btheta'})\right|
&\leq
\frac{Ce^{Ct}}{\sqrt{n}}\|\btheta-{\btheta'}\|_2\label{eq:gradRbound}
\end{align}
for any $\btheta,{\btheta'} \in \R^n$ and $t \geq 0$.
Let $\{\btheta^t\}_{t \geq 0}$ and
$\{{\btheta'}^t\}_{t \geq 0}$ be two diffusions with initial states
$\btheta,\btheta'$ coupled by the same Brownian motion, and define
correspondingly $\{{\J}^{t,0}\}_{t \geq 0}$ and $\{{\J'}^{t,0}\}_{t \geq 0}$
(restricting the above Jacobian processes to $s=0$).
The above arguments show the representations
$\frac{1}{n}\sum_{i=1}^n (\e_i^\top \nabla P_\tau e_i)(\btheta)
=\allowbreak\frac{1}{n}\E[\Tr \J^{\tau,0} \mid \X,\allowbreak\btheta^0=\btheta]$,
$\frac{1}{n}\sum_{i=1}^n (\e_i^\top \nabla P_\tau e_i)({\btheta'})
=\allowbreak\frac{1}{n}\E[\Tr {\J'}^{\tau,0} \mid \X,\allowbreak{\btheta^0}'=\btheta']$.
Under this coupling,
\[\frac{\d}{\d t}(\J^{t,0}-{\J'}^{t,0})
=\diag(U''(\btheta^t)-U''({\btheta'}^t))\J^{t,0}
+\diag(U''({\btheta'}^t))(\J^{t,0}-{\J'}^{t,0}).\]
Since $U''(\cdot)$ is bounded and Lipschitz under Assumption
\ref{assump:dynamics} for $f=-U'$,
the right side is bounded in Frobenius norm as
\begin{align*}
&\|\diag(U''(\btheta^t)-U''({\btheta'}^t))\J^{t,0}
+\diag(U''({\btheta'}^t))(\J^{t,0}-{\J'}^{t,0})\|_F\\
&\leq \|\diag(U''(\btheta^t)-U''({\btheta'}^t))\|_F \|\J^{t,0}\|_\op
+\|\diag(U''({\btheta'}^t))\|_\op\|\J^{t,0}-{\J'}^{t,0}\|_F\\
&\leq e^{Ct}\|\btheta^t-{\btheta'}^t\|_2+C\|\J^{t,0}-{\J'}^{t,0}\|_F.
\end{align*}
Then Gr\"onwall's inequality and \eqref{eq:thetainitcoupling} imply
\begin{align*}
&\left|\frac{1}{n}\E[\Tr \J^{t,0} \mid \X,\btheta^0=\btheta]-
\frac{1}{n}\E[\Tr {\J'}^{t,0} \mid \X,{\btheta^0}'=\btheta']\right|\\
&\leq \frac{1}{\sqrt{n}} \|\J^{t,0}-{\J'}^{t,0}\|_F
\leq \frac{Ce^{Ct}}{\sqrt{n}} \sup_{s \in [0,t]} \|\btheta^s-{\btheta'}^s\|_2
\leq \frac{Ce^{C't}}{\sqrt{n}}\|\btheta-{\btheta'}\|_2,
\end{align*}
which shows \eqref{eq:gradRbound}. Then, recalling the coupling of
$(\btheta^0,\btheta^s)$ and $\tilde\btheta^0$ which 
gives \eqref{eq:thetatthetacouple}, let us apply \eqref{eq:gradRbound}
with $\btheta=\btheta^s$ and $\btheta'=\tilde\btheta^0$. This shows
\begin{equation}\label{eq:Rrcompare}
\begin{aligned}
\left|\frac{1}{n}\sum_{i=1}^n \langle \e_i^\top \nabla P_\tau
e_i(\btheta^s)\rangle-\frac{1}{n}\sum_{i=1}^n \langle \e_i^\top \nabla P_\tau
e_i(\tilde \btheta^0) \rangle\right|
&\leq \frac{Ce^{C\tau}}{\sqrt{n}}\,
\E[\|\btheta^s-\tilde \btheta^0\|_2 \mid \X,\btheta^0]
\\
&
\leq Ce^{C\tau} \cdot e^{-cs}.
\end{aligned}
\end{equation}
By the existence of the almost sure limit
$R_\theta(s+\tau,s)$ in \eqref{eq:Rthetarepr}, we then have
\[\limsup_{n \to \infty} \frac{1}{n}\sum_{i=1}^n \langle \e_i^\top \nabla P_\tau
e_i(\tilde \btheta^0) \rangle
-\liminf_{n \to \infty} \frac{1}{n}\sum_{i=1}^n \langle \e_i^\top \nabla P_\tau
e_i(\tilde \btheta^0) \rangle \leq 2Ce^{C\tau} \cdot e^{-cs} \text{ a.s.}\]
Since $s \geq 0$ is arbitrary, this shows that for each fixed $\tau \geq 0$,
there exists a limit
\begin{equation}\label{eq:rthetadef}
r_\theta(\tau)=\lim_{n \to \infty}
\frac{1}{n}\sum_{i=1}^n \langle \e_i^\top \nabla P_\tau
e_i(\tilde \btheta^0) \rangle \text{ a.s.}
\end{equation}
Taking the limit $n \to \infty$ followed by $s \to \infty$ in
\eqref{eq:Rrcompare} shows that $r_\theta(\tau)$ is deterministic, with
$r_\theta(\tau)=\lim_{s \to \infty} R_\theta(s+\tau,s)$.

To bound $R_\theta(s+\tau,s)$, $r_\theta(\tau)$, and their difference, let us
check that
\begin{align}
\left\langle \left|\frac{1}{n}\sum_{i=1}^n \e_i^\top \nabla P_\tau
e_i(\btheta^s)  \right|\right\rangle,
\left\langle\left|\frac{1}{n}\sum_{i=1}^n \e_i^\top \nabla P_\tau e_i(\tilde
\btheta^0)\right| \right \rangle
&\leq Ce^{-c\tau}.\label{eq:Rbound}
\end{align}
Define as above $\{\J^{t,s}\}_{t \geq s \geq 0}$ from
$\{\btheta^t\}_{t \geq 0}$. Recalling the representation
$\frac{1}{n}\sum_{i=1}^n \e_i^\top \nabla P_\tau e_i(\btheta^s)
=\allowbreak\frac{1}{n}\langle \Tr \J^{s+\tau,s} \rangle$,
if $\tau \leq 1$ then the first statement of
\eqref{eq:Rbound} follows immediately from
the operator norm bound \eqref{eq:Jopbound}, which implies
$|\frac{1}{n}\langle \Tr \J^{s+\tau,s} \rangle| \leq \|\J^{s+\tau,s}\|_\op
\leq C$. For $\tau \geq 1$, we apply
a Bismut-Elworthy-Li representation for $\nabla P_t$: 
For any $f:\R^n \to \R$ with uniformly bounded first and second derivatives,
\[\nabla P_t f(\btheta^0)
=\frac{1}{t \sqrt{2}}\left \langle f(\btheta^t)\int_0^t (\J^{r,0})^\top
\d \b^r \right \rangle\]
where $\{\b^t\}_{t \geq 0}$ is the Brownian motion defining
$\{\btheta^t,\J^{t,0}\}_{t \geq 0}$ \cite[Lemma 4.2]{fan2025dynamicalII}.
Applying this with $t=1$ and $f_i=P_{\tau-1}e_i$, with the needed regularity of $f_i$ ensured by  \cite[Proposition A.2(c)]{fan2025dynamicalI},  we get
\[\frac{1}{n}\sum_{i=1}^n \e_i^\top \nabla P_\tau e_i(\btheta^0)
=\frac{1}{n}\sum_{i=1}^n \e_i^\top \nabla P_1 f_i(\btheta^0)
=\frac{1}{\sqrt{2}}\,\frac{1}{n}\left \langle 
P_{\tau-1}\id(\btheta^1)^\top\int_0^1 (\J^{r,0})^\top
\d \b^r \right \rangle.\]
Thus, noting that $\langle \int_0^1 (\J^{r,0})^\top \d \b^r \rangle=0$ and
applying Cauchy-Schwarz, Jensen's inequality, and the It\^o isometry, we have
\begin{align}
\sqrt{2} 
\left|\frac{1}{n}\sum_{i=1}^n \e_i^\top \nabla P_\tau e_i(\btheta^0)\right|
&=\left|\frac{1}{n}\left \langle 
\Big(P_{\tau-1}\id(\btheta^1)-\langle \btheta \rangle\Big)^\top\int_0^1
(\J^{r,0})^\top \d \b^r \right \rangle\right|\notag\\
&\leq \left\langle \frac{1}{n}\|P_{\tau-1}\id(\btheta^1)-\langle \btheta
\rangle\|_2^2 \right\rangle^{1/2} \cdot
\left\langle \frac{1}{n}\left\|\int_0^1 (\J^{r,0})^\top \d \b^r\right\|_2^2
\right\rangle^{1/2}\notag\\
&\leq \left\langle \frac{1}{n}\|\btheta^\tau-\langle \btheta
\rangle\|_2^2 \right\rangle^{1/2} \cdot
\left\langle \frac{1}{n}\int_0^1 \|\J^{r,0}\|_F^2\d r \right\rangle^{1/2}\notag\\
&\leq C\left\langle \frac{1}{n}\|\btheta^\tau-\langle \btheta
\rangle\|_2^2 \right\rangle^{1/2}\label{eq:responsebound},
\end{align}
the last inequality applying \eqref{eq:Jopbound}.
Applying this bound with $\btheta^s$ in place of $\btheta^0$ shows
\[\left|\frac{1}{n}\sum_{i=1}^n \e_i^\top \nabla P_\tau e_i(\btheta^s)\right|
\leq C\,\E\left[\frac{1}{n}\|\btheta^{s+\tau}-\langle \btheta
\rangle\|_2^2 \,\bigg|\,\X,\btheta^s\right]^{1/2}.\]
Then applying the coupling bound \eqref{eq:thetascoupling}, 
we obtain the first statement of
\eqref{eq:Rbound}. The second statement of \eqref{eq:Rbound}
follows similarly, applying the above
bound \eqref{eq:responsebound}
with $\tilde\btheta^0$ in place of $\btheta^0$ and using instead
the coupling \eqref{eq:thetaequilibriumcoupling}.

Taking the $n \to \infty$ limit in \eqref{eq:Rrcompare} and 
\eqref{eq:Rbound}, we have
\begin{equation}\label{eq:Rrcompare2}
|R_\theta(s+\tau,s)-r_\theta(\tau)| \leq Ce^{C\tau} \cdot e^{-cs},
\qquad R_\theta(s+\tau,s) \leq Ce^{-c\tau},
\qquad r_\theta(\tau) \leq Ce^{-c\tau}.
\end{equation}
Then, choosing a small enough constant $\iota>0$ such that the first bound
satisfies
$Ce^{C\tau} \cdot e^{-cs} \leq C'e^{-c's}$ when 
$\tau \leq \iota s/(1-\iota)$ (i.e.\ $s \geq (1-\iota)(s+\tau)$), we may write
\[\begin{aligned}
&\int_0^t |R_\theta(t,s)-r_\theta(t-s)| \d s\\
&=\int_{(1-\iota)t}^t |R_\theta(t,s)-r_\theta(t-s)| \d s+\int_0^{(1-\iota)t} |R_\theta(t,s)-r_\theta(t-s)| \d s.
\end{aligned}\]
Applying the first bound of \eqref{eq:Rrcompare2} for the first integral
and the latter two bounds of \eqref{eq:Rrcompare2} for the second integral,
this shows
\[\int_0^t |R_\theta(t,s)-r_\theta(t-s)| \d s \leq Ce^{-ct}.\]
Applying the first bound of \eqref{eq:Rrcompare2} when $\tau \leq \iota
s/(1-\iota)$ and the latter two bounds of \eqref{eq:Rrcompare2}
when $\tau \geq \iota s/(1-\iota)$, we have also
\[|R_\theta(s+\tau,s)-r_\theta(\tau)| \leq C'e^{-c's}.\]
This shows all statements of part (c).\\

{\bf Part (d).} Let
$L^2(\mu_\Gibbs)$ be the space of functions $f:\R^n \to \R$ for which
$\langle f(\btheta)^2 \rangle<\infty$. Since
$\{\tilde\btheta^t\}_{t \in \R}$ is reversible, for each fixed $n$, there is a
family $\{E_a\}_{a \geq 0}$ of orthogonal projections onto increasing closed
linear subspaces of $L^2(\mu_\Gibbs)$ (corresponding to the spectral
decomposition of the generator) such that for all $f,g \in
L^2(\mu_\Gibbs)$ and $\tau \geq 0$,
\[\langle f(\tilde\btheta^0) P_\tau g(\tilde\btheta^0) \rangle
=\int_0^\infty e^{-a \tau} \d \langle f(\btheta)E_a g(\btheta)\rangle\]
\cite[Theorem A.4.2]{bakry2014analysis}. Applying this with the coordinate
functions $f=g=e_i$ and averaging over $i=1,\ldots,n$, there is
then a positive measure $\mu_n$ (depending on $n$ and $\X$) for which
\begin{equation}\label{eq:Ptaurepr}
\frac{1}{n}\langle {\tilde \btheta}^{0\top} P_\tau\id(\tilde
\btheta^0)\rangle=\int_0^\infty e^{-a\tau} \mu_n(\d a).
\end{equation}
The bound \eqref{eq:thetanormbound} implies
$\int_0^\infty \mu_n(\d a)=\frac{1}{n}\,\langle \|\tilde\btheta^0\|_2^2 \rangle \leq C$. For any $M,\tau>0$, we have also
\[\int_0^M \mu_n(\d a)+e^{-M\tau}\int_M^\infty \mu_n(\d a)
\geq \int_0^\infty e^{-a\tau} \mu_n(\d a)=
\frac{1}{n}\langle {\tilde \btheta}^{0\top} P_\tau \id(\tilde \btheta^0)
\rangle\]
so
\begin{equation}\label{eq:tightnessbound}
\int_M^\infty \mu_n(\d a)
\leq \frac{1}{1-e^{-M\tau}}\left(\frac{1}{n}\,\langle \|\tilde\btheta^0\|_2^2
\rangle-\frac{1}{n}\langle {\tilde \btheta}^{0\top} P_\tau
\id(\tilde \btheta^0) \rangle\right).
\end{equation}
By It\^o's lemma, Cauchy-Schwarz, and \eqref{eq:thetanormbound},
\begin{align*}
\frac{\d}{\d t}\,\frac{1}{n}\langle \|\tilde\btheta^t-\tilde\btheta^0\|_2^2 \rangle
&=\frac{2}{n}\langle (\tilde\btheta^t-\tilde\btheta^0)^\top[\X\tilde\btheta^t-U'(\tilde\btheta^t)]
\rangle+2
\leq C\left(\frac{1}{n}\langle \|\tilde\btheta^t-\tilde\btheta^0\|_2^2 \rangle+1\right),
\end{align*}
and hence Gr\"onwall's inequality shows
\begin{equation}\label{eq:thetatildeequicontinuous}
\frac{1}{n}\langle \|\tilde\btheta^t-\tilde\btheta^0\|_2^2 \rangle
\leq C(e^{Ct}-1).
\end{equation}
Then by Jensen's inequality, also
$\frac{1}{n}\langle \|\tilde\btheta^0-P_t \id(\tilde\btheta^0)\|_2^2 \rangle
\leq C(e^{Ct}-1)$. Applying this to \eqref{eq:tightnessbound}, for all $n$,
\[\int_M^\infty \mu_n(\d a) \leq \frac{C'(e^{C\tau}-1)^{1/2}}{1-e^{-M\tau}}.\]
Taking $\tau=1/M$ and $M \to \infty$ shows that the sequence of measures
$\{\mu_n\}_{n=1}^\infty$ is uniformly tight, and hence converges weakly along a
subsequence of $n$ to a positive finite measure $\mu$ on $[0,\infty)$. Then
\eqref{eq:cthetadef} and \eqref{eq:Ptaurepr}
show that $c_\theta(\tau)$ admits a representation
\begin{equation}\label{eq:cthetarepr}
c_\theta(\tau)=\int_0^\infty e^{-a|\tau|} \mu(\d a) \text{ for all } \tau \in
\R.
\end{equation}
From this representation and dominated convergence, we see that
$c_\theta:\R \to \R$ is continuous, and its restriction to $(0,\infty)$ is
smooth where derivatives may be computed under the integral. Since
$\mu$ is a positive measure, this implies also for all $\tau>0$ that
\[c_\theta(\tau) \geq 0,
\qquad c_\theta'(\tau) \leq 0, \qquad c_\theta''(\tau) \geq 0.\]

By the fluctuation-dissipation relation for the diffusion
$\{\tilde \btheta^t\}_{t \in \R}$ (c.f.\ \cite[Lemma
4.1]{fan2025dynamicalII} applied with the coordinate functions $A=B=e_i$),
for any $\tau>0$ we have
\begin{equation}\label{eq:diffusionFDT}
\frac{1}{n}\sum_{i=1}^n \langle \e_i^\top \nabla P_\tau e_i(\tilde \btheta^0)
\rangle={-}\frac{\d}{\d \tau}
\left[\frac{1}{n}\langle \tilde \btheta^{0\top}\tilde \btheta^\tau
\rangle\right].
\end{equation}
As a function of $\tau$, the left side is uniformly equicontinuous 
on any compact interval of $[0,\infty)$ for all $n$, by the same argument as
above that established the equicontinuity of \eqref{eq:response}. Then
by Arzel\`a-Ascoli and the identities \eqref{eq:cthetadef}
and \eqref{eq:rthetadef}, we have
\[r_\theta(\tau)=-c_\theta'(\tau) \text{ for } \tau>0.\]
Thus for all $\tau>0$,
$r_\theta(\tau) \geq 0$ and $r_\theta'(\tau)={-}c_\theta''(\tau) \leq 0$.
Since the equicontinuity of the left side of \eqref{eq:diffusionFDT}
includes $\tau=0$, the identity
\eqref{eq:rthetadef} implies that $r_\theta:[0,\infty) \to \R$ is continuous at
0, so there exists a right derivative $c_\theta'(0^+)={-}r_\theta(0)$. Since
$\nabla P_\tau e_i=\e_i$ at $\tau=0$, \eqref{eq:rthetadef} shows that
$r_\theta(0)=1$, and hence $c_\theta'(0^+)=-1$.
\end{proof}

We now turn to Lemma \ref{lemma:ttig} on the analysis of $R_g,C_g$.

\begin{proof}[Proof of Lemma \ref{lemma:ttig}]
We extend the definitions of $R_\theta,C_\theta$ to $\R^2$ by setting
$R_\theta(t,s)=0$ for all $t<0$ or $s \notin [0,t]$, and
$C_\theta(t,s)=0$ for all $s<0$ or $t<0$. We also extend $r_\theta,r_\theta'$ to
$\R$ by setting $r_\theta(\tau)=0$ for $\tau<0$
and $r_\theta'(\tau)=0$ for $\tau \leq 0$.

We will apply the following convolution bound: If
$A_1,\ldots,A_p:\R^2 \to \R^2$ satisfy
$|A_k(t,s)| \leq C_0e^{-c_0|t-s|}$ for each $k=1,\ldots,p$ and all $s,t \in \R$,
then
\[\int e^{(c_0/2)|t-s|}|A_k(t,s)| \d s \leq C_0\int e^{-(c_0/2)|t-s|}
\d r=4C_0/c_0,\]
and hence
\begin{align*}
\sup_{s,t}e^{(c_0/2)|t-s|} |A_1*\ldots*A_p(t,s)| 
&\leq \sup_{s,t}\int e^{(c_0/2)|t-r|}|A_1(t,r)|e^{(c_0/2)|r-s|} |A_2*\ldots*A_p(r,s)| \d r\\
&\leq (4C_0/c_0)\sup_{r,s} e^{(c_0/2)|r-s|}|A_2*\ldots*A_p(r,s)|\\
&\leq \ldots \leq (4C_0/c_0)^{p-1} \cdot C_0.
\end{align*}
Thus for all $s,t \in \R$,
\begin{equation}\label{eq:convolutionbound}
|A_1*\ldots*A_p(t,s)| \leq
\left(\frac{4C_0}{c_0}\right)^{p-1}C_0e^{-(c_0/2)|t-s|}.
\end{equation}
We will apply also the basic identity, if $a_1,\ldots,a_k:\R \to \R$ are
integrable, then
\begin{equation}\label{eq:integralconvolution}
\int a_1*\ldots*a_k(\tau)\d \tau
=\prod_{i=1}^k \int a_i(\tau)\d \tau.
\end{equation}\\

{\bf Part (a):} We have
\[r_g(\tau)=\sum_{p \geq 1} \kappa_{p+1}\,r_\theta^{*p}(\tau).\]
By the fluctuation-dissipation relation \eqref{eq:crthetaFDT} of
Lemma \ref{lemma:tti},
\[\int r_\theta(\tau)\d \tau
=\int_0^\infty r_\theta(\tau)\d \tau=c_\theta(0)
-\lim_{\tau \to \infty} c_\theta(\tau)=v_*-q_*.\]
Then by \eqref{eq:integralconvolution}, for each $p \geq 1$,
\begin{equation}\label{eq:rthetapintegral}
\int_0^\infty r_\theta^{*p}(\tau)\d \tau
=\int r_\theta^{*p}(\tau)\d \tau=(v_*-q_*)^p.
\end{equation}
Lemma \ref{lemma:tti} implies $0 \leq r_\theta(\tau) \leq
r_\theta(0)=1$ for all $\tau \in \R$. Then also for each $p \geq 1$,
\begin{equation}
0 \leq r_\theta^{*p}(\tau)
\leq \int r_\theta^{*(p-1)}(s) \d s
=(v_*-q_*)^{p-1}.\label{eq:rthetapbound1}
\end{equation}
The condition \eqref{eq:Rdomain} then implies that the above series defining
$r_g(\tau)$ is absolutely and uniformly convergent over $\tau \in [0,\infty)$.
Since each summand is continuous in $\tau$, this shows $r_g(\tau)$ is continuous
on $[0,\infty)$. The bound \eqref{eq:intrgbound} follows from
\eqref{eq:rthetapintegral} and \eqref{eq:Rdomain},
\[\int_0^\infty |r_g(\tau)| \d \tau
\leq \sum_{p \geq 1} |\kappa_{p+1}| \int_0^\infty r_\theta^{*p}(\tau) \d \tau
<\alpha.\]

For any integrable $a:\R \to \R$ that is continuous and supported on
$[0,\infty)$ and continuously-differentiable on $(0,\infty)$,
we have $[a*b]'(\tau)=a(0)b(\tau)+[a'*b](\tau)$. Thus
the derivative of $r_g$ exists on $(0,\infty)$ and is given by 
\begin{equation}\label{eq:rgderiv}
r_g'(\tau)=\kappa_2r_\theta'(\tau)+\sum_{p \geq 2} \kappa_{p+1}
\Big(r_\theta^{*(p-1)}(\tau)+[r_\theta'*r_\theta^{*(p-1)}](\tau)\Big)
\end{equation}
provided that this series converges absolutely and uniformly over
compact subsets of $(0,\infty)$. Note that by Lemma \ref{lemma:tti},
\[\int r_\theta'(\tau)\d \tau
=\int_0^\infty r_\theta'(\tau)\d \tau={-}r_\theta(0)=-1.\]
Then since $r_\theta'(\tau) \leq 0$ and $r_\theta(\tau) \geq 0$,
for any $p \geq 1$, \eqref{eq:integralconvolution} implies
\begin{align*}
0 \leq {-}[r_\theta'*r_\theta^{*p}](\tau)
\leq -\int r_\theta'*r_\theta^{*(p-1)}(\tau)\d\tau
=(v_*-q_*)^{p-1}.
\end{align*}
Then \eqref{eq:rgderiv} is absolutely
and uniformly convergent on $(0,\infty)$ as claimed.
Each summand is continuous in $\tau$, implying 
that $r_g'$ is continuous. Furthermore, from the
signs of $r_\theta$ and $r_\theta'$ and the above calculations, we have
\[\int_0^\infty |r_g'(\tau)|\,\d \tau
\leq \kappa_2+\sum_{p \geq 2} |\kappa_{p+1}|
\left(\int_0^\infty r_\theta^{*(p-1)}(\tau)\d\tau
-\int_0^\infty r_\theta'*r_\theta^{*(p-1)}(\tau)\d \tau\right)<\infty.\]

To establish exponential decay of $r_g(\tau)$, note that the convergence
condition of \eqref{eq:Rdomain} implies
$\limsup_{p \to \infty} |\kappa_{p+1}|^{1/p}(v_*-q_*+\eps) \leq 1$.
Thus, for some $p_0>0$, we have
\begin{equation}\label{eq:p0def}
|\kappa_{p+1}| \leq \left(\frac{1}{v_*-q_*+\eps/2}\right)^p \text{ for all }
p \geq p_0.
\end{equation}
Then, letting $\delta>0$ be a small enough constant to be determined,
the bound \eqref{eq:rthetapbound1} implies
\[\sum_{p \geq \max(\delta\tau,p_0)} |\kappa_{p+1}|
\cdot r_\theta^{*p}(\tau) \leq \sum_{p \geq \delta\tau}
\left(\frac{1}{v_*-q_*+\eps/2}\right)^p
\cdot (v_*-q_*)^{p-1} \leq Ce^{-c\delta\tau}\]
for some constants $C,c>0$ depending on $v_*,q_*,\eps$. For
$p<\max(\delta\tau,p_0)$, we have by Lemma \ref{lemma:tti} and
\eqref{eq:convolutionbound} applied to $A_k(t,s)=r_\theta(t-s)$
that $r_\theta^{*p}(\tau) \leq C^p e^{-c\tau}$
for some $C,c>0$, any $p \geq 1$, and all $\tau \geq 0$.
We also have $|\kappa_p| \leq C^p$ for a constant $C>0$ and all
$p \geq 1$, by Proposition \ref{prop:freecumulants}. Thus
\begin{align*}
\sum_{p=1}^{\max(\delta\tau,p_0)}
|\kappa_{p+1}| \cdot r_\theta^{*p}(\tau)
\leq \sum_{p=1}^{\max(\delta\tau,p_0)}
C^{p+1} \cdot C^pe^{-c\tau}
&\leq C'e^{-c\tau} \cdot e^{C'\max(\delta \tau,p_0)}
\end{align*}
for some constants $C',c>0$ not depending on $\delta$. Choosing $\delta$
small enough ensures this is at most $C''e^{-(c/2)\tau}$ for all $\tau \geq 0$
and a constant $C''>0$ (depending on $p_0$, in the case $\delta\tau<p_0$). Combining the above two statements, for some $C,c>0$,
\begin{equation}\label{eq:rgdecayproof}
|r_g(\tau)|\leq \sum_{p \geq 1} |\kappa_{p+1}| \cdot r_\theta^{*p}(\tau)
\leq Ce^{-c\tau}.
\end{equation}

To bound $R_g(t,s)$, recall that
$R_g(t,s)=\sum_{p \geq 1} \kappa_{p+1}R_\theta^{*p}(t,s)$ where
\[R_\theta^{*p}(t,s)
=\int_s^t \d t_1 \int_s^{t_1} \d t_2 \ldots \int_s^{t_{p-2}} \d t_{p-1}
R_\theta(t,t_1)R_\theta(t_1,t_2)\ldots R_\theta(t_{p-1},s).\]
We establish analogues of \eqref{eq:rthetapintegral} and
\eqref{eq:rthetapbound1} for $R_\theta$:
Letting $\eps>0$ be the constant in \eqref{eq:Rdomain}, Lemma
\ref{lemma:tti} ensures there exists $s_0>0$ such that if $t \geq
s_0$, then $\int_0^t |R_\theta(t,s)-r_\theta(t-s)|\d s<\eps/4$.
Hence
\[\int_0^t |R_\theta(t,s)|\d s \leq v_*-q_*+\eps/4 \text{ for } t \geq s_0.\]
Then
\begin{align*}
\int_0^t |R_\theta^{*p}(t,s)| \d s
&\leq \int_0^t \d t_1
\ldots \int_0^{t_{p-2}}\d t_{p-1}
\int_0^{t_{p-1}}\d s\,
|R_\theta(t,t_1)| \ldots |R_\theta(t_{p-1},s)|\\
&\leq \sum_{k=0}^{p-1} \int_{s_0}^t \d t_1
\ldots \int_{s_0}^{t_{k-1}} \d t_k
\int_0^{s_0} \d t_{k+1} \ldots \int_0^{t_{p-1}}\d t_p\,
|R_\theta(t,t_1)| \ldots |R_\theta(t_{p-1},t_p)|,
\end{align*}
where we set $t_p=s$ and
stratify the domain of integration $t \geq t_1 \geq \ldots \geq t_p \geq 0$ by
the largest index $k$ for which $t_k>s_0$, with $k=0$ if $t_1\leq s_0$.
Bounding $|R_\theta(t_j,t_{j+1})| \leq C$ for all $j \geq k$, this gives
\begin{align}
\int_0^t |R_\theta^{*p}(t,s)|\d s
&\leq \sum_{k=0}^{p-1} \int_{s_0}^t \d t_1
\ldots \int_{s_0}^{t_{k-1}} \d t_k\,|R_\theta(t,t_1)|\ldots
|R_\theta(t_{k-1},t_k)| \cdot \frac{(Cs_0)^{p-k}}{(p-k)!}\notag\\
&\leq \sum_{k=0}^{p-1} (v_*-q_*+\eps/4)^k \cdot \frac{(Cs_0)^{p-k}}{(p-k)!}\notag\\
&=\sum_{k=0}^{p-1} \frac{[Cs_0/(v_*-q_*+\eps/4)]^{p-k}}{(p-k)!}
(v_*-q_*+\eps/4)^p
\leq C'(v_*-q_*+\eps/4)^p\label{eq:Rthetapintegral}
\end{align}
for a constant $C'>0$ depending on $v_*,q_*,\eps,s_0$, which is an analogue of
\eqref{eq:rthetapintegral}.
Similarly, first bounding $|R_\theta(t_{p-1},s)| \leq C$ and applying the same
argument, we have
\begin{equation}\label{eq:Rthetapbound1}
|R_\theta^{*p}(t,s)| \leq C''(v_*-q_*+\eps/4)^{p-1}
\end{equation}
which is an analogue of \eqref{eq:rthetapbound1}.

Choosing a small enough constant $\delta>0$, and
applying \eqref{eq:Rthetapbound1} for $p \geq \max(\delta(t-s),p_0)$
and \eqref{eq:convolutionbound} for $p<\max(\delta(t-s),p_0)$,
the same argument as for $r_g(\tau)$ above shows
\begin{equation}\label{eq:Rgdecayproof}
|R_g(t,s)| \leq Ce^{-c(t-s)}.
\end{equation}
To bound $|R_g(t,s)-r_g(t-s)|$, again fix a constant $\delta>0$ small
enough. Observe that \eqref{eq:rthetapbound1} and \eqref{eq:Rthetapbound1} imply
\[\sum_{p \geq \max(\delta s,p_0)}
|\kappa_{p+1}| \cdot |R^{*p}_\theta(t,s)-r_\theta(t-s)^{*p}|
\leq Ce^{-c\delta s}.\]
For the summands $p<\max(\delta s,p_0)$, identify $r_\theta(t,s) \equiv
r_\theta(t-s)$ and write
\begin{align*}
R_\theta^{*p}(t,s)-r_\theta^{*p}(t-s)
&=\sum_{k=0}^{p-1} R_\theta^{*k} * (R_\theta-r_\theta)
* r_\theta^{*(p-1-k)}(t,s)\\
&=\sum_{k=0}^{p-1} \int_s^t \d t_1
\int_s^{t_1} \d t_2\,R_\theta^{*k}(t,t_1)
[R_\theta-r_\theta](t_1,t_2)
r_\theta^{*(p-1-k)}(t_2,s).
\end{align*}
By Lemma \ref{lemma:tti}, we have $|[R_\theta-r_\theta](t_1,t_2)|
\leq Ce^{-ct_2} \leq Ce^{-cs}$. Then,
together with \eqref{eq:rthetapintegral} and \eqref{eq:Rthetapintegral},
this gives
\begin{align*}
|R_\theta^{*p}(t,s)-r_\theta^{*p}(t-s)|
&\leq \sum_{k=0}^{p-1}
Ce^{-cs} \int_s^t \d t_1 \int_s^{t_1} \d t_2 \, |R_\theta^{*k}(t,t_1)|
\cdot r_\theta^{*(p-1-k)}(t_2-s)\\
&\leq C'pe^{-cs}(v_*-q_*+\eps/4)^{p-1}.
\end{align*}
Applying again $|\kappa_p| \leq C^p$, we then have
\[\sum_{p=1}^{\max(\delta s,p_0)}
|\kappa_{p+1}| \cdot |R_\theta(t,s)^{*p}-r_\theta(t-s)^{*p}|
\leq Ce^{-cs} \cdot e^{C\max(\delta s,p_0)}.\]
This is at most $C'e^{-(c/2)s}$ for $\delta>0$ small enough, and combining with
the above yields
\[|R_g(t,s)-r_g(t-s)| \leq Ce^{-cs}.\]
By this bound, \eqref{eq:rgdecayproof}, and \eqref{eq:Rgdecayproof},
also $\int_0^t |R_g(t,s)-r_g(t-s)|\d s \leq Ce^{-ct}$ by
the same argument as following \eqref{eq:Rrcompare2}, 
showing all statements of part (a).\\

{\bf Part (b):} Recall that
\[c_g(\tau)=\sum_{p,q \geq 0} \kappa_{p+q+2}\,
r_\theta^{*p}*c_\theta*\bar r_\theta^{*q}(\tau).\]
Bounding $0 \leq c_\theta(\tau) \leq C$ shows that
$0 \leq r_\theta^{*p}*c_\theta*\bar r_\theta^{*q}(\tau) \leq
C\int r_\theta^{*p}*\bar r_\theta^{*q}(s)\d s=C(v_*-q_*)^{p+q}$,
so by \eqref{eq:Rdomain}, the above series defining
$c_g(\tau)$ is absolutely and uniformly convergent over $\tau \in \R$.
Since each summand is continuous, $c_g:\R \to \R$ is continuous.

Write $c_\theta(\tau)=c_\theta^0(\tau)+q_*$. Note that
for the constant function $q_*(\tau) \equiv q_*$, we have
\begin{align*}
\sum_{p,q \geq 0} \kappa_{p+q+2}\,r_\theta^{*p}
*q_* *\bar r_\theta^{*q}(\tau)
&=\sum_{p,q \geq 0} \kappa_{p+q+2}
\cdot q_*\int r_\theta^{*p}*\bar r_\theta^{*q}(s)\d s\\
&=q_*\sum_{p,q \geq 0} \kappa_{p+q+2}(v_*-q_*)^{p+q}\\
&=q_*\sum_{k \geq 0} \kappa_{k+2} (k+1)(v_*-q_*)^k
=q_*\cR_\Lambda'(v_*-q_*),
\end{align*}
where the last identity holds since \eqref{eq:Rdomain} implies
$\cR_\Lambda(x)=\sum_{p \geq 0} \kappa_{p+1} x^p$ is absolutely convergent in a
neighborhood of $v_*-q_*$. Thus
\[c_g(\tau)=c_g^0(\tau)+q_*\cR_\Lambda'(v_*-q_*),
\qquad c_g^0(\tau):=\sum_{p,q \geq 0} \kappa_{p+q+2}\,r_\theta^{*p}
*c_\theta^0*\bar r_\theta^{*q}(\tau).\]

To show exponential decay of $c_g^0(\tau)$, fix a constant $\delta>0$,
and recall $p_0$ for which \eqref{eq:p0def} holds. Then, for summands where $p+q
\geq \max(\delta\tau,p_0)$,
bounding $|c_\theta^0(\tau)| \leq C$ for all $\tau \in \R$, we have as above
\[\begin{aligned}
&\sum_{p,q:\,p+q \geq \max(\delta\tau,p_0)}
|\kappa_{p+q+2}| \cdot |[r_\theta^{*p}*c_\theta^0
*\bar r_\theta^{*q}](\tau)|
\\
&\qquad\leq C\sum_{p,q:\,p+q \geq \max(\delta\tau,p_0)}
|\kappa_{p+q+2}| \cdot (v_*-q_*)^{p+q} \leq C'e^{-c\delta\tau}.
\end{aligned}\]
For $p+q \leq \max(\delta\tau,p_0)$, 
applying the bounds $|r_\theta(\tau)|,|c_\theta^0(\tau)| \leq Ce^{-c\tau}$ from
Lemma \ref{lemma:tti} together with \eqref{eq:convolutionbound},
\begin{align*}
&\sum_{p,q:\,p+q<\max(\delta\tau,p_0)}
|\kappa_{p+q+2}| \cdot |[r_\theta^{*p}*c_\theta^0 *\bar r_\theta^{*q}](\tau)|
\\
&\hspace{1in}\leq \sum_{p,q:\,p+q<\max(\delta\tau,p_0)}
C^{p+q+1}e^{-c\tau}
\leq C'e^{-c\tau} \cdot e^{C'\max(\delta\tau,p_0)}.
\end{align*}
Choosing $\delta>0$ small enough and combining with the above shows
$\lim_{\tau \to \infty} c_g(\tau)=q_*\cR_\Lambda'(v_*-q_*)$ and
\[|c_g(\tau)-q_*\cR_\Lambda'(v_*-q_*)|
=|c_g^0(\tau)| \leq Ce^{-c\tau}.\]

To bound $|C_g(t,s)-c_g(t-s)|$, identify $r_\theta(t,s) \equiv r_\theta(t-s)$
and $c_\theta(t,s) \equiv c_\theta(t-s)$, and note that
\[C_g(t,s)-c_g(t-s)=\sum_{p,q \geq 0}
\kappa_{p+q+2} \big(R_\theta^{*p} * C_\theta * \bar R_\theta^{*q}
-r_\theta^{*p} * c_\theta * \bar r_\theta^{*q}\big)(t,s).\]
Fixing a small constant $\delta>0$, for summands where $p+q \geq
\max(\delta\min(s,t),p_0)$, we may bound $|C_\theta(t,s)|,|c_\theta(t,s)| \leq
C$ and apply \eqref{eq:rthetapintegral} and \eqref{eq:Rthetapintegral} to get
\[\sum_{p,q:\,p+q \geq \max(\delta\min(s,t),p_0)}
|\kappa_{p+q+2}| \cdot \big|\big(R_\theta^{*p} * C_\theta * \bar R_\theta^{*q}
-r_\theta^{*p} * c_\theta * \bar r_\theta^{*q}\big)(t,s)\big|
\leq Ce^{-c\delta\min(s,t)}.\]
For the summands where $p+q<\max(\delta \min(s,t),p_0)$, write
\[\big(R_\theta^{*p} * C_\theta * \bar R_\theta^{*q}
-r_\theta^{*p} * c_\theta * \bar r_\theta^{*q}\big)(t,s)=\I+\II+\III\]
where
\begin{align*}
\I&=R_\theta^{*p}*(C_\theta-c_\theta)*\bar R_\theta^{*q}(t,s),\\
\II&=\sum_{k=0}^{p-1} R_\theta^{*k} * (R_\theta-r_\theta)
*r_\theta^{*(p-1-k)}*c_\theta*\bar R_\theta^{*q}(t,s),\\
\III&=\sum_{k=0}^{q-1} r_\theta^{*p}*c_\theta*\bar r_\theta^{*(q-1-k)}
*(\bar R_\theta-\bar r_\theta)*\bar R_\theta^{*k}(t,s).
\end{align*}
For $\I$, applying $|C_\theta(t',s')-c_\theta(t',s')| \leq Ce^{-c\min(s',t')}$
from Lemma \ref{lemma:tti}, and \eqref{eq:convolutionbound} to bound
$|R_\theta^{*p}|$ and $|R_\theta^{*q}|$, we have
\begin{align*}
|\I|& \leq \int_0^t \d t' \int_0^s \d s'\,
|R_\theta^{*p}(t,t')| \cdot |C_\theta(t',s')-c_\theta(t',s')|
\cdot |R_\theta^{*q}(s,s')|\\
&\leq \int_0^t\d t' \int_0^s\d s'
\,Ce^{-c\min(s',t')} \cdot C^{p+q}e^{-c(t-t')-c(s-s')}.
\end{align*}
Considering separately the regions of integration where $\min(s',t') \geq
\frac{1}{2}\min(s,t)$, where
$s' \leq t'$ and $s' \leq \frac{1}{2}\min(s,t)$ in which
case $s-s' \geq \frac{1}{2}\min(s,t)$,
and where $t' \leq s'$ and $t' \leq \frac{1}{2}\min(s,t)$ in which
case $t-t' \geq \frac{1}{2}\min(s,t)$, we get
\[|\I|\leq C^{p+q} \cdot C'e^{-(c/2)\min(s,t)}\]
For $\II$, let us apply from Lemma \ref{lemma:tti} that
$|(R_\theta-r_\theta)(t,s)| \leq Ce^{-cs}$ when $s \geq 0$,
and $|(R_\theta-r_\theta)(t,s)|=|r_\theta(t-s)| \leq C$ when $s<0$.
Let us also apply \eqref{eq:convolutionbound} to bound
$r_\theta^{*(p-1-k)}*c_\theta*\bar R_\theta^{*q}$.
Then each term of $\II$ is bounded similarly as
\begin{align*}
&|R_\theta^{*k} * (R_\theta-r_\theta)
*r_\theta^{*(p-1-k)}*c_\theta*\bar R_\theta^{*q}(t,s)|\\
&\leq \int_0^t \d t_1 \int_{-\infty}^{t_1} \d t_2\,
|R_\theta^{*k}(t,t_1)| \cdot |(R_\theta-r_\theta)(t_1,t_2)|
\cdot |r_\theta^{*(p-1-k)}*c_\theta*\bar R_\theta^{*q}(t_2,s)|\\
&\leq \int_0^t \d t_1 \int_{-\infty}^{t_1} \d t_2\,
C^k e^{-c(t-t_1)} \cdot Ce^{-c\max(t_2,0)}
\cdot C^{p+q-k}e^{-c|t_2-s|}\\
&\leq C^{p+q}\int_{-\infty}^\infty
C'e^{-c\max(t_2,0)-c|t_2-s|}\d t_2.
\end{align*}
Considering separately the regions where $t_2 \leq s/2$ and $t_2>s/2$, this is
at most $C^{p+q} \cdot C''e^{-(c/2)s}$. Thus
\[|\II| \leq C^{p+q} \cdot p \cdot C''e^{-(c/2)s}.\]
Similarly, we have $|\III| \leq C^{p+q} \cdot q \cdot C''e^{-(c/2)t}$. Combining
these bounds, for some constants $C,c>0$,
\[|\I+\II+\III| \leq C^{p+q+1} \cdot (p+q+1) \cdot e^{-c\min(s,t)}.\]
Thus
\begin{align*}
&\sum_{p,q:\,p+q<\max(\delta\min(s,t),p_0)}
|\kappa_{p+q+2}| \cdot \big|\big(R_\theta^{*p} * C_\theta * \bar R_\theta^{*q}
-r_\theta^{*p} * c_\theta * \bar r_\theta^{*q}\big)(t,s)\big|\\
&\hspace{2in}\leq C'e^{-c\min(s,t)} \cdot e^{C'\max(\delta\min(s,t),p_0))},
\end{align*}
which is at most $C''e^{-(c/2)\min(s,t)}$ for small enough $\delta>0$. This
establishes
\[|C_g(t,s)-c_g(t-s)| \leq Ce^{-c\min(s,t)},\]
showing both bounds of \eqref{eq:cgbounds}.
Theorem \ref{thm:dmft-approx} implies $C_g$ is a positive-semidefinite kernel.
Since $c_g(\tau)=\lim_{s \to \infty} C_g(s+\tau,s)$ for each fixed $\tau
\in \R$, this implies that $c_g$ is symmetric positive-semidefinite
on $\R$. Since $q_*\cR_\Lambda'(v_*-q_*)=\lim_{\tau \to \infty}
c_g(\tau)$, we then have
\[q_*\cR_\Lambda'(v_*-q_*) \geq 0.\]

Finally, let us check the differentiability of $c_g^0(\tau)=c_g(\tau)-q_*\cR_{\Lambda}'(v_*-q_*)$,
the fluctuation-dissipation relation \eqref{eq:cgrgFDT}, and the value of
$c_g(0)$. Analogously to \eqref{eq:rgderiv}, $c_g^0$ is continuously-differentiable
on $\R \setminus \{0\}$, with derivative
\begin{align}
{c_g^0}'(\tau)&=
\sum_{p,q \geq 0} \kappa_{p+q+2} (r_\theta^{*p}*c_\theta^0*\bar r_\theta^{*q})'(\tau)\notag\\
&=\sum_{p \geq 1,\,q \geq 0}
\kappa_{p+q+2} \left(r_\theta^{*(p-1)}*c_\theta^0*\bar r_\theta^{*q}(\tau)
+r_\theta'*r_\theta^{*(p-1)}*c_\theta^0*\bar r_\theta^{*q}(\tau)\right)\notag\\
&\hspace{0.5in}+\sum_{q \geq 1} \kappa_{q+2}
\left(c_\theta^0 * \bar r_\theta^{*(q-1)}(\tau)
+c_\theta^0*\bar r_\theta^{*(q-1)} * \bar r_\theta'(\tau)\right)
+\kappa_2 c_\theta'(\tau).\label{eq:cgderiv}
\end{align}
This holds since the bound $0 \leq c_\theta(\tau) \leq C$
and \eqref{eq:rthetapintegral}
imply that \eqref{eq:cgderiv} is absolutely and uniformly convergent over
$\tau \in \R \setminus \{0\}$.
From the existence of $c_\theta'(0^+)$ in Lemma \ref{lemma:tti},
this shows also the existence of $c_g'(0^+)={c_g^0}'(0^+)$. 
Setting $r_g(\tau)=0$ for $\tau<0$ and $\bar r_g(\tau)=r_g(-\tau)$,
and applying \eqref{eq:integralconvolution} and
\eqref{eq:rthetapintegral}, both the series \eqref{eq:cgderiv} and
\[r_g(\tau)=\sum_{p \geq 1} \kappa_{p+1} r_\theta^{*p}(\tau),
\qquad \bar r_g(\tau)=\sum_{p \geq 1} \kappa_{p+1} \bar r_\theta^{*p}(\tau)\]
are also absolutely convergent in $L^1(\R)$. Then, writing
$\cF[f](\omega)=\int f(t) e^{-i\omega t} \d t$ for the Fourier transform
and applying $\cF[f'](\omega)=i\omega \cF[f](\omega)$ on the Sobolev space $f
\in W^{1,1}$, we have 
\begin{equation}\label{eq:Fouriercalculations}
\begin{aligned}
\cF[c_g^0](\omega)
&=\sum_{p,q \geq 0} \kappa_{p+q+2}
\cF[(r_\theta^{*p}*c_\theta^0*\bar r_\theta^{*q})](\omega)\\
&=\sum_{p,q \geq 0} \kappa_{p+q+2}
\cF[r_\theta](\omega)^p\cF[c_\theta^0](\omega)\cF[\bar r_\theta](\omega)^q\\
&=\cF[c_\theta^0](\omega)
\sum_{k \geq 0} \kappa_{k+2} \frac{\cF[r_\theta](\omega)^{k+1}
-\cF[\bar r_\theta](\omega)^{k+1}}{\cF[r_\theta](\omega)-\cF[\bar r_\theta](\omega)},\\
\cF[{c_g^0}'](\omega)&=
i\omega \cF[c_\theta^0](\omega)
\sum_{k \geq 0} \kappa_{k+2} \frac{\cF[r_\theta](\omega)^{k+1}
-\cF[\bar r_\theta](\omega)^{k+1}}{\cF[r_\theta](\omega)-\cF[\bar r_\theta](\omega)},\\
\cF[r_g](\omega)&=\sum_{p \geq 1} \kappa_{p+1}\cF[r_\theta](\omega)^p,
\qquad \cF[\bar r_g](\omega)=\sum_{p \geq 1} \kappa_{p+1}\cF[\bar r_\theta](\omega)^p.
\end{aligned}
\end{equation}
Applying the fluctuation-dissipation relation
$c_\theta'(\tau)={c_\theta^0}'(\tau)={-}r_\theta(\tau)$ from \eqref{eq:crthetaFDT}
for $\tau>0$, we have
\[\cF[r_\theta](\omega)=\int_0^\infty r_\theta(\tau)e^{-i\omega \tau}
\d\tau=\int_0^\infty {-}{c_\theta^0}'(\tau)e^{-i\omega \tau}
\d\tau=c_\theta^0(0)-\int_0^\infty i\omega c_\theta^0(\tau)e^{-i\omega \tau}
\d \tau,\]
and likewise
\[\cF[\bar r_\theta](\omega)=\int_{-\infty}^0 r_\theta(-\tau)
e^{-i\omega\tau}\d \tau=\int_{-\infty}^0 {c_\theta^0}'(\tau)e^{-i\omega \tau}\d \tau
=c_\theta^0(0)+\int_{-\infty}^0 i\omega c_\theta^0(\tau)e^{-i\omega \tau}\d \tau.\]
Thus
\[\cF[r_\theta](\omega)-\cF[\bar r_\theta](\omega)
={-}i\omega \int c_\theta^0(\tau)e^{-i\omega \tau}\d \tau
={-}i\omega \cF[c_\theta^0](\omega).\]
Applying this above shows
\[\cF[{c_g^0}'](\omega)=-(\cF[r_g](\omega)-\cF[\bar r_g](\omega))
=-\int \sign(\tau)r_g(|\tau|) e^{-i\omega\tau} \d \tau.\]
Thus $c_g'(\tau)={c_g^0}'(\tau)={-}\sign(\tau)r_g(|\tau|)$ for a.e.\ $\tau \in \R$.
The continuity of both functions restricted to $(0,\infty)$, and also
from the right at $\tau=0$, implies ${-}c_g'(\tau)=r_g(\tau)$ for $\tau>0$ and
${-}c_g'(0^+)=r_g(0)$, verifying \eqref{eq:cgrgFDT}. Finally, 
recalling that $\lim_{\tau \to \infty} c_g(\tau)=q_*\cR_\Lambda'(v_*-q_*)$,
this implies that
\[\begin{aligned}
c_g(0)-q_*\cR_\Lambda'(v_*-q_*)&=\int_0^\infty r_g(\tau) \d \tau\\
&=\sum_{p \geq 1} \kappa_{p+1} \int_0^\infty r_\theta^{*p}(\tau) \d \tau=\sum_{p \geq 1} \kappa_{p+1} (v_*-q_*)^p=\cR_\Lambda(v_*-q_*).
\end{aligned}\]
Since $c_g:\R \to \R$ is positive-semidefinite, we must have
\[\cR_\Lambda(v_*-q_*)=\allowbreak c_g(0)-\allowbreak
\lim_{\tau \to \infty} c_g(\tau) \geq 0.\]
This shows all statements of part (b).
\end{proof}

\section{Replica-symmetric overlaps and free energy}\label{sec:replica}

In this section, we prove Theorem \ref{thm:replica} on the asymptotic overlaps
and free energy of the Gibbs measure \eqref{eq:gibbsmeasure}.

The argument is to establish a Markovian approximation of the generalized
Langevin
equation \eqref{eq:dmft-theta}, using the fluctuation-dissipation relation between $c_g,r_g$ in
Lemma \ref{lemma:ttig}, an approximation of $c_g^0(\tau)=c_g(\tau)-\lim_{s \to \infty}
c_g(s)$ by the correlation function
of observables of the lifted Markovian process
$\bxi^t=\{\btheta^s\}_{s \in [t-B,t]}$ over path histories, a further
approximation of this correlation function by a kernel $\tilde c_g$ defined
via a finite-dimensional approximation of its infinitesimal generator,
and a stable coupling between the dynamics driven by $c_g^0$ and $\tilde c_g$.

The approximation of $c_g^0$ (and its derivative ${-}r_g$) is carried out in the
following lemma.

\begin{lemma}\label{lemma:semigroupapprox}
Suppose the conditions of Theorem \ref{thm:replica} hold. Let
$v_*,q_*,c_g(\tau),r_g(\tau)$ be as defined in Lemmas \ref{lemma:tti} and
\ref{lemma:ttig}, and set
\[c_g^0(\tau)=c_g(\tau)-q_*\cR_\Lambda'(v_*-q_*).\]
Fix any $\eps>0$. Then there exists a constant $C>0$
independent of $\eps$, and some $T_0=T_0(\eps)>0$
and function $\ell_*:[0,T_0]^2 \to \R$, such that the following holds: 

For any $T>0$ and $\delta>0$, there exist
$m \geq 1$, $u_m \in \R^m$, and $A_m \in \R^{m \times m}$ with
$A_m+A_m^\top$ strictly positive definite,
all depending on $\delta,T,\eps,T_0$, such that:
\begin{enumerate}[(i)]
\item $|u_m^\top e^{-\tau A_m}u_m
-c_g^0(\tau)|<\delta$ for all $\tau \in [0,T]$.
\item $\int_0^\infty |u_m^\top A_me^{-\tau A_m}u_m- r_g(\tau)|\d \tau<\delta$.
\item $\|u_m\|_2<C$.
\item $\|e^{-tA_m}u_m\|_2<\eps$ for all $t \geq T_0$.
\item $\|(e^{-tA_m}-e^{-sA_m})u_m\|_2<C|t-s|^{1/2}+\delta$
and $\|A_m(e^{-tA_m}-e^{-sA_m})u_m\|_2<C|t-s|^{1/2}+\delta$
for all $s,t \in [0,T_0]$.
\item $|u_m^\top e^{-tA_m^\top}A_me^{-sA_m}u_m-\ell_*(t,s)|<\delta$
for all $s,t \in [0,T_0]$.
\end{enumerate}
\end{lemma}
We emphasize that for our later coupling argument, it is important that the
constants $C,\eps,T_0$ in statements (iii)--(vi) are fixed irrespective of $m$.

\begin{proof}
We again assume, without loss of generality, that Assumption \ref{assump:as}
holds. The process $\{\tilde\btheta^t\}_{t \in \R}$ and conventions
for $\langle \cdot \rangle$ are the same as in the proof
of Lemma \ref{lemma:tti}.\\

{\bf Step 1 (Markov semigroup setup):} 
For a value $B=B(n)>0$ to be determined, let
$\cS=C([-B,0],\R^n)$ denote the space of continuous bounded
$\R^n$-valued functions on $[-B,0]$. We denote elements of this space by
$\bxi \in \cS$, indexed as $\bxi=\{\bxi_\tau\}_{\tau \in [-B,0]}$.
$\cS$ is equipped with the norm
\[\|\bxi\|=\sup_{\tau \in [-B,0]} \|\bxi_\tau\|_2\]
and its associated Borel sigma-algebra. Denote
\[\bxi^t=\{\tilde \btheta^{t+\tau}\}_{\tau \in [-B,0]},
\text{ i.e. } \bxi_\tau^t=\tilde \btheta^{t+\tau} \text{ for } \tau \in
[-B,0].\]
We use subscripts to index elements of $\cS$
and superscripts to index the process $\{\bxi^t\}_{t \in \R}$ on $\cS$.

Let $\nu$ be the law of $\bxi^0=\{\tilde\btheta^\tau\}_{\tau \in [-B,0]}$
on $\cS$. Then $\{\bxi^t\}_{t \in \R}$ is a
$\cS$-valued Markov process with stationary law $\nu$.
Let $L^2(\mu_\Gibbs)$ and $L^2(\nu)$ be the spaces of
Borel-measurable functions $f_0:\R^n \to \R^n$ and
$f:\cS \to \R^n$, respectively, such that
\[\|f_0\|_{L^2}^2:=\frac{1}{n}\,\E_{\btheta \sim \mu_\Gibbs}
\|f_0(\btheta)\|_2^2<\infty,\qquad \|f\|_{L^2}^2:=
\frac{1}{n}\,\E_{\bxi \sim \nu} \|f(\bxi)\|_2^2<\infty,\]
and write $\langle \cdot \rangle_{L^2}$ for the associated inner-products.
Consider the Markov semigroup $\{P_t\}_{t \geq 0}$ of
bounded operators on $L^2(\nu)$ given by
\[P_t f(\bxi)=\E[f(\bxi^t) \mid \bxi^0=\bxi].\]
Note that for any $f \in L^2(\nu)$, we have the contractivity property
$\|P_tf\|_{L^2} \leq \|f\|_{L^2}$ by Jensen's inequality. Furthermore,
this semigroup is strongly continuous in the sense
\begin{equation}\label{eq:strongcontinuity}
\lim_{t \to 0^+} \|P_tf-f\|_{L^2}=0 \text{ for each fixed } f \in L^2(\nu).
\end{equation}
Indeed, as $t \to 0^+$, we have $\bxi^t \to \bxi^0$ in
$\cS=C([-B,0],\R^n)$ by the uniform continuity of sample paths of
$\{\tilde \btheta^t\}_{t \in \R}$ on compact intervals. Thus if
$f:\cS \to \R^n$ is continuous and bounded, then by dominated convergence,
\[\lim_{t \to 0^+} \|P_tf-f\|_{L^2}^2
=\lim_{t \to 0^+} \frac{1}{n}\E\|\E[f(\bxi^t) \mid \bxi^0]-f(\bxi^0)\|_2^2
\leq \lim_{t \to 0^+} \frac{1}{n}\E\|f(\bxi^t)-f(\bxi^0)\|_2^2=0.\]
As $\cS$ is separable,
the space of such continuous and bounded functions on $\cS$
is dense in $L^2(\nu)$. Then, since $P_t$ is also contractive,
this implies the strong continuity \eqref{eq:strongcontinuity}. 

Define the generator $-A:D(A) \to L^2(\nu)$ for $\{P_t\}_{t \geq 0}$ by
\[-Af=\lim_{t \to 0^+} \frac{1}{t}(P_tf-f)\]
where $D(A) \subset L^2(\nu)$ is the domain where this limit exists in
$L^2(\nu)$. Let $r_g^B:[0,\infty) \to \R$ be a compactly supported,
smooth approximation of $r_g$ such that
\[r_g^B(\tau)=\begin{cases}r_g(\tau) & \text{ for } \tau \in [0,B-1],\\
0 & \text{ for } \tau \geq B,\end{cases}\]
and $|r_g^B(\tau)| \leq |r_g(\tau)|$ 
and $|{r_g^B}'(\tau)| \leq C|r_g'(\tau)|$  for a constant $C>0$ and
all $\tau \geq 0$. Consider the function $f_B \in L^2(\nu)$ given by
\begin{equation}\label{eq:Fdef}
f_B(\bxi)=\X\bxi_0-\int_{-B}^0 r_g^B(-s)\bxi_s \d s
-\E_{\bxi \sim \nu}\left[\X\bxi_0-\int_{-B}^0 r^B_g(-s)\bxi_s \d s\right].
\end{equation} 
We claim that $f_B \in D(A)$. To see this, let
$P_t^\theta f_0(\btheta)=\E[f_0(\tilde \btheta^t) \mid \tilde \btheta^0=\btheta]$ denote the Markov diffusion semigroup on $L^2(\mu_\Gibbs)$,
and let
\begin{equation}\label{eq:diffusiongenerator}
-A^\theta f_0(\btheta)=\Big((\X\btheta-U'(\btheta))^\top \nabla f_{0,i}(\btheta)
+\Tr \nabla^2 f_{0,i}(\btheta)\Big)_{i=1}^n
\end{equation}
denote its generator with domain $D(A^\theta) \subset
L^2(\mu_\Gibbs)$. Denote $f_{B,0}(\btheta)=\X\btheta$, and write $c_0$ for the
last constant term of \eqref{eq:Fdef} (not depending on $\bxi$). Then
\begin{align*}
P_tf_B(\bxi)&=\E[f_{B,0}(\bxi_0^t) \mid \bxi^0=\bxi]
-\int_{-B}^0 r_g^B({-}s)\E[\bxi_s^t \mid \bxi^0=\bxi]\d s-c_0\\
&=P_t^\theta f_{B,0}(\bxi_0)
-\int_{\max(-t,-B)}^0 r_g^B({-}s)\E[\tilde \btheta^{t+s} \mid
\tilde\btheta^0=\bxi_0]\d s\\
&\quad-\int_{-B}^{\max(-t,-B)} r_g^B({-}s)\bxi_{t+s}\d
s-c_0.
\end{align*}
For $t \in (0,B)$, we have
\begin{align*}
\frac{1}{t}(P_tf_B(\bxi)-f_B(\bxi))
&=\frac{1}{t}(P_t^\theta f_{B,0}(\bxi_0)-f_{B,0}(\bxi_0))
-\frac{1}{t}\int_{-t}^0 r_g^B({-}s)\E[\tilde \btheta^{t+s} \mid \tilde \btheta^0
=\bxi_0]\d s\\
&-\frac{1}{t}\left(\int_{-B+t}^0 (r_g^B(t-s)-r_g^B({-}s))\bxi_s\d s\right)
+\frac{1}{t}\int_{-B}^{-B+t} r_g^B({-}s)\bxi_s \d s.
\end{align*}
Under the condition that Assumption \ref{assump:dynamics} holds for ${-}U'$,
$f_{B,0}$ belongs to $D(A^\theta)$, so
\[\lim_{t \to 0^+} \frac{1}{t}(P_t^\theta f_{B,0}-f_{B,0})(\btheta)
={-}A^\theta f_{B,0}(\btheta)=\X(\X\btheta-U'(\btheta)).\]
Pointwise over $\bxi \in \cS$, since $\bxi_s$ and
$\E[\tilde \btheta^s \mid \tilde \btheta^0=\bxi_0]$ are
continuous in $s$ and $r_g^B(s)$ is continuously-differentiable with
$r_g^B(B)=0$, we have
\begin{align*}
&\lim_{t \to 0^+}\Bigg(
{-}\frac{1}{t}\int_{-t}^0 r_g^B({-}s)\E[\tilde \btheta^{t+s} \mid \tilde \btheta^0=\bxi_0]\d s
-\frac{1}{t}\left(\int_{-B+t}^0 (r_g^B(t-s)-r_g^B({-}s))\bxi_s\d s\right)\\
&
\hspace{0.5in}+\frac{1}{t}\int_{-B}^{-B+t} r_g^B({-}s)\bxi_s \d s\Bigg)={-}r_g^B(0)\bxi_0-\int_{-B}^0 {r_g^B}'(s)\bxi_s \d s+0.
\end{align*}
By dominated convergence, this holds also as functions of $\bxi$ in
$L^2(\nu)$. Thus ${-}Af_B$ exists in $L^2(\nu)$, and is given explicitly by
\begin{equation}\label{eq:generatorformula}
-Af_B(\bxi)
=\lim_{t \to 0^+} \frac{1}{t}(P_tf_B(\bxi)-f_B(\bxi))
=\X(\X\bxi_0-U'(\bxi_0))-r_g^B(0)\bxi_0-\int_{-B}^0 {r_g^B}'(s)\bxi_s \d s.
\end{equation}
This checks that $f_B \in D(A)$.\\

{\bf Step 2 (semigroup estimates):} Let us establish bounds independent of
$n,B$ for the quantities
\begin{align*}
\|P_t f_B\|_{L^2},
\quad \|Af_B\|_{L^2},
\quad \|(P_t-P_s)f_B\|_{L^2},
\quad \|(P_t-P_s)Af_B\|_{L^2}.
\end{align*}
For convenience, define
\begin{align}
\tilde \g_B^t&=\X\tilde \btheta^t-\int_{-\infty}^t r_g^B(t-s)\tilde\btheta^s
\d s,\\
\tilde \a_B^t&=\X(\X\tilde \btheta^t-U'(\tilde \btheta^t))
-r_g(0)\tilde \btheta^t-\int_{-\infty}^t {r_g^B}'(t-s)\tilde \btheta^s \d s,
\label{eq:tildeaB}
\end{align}
and recall the notation $\langle \cdot \rangle$ for the
expectation over $\{\tilde \btheta^t\}_{t \in \R}$ conditional on $\X$. Then
\begin{equation}\label{eq:gareprs}
f_B(\bxi^t)=\tilde \g_B^t-\langle \tilde \g_B^0 \rangle,
\qquad {-}Af_B(\bxi^t)=\tilde \a_B^t.
\end{equation}
For $\|P_tf_B\|_{L^2}$, write
\begin{align*}
\tilde \g_B^t-\langle \tilde \g_B^0\rangle&=\g_{B,1}^t+\g_{B,2}^t,\\
\g_{B,1}^t&=\X(\tilde \btheta^t-\langle \btheta \rangle)
-\int_0^{\min(t/2,B)} r_g^B(s)(\tilde\btheta^{t-s}
-\langle \btheta \rangle)\d s,\\
\g_{B,2}^t&={-}\int_{\min(t/2,B)}^B r_g^B(s)(\tilde \btheta^{t-s}
-\langle \btheta \rangle) \d s
\quad \text{ if } t/2<B, \text{ or } \g_{B,2}^t=0 \text{ otherwise}.
\end{align*}
When $\g_{B,2}^t \neq 0$, we may bound
\begin{align*}
&\left\langle \frac{1}{n}\,\|\E[\g_{B,2}^t \mid \X,\{\tilde\btheta^r\}_{r \leq
0}] \|_2^2\right\rangle
\leq \frac{1}{n}\,\langle \|\g_{B,2}^t\|_2^2\rangle\\
&\leq \int_{\min(t/2,B)}^B \d t'
\int_{\min(t/2,B)}^B \d s'\,|r_g^B(t')|\,|r_g^B(s')|
\,\frac{1}{n}\Big|\langle (\tilde \btheta^{t-t'}-\langle \btheta \rangle)^\top
(\tilde\btheta^{s-s'}-\langle \btheta \rangle)\rangle\Big|\\
&\leq \frac{1}{n}\langle \|\btheta\|_2^2 \rangle
\left(\int_{t/2}^\infty |r_g(t')|\d t'\right)^2.
\end{align*}
Applying \eqref{eq:thetanormbound} and the exponential bound for $r_g$
from Lemma \ref{lemma:ttig},
\begin{equation}\label{eq:gB2bound}
\sup_{B>0}
\left\langle\frac{1}{n}\|\E[\g_{B,2}^t \mid \X,\{\tilde\btheta^r\}_{r \leq
0}]\|_2^2 \right\rangle^{1/2} \leq Ce^{-ct}
\text{ for all } t \geq 0.
\end{equation}
For $\g_{B,1}^t$, recall that for each $t \geq 0$,
there exists a coupling \eqref{eq:thetaequilibriumcoupling}
of $(\tilde\btheta^0,\tilde\btheta^t)$ with $\btheta \sim \mu_\Gibbs$
for which $\btheta$ is independent of $\tilde\btheta^0$ and
$\frac{1}{n}\E[\|\tilde \btheta^t-\btheta\|_2^2 \mid \X] \leq 
Ce^{-ct}$.
Then by Jensen's inequality, also
$\frac{1}{n}\langle \|\E[\tilde \btheta^t \mid \X,\tilde \btheta^0]-\langle
\btheta \rangle\|_2^2 \rangle \leq Ce^{-ct}$ for all $t \geq 0$. Hence
\begin{align*}
&\left\langle \frac{1}{n}\,\|\E[\g_{B,1}^t \mid \X,\{\tilde\btheta^r\}_{r \leq
0}]\|_2^2 \right\rangle\\
&=\left\langle\frac{1}{n}\,
\left\|\X(\E[\tilde \btheta^t \mid \X,\tilde \btheta^0]-\langle \btheta \rangle)
-\int_0^{\min(t/2,B)} r_g^B(s)(\E[\tilde \btheta^{t-s} \mid \X,\tilde \btheta^0]
-\langle \btheta \rangle)\d s\right\|_2^2\right\rangle\\
&\leq \frac{C}{n}\,\langle
\|\E[\tilde \btheta^t \mid \X,\tilde \btheta^0]-\langle \btheta \rangle\|_2^2
\rangle+\frac{C}{n}\sup_{s \in [t/2,t]}
\langle \|\E[\tilde \btheta^s \mid \X,\tilde \btheta^0]-\langle \btheta
\rangle\|_2^2 \rangle
\left(\int_0^\infty |r_g(s)|\,\d s\right)^2\\
&\leq C'e^{-c't}.
\end{align*}
Thus
\begin{equation}\label{eq:gB1bound}
\sup_{B>0}
\left\langle \frac{1}{n}\|\E[\g_{B,1}^t \mid \X,\{\tilde\btheta^r\}_{r \leq
0}]\|_2^2\right\rangle^{1/2} \leq Ce^{-ct} \text{ for all } t \geq 0.
\end{equation}
Combining \eqref{eq:gB2bound} and \eqref{eq:gB1bound}, for some constants $C,c>0$,
\begin{equation}\label{eq:IBbound}
\sup_{B>0} \|P_tf_B\|_{L^2} \leq Ce^{-ct} \text{ for all } t \geq 0.
\end{equation}

For $\|Af_B\|_{L^2}$,
by \eqref{eq:thetanormbound} and integrability of $|r_g'|$ shown in
Lemma \ref{lemma:ttig}, for a constant $C>0$ we have similarly
\begin{align}\label{eq:IIBbound}
\sup_{B>0} \|Af_B\|_{L^2}
=\sup_{B>0} \left\langle \frac{1}{n}\|\tilde \a_B^0\|_2^2
\right\rangle^{1/2} \leq C.
\end{align}
For $\|(P_t-P_s)f_B\|_{L^2}$ and $\|(P_t-P_s)Af_B\|_{L^2}$,
note that \eqref{eq:thetanormbound}
and \eqref{eq:thetatildeequicontinuous} imply
\begin{equation}\label{eq:thetacont}
\frac{1}{n}\langle \|\tilde \btheta^\tau-\tilde\btheta^0\|_2^2
\rangle\leq C\min(\tau,1) \text{ for all } \tau \geq 0.
\end{equation}
Then also by the integrability of $|r_g|$,
\begin{align}\label{eq:IIIBgbound}
&\sup_{B>0} \|(P_{s+\tau}-P_s)f_B\|_{L^2}
\leq \sup_{B>0} \left\langle\frac{1}{n}\|\tilde\g_B^{\tau}-\tilde\g_B^0\|_2^2
\right\rangle^{1/2}\notag\\
 &\leq \sup_{B>0}
 \left\langle \frac{1}{n}\|\X(\tilde\btheta^\tau-\tilde\btheta^0)\|_2^2
 \right\rangle^{1/2}+\left\langle\frac{1}{n}\left\|\int_0^\infty r_g^B(s)(\tilde \btheta^{\tau-s}
-\tilde\btheta^{-s})\d s \right\|_2^2 \right\rangle^{1/2}\notag\\
&\leq C\min(\tau^{1/2},1) \text{ for all } s,\tau \geq 0.
\end{align}
Similarly by the integrability of $|r_g'|$,
\begin{equation}\label{eq:IIIBabound}
\sup_{B>0} \|(P_{s+\tau}-P_s)Af_B\|_{L^2}
\leq \sup_{B>0} \left\langle\frac{1}{n}\|\tilde\a_B^{\tau}-\tilde\a_B^0\|_2^2
\right\rangle^{1/2}
\leq C\min(\tau^{1/2},1) \text{ for all } s,\tau \geq 0.
\end{equation}

{\bf Step 3 (diffusion approximation of $c_g^0$ and $r_g$):} We claim that for any 
sequence $B(n)$ satisfying $B(n) \to \infty$ as $n \to \infty$, and for
any $T>0$,
\begin{align}
&\lim_{n \to \infty}
\sup_{\tau \in [0,T]} \left|c_g^0(\tau)-
\left(\frac{1}{n}\langle \tilde{\g}_{B(n)}^{\tau\top} \tilde \g_{B(n)}^0
\rangle-\frac{1}{n}\|\langle \tilde{\g}_{B(n)}^0 \rangle\|_2^2
\right)\right|=0 \text{ a.s.},\label{eq:cgrepr}\\
&\lim_{n \to \infty}
\sup_{\tau \in [0,T]} \left|r_g(\tau)+
\frac{\d}{\d\tau}
\left(\frac{1}{n}\langle \tilde{\g}_{B(n)}^{\tau\top} \tilde \g_{B(n)}^0
\rangle\right)\right|=0 \text{ a.s.}\label{eq:rgrepr}
\end{align}
To show this claim, observe that
Theorem \ref{thm:dmft-approx}(c) and Lemma \ref{lemma:ttig} imply that
\begin{equation}\label{eq:cgrepr1}
c_g(\tau)=\lim_{t \to \infty}
C_g(t+\tau,t)=\lim_{t \to \infty} \lim_{n \to \infty} \frac{1}{n}\,
\langle \g^{t+\tau\top} \g^t \rangle \text{ a.s.}
\end{equation}
where here and below, all limits in $n$ exist a.s.\ over $(\X,\btheta^0)$.
First fixing some constant $B>0$, define for $t \geq B$
\[\g_B^t=\X\btheta^t-\int_{-\infty}^t r_g^B(t-s)\btheta^s \d s.\]
Denote $\Delta_B(t,s)=R_g(t,s)-r_g^B(t-s)$ for $s \in [0,t]$.
By the same argument as \eqref{eq:thetatildeequicontinuous}, using
\eqref{eq:thetatnormbound} to bound $\frac{1}{n}\langle
\|\btheta^t\|_2^2 \rangle$, we have
\begin{equation}\label{eq:thetaequicontinuous}
\frac{1}{n}\langle \|\btheta^t-\btheta^s\|_2^2 \rangle
\leq C(e^{C(t-s)}-1) \text{ for all } t \geq s \geq 0.
\end{equation}
This equicontinuity and Theorem \ref{thm:dmft-approx}(c) imply, for any $t \geq
0$,
\[\lim_{n \to \infty} \sup_{s \in [0,t]}
\left|\frac{1}{n}\btheta^{t\top}\btheta^s-C_\theta(t,s)\right|=0 \text{ a.s.}\]
Then for any $t \geq B$,
\begin{align*}
\lim_{n \to \infty} \frac{1}{n}\langle \|\g^t-\g_B^t\|^2 \rangle
&=\lim_{n \to \infty} \frac{1}{n}\left\langle \left\|\int_0^t \Delta_B(t,s)
 \btheta^s \d s\right\|^2\right\rangle
 \\
&=\int_0^t \d s \int_0^t \d s' \Delta_B(t,s)\Delta_B(t,s')C_\theta(s,s') \text{
a.s.}
\end{align*}
Note that $\Delta_B(t,s)=R_g(t,s)$ for $s>t-B+1$.
Then applying the properties for $C_\theta,c_\theta,R_g,r_g$
in Lemmas \ref{lemma:tti} and \ref{lemma:ttig},
it is readily checked that for any fixed $B>0$:
\begin{align*}
&\lim_{t \to \infty}
\int_0^t \d s \int_0^t \d s' \Delta_B(t,s)\Delta_B(t,s')C_\theta(s,s')\\
&=\lim_{t \to \infty}
\int_0^{t-B+1} \d s \int_0^{t-B+1} \d s'
(R_g(t,s)-r_g^B(t-s))(R_g(t,s')-r_g^B(t-s'))C_\theta(s,s')\\
&=\lim_{t \to \infty}
\int_{B-1}^t \d s \int_{B-1}^t \d s' (R_g(t,t-s)-r_g^B(s))
(R_g(t,t-s')-r_g^B(s'))C_\theta(t-s,t-s')\\
&=\int_{B-1}^\infty \d s \int_{B-1}^\infty \d s'\,(r_g(s)-r_g^B(s))
(r_g(s')-r_g^B(s'))c_\theta(s-s').
\end{align*}
Then by the exponential bound for $r_g$ in Lemma \ref{lemma:ttig}, also
\begin{align*}
&\lim_{B \to \infty} \lim_{t \to \infty}
\int_0^t \d s \int_0^t \d s' \Delta_B(t,s)\Delta_B(t,s')C_\theta(s,s')=0.
\end{align*}
This shows that
$\lim_{B \to \infty} \lim_{t \to \infty}
\lim_{n \to \infty} \frac{1}{n}\langle \|\g^t-\g_B^t\|_2^2 \rangle=0$ a.s.
Applying this bound, \eqref{eq:cgrepr1}, and Cauchy-Schwarz, we have
\begin{align}\label{eq:cgrepr2}
\left|c_g(\tau)-\lim_{t \to \infty}
\lim_{n \to \infty} \frac{1}{n}\langle \g_B^{t+\tau\top}\g_B^t \rangle\right|
=\left|\lim_{t \to \infty} \lim_{n \to \infty}
\frac{1}{n}\langle \g^{t+\tau\top}\g^t \rangle
-\frac{1}{n}\langle \g_B^{t+\tau\top}\g_B^t \rangle\right| \leq o_B(1) \text{
a.s.}
\end{align}
where here and below, $o_B(1)$ is a constant bound not depending on $n,\tau$
and vanishing as $B \to \infty$.

For $t>B$, Assumption \ref{assump:convergence} implies
there exists a coupling of
$(\btheta^0,\btheta^{t-B})$ with $\tilde \btheta^{-B}$ 
having law $\mu_\Gibbs$ conditional on $\btheta^0$, for which
\[\frac{1}{n}\,\E[\|\btheta^{t-B}-\tilde \btheta^{-B}\|_2^2 \mid \X,\btheta^0]
\leq Ce^{-ct}.\]
Define from this coupling of $\btheta^{t-B}$ and $\tilde\btheta^{-B}$ a coupling
of $\{\btheta^{t-B+s}\}_{s \in [0,B+\tau]}$ and
$\{\tilde \btheta^{-B+s}\}_{s \in [0,B+\tau]}$ where both processes are driven
by the same Brownian motion. Then by Gr\"onwall's inequality,
\[\|\btheta^{t-B+s}-\tilde \btheta^{-B+s}\|_2 \leq e^{Cs}
\|\btheta^{t-B}-\tilde \btheta^{-B}\|_2 \text{ for all } s \in [0,B+\tau].\]
Thus for any fixed $B>0$ and $\tau>0$,
there exists a constant $C(B,\tau)>0$ such that also
\[\sup_{s \in [0,B+\tau]} \frac{1}{n}\E[\|\btheta^{t-B+s}
-\tilde\btheta^{-B+s}\|_2^2 \mid \X,\btheta^0] \leq C(B,\tau)e^{-ct}.\]
This implies from the definitions of $\g_B^t$ and $\tilde \g_B^t$ that
\[\lim_{t \to \infty}\limsup_{n \to \infty}
\left|\frac{1}{n}\,\langle \g_B^{t+\tau\top} \g_B^t\rangle
-\frac{1}{n}\,\langle \tilde \g_B^{\tau\top} \tilde \g_B^0\rangle
\right|=0,\]
so \eqref{eq:cgrepr2} implies 
\begin{equation}\label{eq:cgrepr3}
\sup_{\tau \geq 0}\left|c_g(\tau)-\lim_{n \to \infty}
\frac{1}{n}\,\langle \tilde\g_B^{\tau\top}\tilde \g_B^0
\rangle\right| \leq o_B(1) \text{ a.s.}
\end{equation}

Since \eqref{eq:cgrepr3} implies that the almost-sure limit
$\lim_{n \to \infty} \frac{1}{n}\,\langle
\tilde \g_B^{\tau\top}\tilde \g_B^0 \rangle$ converges uniformly to
$c_g(\tau)$ as $B \to \infty$, we have
\begin{equation}\label{eq:cgrepr4}
\begin{aligned}
q_*\cR_\Lambda'(v_*-q_*)=\lim_{\tau \to \infty} c_g(\tau)&=\lim_{\tau \to \infty} \lim_{B \to \infty}
\lim_{n \to \infty}
\frac{1}{n}\,\langle \tilde \g_B^{\tau\top}\tilde \g_B^0 \rangle\\
&=\lim_{B \to \infty}
\lim_{\tau \to \infty} \lim_{n \to \infty}
\frac{1}{n}\,\langle \tilde \g_B^{\tau\top}\tilde \g_B^0 \rangle \text{ a.s.}
\end{aligned}
\end{equation}
Now let $\{\check \btheta^t\}_{t \in \R}$ be a second diffusion following
\eqref{eq:langevin} with equilibrium
initial state $\check \btheta^0 \sim \mu_\Gibbs$. Again fix $B>0$, and
consider $\tau>B$. Applying again
Assumption \ref{assump:convergence} and the same coupling argument as above,
there exists a coupling of the laws of $\{\tilde \btheta^{\tau-B+s}\}_{s \in
[0,B]}$ and $\{\check{\btheta}^{\tau-B+s}\}_{s \in [0,B]}$ 
conditional on $\{\tilde\btheta^t\}_{t \leq 0}$,
where $\{\check{\btheta}^{\tau-B+s}\}_{s \in [0,B]}$ is independent of
$\{\tilde\btheta^t\}_{t \leq 0}$, for which
\[\sup_{s \in [0,B]} \frac{1}{n}\,\E[\|\tilde \btheta^{\tau-B+s}
-\check\btheta^{\tau-B+s}\|_2^2 \mid \X] \leq Ce^{-c\tau}.\]
Letting $\{\check \g_B^t\}_{t \in \R}$ be the copy of $\{\tilde \g_B^t\}_{t \in
\R}$ defined from $\{\check \btheta^t\}_{t \in \R}$,
this implies by Cauchy-Schwarz that
\[\lim_{\tau \to \infty} \limsup_{n \to \infty}
\left|\frac{1}{n}\langle \tilde \g_B^{\tau\top}\tilde \g_B^0 \rangle
-\frac{1}{n}\E[\check \g_B^{\tau\top}\tilde \g_B^0 \mid \X]\right|=0.\]
Here $\E[\check \g_B^{\tau\top}\tilde \g_B^0 \mid \X]=
\|\langle \tilde \g_B^0 \rangle\|_2^2$ since
$\check \g_B^\tau$ is a function of $\{\check \btheta^{\tau-B+s}\}_{s \in
[0,B]}$, $\tilde \g_B^0$ is a function of $\{\tilde \btheta^t\}_{t \leq 0}$, and
these are independent. Then, applying this to \eqref{eq:cgrepr4}, we have
\begin{equation}\label{eq:cgrepr5}
\left|q_*\cR_\Lambda'(v_*-q_*)-\lim_{n \to \infty}
\frac{1}{n}\|\langle \tilde \g_B^0 \rangle\|_2^2\right| \leq o_B(1) \text{ a.s.}
\end{equation}

For any fixed $B>0$ and $B(n)$ satisfying $B(n) \to \infty$ as $n \to \infty$,
\[\begin{aligned}
\frac{1}{n}\langle \|\tilde \g_B^\tau-\tilde \g_{B(n)}^\tau\|_2^2 \rangle
&=\frac{1}{n}\,\left\langle\left\|\int_{-\infty}^\tau
 (r_g^B(\tau-s)-r_g^{B(n)}(\tau-s))\tilde\btheta^s \d s\right\|_2^2\right\rangle
\\
&\leq C\left(\int_{B-1}^\infty |r_g(s)| \d s\right)^2
\cdot \frac{1}{n}\langle \|\btheta\|^2 \rangle
\end{aligned}\]
which is uniformly bounded by \eqref{eq:thetanormbound}
and tends to 0 as $B \to \infty$. Thus
\begin{align*}
\sup_{\tau \geq 0} \left|\frac{1}{n}\,\langle \tilde\g_B^{\tau\top}
\tilde \g_B^0 \rangle
-\frac{1}{n}\,\langle \tilde\g_{B(n)}^{\tau\top}\tilde \g_{B(n)}^0
\rangle\right|& \leq o_B(1),\\
\sup_{\tau \geq 0} \left|\frac{1}{n}\|\langle \tilde\g_B^0
\rangle \|_2^2
-\frac{1}{n}\|\langle \tilde\g_{B(n)}^0\rangle\|_2^2\right| & \leq o_B(1).
\end{align*}
Combining with \eqref{eq:cgrepr3} and \eqref{eq:cgrepr4}
and noting that the fixed value $B>0$ may be taken arbitrarily large, this shows
\begin{equation}\label{eq:cgpointwiseapprox}
c_g^0(\tau)=c_g(\tau)-q_*\cR_\Lambda'(v_*-q_*)
=\lim_{n \to \infty} \frac{1}{n}\,\langle \tilde\g_{B(n)}^{\tau\top}
\tilde \g_{B(n)}^0 \rangle
-\frac{1}{n}\|\langle \tilde\g_{B(n)}^0\rangle\|_2^2 \text{ a.s.}
\end{equation}
The bound \eqref{eq:IIIBgbound} implies that
for all $\tau,\tau' \in \R$ and $n \geq 1$,
\[\left|\frac{1}{n}\langle \tilde{\g}_{B(n)}^{\tau\top} \tilde \g_{B(n)}^0
\rangle-\frac{1}{n}\langle \tilde{\g}_{B(n)}^{\tau'\top}
\tilde\g_{B(n)}^0 \rangle\right|
\leq C|\tau-\tau'|^{1/2}.\]
Then $\tau \mapsto \frac{1}{n}\langle \tilde{\g}_{B(n)}^{\tau\top}
\tilde\g_{B(n)}^0 \rangle$
is uniformly equicontinuous for all $n$, so the convergence in
\eqref{eq:cgpointwiseapprox} is
uniform over any compact interval $[0,T]$,
establishing the claim \eqref{eq:cgrepr} for $c_g^0$.

To show the second claim \eqref{eq:rgrepr} for $r_g$,
note that Lemma \ref{lemma:ttig} shows $r_g(\tau)={-}{c_g^0}'(\tau)$. Using
$B=B(n)$ to define the preceding semigroup $\{P_t\}_{t \geq 0}$ on $L^2(\nu)$
in Step 1, the identities \eqref{eq:gareprs} imply that
\[\frac{\d}{\d\tau}\,\frac{1}{n}\langle \tilde \g_{B(n)}^{\tau\top}
\tilde \g_{B(n)}^0 \rangle
=\frac{\d}{\d\tau} \langle f_{B(n)},P_\tau f_{B(n)} \rangle_{L^2}
={-}\langle f_{B(n)},P_\tau Af_{B(n)} \rangle_{L^2}
=\frac{1}{n}\langle \tilde \g_{B(n)}^{0\top}\tilde \a_{B(n)}^\tau \rangle.\]
The bound \eqref{eq:IIIBabound} implies that this is also uniformly
equicontinuous in $\tau$ for all $n$. Then by \eqref{eq:cgrepr} and
Arzel\`a-Ascoli, almost surely
\[\lim_{n \to \infty} \sup_{\tau \in [0,T]} \left|{c_g^0}'(\tau)
-\frac{\d}{\d\tau}\,\frac{1}{n}\langle \tilde \g_{B(n)}^{\tau\top}
\tilde \g_{B(n)}^0 \rangle\right|=0.\]
This shows \eqref{eq:rgrepr}.\\

{\bf Step 4 (semigroup regularization):} 
Take any $B(n)$ such that $B(n) \to \infty$ as $n \to \infty$.
Henceforth, we define the semigroup $\{P_t\}_{t \geq 0}$ with this $B(n)$ and
write simply $B \equiv B(n)$. For the given value of
$\eps$ in the lemma, let $T_0=T_0(\eps)>0$ be large enough such that
\eqref{eq:IBbound} ensures
\begin{equation}\label{eq:IBboundT0}
\|P_{T_0}f_B\|_{L^2}<\eps/3.
\end{equation}
The bounds
\eqref{eq:IBbound}, \eqref{eq:IIBbound},
\eqref{eq:IIIBgbound}, and \eqref{eq:IIIBabound} imply
that almost surely, $\langle P_tf_B,P_sAf_B \rangle_{L^2}$
is uniformly bounded and uniformly equicontinuous over $(t,s)$ for all $n$.
Hence along a subsequence of $n$,
there exists a limit function $\ell_*:[0,T_0]^2 \to \R$ 
(possibly random, depending on the infinite sequence of realizations
$\X \in \R^{n \times n}$ as $n \to \infty$) for which
\begin{equation}\label{eq:lstarlimit}
\sup_{s,t \in [0,T_0]} |\langle P_tf_B,P_sAf_B
\rangle_{L^2}-\ell_*(t,s)| \to 0.
\end{equation}

Now let $\delta>0$ and $T>0$ be as given in the lemma. For a sufficiently
small regularization parameter $\rho=\rho(\eps,T_0,\delta,T)>0$ to be
determined, define
\[A^\rho=A+\rho\,\Id, \qquad P_t^\rho=e^{-\rho t}P_t.\]
Note that ${-}A^\rho$ is then the generator of the regularized semigroup
$\{P_t^\rho\}_{t \geq 0}$ on $L^2(\nu)$, with the same domain $D(A^\rho)=D(A)$.
(Here $P_t^\rho$ no longer preserves constant functions, which no longer
belong to the kernel of $A^\rho$, but this will be irrelevant for the rest of
the analysis.) Then we have
\[\frac{\d}{\d t} \langle f_B,P_t^\rho f_B \rangle_{L^2}
={-}\langle f_B,P_t^\rho A^\rho f_B \rangle_{L^2}
={-}e^{-\rho t}\Big(\langle f_B,P_tAf_B \rangle_{L^2}
+\rho \langle f_B,P_tf_B \rangle_{L^2}\Big).\]
Define correspondingly
\[c_g^\rho(\tau)=e^{-\rho \tau}c_g^0(\tau),
\qquad r_g^\rho(\tau)={-}(c_g^\rho)'(\tau)
=e^{-\rho\tau}(r_g(\tau)+\rho c_g^0(\tau)),\]
where the last equality uses ${-}{c_g^0}'(\tau)=r_g(\tau)$.
Then \eqref{eq:cgrepr} and \eqref{eq:rgrepr} imply also, for any fixed
$\rho>0$ and $T>0$, almost surely
\begin{equation}\label{eq:rgrhorepr}
\lim_{n \to \infty} \sup_{\tau \in [0,T]}
|r_g^\rho(\tau)-\langle f_B,P_\tau^\rho A^\rho f_B \rangle_{L^2}|=0.
\end{equation}
Note that
\[\int_0^\infty |r_g^\rho(\tau)-r_g(\tau)|\,\d \tau
\leq \int_0^\infty (1-e^{-\rho \tau})|r_g(\tau)| \d \tau
+\int_0^\infty \rho e^{-\rho \tau}|c_g^0(\tau)| \d \tau.\]
Lemma \ref{lemma:ttig} shows that both $|r_g|,|c_g^0|$ are integrable, so
by dominated convergence, this approaches to 0 as $\rho \to 0$. Thus, for
all small enough $\rho$, we have
\begin{equation}\label{eq:rgrgrho}
\int_0^\infty |r_g^\rho(\tau)-r_g(\tau)|\,\d \tau<\delta/8.
\end{equation}
Given any $\rho>0$, we may pick $T(\rho)>0$ large enough such that
\begin{equation}\label{eq:Trhodef}
\int_{T(\rho)}^\infty |r_g(\tau)|\d \tau
<\delta/8, \qquad
\|f_B\|_{L^2} \|A^\rho f_B\|_{L^2} \int_{T(\rho)}^\infty e^{-\rho\tau}
\d \tau <\delta/8.
\end{equation}
Note finally that
$\|P_t^\rho f-P_t f\|_{L^2} \leq (1-e^{-\rho t})\|f\|_{L^2}$
for all $f \in L^2(\nu)$, and $\|(A^\rho-A)f\|_{L^2} \leq \rho\|f\|_{L^2}$
for all $f \in D(A)$. Then the statements \eqref{eq:cgrepr},
\eqref{eq:rgrhorepr},
\eqref{eq:lstarlimit}, \eqref{eq:IBbound}, \eqref{eq:IIBbound},
\eqref{eq:IIIBgbound}, and \eqref{eq:IIIBabound}
together imply that there exists a realization of $\ell_*:[0,T_0]^2 \to \R$ depending only on $\eps,T_0$ and some $n \geq 1$, $B=B(n)$, $\rho>0$,
and realization of $\X \in \R^{n \times n}$ depending on $\eps,T_0,\delta,T$   such that \eqref{eq:rgrgrho} holds for
this $\rho>0$,
\eqref{eq:Trhodef} holds for the corresponding $T(\rho)>0$, and  
\begin{align}
\sup_{\tau \in [0,T]} |c_g^0(\tau)
-\langle f_B,P_\tau^\rho f_B \rangle_{L^2}|&<\delta/2,\label{eq:cgsemigroupapprox}\\
\int_0^{T(\rho)} |r_g^\rho(\tau)-\langle f_B,P_\tau^\rho A^\rho f_B
\rangle_{L^2}|\d\tau&<\delta/8,\label{eq:rgsemigroupapprox}\\
\sup_{s,t \in [0,T_0]} |\langle P_t^\rho f_B,P_s^\rho A^\rho f_B
\rangle-\ell_*(t,s)|&<\delta/2,\label{eq:lstarapprox}\\
\|P_{T_0}^\rho f_B\|_{L^2}&<\eps/2,\label{eq:contractionatT0}\\
\|f_B\|_{L^2},\|A^\rho f_B\|_{L^2}&<C,\label{eq:AfBbound}\\
\|(P_t^\rho-P_s^\rho)f_B\|_{L^2} &\leq C|t-s|^{1/2}+\delta/2
\text{ for all } s,t \in [0,T_0],\label{eq:PtPsbound}\\
\|(P_t^\rho-P_s^\rho)A^\rho f_B\|_{L^2} &\leq C|t-s|^{1/2}+\delta/2
\text{ for all } s,t \in [0,T_0].\label{eq:PtPsAbound}
\end{align}
In the remainder of the proof, we fix these values of $n,B,\rho$
and realizations of $\X$ and $\ell_*$.\\

{\bf Step 5 (Trotter-Kato approximation):}
We now approximate ${-}A^\rho$ by a finite-dimensional generator ${-}A_m$, and
apply the Trotter-Kato theorem to conclude the proof.

Since $\{P_t^\rho\}_{t \geq 0}$ is a strongly continuous contractive semigroup, 
the resolvent $(\lambda\,\Id+A^\rho)^{-1}$ of its generator
${-}A^\rho$ exists for any
$\lambda>1$ as a bounded linear operator $(\lambda\,\Id+A^\rho)^{-1}:L^2(\nu) \to
D(A)$, with inverse $(\lambda\,\Id+A^\rho):D(A) \to L^2(\nu)$ \cite[Theorem
1.10]{engel2000one}. Let $\{h_i\}_{i=1}^\infty$ be any orthonormal basis of
$L^2(\nu)$, and for each $m \geq 1$, let
\[V_m=\sp(\{(\lambda\,\Id+A^\rho)^{-1}h_i\}_{i=1}^m) \subset D(A).\]
We note that for any $f \in D(A)$,
there exists $f_m \in V_m$ for which
\begin{equation}\label{eq:graphnormapprox}
f_m \to f, \qquad A^\rho f_m \to A^\rho f \qquad \text{ in } L^2(\nu)
\text{ as } m \to \infty.
\end{equation}
Indeed, let
$g=(\lambda\,\Id+A^\rho)f$, let $g_m \in \sp(h_1,\ldots,h_m)$ be such that
$g_m \to g$ in $L^2(\nu)$ as $m \to \infty$,
and let $f_m=(\lambda\,\Id+A^\rho)^{-1}g_m$. Then, since
$(\lambda\,\Id+A^\rho)^{-1}$ is a bounded operator, $f_m \to
(\lambda\,\Id+A^\rho)^{-1}g=f$. Furthermore $A^\rho
f_m=g_m-\lambda f_m$, so $A^\rho f_m \to g-\lambda f=A^\rho f$, showing
\eqref{eq:graphnormapprox}.

To define $A_m$, let $\{e_i\}_{i=1}^\infty$ be elements
of $D(A)$ such that $\{e_i\}_{i=1}^m$ forms an orthonormal basis of $V_m$ for
each $m \geq 1$ (e.g.\ obtained from a Gram-Schmidt procedure).
Define $\Pi_m:L^2(\nu) \to \R^m$ 
and its adjoint $\Pi_m^*:\R^m \to L^2(\nu)$ (in the sense
$v^\top \Pi_m f=\langle \Pi_m^* v,f \rangle_{L^2}$) by
\[\Pi_m f=(\langle e_i,f \rangle_{L^2})_{i=1}^m,
\qquad \Pi_m^* v=\sum_{i=1}^m v(i)e_i.\]
Thus $\Pi_m\Pi_m^*=\Id$ on $\R^m$, and
$\Pi_m^*\Pi_m:L^2(\nu) \to L^2(\nu)$ is the orthogonal projector onto
$V_m$. Define $A_m \in \R^{m \times m}$ by
\[A_m=\Pi_m A^\rho \Pi_m^*,
\qquad (A_m)(i,j)=\langle e_i,A^\rho e_j \rangle_{L^2}.\]
Note that for any $f \in L^2 (\nu)$, we have $\langle f,f-P_t^\rho f \rangle_{L^2} \geq
(1-e^{-\rho t})\|f\|_{L^2}^2$
by the contractivity $\|P_t^\rho f\|_{L^2}
=e^{-\rho t}\|P_t f\|_{L^2} \leq e^{-\rho t}\|f\|_{L^2}$
and Cauchy-Schwarz. Hence for any $f \in D(A)$, we have from definition of the
generator ${-}A^\rho$ that $\langle f,A^\rho f \rangle_{L^2} \geq
\rho\|f\|_{L^2}^2$, implying
\begin{equation}\label{eq:Apsd}
v^\top (A_m+A_m^\top) v 
=2\langle \Pi_m^* v,A^\rho \Pi_m^* v \rangle_{L^2}
\geq 2\rho\|v\|_2^2 \text{ for all } v \in \R^m.
\end{equation}
Thus $\lambda_{\min}(A_m+A_m^\top) \geq 2\rho$, so $A_m+A_m^\top$ is
strictly positive definite. This implies also, for any $v \in \R^m$, that
\begin{align*}
\frac{\d}{\d t}\|e^{-tA_m}v\|_2^2
&=-2(e^{-tA_m}v)^\top A_me^{-tA_m}v\\
&=-(e^{-tA_m}v)^\top(A_m+A_m^\top)(e^{-tA_m}v)\leq -2\rho\,\|e^{-tA_m}v\|_2^2,
\end{align*}
so $\|e^{-tA_m}v\|_2^2 \leq e^{-2\rho t}\|v\|_2^2$
and hence $\|e^{-tA_m}\|_\op \leq e^{-\rho t} \leq 1$ for all $t \geq 0$.
Thus $\{P_t^m\}_{t \geq 0}:=\{e^{-tA_m}\}_{t \geq 0}$ defines
a contraction semigroup on $\R^m$ with generator ${-}A_m$.
The property \eqref{eq:graphnormapprox} implies that for any $f \in D(A)$,
there exists $u_m \in \R^m$ with $\Pi_m^*u_m=f_m \in V_m$
for each $m \geq 1$ such that as $m \to \infty$,
\begin{equation}\label{eq:umapprox}
\|\Pi_m^*u_m-f\|_{L^2}=\|f_m-f\|_{L^2} \to 0,
\qquad \|A^\rho \Pi_m^*u_m-A^\rho f\|_{L^2}=\|A^\rho f_m-A^\rho f\|_{L^2} \to 0.
\end{equation}
Hence, from the above definition $A_m=\Pi_mA^\rho \Pi_m^*$, also
\begin{align}
\|\Pi_m^*A_mu_m-A^\rho f\|_{L^2}
&=\|\Pi_m^*A_m\Pi_mf_m-A^\rho f\|_{L^2}
=\|\Pi_m^*\Pi_m A^\rho f_m-A^\rho f\|_{L^2}\notag\\
&\leq \|\Pi_m^*\Pi_m A^\rho f-A^\rho f\|_{L^2}
+\|\Pi_m^*\Pi_m(A^\rho f_m-A^\rho f)\|_{L^2} \to 0.\label{eq:Aumapprox}
\end{align}
This checks property (C2) of \cite[Proposition 3.1]{ito1998trotter}
for $D=D(A)$. Property (C1) is automatic, since $D(A)$ is dense in $L^2(\nu)$
\cite[Theorem 1.4]{engel2000one}
and $(\lambda\,\Id+A^\rho):D(A) \to L^2(\nu)$ is surjective for any $\lambda$
in the resolvent set of ${-}A^\rho$. Thus, by the version of
the Trotter-Kato theorem stated in \cite[Theorem 2.1, Proposition
3.1]{ito1998trotter}, for any fixed $f \in L^2(\nu)$ and $T>0$,
\begin{equation}\label{eq:trotterkato}
\lim_{m \to \infty} \sup_{\tau \in [0,T]} \|\Pi_m^*P_\tau^m \Pi_m f
-P_\tau^\rho f\|_{L^2}=0.
\end{equation}

We now verify conditions (i)--(vi) of the lemma.
Recall the function $f_B$ in \eqref{eq:Fdef}, and let $u_m \in \R^m$
be a sequence satisfying \eqref{eq:umapprox} and \eqref{eq:Aumapprox}
for $f=f_B$. Let $f_{m,B} \in V_m$ be such that
\[u_m=\Pi_m f_{m,B} \in \R^m, \qquad f_{m,B}=\Pi_m^* u_m.\]
Recall that \eqref{eq:cgsemigroupapprox} shows
$\sup_{\tau \in [0,T]}
|c_g^0(\tau)-\langle f_B,P_\tau^\rho f_B \rangle_{L^2}|<\delta/2$.
Applying the Trotter-Kato result \eqref{eq:trotterkato} with $f=f_B$ and
Cauchy-Schwarz,
\begin{equation}\label{eq:trottercovapprox1}
\lim_{m \to \infty}
\sup_{\tau \in [0,T]} |\langle f_B,P_\tau^\rho f_B \rangle_{L^2}
-\langle f_B,\Pi_m^*P_\tau^m \Pi_m f_B \rangle_{L^2}|=0.
\end{equation}
Here $\langle f_B,\Pi_m^*P_\tau^m \Pi_mf_B \rangle_{L^2}=(\Pi_m f_B)^\top
e^{-\tau A_m} (\Pi_m f_B)$, and the first statement of
\eqref{eq:umapprox} implies
\[\lim_{m \to \infty}
\|\Pi_m f_B-u_m\|_2=\lim_{m \to \infty} 
\|\Pi_m^*u_m-f_B\|_{L^2}=0.\]
Then, using $\|e^{-\tau A_m}\|_\op \leq 1$ and Cauchy-Schwarz,
\begin{equation}\label{eq:trottercovapprox2}
\lim_{m \to \infty} \sup_{\tau \in [0,T]}
\Big|\langle f_B,\Pi_m^* P_\tau^m \Pi_m f_B \rangle_{L^2}-u_m^\top e^{-\tau A_m}u_m\Big|=0.
\end{equation}
Combining \eqref{eq:cgsemigroupapprox}, 
\eqref{eq:trottercovapprox1}, and \eqref{eq:trottercovapprox2} shows
statement (i) for all large enough $m$.

For (ii), note that 
\[|u_m^\top A_me^{-\tau A_m}u_m|
\leq \|e^{-\tau A_m}\|_{\op} \|u_m\|_2 \|A_mu_m\|_2
\leq e^{-\tau\rho}\|u_m\|_2 \|A_mu_m\|_2.\]
By \eqref{eq:umapprox} and \eqref{eq:Aumapprox}, $\|u_m\|_2 \to \|f_B\|_{L^2}$
and $\|A_mu_m\|_2 \to \|A^\rho f_B\|_{L^2}$ as $m \to \infty$. Then the second
statement of \eqref{eq:Trhodef} ensures that for all large enough $m$,
\[\int_{T(\rho)}^\infty
|u_m^\top A_me^{-\tau A_m}u_m|\d\tau<\delta/4.\]
Combining this with the first statement of
\eqref{eq:Trhodef} and with \eqref{eq:rgrgrho} and
\eqref{eq:rgsemigroupapprox},
\begin{align}
&\int_0^\infty |r_g(\tau)-u_m^\top A_me^{-\tau A_m}u_m|\d\tau\notag\\
&\leq \int_0^{T(\rho)} |r_g(\tau)-u_m^\top A_me^{-\tau A_m}u_m|\d\tau
+\int_{T(\rho)}^\infty |r_g(\tau)|\d\tau
+\int_{T(\rho)}^\infty |u_m^\top A_me^{-\tau A_m}u_m|\d\tau\notag\\
&\leq \int_0^{T(\rho)} |\langle f_B,P_\tau^\rho A^\rho f_B \rangle_{L^2}
-u_m^\top A_me^{- \tau A_m}u_m|\d\tau+5\delta/8.\label{eq:rgsemigroupapprox1}
\end{align}
Applying the Trotter-Kato result \eqref{eq:trotterkato} with $f=A^\rho f_B$   and
$T=T(\rho)$,
\[\lim_{m \to \infty}\sup_{\tau \in [0,T(\rho)]}
|\langle f_B,P_\tau^\rho A^\rho f_B \rangle_{L^2}
-\langle f_B,\Pi_m^* P_\tau^m\Pi_m A^\rho f_B\rangle_{L^2}|=0.\]
Here $\langle f_B,\Pi_m^* P_\tau^m\Pi_m A^\rho f_B\rangle_{L^2}
=(\Pi_m f_B)^\top e^{-\tau A_m}(\Pi_m A^\rho f_B)$, and \eqref{eq:umapprox}
and \eqref{eq:Aumapprox} imply $\|\Pi_m f_B-u_m\|_2 \to 0$
and $\|\Pi_m A^\rho f_B-A_mu_m\|_2 \to 0$ as $m \to \infty$. Thus
\[\lim_{m \to \infty}\sup_{\tau \in [0,T(\rho)]}
|\langle f_B,P_\tau^\rho A^\rho f_B \rangle_{L^2}
-u_m^\top e^{-\tau A_m}A_mu_m|=0.\]
Applying this to \eqref{eq:rgsemigroupapprox1} shows that (ii) holds for all
large $m$.

Statement (iii) follows immediately from the convergence $\|u_m\|_2 \to
\|f_B\|_{L^2}$ implied by \eqref{eq:umapprox}, and \eqref{eq:AfBbound}.

For statement (iv),
observe that at $t=T_0$,
\begin{align*}
\|e^{-T_0A_m}u_m\|_2
&=\sup_{v \in \R^m:\|v\|_2=1} v^\top e^{-T_0A_m}u_m\\
&\leq \sup_{g \in L^2(\nu):\|g\|_{L^2}=1} (\Pi_m g)^\top e^{-T_0A_m} u_m=\|\Pi_m^*P_{T_0}^m u_m\|_{L^2}\\
&\leq \|\Pi_m^* P_{T_0}^m \Pi_m f_B\|_{L^2}
+\|\Pi_m^* P_{T_0}^m(\Pi_m f_B-u_m)\|_{L^2}\\
&\leq \|\Pi_m^* P_{T_0}^m \Pi_m f_B\|_{L^2}
+\|\Pi_m^*u_m-f_B\|_{L^2}.
\end{align*}
Then by the Trotter-Kato result \eqref{eq:trotterkato} applied with $f=f_B$,
$T=T_0$, and \eqref{eq:umapprox},
\[\limsup_{m \to \infty}
\|e^{-T_0A_m}u_m\|_2 \leq 
\lim_{m \to \infty} \|\Pi_m^* P_{T_0}^m \Pi_mf_B\|_{L^2}
+\|\Pi_m^*u_m-f_B\|_{L^2}=\|P_{T_0}^\rho f_B\|_{L^2}.\]
Then (iv) holds for all large $m$ by \eqref{eq:contractionatT0}.

For statement (v), note similarly as above that for any $s,t \geq 0$,
\begin{align*}
\|(e^{-tA_m}-e^{-sA_m})u_m\|_2
&=\|\Pi_m^*(P_t^m-P_s^m)u_m\|_{L^2}\\
&\leq \|\Pi_m^*(P_t^m-P_s^m)\Pi_mf_B\|_{L^2}
+\|\Pi_m^* u_m-f_B\|_{L^2}.
\end{align*}
Then by \eqref{eq:trotterkato} applied with $f=f_B$, $T=T_0$,
and \eqref{eq:umapprox}, for any fixed $\delta>0$ and large enough $m$,
\[\|(e^{-tA_m}-e^{-sA_m})u_m\|_2<\|(P_t^\rho-P_s^\rho)f_B\|_{L^2}+\delta/2
\text{ for all } s,t \in [0,T_0].\]
Then by \eqref{eq:PtPsbound}, 
$\|(e^{-tA_m}-e^{-sA_m})u_m\|_2<C|t-s|^{1/2}+\delta$. Similarly, 
\begin{align*}
\|A_m(e^{-tA_m}-e^{-sA_m})u_m\|_2
&=\|\Pi_m^*(P_t^m-P_s^m)A_mu_m\|_{L^2}\\
&\leq \|\Pi_m^*(P_t^m-P_s^m)\Pi_m A^\rho f_B\|_{L^2}+\|\Pi_m^* A_mu_m-A^\rho f_B\|_{L^2}.
\end{align*}
Applying \eqref{eq:trotterkato} with $f=A^\rho f_B$, $T=T_0$, and
\eqref{eq:Aumapprox}, for any $\delta>0$ and all large enough $m$,
\[\|A_m(e^{-tA_m}-e^{-sA_m})u_m\|_2<
\|(P_t^\rho-P_s^\rho)A^\rho f_B\|_{L^2}+\delta/2 \text{ for all } s,t \in [0,T_0].\]
Then by \eqref{eq:PtPsAbound},
$\|A_m(e^{-tA_m}-e^{-sA_m})u_m\|_2<C|t-s|^{1/2}+\delta$,
showing statement (v).

For statement (vi), note that
\[u_m^\top e^{-tA_m^\top}A_me^{-sA_m}u_m
=\langle \Pi_m^* P_t^m u_m,\Pi_m^*P_s^mA_mu_m \rangle_{L^2}.\]
Then by the same arguments as above and Cauchy-Schwarz,
for any $\delta>0$ and all large enough $m$, 
\begin{align*}
\Big|u_m^\top e^{-tA_m^\top}A_me^{-sA_m}u_m
-\langle P_t^\rho f_B,P_s^\rho A f_B \rangle_{L^2}\Big|<\delta/2
\text{ for all } s,t \in [0,T_0].
\end{align*}
The bound \eqref{eq:lstarapprox} shows
$|\langle P_t^\rho f_B,P_s^\rho Af_B \rangle_{L^2}-\ell_*(t,s)|<\delta/2$
for all $s,t \in [0,T_0]$, showing (vi).
\end{proof}

The above lemma implies the following coupling result for stationary
Gaussian processes with covariances $c_g^0(\tau)$ and $\tilde c_g^m(\tau)=u_m^\top e^{-|\tau|A_m}u_m$.

\begin{corollary}\label{cor:semigroupapprox}
Suppose the conditions of Theorem \ref{thm:replica} hold.
Then there exists a constant $C>0$ such that
given any $\eps>0$, there exist $m \geq 1$, 
$u_m \in \R^m$ with $\|u_m\|_2<C$, and $A_m \in \R^{m \times m}$ with
$A_m+A_m^\top$ strictly positive definite, such that the following holds:

For all $\tau \in \R$ let
\[c_g^0(\tau)=c_g(\tau)-q_*\cR_\Lambda'(v_*-q_*),
\qquad \tilde c_g^m(\tau)=u_m^\top e^{-|\tau|A_m}u_m,\]
and for $\tau \geq 0$ let
\[r_g(\tau)=-c_g'(\tau),
\qquad \tilde r_g^m(\tau)=-(\tilde c_g^m)'(\tau)
=u_m^\top A_me^{-\tau A_m}u_m\]
(with derivatives understood from the right at $\tau=0$). Then
\begin{equation}\label{eq:cg0approx}
|\tilde c_g^m(0)-\cR_{\Lambda}(v_*-q_*)|<\eps,
\end{equation}
\begin{equation}\label{eq:rgapprox}
\int_0^\infty |r_g(\tau)-\tilde r_g^m(\tau)|\,\d \tau<\eps.
\end{equation}
Furthermore, $c_g^0$ and $\tilde c_g^m$ are both symmetric and
positive-semidefinite on $\R$, and for any $T>0$,
there exists a coupling of two stationary Gaussian processes
$\{g_0^t\}_{t \in \R} \sim \GP(0,c_g^0)$ and $\{\tilde g^t\}_{t \in \R}
\sim \GP(0,\tilde c_g^m)$ such that
\begin{equation}\label{eq:cgcoupling}
\sup_{t \in [0,T]} \E(g_0^t-\tilde g^t)^2<\eps^2.
\end{equation}
\end{corollary}
For our later arguments, we emphasize that the coupling \eqref{eq:cgcoupling} is
ensured for a single fixed $m=m(\eps)$, and arbitrarily large $T$.

\begin{proof}
Let $\eps$ be the given constant, and
let $T_0=T_0(\eps)$ be as given in Lemma \ref{lemma:semigroupapprox}. Let
$\{(A_m,u_m)\}_{m=1}^\infty$ be a sequence of pairs in $\R^{m \times m} \times
\R^m$, such that each $(A_m,u_m)$ attains the guarantees of Lemma
\ref{lemma:semigroupapprox} for $(\eps,T_0)$ and $(\delta,T)=(\delta_m,T_m)$,
where
\[\delta_m \to 0 \text{ and } T_m \to \infty \text{ as }
m \to \infty.\]
Recall from Lemma \ref{lemma:ttig} that $c_g^0(0)=\cR_\Lambda(v_*-q_*)$.
Then for any $m$ large enough such that $\delta_m<\eps$, \eqref{eq:cg0approx}
and \eqref{eq:rgapprox}
hold by statements (i) and (ii) of Lemma \ref{lemma:semigroupapprox}.

For the coupling statement \eqref{eq:cgcoupling},
for each $m \geq 1$, let $\{b_m^t\}_{t \in \R}$ be a standard bi-directional
Brownian motion on $\R^m$ with $b_m^0=0$,
and let $\Sigma_m \in \R^{m \times m}$ be any matrix such that
$A_m+A_m^\top=\Sigma_m\Sigma_m^\top$. Denote
\[v_m^t=\Sigma_m^\top e^{-t A_m}u_m \in \R^m.\]
We note that $\tilde c_g^m$ is the
covariance of the Gaussian process
\[\tilde g_m^t=\int_{-\infty}^t v_m^{t-s\top}\d b_m^s.\]
Indeed, note that for any $t_0<t$,
\begin{align*}
\E\left(\int_{t_0}^t v_m^{t-s\top}\d b_m^s\right)^2
&=\int_{t_0}^t u_m^\top e^{-(t-s)A_m^\top}\Sigma_m\Sigma_m^\top
e^{-(t-s)A_m}u_m\,\d s\\
&=\int_0^{t-t_0} u_m^\top e^{-sA_m^\top}(A_m+A_m^\top)
e^{-sA_m}u_m\,\d s\\
&=u_m^\top(I-e^{-(t-t_0)A_m^\top}e^{-(t-t_0)A_m})u_m\\
&=\|u_m\|_2^2-\|e^{-(t-t_0)A_m}u_m\|_2^2 \leq \|u_m\|_2^2.
\end{align*}
Since $A_m+A_m^\top$ is positive definite, we have
$\|e^{-tA_m}\|_\op<1$ and 
$\|e^{-(t-t_0)A_m}u_m\|_2^2 \to 0$ as $t_0 \to -\infty$. Then $\tilde g_m^t$ is
well-defined, and $\E(\tilde g_m^t)^2=\|u_m\|_2^2=\tilde c_g^m(0)$.
An analogous calculation shows, for $s \leq t$,
\begin{align*}
\E\tilde g_m^s\tilde g_m^t&=u_m^\top 
\left(\int_{-\infty}^s e^{-(s-r)A_m^\top}(A_m+A_m^\top)e^{-(s-r)A_m}\d r\right)
e^{-(t-s)A_m}u_m\\
&=u_m^\top e^{-(t-s)A_m}u_m=\tilde c_g^m(t-s),
\end{align*}
so $\tilde c_g^m(\tau)$ is the covariance of $\{\tilde g_m^t\}_{t \in \R}$ as
claimed. In particular, $\tilde c_g^m$ is symmetric positive-semidefinite.

Let us first couple $\{\tilde g_m^t\} \sim \GP(0,\tilde c_g^m)$ and
$\{\tilde g_n^t\} \sim \GP(0,\tilde c_g^n)$: 
Property (iv) of Lemma \ref{lemma:semigroupapprox} ensures
\[\int_{T_0}^\infty \|v_m^t\|_2^2 \d t
=\int_{T_0}^\infty u_m^\top e^{-tA_m^\top}(A_m+A_m^\top)e^{-tA_m}u_m \d t
=\|e^{-T_0A_m}u_m\|_2^2<\eps^2\]
for all $m \geq 1$. Set $\gamma=\eps^2/T_0$.
For any $t \in [0,T_0]$, let
$\lfloor t \rfloor=\max\{k\gamma:k \in \N,k\gamma \leq t\}$ be the largest
integer multiple of $\gamma$ that is at most $t$. Then by property (v) of Lemma
\ref{lemma:semigroupapprox}, for all $m \geq 1$ for which $\delta_m<\gamma^{1/2}$
and $T_m>T_0$, we have
\[\begin{aligned}
\int_0^{T_0} \|v_m^t-v_m^{\lfloor t \rfloor}\|_2^2\d t
&=\int_0^{T_0} 2u_m^\top (e^{-tA_m^\top}
-e^{-\lfloor t \rfloor A_m^\top})A_m(e^{-tA_m}-e^{-\lfloor t \rfloor A_m})u_m \d t\\
&<CT_0\gamma=C\eps^2.
\end{aligned}\]
Let $E=[0,T_0] \cap \{k\gamma:k \in \N\}$. Then property
(vi) of Lemma \ref{lemma:semigroupapprox} ensures that
\[\begin{aligned}
&\lim_{m,n \to \infty}
\sup_{s,t \in E} |v_m^{s\top} v_m^t
-v_n^{s\top} v_n^t|
\\
&\qquad=\lim_{m,n \to \infty}
\sup_{s,t \in E} |2u_m^\top e^{-sA_m^\top}A_m e^{-tA_m}u_m
-2u_n^\top e^{-sA_n^\top}A_n e^{-tA_n}u_n|=0.
\end{aligned}\]
Then we may construct, for each pair $m,n \geq 1$, some
$N=N(n,m)$ and isometric embeddings $\iota_m:\R^m \to \R^N$ and
$\iota_n:\R^n \to \R^N$, such that
\[\lim_{m,n \to \infty} \sup_{t \in E} \|\iota_m(v_m^t)-\iota_n(v_n^t)\|_2=0.\]
(For example, letting $K_m \in \R^{|E| \times |E|}$ and $K_n \in \R^{|E| \times
|E|}$ have entries $(v_m^{s\top}v_m^t)_{s,t \in E}$
and $(v_n^{s\top}v_n^t)_{s,t \in E}$, we may set $N(n,m)=|E|$,
$\iota_m(v_m^t)=K_m^{1/2}e_t$, and $\iota_n(v_n^t)=K_n^{1/2}e_t$
where $e_t$ is the standard basis vector in $\R^{|E|}$ for each $t \in E$.)
Then, letting $\{b_N^t\}_{t \geq 0}$ be a standard Brownian motion on $\R^N$,
we may couple $\{\tilde g_m^t\} \sim \GP(0,\tilde c_g^m)$ and
$\{\tilde g_n^t\} \sim \GP(0,\tilde c_g^n)$ as
\[\tilde g_m^t=\int_{-\infty}^t \iota_m(v_m^{t-s})^\top \d b_N^s,
\qquad \tilde g_n^t=\int_{-\infty}^t \iota_n(v_n^{t-s})^\top \d b_N^s.\]
Denoting by $\E_{\pi(m,n)}$ the expectation over this coupling,
for any $m,n \geq 1$ such that
$\delta_m,\delta_n<\iota^{1/2}$ and $T_m,T_n>T_0$, we then have
\begin{align*}
\E_{\pi(m,n)}(\tilde g_m^t-\tilde g_n^t)^2
&=\int_{-\infty}^t \|\iota_m(v_m^{t-s})-\iota_n(v_n^{t-s})\|_2^2 \d s=\int_0^\infty \|\iota_m(v_m^s)-\iota_n(v_n^s)\|_2^2 \d s\\
&\leq 3\int_0^{T_0} \left(
\|v_m^s-v_m^{\lfloor s \rfloor}\|_2^2
+\|v_n^s-v_n^{\lfloor s \rfloor}\|_2^2+\|\iota_m(v_m^{\lfloor s
\rfloor})-\iota_n(v_n^{\lfloor s \rfloor})\|_2^2
\right)\d s\\
&\hspace{1in}+2\int_{T_0}^\infty \left(\|v_m^s\|_2^2+\|v_n^s\|_2^2\right)\d s\\
&\leq C'\eps^2+3\int_0^{T_0} 
\|\iota_m(v_m^{\lfloor s \rfloor})-\iota_n(v_n^{\lfloor s \rfloor})\|_2^2
\d s
\end{align*}
for a constant $C>0$. Thus
\[\lim_{m,n \to \infty} \sup_{t \in \R}
\E_{\pi(m,n)}(\tilde g_m^t-\tilde g_n^t)^2 \leq C'\eps^2,\]
implying that there exists a fixed index $m=m(\eps)$ large enough so that
$\delta_m<\iota^{1/2}$, $T_m>T_0$, and
\[\lim_{n \to \infty} \sup_{t \in \R}
\E_{\pi(m,n)}(\tilde g_m^t-\tilde g_n^t)^2 \leq C'\eps^2.\]

Since Lemma \ref{lemma:semigroupapprox} ensures that $c_g^0$ is the pointwise
limit of positive-semidefinite functions $\tilde c_g^m$, it is also
positive-semidefinite. For any given time horizon $T>0$,
we may couple $\tilde g_n^t \sim \GP(0,\tilde c_g^n)$
and $g_0^t \sim \GP(0,c_g^0)$ generically over $[0,T]$, by noting that Lemma
\ref{lemma:ttig} implies
\begin{equation}\label{eq:cgcontinuity}
|c_g^0(t-t)+c_g^0(s-s)-2c_g^0(t-s)|
=2|c_g(0)-c_g(t-s)| \leq C|t-s|
\end{equation}
for a constant $C>0$. Then there exists a coupling of $\tilde g_n^t$ and $g_0^t$ such that
\[\sup_{t \in [0,T]} \E(\tilde g_n^t-g_0^t)^2
\leq C'(1+\sqrt{T}) \cdot \sqrt{\sup_{t,s \in [0,T]}
|c_g^0(t-s)-\tilde c_g^n(t-s)|}\]
(c.f.\ \cite[Lemma D.1]{fan2025dynamicalII}). Since $[0,T] \subset [0,T_n]$ for all large $n$, property
(i) of Lemma \ref{lemma:semigroupapprox} then ensures that there exists
$n=n(\eps,T)$ large enough such that
\[\sup_{t \in [0,T]} \E(\tilde g_n^t-g_0^t)^2<\eps^2.\]
Combining with the above, there exists a coupling of $\{\tilde g_m^t\} \sim
\GP(0,\tilde c_g^m)$ and $\{g_0^t\} \sim \GP(0,c_g^0)$ such that
\[\sup_{t \in [0,T]}\E(\tilde g_m^t-g_0^t)^2<C''\eps^2.\]
The lemma follows upon adjusting $\eps$.
\end{proof}

We now couple the generalized Langevin process $\{\theta^t\}_{t \geq 0}$ in the dynamical mean-field limit of
Definition \ref{def:DMFT} with a process $\tilde\theta^t$ defined from
$(\tilde c_g^m,\tilde r_g^m)$, which will admit a Markovian representation.

\begin{lemma}\label{lemma:thetacoupling}
Suppose the conditions of Theorem \ref{thm:replica} hold.
Fix any sufficiently large $T_0,T>0$ and sufficiently small $\eps>0$.
Let $(\tilde c_g,\tilde r_g) \equiv (\tilde c_g^m,\tilde r_g^m)$ be
the kernels of Corollary \ref{cor:semigroupapprox} for this $\eps$. Let
$\{\tilde g^t\}_{t \in \R},\{\tilde b^t\}_{t \geq T_0},z$
be mutually independent such that
$\{\tilde g^t\}_{t \in \R} \sim \GP(0,\tilde c_g)$,
$\{\tilde b^t\}_{t \geq T_0}$ is a standard Brownian motion with $\tilde b^{T_0}=0$, and $z \sim \cN(0,q_*\cR_\Lambda'(v_*-q_*))$.

Then there is a coupling of
$(\{\tilde g^t\}_{t \in \R},\{\tilde b^t\}_{t \geq T_0},z)$
with the process $\{\theta^t\}_{t \geq 0}$ in \eqref{eq:dmft-theta} of
Definition \ref{def:DMFT} such that the following holds: Define the process
$\{\tilde \theta^t\}_{t \geq 0}$ by
\begin{align*}
\tilde\theta^t&=\theta^t \text{ for all } t \in [0,T_0],\\
\d\tilde \theta^t&=
\left[-U'(\tilde \theta^t)+\int_0^t \tilde r_g(t-s)
\tilde \theta^s ds+\tilde g^t+z\right]
\d t+\sqrt{2}\,\d \tilde b^t \text{ for } t \geq T_0.
\end{align*}
Then for some constants $C,c>0$ not depending on $T_0,T,\eps$,
\begin{equation}\label{eq:thetafourthmoment}
\sup_{t \geq 0} \E(\theta^t)^4<C,
\qquad \sup_{t \geq 0} \E(\tilde \theta^t)^4<C,
\end{equation}
and
\begin{equation}\label{eq:thetacoupling}
\sup_{t \in [T_0,T_0+T]} \E(\tilde \theta^t-\theta^t)^2
<C(\eps+T^{1/4}e^{-cT_0})^{2/3}.
\end{equation}
\end{lemma}
\begin{proof}
We construct the coupling as follows: By Lemma \ref{lemma:ttig}, there exist
$C,c>0$ such that
\[\sup_{s,t \in [T_0,T_0+T]} |C_g(t,s)-c_g(t-s)|<Ce^{-cT_0}.\]
Writing $c_g(\tau)=c_g^0(\tau)+q_*\cR_\Lambda'(v_*-q_*)$
where $q_*\cR_\Lambda'(v_*-q_*)=\lim_{\tau \to \infty} c_g(\tau) \geq 0$,
we may represent a centered stationary Gaussian process with covariance
$c_g$ on $\R$ as
\[\{g_0^t+z\}_{t \in \R} \sim \GP(0,c_g)\]
where
\[\{g_0^t\}_{t \in \R} \sim \GP(0,c_g^0),
\qquad z \sim \cN(0,q_*\cR_\Lambda'(v_*-q_*)),\]
and these are independent. Then by the bound \eqref{eq:cgcontinuity} and
\cite[Lemma D.1]{fan2025dynamicalII}, for some constants
$C,c>0$, there exists a coupling of $\{g^t\}_{t \in [T_0,T_0+T]}
\sim \GP(0,C_g)$ and $\{g_0^t+z\}_{t \in \R} \sim \GP(0,c_g)$ such that
\[\sup_{t \in [T_0,T_0+T]} \E(g_0^t+z-g^t)^2<C\sqrt{T}\,e^{-cT_0}.\]
By Corollary \ref{cor:semigroupapprox}, we may couple $\{g_0^t\}_{t \in \R} \sim
\GP(0,c_g^0)$ and $\{\tilde g^t\}_{t \in \R} \sim \GP(0,\tilde c_g)$
such that
\[\E(g_0^t-\tilde g^t)^2<\eps^2\]
for all $t \in [T_0,T_0+T]$, yielding a coupling of
$(\{\tilde g^t\}_{t \in \R},z)$ with $\{g^t\}_{t \in [T_0,T_0+T]}$
such that
\begin{equation}\label{eq:finalgcoupling}
\sup_{t \in [T_0,T_0+T]} \E(\tilde g^t+z-g^t)^2
<C(\eps^2+\sqrt{T} \cdot e^{-cT_0}).
\end{equation}

Next, we adapt a sticky/reflection coupling
argument of \cite{eberle2016reflection,eberle2019sticky} and \cite[Lemma
3.3]{fan2025dynamicalII} to couple $\{\theta^t\}_{t \geq T_0}$ and
$\{\tilde \theta^t\}_{t \geq T_0}$.
Fix $\delta>0$, and let $h_\delta:[0,\infty) \to [0,1]$ be a smooth
Lipschitz function such that
\begin{equation}\label{eq:couplingboundary}
h_\delta(0)=0, \qquad h_\delta(x)=1 \text{ for } x \geq \delta.
\end{equation}
Let $\{b_1^t\}_{t \geq T_0}$ and $\{b_2^t\}_{t \geq T_0}$
be two standard Brownian motions independent of each other and of
$(\{\theta^t\}_{t \in [0,T_0]},\{b^t\}_{t \in [0,T_0]},
\{g^t\}_{t \in [0,T_0+T]},\{\tilde g^t\}_{t \in \R},z)$.
We couple $\{\theta^t\}_{t \geq T_0}$ and
$\{\tilde \theta^t\}_{t \geq T_0}$ by the evolutions
\begin{align*}
\d \theta^t&=\left[{-}U'(\theta^t)+\int_0^t R_g(t,s)\theta^s \d s+g^t\right]\d t
\\
&\hspace{0.5in}\underbrace{+h_\delta(|\theta^t-\tilde \theta^t|)\sqrt{2}\,\d b_1^t
+\sqrt{2(1-h_\delta(|\theta^t-\tilde \theta^t|)^2)}\,\d b_2^t}_{=\sqrt{2} \d b^t},\\
\d \tilde \theta^t&=\left[{-}U'(\tilde \theta^t)+\int_0^t \tilde r_g(t-s)
\tilde \theta^s \d s+\tilde g^t+z\right]\d t
\\
&\hspace{0.5in}\underbrace{-h_\delta(|\theta^t-\tilde \theta^t|)\sqrt{2}\,\d b_1^t
+\sqrt{2(1-h_\delta(|\theta^t-\tilde \theta^t|)^2)}\,\d b_2^t}_{=\sqrt{2}\d \tilde b^t}.
\end{align*}
L\'evy's characterization of Brownian motion shows that $\{b^t\}_{t \geq T_0}$
and $\{\tilde b^t\}_{t \geq T_0}$ defined above
are both standard Brownian motions adapted to the filtration generated by
$\theta^t,\tilde \theta^t,g^t,\tilde g^t,b_1^t,b_2^t$,
and hence the above processes indeed
coincide in law with $\{\theta^t\}_{t \geq T_0}$ and
$\{\tilde \theta^t\}_{t \geq T_0}$.

Let us first check the fourth moment
bound \eqref{eq:thetafourthmoment} for all sufficiently small
$\eps>0$. For $\{\theta^t\}_{t \geq 0}$,
starting with the second moment, It\^o's formula gives
\begin{equation}\label{eq:itothetasq}
\frac{\d}{\d t}\E(\theta^t)^2=
\E\left[2\theta^t\left(-U'(\theta^t)+\int_0^t R_g(t,s)\theta^s \d
s+g^t\right)\right]+2.
\end{equation}
By Lemma \ref{lemma:ttig}, there exists some $\iota>0$ small enough
such that $\int_0^\infty |r_g(s)|\d s<\alpha-4\iota$, and hence also for all
large $t$,
\[\int_0^t |R_g(t,s)|\d s
\leq \int_0^t |R_g(t,s)-r_g(t-s)|\d s+\int_0^t |r_g(t-s)|\d s
<\alpha-3\iota.\]
By the condition \eqref{eq:Uconvexity} for $U(\cdot)$, there exists a
constant $C=C(\alpha,\iota)>0$ such that $\theta U'(\theta) \geq (\alpha-\iota)
\theta^2-C$ for all $\theta \in \R$.
Furthermore Lemma \ref{lemma:tti} implies $\E(g^t)^2=C_g(t,t)<C$
for some constant $C>0$ and all $t \geq 0$,
so $\E\theta^t g^t \leq \iota \E(\theta^t)^2+C'$ for some
$C'=C'(\iota)>0$.
Then for a constant $C(\iota)>0$ and all large $t$,
\begin{align*}
\frac{\d}{\d t}\E(\theta^t)^2
&\leq -2(\alpha-2\iota)\,\E(\theta^t)^2
+2(\alpha-3\iota)\sup_{s \in [0,t]} \E|\theta^t\theta^s|+C(\iota)\\
&\leq -2(\alpha-2\iota)\,\E(\theta^t)^2
+2(\alpha-3\iota)[\E(\theta^t)^2]^{1/2} \cdot \sup_{s \in [0,t]}
[\E(\theta^s)^2]^{1/2}+C(\iota).
\end{align*}
Consider $M(t)=\sup_{s \in [0,t]} \E(\theta^s)^2$. Since
$t \mapsto \E(\theta^t)^2$ is continuously-differentiable by
\eqref{eq:itothetasq}, then so is $M(t)$
with $\frac{\d}{\d t}M(t)=\frac{\d}{\d t} \E(\theta^t)^2$ if
$M(t)=\E(\theta^t)^2$, and $\frac{\d}{\d t}M(t)=0$ otherwise. Thus the above
gives
\[\frac{\d}{\d t}M(t)
\leq \max\left(0,{-}2\iota M(t)+C(\iota)\right)\]
for all large $t$, implying that $M(t)$ is uniformly bounded over $t \geq 0$.
Thus for a constant $C>0$,
\[\sup_{t \geq 0} \E(\theta^t)^2 \leq C.\]
Now for the fourth moment, It\^o's formula gives
\[\frac{\d}{\d t}\E(\theta^t)^4
=\E\left[4(\theta^t)^3\left({-}U'(\theta^t)+\int_0^t R_g(t,s)\theta^s \d s
+g^t\right)+12(\theta^t)^2\right].\]
Applying $\theta^3 U'(\theta) \geq (\alpha-\iota) \theta^4-C\theta^2$,
$\E(\theta^t)^3g^t
\leq \iota \E(\theta^t)^4+C\E(g^t)^4 \leq \iota \E(\theta^t)^4+C'$, H\"older's
inequality, and the above second moment bound,
\[\frac{\d}{\d t}\E(\theta^t)^4
\leq -4(\alpha-2\iota)\E(\theta^t)^4+4(\alpha-3\iota)[\E(\theta^t)^4]^{3/4}
\sup_{s \in [0,t]} [\E(\theta^s)^4]^{1/4}+C(\iota).\]
Then the same argument as above shows $\sup_{t \geq 0} \E(\theta^t)^4 \leq C$,
giving the first statement of \eqref{eq:thetafourthmoment}. The argument for
the second statement is the same, noting that
Corollary \ref{cor:semigroupapprox} implies
$\E(\tilde g^t)^2=\|u_m\|_2^2<C$ for all $t \in \R$ and some $C>0$ 
(thus $\E(\tilde g^t)^4 \leq C'$), and also for $\eps>0$ sufficiently small,
\[\int_0^t |\tilde r_g(t-s)|\d s
\leq \int_0^t |\tilde r_g(t-s)-r_g(t-s)|\d s
+\int_0^t |r_g(t-s)|\d s<\eps+(\alpha-4\iota)<\alpha-3\iota.\]
Then the same arguments as above applied for $\tilde\theta^t$ and
$t \geq T_0$ show the second statement of \eqref{eq:thetafourthmoment}.

Now set
\[\xi^t=\theta^t-\tilde \theta^t,\]
\[v^t=-U'(\theta^t)+\int_0^t R_g(t,s)\theta^s \d s+g^t,
\qquad \tilde v^t=-U'(\tilde \theta^t)+\int_0^t \tilde r_g(t-s)
\tilde \theta^s \d s+\tilde g^t+z.\]
Thus
\[\d \xi^t=(v^t-\tilde v^t)\d t+2\sqrt{2}h_\delta(|\xi^t|)\,\d b_1^t.\]
Applying It\^o's formula with a smooth approximation $S_{\delta'}(x)$
of $|x|$, and then taking $\delta' \to 0$, gives
(c.f.\ \cite[Eq.\ (4.4)]{eberle2019sticky})
\begin{equation}\label{eq:dabsxi}
\d|\xi^t|=\sign(\xi^t)(v^t-\tilde v^t)\d t
+2\sqrt{2}\sign(\xi^t)h_\delta(|\xi^t|)\d b_1^t,
\end{equation}
where the property $h_\delta(0)=0$ ensures there is no additional term from the
local time of $\xi^t$ at 0.  
Let $A:[0,\infty) \to [0,\infty)$ be a function chosen later, and set
\[r^t=|\xi^t|+\int_0^t A(t-s)|\xi^s|\d s\]
Let $f:[0,\infty) \to [0,\infty)$ be a function chosen later, such that
$f(r)$ is differentiable and $f'(r)$ is absolutely continuous, and
\[f'(r) \in [0,1], \qquad f''(r) \leq 0 \text{ for a.e. } r \geq 0.\]
Then applying It\^o's formula again with \eqref{eq:dabsxi} for $\d|\xi^t|$ gives
\begin{equation}\label{eq:itobound1}
\begin{aligned}
\frac{\d}{\d t}
\E f(r^t)&=\E\left[f'(r^t)\left(\sign(\xi^t)(v^t-\tilde v^t)+A(0)|\xi^t|
+\int_0^t A'(t-s)|\xi^s|\d s\right)\right.\\
&\qquad\left.+4f''(r^t)h_\delta(|\xi^t|)^2\right].
\end{aligned}
\end{equation}

By Lemma \ref{lemma:ttig} and dominated convergence,
$\lim_{\iota \to 0} \int_0^\infty e^{\iota s} |r_g(s)|\d s<\alpha$. Then
there must exist a sufficiently small constant $\iota>0$ such that
\[\int_0^\infty e^{\iota s} |r_g(s)|\d s<\alpha-2\iota.\]
Let us set
\[\qquad A(t)=e^{-\iota t}\left(\alpha-2\iota-\int_0^t e^{\iota s}|r_g(s)|\d s
\right).\]
This satisfies
\[A(0)=\alpha-2\iota, \qquad A(t)>0, \qquad A'(t)={-}\iota A(t)-|r_g(t)|.\]
Define
\[\Delta_t=\int_0^t |R_g(t,s)-r_g(t-s)| \cdot \E|\theta^s|\d s
+\int_0^t |\tilde r_g(t-s)-r_g(t-s)| \cdot \E|\tilde \theta^s|\d s
+\E|\tilde g^t+z-g^t|.\]
Then, applying a triangle inequality and $f'(r) \in [0,1]$,
\begin{align}
&\E[f'(r^t)\sign(\xi^t)(v^t-\tilde v^t)]\notag\\
&\leq \E\Bigg[f'(r^t)\sign(\xi^t)\Big({-}U'(\theta^t)+U'(\tilde \theta^t)\Big)
\notag\\
&\hspace{1in}+f'(r^t)\left|\int_0^t R_g(t,s)\theta^s\d s
-\int_0^t \tilde r_g(t-s)\tilde \theta^s\d s+g^t-(\tilde g^t+z)\right|
\Bigg]\notag\\
&\leq \E\left[f'(r^t)\sign(\xi^t)\Big({-}U'(\theta^t)+U'(\tilde \theta^t)\Big)
+f'(r^t)\int_0^t |r_g(t-s)| \cdot |\xi^s|\d s\right]+\Delta_t.
\label{eq:itobound2}
\end{align}
Let
\[\kappa(r)=\inf_{x,y:|x-y|=r} \frac{U'(x)-U'(y)}{x-y}.\]
Then
\[\sign(\xi^t)\Big({-}U'(\theta^t)+U'(\tilde \theta^t)\Big)
={-}|\theta^t-\tilde \theta^t|
\frac{U'(\theta^t)-U'(\tilde \theta^t)}{\theta^t-\tilde \theta^t}
\leq -|\xi^t|\kappa(|\xi^t|).\]
Applying this bound and the identities $A(0)=\alpha-2\iota$ and
$|r_g(t-s)|={-}\iota A(t-s)-A'(t-s)$ into
\eqref{eq:itobound2} and \eqref{eq:itobound1},
and re-identifying the definition of $r^t$,
\begin{equation}\label{eq:itoboundfinal}
\frac{\d}{\d t}
\E f(r^t)\leq \E\left[-\iota f'(r^t)r^t
+f'(r^t)|\xi^t|(\alpha-\iota-\kappa(|\xi^t|))
+4f''(r^t)h_\delta(|\xi^t|)^2\right]+\Delta_t.
\end{equation}
Here, for some constants $C,c>0$, note that $\Delta_t$ satisfies the bound
\begin{equation}\label{eq:Deltatbound}
\sup_{t \in [T_0,T_0+T]} \Delta_t \leq C(\eps+T^{1/4}e^{-cT_0})
\end{equation}
by \eqref{eq:thetafourthmoment}, \eqref{eq:finalgcoupling}, 
$\int_0^t |R_g(t,s)-r_g(t-s)|\d s<Ce^{-cT_0}$
from Lemma \ref{lemma:ttig}, and
$\int_0^t |r_g(t-s)-\tilde r_g(t-s)|\d s<\eps$ from
Corollary \ref{cor:semigroupapprox}.

We proceed to derive a differential inequality for $\E f(r^t)$ from
\eqref{eq:itoboundfinal}.
By the condition \eqref{eq:Uconvexity} for $U(\cdot)$,
for some $C,R>0$, we must have
\[\kappa(r) \geq -C \text{ for all }
r \geq 0, \qquad \kappa(r) \geq \alpha-\iota \text{ for all } r \geq R.\]
Denote
\[\tilde \kappa(r)=\max(\alpha-\iota-\kappa(r),0)\]
so that
\[0 \leq \tilde \kappa(r) \leq C_0:=C+\alpha-\iota \text{ for all } r \geq 0,
\qquad \tilde \kappa(r)=0 \text{ for all } r \geq R.\]
Note that this implies
\[r\tilde \kappa(r) \leq C_0R \text{ for all } r \geq 0.\]
Recalling the cutoff $\delta>0$ defining $h_\delta(x)$ in
\eqref{eq:couplingboundary}, denote
\[K(r)=\sup_{t \in [T_0,T_0+T]} \E\Big[|\xi^t|\tilde \kappa(|\xi^t|)\,\Big|\,
r^t=r,\,|\xi^t|>\delta\Big] \in [0,C_0R].\]
Then, using that $h_\delta(x)=1$ when $x>\delta$, for any $t \in [T_0,T_0+T]$ we have
\begin{align*}
&\E\left[-\iota f'(r^t)r^t
+f'(r^t)|\xi^t|(\alpha-\iota-\kappa(|\xi^t|))
+4f''(r^t)h_\delta(|\xi^t|)^2 \,\Big|\,r^t=r,\,|\xi^t|>\delta \right]\\
&\leq \E\left[-\iota f'(r^t)r^t
+f'(r^t)|\xi^t|\tilde \kappa(|\xi^t|)
+4f''(r^t) \,\Big|\,r^t=r,\,|\xi^t|>\delta \right]\\
&={-}\iota f'(r)r+f'(r)K(r)+4f''(r).
\end{align*}
Let us set $f:[0,\infty) \to [0,\infty)$ as
\[f(0)=0,
\qquad f'(r)=\exp\left(-\frac{1}{4}\int_0^{\max(r,2C_0R/\iota)}
K(s)\d s\right)\]
Thus $f'(r)$ is absolutely continuous and satisfies $f'(r) \in [c_0,1]$ for a
constant $c_0>0$ (depending on $C_0,R,\iota$), and
\[f''(r)=\begin{cases} {-}\frac{1}{4}f'(r)K(r) & \text{ if } r \leq
2C_0R/\iota \\
0 & \text{ if } r>2C_0R/\iota. \end{cases}\]
For this choice of $f(r)$, substituting above gives
\begin{align*}
&\E\left[-\iota f'(r^t)r^t
+f'(r^t)|\xi^t|(\alpha-\iota-\kappa(|\xi^t|))
+4f''(r^t)h_\delta(|\xi^t|)^2 \,\Big|\,r^t=r,\,|\xi^t|>\delta \right]\\
&\leq {-}\iota f'(r)r
+\begin{cases} 0 & r \leq 2C_0R/\iota \\
f'(r)K(r) & r>2C_0R/\iota \end{cases}\\
&\leq {-}(\iota/2)f'(r)r
\end{align*}
where the second inequality applies $K(r) \leq C_0R<\iota
r/2$ when $r>2C_0R/\iota$. On the complementary event $|\xi^t| \leq \delta$,
let us simply apply $f'(r) \leq 1$, $f''(r) \leq 0$, and
$\tilde \kappa(r) \leq C_0$ to get
\[\begin{aligned}
&\E\left[-\iota f'(r^t)r^t
+f'(r^t)|\xi^t|(\alpha-\iota-\kappa(|\xi^t|))
+4f''(r^t)h_\delta(|\xi^t|)^2 \,\Big|\,r^t=r,\,|\xi^t| \leq \delta \right]
\\
&\qquad\leq -\iota f'(r)r+C_0\delta.
\end{aligned}\]
Combining these bounds, taking the expectation over $r^t$, and applying this to
\eqref{eq:itoboundfinal} yields
\[\frac{\d}{\d t}
\E f(r^t)\leq {-}(\iota/2)\E[f'(r^t)r^t]+C_0\delta+\Delta_t.\]
Finally, applying $f(0)=0$ and $f'(r^t) \in [c_0,1]$, hence
$r^t \geq f(r^t) \geq c_0r^t$, we obtain the desired differential inequality
\[\frac{\d}{\d t}
\E f(r^t)\leq {-}(c_0\iota/2)\E[f(r^t)]+C_0\delta+\Delta_t.\]
Here $\delta>0$ is arbitrary, $C_0,c_0,\iota$ do not depend on $\delta$,
and $f(r^{T_0})=f(0)=0$, so Gr\"onwall's inequality and \eqref{eq:Deltatbound}
show
\[\sup_{t \in [T_0,T_0+T]} \E f(r^t) \leq C \sup_{t \in [T_0,T_0+T]} \Delta_t
\leq C'(\eps+T^{1/4}e^{-cT_0}).\]
Since
also $\E|\xi^t| \leq \E r^t \leq c_0^{-1}\E f(r^t)$, this establishes for some
constants $C,c>0$ that
\[\sup_{t \in [T_0,T_0+T]} \E|\theta^t-\tilde \theta^t|
\leq C(\eps+T^{1/4}e^{-cT_0}).\]
The desired bound for $\E(\theta^t-\tilde\theta^t)^2$
then follows from $\E(\theta^t-\tilde\theta^t)^2 \leq
\E[|\theta^t-\tilde\theta^t|]^{2/3}\E[|\theta^t-\tilde\theta^t|^4]^{1/3}$ and
the fourth moment bounds of \eqref{eq:thetafourthmoment}.
\end{proof}

We now prove Theorem \ref{thm:replica}.

\begin{proof}[Proof of Theorem \ref{thm:replica}]

Fix constants $T_0,T,\eps>0$ to be determined, and let
$(\{\tilde \theta^t\}_{t \geq 0},\allowbreak z)$
be as constructed in Lemma \ref{lemma:thetacoupling} for these values.
Recall that $z \sim \cN(0,q_*\cR_\Lambda'(v_*-q_*))$, and that
 $\{\tilde \theta^t\}_{t \geq 0}$ is defined via the kernels
$(\tilde c_g,\tilde r_g) \equiv (\tilde c_g^m,\tilde r_g^m)$
corresponding to a pair $(A,u) \in \R^{m \times m} \times
\R^m$ in Corollary \ref{cor:semigroupapprox},
where $m,A,u,\tilde c_g,\tilde r_g$ depend only on $\eps$ and not on $T_0,T$.

Recall from Corollary \ref{cor:semigroupapprox} that
$|\tilde c_g(0)-\cR_\Lambda(v_*-q_*)|<\eps$.
Then for $\eps>0$ sufficiently small, the second condition in
\eqref{eq:alphacondition} implies
\begin{equation}\label{eq:Uconvexityreplica}
\liminf_{|\theta| \to \infty} U''(\theta)>\tilde c_g(0)+\eps.
\end{equation}
Then given any $z \in \R$, we may define a probability density on $\R$ by
\begin{equation}\label{eq:mum}
\mu_m(\tilde \theta \mid z) \propto \exp\bigg({-}U(\tilde\theta)
+z\tilde\theta+\frac{\tilde c_g(0)}{2}\tilde\theta^2\bigg).
\end{equation}
Let $(\tilde\theta,\tilde\theta',z)$ be random variables such that
$z \sim \cN(0,q_*\cR_\Lambda'(v_*-q_*))$ as above, and
$\tilde\theta,\tilde\theta'$ are conditionally independent given $z$, each
having law $\mu_m(\tilde\theta \mid z)$. We claim that for some constants
$C(\eps),c(\eps)>0$ depending only on $\eps$ and not on $T_0,T$,
\begin{equation}\label{eq:W2tildethetabound}
W_2(\Law(\tilde \theta^{T_0+T/2},\tilde \theta^{T_0+T},z),\,
\Law(\tilde\theta,\tilde\theta',z)) \leq C(\eps)e^{-c(\eps)T}.
\end{equation}

To show this claim, let $\Sigma \in \R^{m \times m}$ be any matrix such that
$\Sigma\Sigma^\top=A+A^\top$. Observe that the
joint law of $(\{\tilde\theta^t\}_{t \geq 0},z)$
is equivalent to that in the system
\begin{equation}\label{eq:auxiliarytheta}
\begin{aligned}
\tilde \theta^t&=\theta^t \text{ for all } t \in [0,T_0]\\
\d \tilde \theta^t&=[{-}U'(\tilde \theta^t)
+u^\top x^t+z]\d t+\sqrt{2}\,\d\tilde b^t \text{ for } t \geq T_0,\\
\d x^t&={-}A[x^t-u\,\tilde \theta^t]\d t+\Sigma\,\d b_x^t \text{ for } t \geq 0,
\qquad x^0 \sim \cN(0, I_m),
\end{aligned}
\end{equation}
where $\{b_x^t\}_{t \geq 0}$ is a standard Brownian motion in $\R^m$ independent
of $(\{\theta^t\}_{t \in [0,T_0]},\allowbreak
\{\tilde b^t\}_{t \geq 0},\allowbreak z,\allowbreak x^0)$,
and $\{x^t\}_{t \geq 0}$ is an auxiliary process in $\R^m$. Indeed, the third
equation of \eqref{eq:auxiliarytheta} has the explicit solution
\begin{equation}\label{eq:xtsolution}
x^t=e^{-At} x^0 + \int_0^t Ae^{-A(t-s)}u\,\tilde\theta^s\,\d s
+\int_0^t e^{-A(t-s)}\Sigma\,\d b_x^s.
\end{equation}
Substituting this solution into the second equation gives
\[\begin{aligned}
\d\tilde\theta^t&=\bigg[{-}U'(\tilde\theta^t)+\int_0^t u^\top
Ae^{-A(t-s)}u\,\tilde \theta^s \d s+u^\top e^{-At} x^0+\int_0^t u^\top e^{-A(t-s)}\Sigma\,\d b_x^s
+z\bigg]\d t\\
&\qquad+\sqrt{2}\,\d\tilde b^t \text{ for } t \geq T_0.
\end{aligned}\]
Since $u^\top Ae^{-A(t-s)}u=\tilde r_g(t-s)$, and since $u^\top e^{-At} x^0+\int_0^t u^\top
e^{-A(t-s)}\Sigma\,\d b_x^s$ is a Gaussian process with covariance kernel
$\tilde c_g$ as shown in the proof of Corollary \ref{cor:semigroupapprox},
this gives the same law for $(\{\tilde\theta^t\}_{t \geq 0},z)$
as defined in Lemma \ref{lemma:thetacoupling}.

Conditional on $z$, the process $\{(\tilde \theta^t,x^t)\}_{t \geq T_0}$ in
\eqref{eq:auxiliarytheta} is a Markov diffusion after time $T_0$, taking
the form
\begin{equation}\label{eq:jointmarkov}
\d \begin{pmatrix} \tilde \theta^t \\ x^t \end{pmatrix}
={-}\begin{pmatrix} 1 & 0 \\ 0 & A \end{pmatrix}
\nabla H_m(\tilde\theta^t,x^t \mid z)+\begin{pmatrix} \sqrt{2} & 0 \\ 0 &
\Sigma \end{pmatrix} \d\begin{pmatrix} \tilde b^t \\ b_x^t
\end{pmatrix}
\end{equation}
where $H_m(\cdot \mid z)$ is the Hamiltonian on $\R^{m+1}$ given by
\begin{align*}
H_m(\tilde\theta,x \mid z)&=U(\tilde\theta)-(u^\top x+z)\tilde\theta+\frac{1}{2}\|x\|_2^2\\
&=U(\tilde\theta)-z\tilde\theta-\frac{\tilde c_g(0)}{2}\tilde\theta^2+\frac{1}{2}\|x-\tilde\theta u\|_2^2,
\end{align*}
the second equality applying $\|u\|_2^2=\tilde c_g(0)$. Then
\[\mu_m(\tilde\theta,x \mid z) \propto e^{-H_m(\tilde\theta,x \mid z)}\]
defines the joint probability density for a pair $(\tilde\theta,x)$ where
$\tilde\theta \sim \mu_m(\tilde\theta \mid z)$ has the conditional
law \eqref{eq:mum} given $z$, and $x \sim \cN(\tilde\theta u,\Id_m)$
given both $(\tilde\theta,z)$. Since $A+A^\top=\Sigma\Sigma^\top$,
it is standard to check that $\mu_m(\tilde\theta,x \mid z)$
is a stationary law of the diffusion process \eqref{eq:jointmarkov}.

To quantify the convergence of \eqref{eq:jointmarkov} to $\mu_m(\tilde\theta,x
\mid z)$,
let us write $C(\eps),c(\eps)>0$ for constants depending on $\eps$
but not $T_0,T$ or $z$. We check that $\mu_m(\tilde\theta,x \mid z)$ satisfies a
(Euclidean) log-Sobolev inequality on $\R^{m+1}$
with some log-Sobolev constant $C(\eps)>0$. Indeed,
the condition \eqref{eq:Uconvexityreplica} implies that there exists a
continuous, compactly supported function
$f:\R \to \R$ for which $\int_\R f=0$
and $f(\theta) \leq U''(\theta)-\tilde c_g(0)-\eps/2$ for all $\theta \in \R$.
Set $W(\theta)=\int_{-\infty}^\theta  \d x \int_{-\infty}^x \d y\,f(y)$, so that
$W'(\theta)=\int_{-\infty}^\theta f(y)\d y$ remains compactly supported (with
the same support as $f$),
and hence $W(\theta)$ is bounded. Define $V(\theta)$ by
\begin{equation}\label{eq:logsobolevdecomp}
U(\theta)-z\theta-\frac{\tilde c_g(0)}{2}\theta^2=V(\theta)+W(\theta),
\end{equation}
so $V''(\theta)=U''(\theta)-\tilde c_g(0)-f(\theta) \geq \eps/2$ for all
$\theta \in \R$. Then by the Holley-Stroock perturbation lemma,
$\mu_m(\tilde\theta \mid z)$ satisfies a log-Sobolev inequality on $\R$ with
log-Sobolev constant $C(\eps)$ not depending on $z$.
Writing $x=\tilde\theta u+w$ where $w \sim \cN(0,\Id_m)$ is independent of
$(\tilde\theta,z)$, we note that $(\tilde\theta,x)$ is a Lipschitz map of
$(\tilde\theta,w)$, so this implies also a log-Sobolev inequality
for $\mu_m(\tilde\theta,x \mid z)$.

Since $\nabla^2 H(\tilde\theta,x \mid z)$ is bounded below by some 
constant ${-}c(\eps)<0$ also not depending on $z$,
a log-Harnack/entropy-cost inequality \cite[Corollary 1.2]{rockner2010log}
shows, for any $t>s \geq T_0$, that $\Law(\tilde \theta^t,x^t \mid \tilde \theta^s,x^s,z)$ admits a density
with respect to the stationary measure $\mu_m(\tilde\theta,x \mid z)$
and satisfies
\begin{align*}
\DKL(\Law(\tilde \theta^t,x^t \mid \tilde \theta^s,x^s,z)\,\|\,\mu_m(\tilde\theta,x \mid z))
&\leq \frac{C(\eps)}{e^{c(\eps)(t-s)}-1}
W_2(\delta_{(\tilde \theta^s,x^s)},\,\mu_m(\tilde\theta,x \mid z))^2\\
&\leq \frac{2C(\eps)}{e^{c(\eps)(t-s)}-1}\Big(\|(\tilde \theta^s,x^s)\|_2^2
+\E[\|(\tilde\theta,x)\|_2^2 \mid z]\Big)
\end{align*}
where $\delta_{(\tilde \theta^s,x^s)}$ denotes the unit point measure
at $(\tilde \theta^s,x^s) \in \R^{m+1}$, and $(\tilde \theta,x)$ denotes a sample
from $\mu_m(\tilde \theta,x \mid z)$. By the log-Sobolev inequality for
$\mu_m(\tilde \theta,x \mid z)$ and $\eps$-dependent lower bound for
$\lambda_{\min}(A+A^\top)$, for any
$t'>t>s \geq T_0$, we have also the decay in relative entropy, 
\begin{align*}
&\DKL(\Law(\tilde \theta^{t'},x^{t'} \mid \tilde \theta^s,x^s,z)\,\|\,\mu_m(\tilde\theta,x \mid z))\\
&\leq e^{-c(\eps)(t'-t)}
\DKL(\Law(\tilde \theta^t,x^t \mid \tilde \theta^s,x^s,z)\,\|\,\mu_m(\tilde \theta,x \mid z)),
\end{align*}
as well as the $T_2$-transportation inequality \cite{otto2000generalization}
\[W_2(\Law(\tilde \theta^{t'},x^{t'} \mid \tilde \theta^s,x^s,z),\,\mu_m(\tilde\theta,x \mid z))^2
\leq C(\eps)\DKL(\Law(\tilde \theta^{t'},x^{t'} \mid \tilde \theta^s,x^s,z)\,\|\,\mu_m(\tilde\theta,x \mid z)).\]
Applying the above with $t=s+1$ and $t'=s+\tau$ shows, for some
$C(\eps),c(\eps)>0$ and any $s \geq T_0$ and $\tau \geq 1$, that
\begin{equation}\label{eq:auxwassersteincontraction}
\begin{aligned}
&W_2(\Law(\tilde \theta^{s+\tau},x^{s+\tau} \mid \tilde \theta^s,x^s,z),\,\mu_m(\tilde\theta,x \mid z))^2\\
&\leq C(\eps)e^{-c(\eps)\tau}\Big(\|(\tilde \theta^s,x^s)\|_2^2
+\E[\|(\tilde\theta,x)\|_2^2 \mid z]\Big).
\end{aligned}
\end{equation}

Let $(\tilde\theta,x,\tilde\theta',x',z)$ be such that $z \sim
\cN(0,q_*\cR_\Lambda'(v_*-q_*))$, and $(\tilde\theta,x),(\tilde\theta',x')$
are conditionally independent given $z$ with law $\mu_m(\tilde\theta,x \mid z)$.
Then, applying \eqref{eq:auxwassersteincontraction} with $s=T_0$ and
$\tau=T/2$ to couple $(\tilde \theta^{T_0+T/2},x^{T_0+T/2})$ with
$(\tilde\theta,x)$ conditional on $(\tilde \theta^{T_0},x^{T_0})$ and $z$,
and then with $s=T_0+T/2$ and $\tau=T/2$ to couple
$(\tilde \theta^{T_0+T},x^{T_0+T})$ with $(\tilde\theta',x')$
conditional on $(\tilde \theta^{T_0+T/2},x^{T_0+T/2})$ and $z$,
and finally marginalizing over $z$, we get for some (different)
constants $C(\eps),c(\eps)>0$ that
\begin{align}
&W_2(\Law(\tilde \theta^{T_0+T/2},\tilde \theta^{T_0+T},z),
\,\Law(\tilde\theta,\tilde\theta',z))^2\notag\\
&\leq \E_{z \sim \cN(0,q_*\cR_\Lambda'(v_*-q_*))}
W_2(\Law(\tilde \theta^{T_0+T/2},x^{T_0+T/2},\tilde \theta^{T_0+T},
x^{T_0+T} \mid z),\,\Law(\tilde\theta,x,\tilde\theta',x' \mid z))^2\notag\\
& \leq C(\eps)e^{-c(\eps)T}\Big(\E\|(\tilde \theta^{T_0},x^{T_0})\|_2^2
+\E\|(\tilde \theta^{T_0+T/2},x^{T_0+T/2})\|_2^2
+\E\|(\tilde\theta,x)\|_2^2
+\E\|(\tilde\theta',x')\|_2^2\Big).\label{eq:W2jointcoupling}
\end{align}
Here, $\E(\tilde \theta^t)^2$ is uniformly bounded
for all $t \geq 0$ by \eqref{eq:thetafourthmoment}.
Then also by the explicit form \eqref{eq:xtsolution}, for any $t \geq T_0$,
\begin{align*}
\E\|x^t\|_2^2
&\leq 2\E\left\|\int_0^t Ae^{-(t-s)A}u\,\theta^s\d s\right\|_2^2
+2\E\left\|\int_0^t e^{-A(t-s)}\Sigma\,\d b_x^s + e^{-At} x^0\right\|_2^2\\
&\leq C\left(\int_0^\infty \|Ae^{-sA}u\|_2\,\d s\right)^2
+2m.%\int_0^\infty \Tr e^{-sA}\Sigma\Sigma^\top e^{-sA^\top}\d s.
\end{align*}
This is at most some $C(\eps)>0$ by the bound $\|e^{-sA}\|_\op \leq
e^{-(s/2)\lambda_{\min}(A+A^\top)}$. Since $(\tilde \theta^t,x^t)$
converges weakly in law to $(\tilde \theta,x)$ as $t \to \infty$,
\eqref{eq:thetafourthmoment} and Fatou's lemma imply also
\begin{equation}\label{eq:staticfourthmoment}
\E\tilde\theta^4=\E(\tilde\theta')^4 \leq C
\end{equation} for a constant $C>0$. Then, since $x \sim \cN(\tilde \theta
u,\Id_m)$ given $(\tilde \theta,z)$, also
$\E\|x\|_2^2=\E\|x'\|_2^2 \leq C(\eps)$.
Applying these bounds to the right side of \eqref{eq:W2jointcoupling}
yields \eqref{eq:W2tildethetabound} as claimed.

Under the second condition in \eqref{eq:alphacondition}, we may define also
the probability density on $\R$
\[\mu(\theta \mid z) \propto
\exp\Big({-}U(\theta)+z\theta+\frac{\cR_\Lambda(v_*-q_*)}{2}\theta^2\Big).\]
Let $(\theta,\theta',z)$ be random variables such that $z \sim
\cN(0,q_*\cR_\Lambda'(v_*-q_*))$ and $\theta,\theta'$ are conditionally
independent given $z$ with law $\mu(\theta \mid z)$, and let
$(\tilde\theta \tilde\theta',z)$ be defined from $\mu_m(\theta \mid z)$
as above. As $\eps \to 0$,
considering a sequence of $m=m(\eps)$ defining $\mu_m(\theta \mid z)$
and recalling that $|\tilde c_g(0)-\cR_\Lambda(v_*-q_*)|<\eps$,
we have $\mu_m(\cdot \mid z) \to \mu(\cdot \mid z)$ weakly as $\eps \to 0$
for each fixed $z \in \R$. Then also marginalizing over $z$,
$\Law(\tilde\theta,\tilde\theta') \to \Law(\theta,\theta')$ weakly as
$\eps \to 0$. Since \eqref{eq:staticfourthmoment} implies that
$(\tilde\theta,\tilde\theta')$ is uniformly integrable in $L^2$ for all
$\eps>0$, this implies
\[\E[\|(\tilde\theta,\tilde\theta')\|_2^2] \to \E[\|(\theta,\theta')\|_2^2]\]
and hence also $W_2(\Law(\tilde\theta,\tilde\theta'),\Law(\theta,\theta')) \to
0$. Thus we have
\begin{equation}\label{eq:replicalawconvergence}
W_2(\Law(\tilde\theta,\tilde\theta'),\Law(\theta,\theta'))^2
\leq o_\eps(1)
\end{equation}
for a quantity $o_\eps(1)$ depending only on $\eps$ and vanishing
as $\eps \to 0$. Combining the conclusion of Lemma \ref{lemma:thetacoupling}
with \eqref{eq:W2tildethetabound} and \eqref{eq:replicalawconvergence},
\[W_2(\Law(\theta^{T_0+T/2},\theta^{T_0+T}),
\Law(\theta,\theta'))^2
\leq C(\eps+T^{1/4}e^{-cT_0})^{2/3}
+C(\eps)e^{-c(\eps)T}+o_\eps(1).\]
Given $\eps>0$, we may take $T=T(\eps)>0$ large enough such that the second
term is small, followed by $T_0=T_0(\eps)$ large enough such that the first term
is small. Thus, there exist some $T(\eps),T_0(\eps) \to \infty$ for which,
as $\eps \to 0$,
\begin{equation}\label{eq:finalW2convergence}
\Law(\theta^{T_0(\eps)+T(\eps)/2},\theta^{T_0(\eps)+T(\eps)}) \to 
\Law(\theta,\theta') \text{ weakly and in Wasserstein-2}.
\end{equation}

We now turn to the statements of Theorem \ref{thm:replica}. We may again assume,
without loss of generality, that Assumption \ref{assump:as} holds.\\

{\bf Part (a).}
Identifying $z=\sqrt{q_*\cR_\Lambda'(v_*-q_*)}\,\xi$ where $\xi
\sim \cN(0,1)$, the above law $\mu(\theta \mid z)$ is precisely
$\mu(\theta \mid v_*,q_*,\xi)$ defined in Theorem \ref{thm:replica}.
Then \eqref{eq:finalW2convergence} implies,
by the characterizations of $v_*$ and $q_*$ in Lemma \ref{lemma:tti}, that
\begin{align*}
v_*&=c_\theta(0)=\lim_{t \to \infty} \E(\theta^t)^2=\E[\theta^2]
=\E_{\xi \sim \cN(0,1)} \langle \theta^2 \rangle_{v_*,q_*,\xi},\\
q_*&=\lim_{\tau \to \infty} c_\theta(\tau)
=\lim_{s,\tau \to \infty} \E(\theta^s\theta^{s+\tau})=\E[\theta\theta']
=\E_{\xi \sim \cN(0,1)} \langle \theta \rangle_{v_*,q_*,\xi}^2,
\end{align*}
establishing the fixed point conditions \eqref{eq:replicafixedpoint} for
$(v_*,q_*)$.

Let $\btheta \sim \mu_\Gibbs$. Assumption
\ref{assump:convergence} implies
that for some constants $C,c>0$,
\begin{equation}\label{eq:langevinWasserstein}
\frac{1}{n}W_2(\Law(\btheta^t \mid \X,\btheta^0),\,\Law(\btheta \mid \X))^2
\leq Ce^{-ct}.
\end{equation}
This implies, for any function $f:\R \to \R$ satisfying $|f(x)-f(y)| \leq
C(1+|x|+|y|)|x-y|$, that
\[\left|\frac{1}{n}\sum_{i=1}^n \langle f(\theta_i^t)
\rangle-\frac{1}{n}\sum_{i=1}^n \langle f(\theta_i) \rangle \right|
\leq C(f)e^{-ct}\]
for a constant $C(f)>0$. Theorem \ref{thm:dmft-approx} implies, for each fixed
$t \geq 0$,
\[\lim_{n \to \infty}
\left|\frac{1}{n}\sum_{i=1}^n \langle f(\theta_i^t) \rangle-\E[f(\theta^t)]
\right|=0 \text{ a.s.}\]
where $\E[\cdot]$ is the expectation over $\theta^t$ in the dynamical mean-field
limit of Definition \ref{def:DMFT}. The Wasserstein convergence
\eqref{eq:finalW2convergence} implies
\[\lim_{t \to \infty} \E[f(\theta^t)]=\E_{\xi \sim \cN(0,1)} \langle f(\theta)
\rangle_{v_*,q_*,\xi}.\]
Combining the above bounds, and taking $n \to \infty$ followed by $t \to
\infty$, shows \eqref{eq:pseudoliptest}.

For the final statement \eqref{eq:thetaXthetalimit}, recall the process
$\g^t=\X\btheta^t-\int_0^t r_g(t-s)\btheta^s \d s$. Then
\[\frac{1}{n} \langle \btheta^{t\top} \X\btheta^t \rangle
=\left\langle \frac{1}{n}\btheta^{t\top}\g^t+
\int_0^t r_g(t-s)\frac{1}{n}\btheta^{t\top} \btheta^s \d s \right\rangle.\]
Recall from the equicontinuity \eqref{eq:thetaequicontinuous} 
and Theorem \ref{thm:dmft-approx}(c) that
\[\lim_{n \to \infty} \sup_{s \in [0,t]}
\left|\frac{1}{n}\langle \btheta^{t\top}\btheta^s
\rangle-C_\theta(t,s)\right|=0 \text{ a.s.},
\qquad \lim_{n \to \infty} \frac{1}{n}\langle \btheta^{t\top} \g^t \rangle
=\E[\theta^t g^t] \text{ a.s.}\]
where $\E[\theta^t g^t]$ is the expectation in the dynamical mean-field limit of
Definition \ref{def:DMFT}. To compute this expectation, note that It\^o's lemma
applied to \eqref{eq:dmft-theta} gives
\[\frac{\d}{\d t}\frac{1}{2}\,\E (\theta^t)^2
=\E\theta^t \left({-}U'(\theta^t)+\int_0^t R_g(t,s)\theta^s \d  s
+g^t\right)+1.\]
Therefore
\[\E[\theta^t g^t]
=\frac{\d}{\d t}\frac{1}{2}\,\E(\theta^t)^2
+\E[\theta^t U'(\theta^t)]-\int_0^t R_g(t,s)C_\theta(t,s) \d s-1.\]
Applying this above,
\[\lim_{n \to \infty} \frac{1}{n}\langle \btheta^{t\top}\X\btheta^t \rangle
=\frac{\d}{\d t}\frac{1}{2}\E(\theta^t)^2
+\E[\theta^tU'(\theta^t)]-1
+\int_0^t (r_g(t-s)-R_g(t,s))C_\theta(t,s)\d s \text{ a.s.}\]
By the equicontinuity \eqref{eq:thetaequicontinuous},
for any $T>0$, this convergence holds uniformly over $t \in [T,2T]$. Then also
\begin{align*}
\lim_{n \to \infty}
\frac{1}{T}\int_T^{2T} \frac{1}{n} \langle \btheta^{t\top}\X\btheta^t \rangle
\d t&=\frac{1}{2T}\big(\E(\theta^{2T})^2-\E(\theta^T)^2\big)
+\frac{1}{T}\int_T^{2T} \Big(\E[\theta^tU'(\theta^t)]-1\Big)\d t\\
&\quad+\frac{1}{T}\int_T^{2T}\d t
\int_0^t (r_g(t-s)-R_g(t,s))C_\theta(t,s)\d s \text{ a.s.}
\end{align*}
The Wasserstein convergence \eqref{eq:finalW2convergence} and bounds for
$|C_\theta(t,s)|$ and $\int_0^t |r_g(t-s)-R_g(t,s)|\d s$ in Lemmas
\ref{lemma:tti} and \ref{lemma:ttig} imply that the first and third
terms on the right side vanish as $T \to \infty$, while the second term
converges to
\[\E_{\xi \sim \cN(0,1)} \langle \theta U'(\theta) \rangle_{v_*,q_*,\xi}-1.\]
This may be evaluated by noting that, for any fixed $\xi \in \R$,
\begin{align*}
&\left\langle\theta\left({-}U'(\theta)+\cR_\Lambda(v_*-q_*)\theta
+\sqrt{q_*\cR_\Lambda'(v_*-q_*)}\,\xi\right)
\right\rangle_{v_*,q_*,\xi}\\
&=\int_\R \theta \mu'(\theta \mid v_*,q_*,\xi) \d \theta
=-\int \mu(\theta \mid v_*,q_*,\xi) \d \theta=-1.
\end{align*}
Therefore, integrating by parts also in $\xi$,
\begin{align*}
&\E_{\xi \sim \cN(0,1)}\langle{\theta U'(\theta)}\rangle_{v_*,q_*,\xi}-1\\
&=\cR_\Lambda(v_*-q_*)\E_{\xi \sim \cN(0,1)}\langle\theta^2
\rangle_{v_*,q_*,\xi}+\sqrt{q_*\cR_\Lambda'(v_*-q_*)}\E_{\xi \sim \cN(0,1)}[\xi \langle \theta
\rangle_{v_*,q_*,\xi}]\\
&=\cR_\Lambda(v_*-q_*)\E_{\xi \sim \cN(0,1)}\langle\theta^2
\rangle_{v_*,q_*,\xi}+q_*\cR_\Lambda'(v_*-q_*)
\E_{\xi \sim \cN(0,1)}[\langle \theta^2 \rangle_{v_*,q_*,\xi}
-\langle \theta \rangle_{v_*,q_*,\xi}^2]\\
&=v_*\cR_\Lambda(v_*-q_*)+q_*(v_*-q_*)\cR_\Lambda'(v_*-q_*).
\end{align*}
This shows
\[\lim_{n \to \infty}
\frac{1}{T}\int_T^{2T} \frac{1}{n} \langle \btheta^{t\top}\X\btheta^t \rangle
=v_*\cR_\Lambda(v_*-q_*)+q_*(v_*-q_*)\cR_\Lambda'(v_*-q_*)+o_T(1) \text{
a.s.},\]
where $o_T(1)$ denotes an error vanishing as $T \to \infty$. The Wasserstein
bound \eqref{eq:langevinWasserstein} and Assumption \ref{assump:as}
imply that for any $t \in [T,2T]$,
\[\left|\frac{1}{n} \langle \btheta^{t\top}\X\btheta^t \rangle
-\frac{1}{n} \langle \btheta^\top\X\btheta \rangle\right|
\leq Ce^{-ct}.\]
Then taking the limit $n \to \infty$ followed by $T \to \infty$ shows
\eqref{eq:thetaXthetalimit}.\\

{\bf Part (b).} Let
\[\mu_\beta(\btheta)=\frac{1}{\cZ_\beta}
\exp\left(\frac{\beta}{2}\btheta^\top \X\btheta
-\sum_{i=1}^n U(\theta_i)\right) \d\btheta,
\quad \cZ_\beta=\int \exp\left(
\frac{\beta}{2}\btheta^\top \X\btheta
-\sum_{i=1}^n U(\theta_i) \right) \d \btheta,\]
and write $\langle f(\btheta) \rangle_\beta$ for the Gibbs average under
$\mu_\beta$. Then
\[\frac{\d}{\d\beta} \frac{1}{n}\log \cZ_\beta
=\frac{1}{2n} \langle \btheta^\top \X\btheta \rangle_\beta.\]
The result \eqref{eq:thetaXthetalimit} for $\beta\X$ in place of $\X$ shows, for
each $\beta \in [0,1]$, almost surely
\[\begin{aligned}
&\lim_{n \to \infty}
\frac{1}{n}\langle \btheta^\top(\beta \X)\btheta \rangle_\beta
\\
&\qquad=v_*(\beta)\cR_{\beta\Lambda}(v_*(\beta)-q_*(\beta))\\
&\qquad\quad+q_*(\beta)(v_*(\beta)-q_*(\beta))
\cR_{\beta\Lambda}'(v_*(\beta)-q_*(\beta))
\end{aligned}\]
where $v_*(\beta),q_*(\beta)$ are the limits ensured by Theorem
\ref{thm:replica}(a) for $\mu_\beta$,
and $\cR_{\beta\Lambda}(x)$ is the R-transform of
$\beta\Lambda$. Writing $\{\kappa_p\}_{p \geq 1}$ for the free cumulants of
$\Lambda$, those of $\beta\Lambda$ are then $\{\beta^p \kappa_p\}_{p \geq 1}$,
so
\[\cR_{\beta\Lambda}(x)=\sum_{p \geq 1} \beta^{p+1}\kappa_{p+1} x^{p}
=\beta \cR_\Lambda(\beta x),
\qquad \cR_{\beta\Lambda}'(x)=\beta^2 \cR'_\Lambda(\beta x).\]
Applying this above, integrating over $\beta \in [0,1]$, and applying dominated
convergence (noting that the bound \eqref{eq:thetanormbound} for
$\langle \frac{1}{n}\|\btheta\|_2^2 \rangle_\beta$ is uniform in $\beta$),
\begin{align}
\lim_{n \to \infty} \frac{1}{n}\log \cZ
&=\lim_{n \to \infty}\frac{1}{n}\log \cZ_0
+\int_0^1 \frac{1}{2n}\langle \btheta^\top \X\btheta \rangle_\beta\d \beta\notag\\
&=\log \int e^{-U(\theta)}\d\theta+\int_0^1 \frac{1}{2}\Big(
v_*(\beta)\cR_\Lambda(\beta(v_*(\beta)-q_*(\beta)))\label{eq:integralfreeenergy}\\
&\hspace{1.5in}+\beta
q_*(\beta)(v_*(\beta)-q_*(\beta))\cR_\Lambda'(\beta(v_*(\beta)-q_*(\beta))\Big)\d\beta.\notag
\end{align}

To compare this to the replica formula, note first that
when Assumption \ref{assump:convergence} holds with
uniform constants $C,c>0$ for the dynamics defined by $\beta\X$ and each
$\beta \in [0,1]$, the functions $v_*(\beta),q_*(\beta)$ must be 
H\"older-continuous over $\beta \in [0,1]$.
Indeed, letting $\{\btheta_\beta^t\}_{t \geq 0}$
denote the solution of \eqref{eq:langevin} for $\beta\X$ and letting
$\btheta_\beta \sim \mu_\beta$ denote a sample from the Gibbs measure,
Assumption \ref{assump:convergence} implies
\begin{equation}\label{eq:vcontinuous1}
\left|\frac{1}{n}\langle \|\btheta_\beta^t\|_2^2 \rangle
-\frac{1}{n}\langle \|\btheta_\beta\|_2^2 \rangle\right| \leq Ce^{-ct}
\text{ for all } t \geq 0,\,\beta \in [0,1].
\end{equation}
For $\beta,\beta' \in [0,1]$, coupling $\{\btheta_\beta^t\}_{t \geq 0}$ and
$\{\btheta_{\beta'}^t\}_{t \geq 0}$ by the same Brownian motion and applying
the bound \eqref{eq:thetatnormbound} which holds for
$\{\btheta^t_\beta\}_{t \geq 0}$ uniformly over $\beta \in [0,1]$, we have
\begin{align*}
\left|\frac{\d}{\d t}\langle \|\btheta_\beta^t-\btheta_{\beta'}^t\|_2^2
\rangle\right|
&=\left|\langle 2(\btheta_\beta^t-\btheta_{\beta'}^t)^\top
(\beta \X\btheta_\beta^t-U'(\btheta_\beta^t)
-\beta' \X\btheta_{\beta'}^t+U'(\btheta_{\beta'}^t))\rangle\right|\\
&\leq C\langle \|\btheta_\beta^t-\btheta_{\beta'}^t\|_2^2 \rangle
+Cn|\beta-\beta'|
\end{align*}
for a constant $C>0$. Then by Gr\"onwall's inequality, $\frac{1}{n}\langle
\|\btheta_\beta^t-\btheta_{\beta'}^t\|_2^2 \rangle \leq Ce^{Ct}|\beta-\beta'|$.
By \eqref{eq:thetatnormbound} and Cauchy-Schwarz, this implies for a (different)
constant $C>0$ that
\begin{equation}\label{eq:vcontinuous2}
\left|\frac{1}{n}\langle \|\btheta_\beta^t\|_2^2 \rangle
-\frac{1}{n}\langle \|\btheta_{\beta'}^t\|_2^2 \rangle\right|
\leq Ce^{Ct}|\beta-\beta'| \text{ for all } t \geq 0,\,\beta,\beta' \in [0,1].
\end{equation}
Combining \eqref{eq:vcontinuous1} and \eqref{eq:vcontinuous2}, taking the limit
$n \to\infty$, and optimizing over $t$ shows that $v_*(\beta)$ is H\"older
continuous over $\beta \in [0,1]$ with some exponent $\rho \in (0,1)$,
i.e.\ $|v_*(\beta)-v_*(\beta')| \leq C|\beta-\beta'|^\rho$ for a constant $C>0$.
 A similar argument establishes H\"older continuity
of $q_*(\beta)$ with the same exponent $\rho \in (0,1)$.

Now let $\alpha$ be the bound for
$U''(\cdot)$ in \eqref{eq:Uconvexity}. Fixing some small enough constant
$\eps>0$, define 
\begin{align*}
&S=\Bigg\{(\beta,v,q):
\beta \in [0,1],\;v>q \geq 0,\;\sum_{p \geq 1} \beta^{p+1}|\kappa_{p+1}|(v-q+\eps)^p<\alpha,\;
q\cR_{\beta\Lambda}'(v-q) \geq 0
\Bigg\}.
\end{align*}
By the given condition \eqref{eq:alphacondition} and Lemma \ref{lemma:ttig}(b),
$(\beta,v_*(\beta),q_*(\beta)) \in S$ for all $\beta \in [0,1]$.
For $(\beta,v,q) \in S$, we may define
\[\begin{aligned}
f(\beta,v,q)&=\E_{\xi \sim \cN(0,1)} \log \int
\exp\left(-U(\theta)+\frac{1}{2}\,\cR_{\beta\Lambda}(v-q)\theta^2
+\sqrt{q\cR_{\beta\Lambda}'(v-q)}\,\xi\theta\right) \d \theta  \\
&\quad
+\frac{1}{2}\int_0^{v-q} \cR_{\beta\Lambda}(s) \d  s
-\frac{1}{2}\,(v-q)\cR_{\beta\Lambda}(v-q)
-\frac{1}{2}\,q(v-q)\cR_{\beta\Lambda}'(v-q)
\end{aligned}\]
and also the probability density on $\R$
\[\mu_\beta(\theta \mid v,q,\xi)
\propto \exp\left(-U(\theta)+\frac{1}{2}\,\cR_{\beta\Lambda}(v-q)\theta^2
+\sqrt{q\cR_{\beta\Lambda}'(v-q)}\,\xi\theta\right) \d \theta.\]
Let us momentarily write $\langle \cdot \rangle$ without subscript
for its associated expectation.
Note that at any $(\beta,v,q) \in S$ where $q\cR_{\beta\Lambda}'(v-q)>0$
strictly, $f(\beta,v,q)$ is continuously-differentiable,
and a calculation via integration-by-parts on $\xi$ shows 
\begin{align*}
\partial_\beta f(\beta,v,q)&=
\frac{1}{2}\bigg(\E_\xi[\partial_\beta \cR_{\beta\Lambda}(v-q)\langle
 \theta^2 \rangle+q\partial_\beta \cR_{\beta\Lambda}'(v-q)
(\langle \theta^2 \rangle-\langle \theta \rangle^2)]\\
 &\qquad
 +\int_0^{v-q} \partial_\beta \cR_{\beta \Lambda}(s)\d s
 -(v-q)\partial_\beta \cR_{\beta\Lambda}(v-q)
 -q(v-q)\partial_\beta\cR_{\beta\Lambda}'(v-q)\bigg),\\
\partial_v f(\beta,v,q)&=
\frac{1}{2}\bigg(\E_\xi[\cR_{\beta\Lambda}'(v-q)\langle \theta^2 \rangle
+q\cR_{\beta\Lambda}''(v-q)
(\langle \theta^2 \rangle-\langle \theta \rangle^2)]\\
&\hspace{1in}
-v\cR_{\beta\Lambda}'(v-q)
-q(v-q)\cR_{\beta\Lambda}''(v-q)\bigg),\\
\partial_q f(\beta,v,q)&=
\frac{1}{2}\bigg(\E_\xi[{-}\cR_{\beta\Lambda}'(v-q)\langle \theta \rangle^2
-q\cR_{\beta\Lambda}''(v-q)
(\langle \theta^2 \rangle-\langle \theta \rangle^2)
]\\
&\hspace{1in}+q\cR_{\beta\Lambda}'(v-q)
+q(v-q)\cR_{\beta\Lambda}''(v-q)\bigg).
\end{align*}
These derivative formulas extend continuously to points $(\beta,v,q) \in S$ where
$q\cR_{\beta\Lambda}'(v-q)=0$, and we take these extensions to be the
definitions of $\partial_\beta f,\partial_v f,\partial_q f$ at such points.
Since $v_*(\beta),q_*(\beta)$ satisfy the fixed point
conditions \eqref{eq:replicafixedpoint}, the above forms of $\partial_v
f,\partial_q f$ imply
\begin{equation}\label{eq:vqderivvanish}
\partial_v f(\beta,v_*(\beta),q_*(\beta))=0,
\qquad \partial_q f(\beta,v_*(\beta),q_*(\beta))=0.
\end{equation}
From the identity $\cR_{\beta\Lambda}(x)=\beta\cR_\Lambda(\beta x)$, we have
$\partial_\beta \cR_{\beta\Lambda}(x)=\cR_{\Lambda}(\beta x)+\beta
x\cR_\Lambda'(\beta x)$ and
\[\int_0^x \partial_\beta \cR_{\beta\Lambda}(s) \d  s
=\int_0^x \left[\cR_\Lambda(\beta s)+\beta s\cR_\Lambda'(\beta s)
\right]\d s=x\cR_\Lambda(\beta x),\]
the last equality applying integration-by-parts for the second term. Then also
\begin{align*}
&\partial_\beta f(\beta,v_*(\beta),q_*(\beta))\\
&=\frac{1}{2}\left(\int_0^{v_*(\beta)-q_*(\beta)}
\partial_\beta \cR_{\beta\Lambda}(s) \d s
+q_*(\beta)\partial_\beta
\cR_{\beta\Lambda}(v_*(\beta)-q_*(\beta))\right)\\
&=\frac{1}{2}\Big(v_*(\beta)\cR_\Lambda(\beta(v_*(\beta)-q_*(\beta)))+\beta q_*(\beta)(v_*(\beta)-q_*(\beta))
\cR_\Lambda'(\beta(v_*(\beta)-q_*(\beta)))
\Big)
\end{align*}
which matches the integrand in \eqref{eq:integralfreeenergy}.

The above continuity of $v_*(\beta),q_*(\beta)$ implies that
$\beta \mapsto f(\beta,v_*(\beta),q_*(\beta))$ is continuous on $[0,1]$. We
claim that it is in fact continuously-differentiable on $(0,1)$, with derivative
\begin{equation}\label{eq:fbetacontdiff}
\frac{\d}{\d \beta}
f(\beta,v_*(\beta),q_*(\beta))
=\partial_\beta f(\beta,v_*(\beta),q_*(\beta)).
\end{equation}
To check differentiability, we extend $f(\beta,v,q)$ smoothly to an open domain: Let $\ell_{\beta, v, q}(z) = \log \int \exp\left(-U(\theta)+\frac{1}{2}\cR_{\beta\Lambda}(v-q)\theta^2 + z\theta\right) d\theta$,  which is well-defined and analytic over $z \in \R$ for all $(\beta, v, q)$ such that $\beta \in (0,1)$, $v>q$, and $\sum_{p \geq 1} \beta^{p+1}|\kappa_{p+1}|(v-q+\eps)^p< \alpha$. Consider $\tilde{\ell}_{\beta,v,q}(s) = \E_{\xi \sim \mathcal{N}(0,1)}[\ell_{\beta,v,q}(\sqrt{s}\xi)]$ for $s \ge 0$. By Gaussian integration by parts, \[
\frac{\partial^j}{\partial s^j} \tilde{\ell}_{\beta,v,q}(s) =\frac{1}{2^j} \E_{\xi \sim \mathcal{N}(0,1)}\left[\frac{\partial^{2j}}{\partial z^{2j}} \ell_{\beta,v,q}(\sqrt{s}\xi)\right], \quad j = 1,2,\ldots
\]
By dominated convergence, we can take the limit $s \to 0+$ and conclude that $\tilde{\ell}_{\beta,v,q}(s)$ has finite derivatives of all orders at $s = 0$ from the positive side. Pick $N$ to be a large enough integer such that $(N+1)\rho > 1$. We extend $\tilde \ell_{\beta,v,q}(s)$ by \begin{align}
   \nonumber \tilde \ell_{\beta,v,q}(s) = \begin{cases}
    \E_{\xi \sim \mathcal{N}(0,1)}[\ell_{\beta,v,q}(\sqrt{s}\xi)], & s \ge 0,\\
    \sum_{j=0}^{N+1} \frac{\tilde \ell^{(j)}(0)}{j!} s^j, & s < 0.
    \end{cases}
\end{align}
From now on, we understand $f(\beta,v,q) = \tilde \ell_{\beta,v,q}(q\cR_{\beta\Lambda}'(v-q)) + \frac{1}{2}\int_0^{v-q} \cR_{\beta\Lambda}(s) ds - \frac{1}{2}(v-q)\cR_{\beta\Lambda}(v-q) - \frac{1}{2}q(v-q)\cR_{\beta\Lambda}'(v-q)$ with this extension on the enlarged domain
\begin{align*}
S'&= \Bigg\{(\beta,v,q):\beta \in (0,1),\; v > q,\\
   &\hspace{1in}\text{there exists } \delta>0\text{ s.t. }
   \sum_{p \ge 1} \beta^{p+1}|\kappa_{p+1}|(v-q+\delta)^p < \alpha
   \Bigg\}.
\end{align*}
Importantly, this domain $S'$ is open, and $f$ is $C^{N+1}$ over $(\beta,v,q) \in S'$. Now fix any $\beta_0 \in (0,1)$. For any $\beta$ sufficiently close to $\beta_0$, we have
\begin{align*}
    &f(\beta,v_*(\beta),q_*(\beta)) - f(\beta_0,v_*(\beta_0),q_*(\beta_0))\\
    &= (f(\beta,v_*(\beta),q_*(\beta)) - f(\beta_0,v_*(\beta),q_*(\beta)) )+ (f(\beta_0,v_*(\beta),q_*(\beta)) - f(\beta_0,v_*(\beta_0),q_*(\beta_0)) )\\
    &=\partial_\beta f(\beta_0,v_*(\beta),q_*(\beta))(\beta - \beta_0) + o(|\beta - \beta_0|) +  (f(\beta_0,v_*(\beta),q_*(\beta)) - f(\beta_0,v_*(\beta_0),q_*(\beta_0)) ).
\end{align*} 
It then suffices to show $f(\beta_0,v_*(\beta),q_*(\beta)) - f(\beta_0,v_*(\beta_0),q_*(\beta_0)) = o(|\beta - \beta_0|)$.   Denote $\Delta v_{\beta} = v_*(\beta) - v_*(\beta_0)$ and $\Delta q_{\beta} = q_*(\beta) - q_*(\beta_0)$. By the H\"older continuity of $v_*$ and $q_*$, we have $|\Delta v_{\beta}| + |\Delta q_{\beta}| = O(|\beta - \beta_0|^{\rho})$.  By Taylor expansion of $f(\beta_0, \cdot, \cdot)$ at $(v_*(\beta_0), q_*(\beta_0))$, we have
\begin{align}
   \nonumber  f(\beta_0,v_*(\beta),q_*(\beta)) - f(\beta_0,v_*(\beta_0),q_*(\beta_0)) = P_N(\Delta v_{\beta}, \Delta q_{\beta}) + O(|\Delta v_{\beta}|^{N+1} + |\Delta q_{\beta}|^{N+1}),
\end{align}
where $P_N$ is a polynomial of degree at most $N$. From the condition \eqref{eq:vqderivvanish}, the first order term of $P_N$ vanishes, and also 
\begin{align}
&|\partial_v  f(\beta_0, v_*(\beta),q_*(\beta))|  \nonumber= |\partial_v  f(\beta_0, v_*(\beta),q_*(\beta))-\partial_v  f(\beta ,  v_*(\beta),q_*(\beta))| = O( |\beta - \beta_0|).   
\end{align}
Combining with the fact that 
$|\partial_v  P_N(\Delta v_{\beta}, \Delta q_{\beta})-\partial_v  f(\beta_0, v_*(\beta),q_*(\beta))| = O(|\Delta v_{\beta}|^{N} + |\Delta q_{\beta}|^{N})$,
we have
$|\partial_v  P_N(\Delta v_{\beta}, \Delta q_{\beta})| = O(|\beta - \beta_0| + |\Delta v_{\beta}|^{N} + |\Delta q_{\beta}|^{N})$.
Similarly, we have
$|\partial_q  P_N(\Delta v_{\beta}, \Delta q_{\beta})| = O(|\beta - \beta_0| + |\Delta v_{\beta}|^{N} + |\Delta q_{\beta}|^{N})$.
Then by the Bochnak--{\L}ojasiewicz inequality
\cite[Lemma~4.3]
{KurdykaMostowskiParusinski2000}, \begin{align*}
|P_N(\Delta v_{\beta}, \Delta q_{\beta})| &\leq C(|\Delta v_{\beta}| + |\Delta q_{\beta}|) \cdot (|\partial_v  P_N(\Delta v_{\beta}, \Delta q_{\beta})| + |\partial_q  P_N(\Delta v_{\beta}, \Delta q_{\beta})|)\\
&= O(|\beta - \beta_0|^{1+\rho}+ |\beta - \beta_0|^{(N+1)\rho}) = o(|\beta - \beta_0|),
\end{align*}
where the last equality uses $\rho \in (0,1)$ and $N$ large enough such that $(N+1)\rho > 1$.  This shows $f(\beta_0,v_*(\beta),q_*(\beta)) - f(\beta_0,v_*(\beta_0),q_*(\beta_0)) = o(|\beta - \beta_0|)$, and hence \eqref{eq:fbetacontdiff} holds immediately by the definition of the derivative.

Note that for $\beta=0$, we have $\cR_{\beta\Lambda}(x)=0$
and $\cR_{\beta\Lambda}'(x)=0$ for all $x \in \R$, so
$f(0,\allowbreak v_*(0),\allowbreak q_*(0))=\log \int e^{-U(\theta)}\d\theta$.
Applying this and \eqref{eq:fbetacontdiff} to \eqref{eq:integralfreeenergy}, and
using the continuous differentiability of $\beta \mapsto
f(\beta,\allowbreak v_*(\beta),\allowbreak q_*(\beta))$,
\[\lim_{n \to \infty} \frac{1}{n}\log \cZ
=f(0,v_*(0),q_*(0))
+\int_0^1 \frac{\d}{\d \beta} f(\beta,v_*(\beta),q_*(\beta))\d \beta
=f(1,v_*(1),q_*(1)).\]
The right side is the desired replica formula in Theorem
\ref{thm:replica}(b), concluding the proof.
\end{proof}

\begin{proof}[Proof of Corollary \ref{cor:hightemp}]
We let $\btheta^0$ have i.i.d.\ entries with (arbitrary) law supported on $[-1,1]$, and restrict to an event holding a.s.\ for all large $n$ where
\[\|\X\|_\op \leq \Lambda_{\max}.\]
Write
\[\d \nu_x(y)=\frac{e^{\alpha xy-\frac{\alpha}{2}y^2}\d \nu(y)}
{\int_{-M}^M e^{\alpha xy-\frac{\alpha}{2}y^2}\d \nu(y)}, \qquad
\langle f(y) \rangle_{\nu_x}=\int_{-M}^M f(y) \d \nu_x(y), \qquad
\Var_{\nu_x}[y]=\langle y^2 \rangle_{\nu_x}
-\langle y \rangle_{\nu_x}^2.\]
Then
\begin{align*}
U''(x)&=\frac{\d^2}{\d x^2}\left[-\log
\int_{-M}^M \sqrt{\frac{\alpha}{2\pi}}e^{-\frac{\alpha}{2}(x-y)^2}\d
\nu(y)\right]\\
&=\frac{\d^2}{\d x^2}\left[\frac{\alpha}{2}x^2-\log \int_{-M}^M
e^{\alpha xy-\frac{\alpha}{2}y^2} \d \nu(y)\right]
=\alpha-\alpha^2\Var_{\nu_x}[y].
\end{align*}
As $x \to \pm\infty$, the measure $\nu_x$ converges weakly to a point mass
at the maximum/minimum point of support of $\nu$. Then $\Var_{\nu_x}[y] \to 0$.
Thus $U(\cdot)$ satisfies
\eqref{eq:Uconvexity}. The same argument shows that the 3rd and 4th cumulants of $\nu_x$
converge to 0 as $x \to \pm \infty$, so also
$\lim_{|x| \to \pm \infty} U'''(x)=0$ and $\lim_{|x| \to \pm \infty} U^{(4)}(x)=0$. Hence $f={-}U'$ satisfies the conditions
of Assumption \ref{assump:dynamics}.

To check the condition \eqref{eq:alphacondition}, observe that we may represent
$\btheta \sim \prod_{i=1}^n e^{-U(\theta_i)}\d\theta_i$
as $\btheta=\y+\z$ where $\y,\z$ are independent vectors with
$y_i \overset{iid}{\sim} \nu$ and $\z \sim \cN(0,\alpha^{-1}\Id)$.
Thus
\begin{align*}
\cZ&=\E\left[\exp\left(\frac{1}{2}(\y+\z)^\top
\X(\y+\z)\right)\bigg|\X\right],\\
\frac{1}{n}\langle \btheta^\top \btheta'\rangle
&=\frac{1}{n\cZ^2}\left\|\E\left[(\y+\z)
\exp\left(\frac{1}{2}(\y+\z)^\top
\X(\y+\z)\right)\bigg|\X\right]\right\|_2^2,\\
\frac{1}{n}\langle \|\btheta\|_2^2\rangle
&=\frac{1}{n\cZ}\,\E\left[\|\y+\z\|_2^2
\exp\left(\frac{1}{2}(\y+\z)^\top
\X(\y+\z)\right)\bigg|\X\right]
\end{align*}
where expectations are over the joint law of $(\y,\z)$.
Evaluating the expectation over $\z$ for fixed $\X,\y$,
where $\|\X\|_\op \leq \Lambda_{\max}<\alpha$, an explicit calculation shows
\begin{align*}
\E\left[\exp\left(\frac{1}{2}(\y+\z)^\top
\X(\y+\z)\right)\bigg|\X,\y\right]
&=\underbrace{\frac{\exp(\frac{1}{2}\y^\top\X(\Id-\alpha^{-1}\X)^{-1}\y)}
{\det(\Id-\alpha^{-1}\X)^{1/2}}}_{:=\cZ(\X,\y)}
\end{align*}
and
\begin{align*}
\E\left[(\y+\z)\exp\left(\frac{1}{2}(\y+\z)^\top
\X(\y+\z)\right)\bigg|\X,\y\right]
&=\cZ(\X,\y)(\Id-\alpha^{-1}\X)^{-1}\y,\\
\E\left[\|\y+\z\|_2^2
\exp\left(\frac{1}{2}(\y+\z)^\top
\X(\y+\z)\right)\bigg|\X,\y\right]
&=\cZ(\X,\y)\left(\begin{aligned}
&\alpha^{-1}\Tr (\Id-\alpha^{-1}\X)^{-1}\\
&+\y^\top(\Id-\alpha^{-1}\X)^{-2}\y
\end{aligned}\right).
\end{align*}
Thus
\begin{equation}\label{eq:gaussianmixtureZ}
\begin{aligned}
\cZ&=\E[\cZ(\X,\y) \mid \X],\\
\frac{1}{n}\langle \btheta^\top \btheta' \rangle
&=\frac{1}{n}\,\frac{\|\E[\cZ(\X,\y)(\Id-\alpha^{-1}\X)^{-1}\y \mid \X]\|_2^2}
{\E[\cZ(\X,\y) \mid \X]^2},\\
\frac{1}{n}\langle \|\btheta\|_2^2 \rangle
&=\frac{1}{n}\,\frac{\E[\cZ(\X,\y)(\alpha^{-1}\Tr (\Id-\alpha^{-1}\X)^{-1}
+\y^\top(\Id-\alpha^{-1}\X)^{-2}\y) \mid \X]}{\E[\cZ(\X,\y) \mid \X]}.
\end{aligned}
\end{equation}
Since
\[\frac{1}{n}\left|\alpha^{-1}\Tr (\Id-\alpha^{-1}\X)^{-1}
+\y^\top(\Id-\alpha^{-1}\X)^{-2}\y\right|
\leq \frac{1}{\alpha-\Lambda_{\max}}+\frac{\alpha^2
M^2}{(\alpha-\Lambda_{\max})^2}\]
for all $\y$ in the support $[-M,M]^n$
of $\prod_{i=1}^n \d\nu(y_i)$, this implies
\[0 \leq \frac{1}{n}\langle \|\btheta\|_2^2\rangle
-\frac{1}{n}\langle \btheta^\top \btheta'\rangle
\leq \frac{1}{n}\langle \|\btheta\|_2^2\rangle
\leq \frac{1}{\alpha-\Lambda_{\max}}+\frac{\alpha^2
M^2}{(\alpha-\Lambda_{\max})^2}.\]
Then the given conditions of \eqref{eq:sufficientalphacondition} in the
corollary imply the required conditions \eqref{eq:alphacondition} for $\alpha$
in Theorem \ref{thm:replica}.

It remains to check Assumption \ref{assump:convergence}. We verify a log-Sobolev inequality for
$\mu_\Gibbs$ following the argument of \cite{bauerschmidt2019very}: Let
\[\omega=\Lambda_{\max}+\eps\]
and write
\[\mu_\Gibbs(\btheta) \propto
\exp\left({-}\frac{1}{2}\btheta^\top(\omega\,\Id-\X)\btheta
+\sum_{i=1}^n \left(\frac{\omega}{2}\theta_i^2-U(\theta_i)\right)\right)\]
Then $\eps\,\Id \preceq \omega\,\Id-\X \preceq (2\omega-\eps)\,\Id$
in the positive-definite ordering, so we have
\[(\omega\,\Id-\X)^{-1}=(2\omega\,\Id)^{-1}+\B^{-1}\]
for some $\B \succ 0$. Hence,
expressing the Gaussian measure with covariance $(\omega\,\Id-\X)^{-1}$
as a convolution of those with covariances $(2\omega\,\Id)^{-1}$ and $\B^{-1}$,
\[\exp\left({-}\frac{1}{2}\btheta^\top(\omega\,\Id-\X)\btheta\right)
\propto \int \exp\left({-}\frac{2\omega}{2}\|\btheta-\bphi\|_2^2
\right)\exp\left({-}\frac{1}{2}\bphi^\top \B\bphi\right)\d\bphi.\]
Thus
\begin{equation}\label{eq:measuredecomp}
\begin{aligned}
\mu_\Gibbs(\btheta) &\propto
\int \d\bphi \exp\underbrace{\left({-}\frac{1}{2}\bphi^\top \B\bphi
+\log \sum_{i=1}^n
\int
\exp\left(\frac{\omega}{2}\theta_i^2-U(\theta_i)-\omega(\theta_i-\varphi_i)^2\right)\d\theta\right)}_{:={-}V(\bphi)}\prod_{i=1}^n\mu_{\varphi_i}(\theta_i)
\end{aligned}
\end{equation}
where
\[\mu_\varphi(\theta)
=\frac{\exp((\omega/2)\theta^2-U(\theta)-\omega(\theta-\varphi)^2)}
{\int \exp((\omega/2)\theta^2-U(\theta)-\omega(\theta-\varphi)^2)\d\theta}.\]
Under the condition \eqref{eq:alphacondition} checked above, each univariate law
$\mu_{\varphi_i}$ satisfies a log-Sobolev inequality on $\R$ with
log-Sobolev constant depending only on $\omega$ and $U(\cdot)$ and not on
$\varphi_i$, by the same argument as \eqref{eq:logsobolevdecomp}. Furthermore,
writing
\[\langle f(\theta) \rangle_{\varphi}
=\int f(\theta)\mu_\varphi(\theta)\d\theta,\quad
\Var_{\varphi}[\theta]=\langle \theta^2 \rangle_{\varphi}-\langle \theta
\rangle_{\varphi}^2,\]
we have
\begin{equation}\label{eq:LSIderivativebound}
\frac{\d^2}{\d\varphi^2}
\left[{-}\log \int
\exp\Big(\frac{\omega}{2}\theta^2-U(\theta)-\omega(\theta-\varphi)^2\Big)\d\theta\right]
=2\omega-4\omega^2\Var_{\varphi}[\theta].
\end{equation}
Recalling that $e^{-U(\theta)}$ is the law of
$y+z$ where $y \sim \nu$ and $z \sim \cN(0,\alpha^{-1})$, we have furthermore
\begin{align*}
\mu_\varphi(\theta)
&\propto \int_{-M}^M
\exp\left(\frac{\omega}{2}\theta^2-\omega(\theta-\varphi)^2\right)
\exp\left({-}\frac{\alpha}{2}(\theta-y)^2\right)\nu(\d y)\\
&\propto \int_{-M}^M
\exp\Big({-}\frac{\alpha+\omega}{2}\theta^2
+(\alpha y+2\omega \varphi)\theta-\frac{\alpha}{2}y^2\Big)\nu(\d y)\\
&\propto \int_{-M}^M \exp\Big({-}\frac{\alpha+\omega}{2}
\Big(\theta-\frac{\alpha y+2\omega \varphi}{\alpha+\omega}\Big)^2\Big)
\tilde \nu_\varphi(\d y),
\end{align*}
for some probability measure $\tilde \nu_\varphi$ supported on $[-M,M]$ and
depending on $(\alpha,\omega,\varphi)$. Thus $\mu_\varphi(\theta)$
is the law of 
\[\theta=\frac{2\omega\varphi}{\alpha+\omega}
+\frac{\alpha}{\alpha+\omega}y+z,
\qquad y \sim \tilde \nu_\varphi,
\qquad z \sim \cN(0,(\alpha+\omega)^{-1}),\]
and $y,z$ are independent. Then
\[\Var_\varphi[\theta]
=\frac{\alpha^2}{(\alpha+\omega)^2}\Var_{\tilde \nu}[y]
+\frac{1}{\alpha+\omega}
\leq \frac{\alpha^2M^2}{(\alpha+\omega)^2}+\frac{1}{\alpha+\omega}
\leq M^2+\frac{1}{\alpha+\omega}.\]
Recalling $\omega=\Lambda_{\max}+\eps$ and choosing $\eps$ small enough, the
second given condition of \eqref{eq:sufficientalphacondition} ensures
\[2\omega-4\omega^2\Var_\varphi[\theta] \geq \eps.\]
Applying this to \eqref{eq:LSIderivativebound} shows that
$V(\bphi)$ defined in \eqref{eq:measuredecomp} is a
uniformly strongly convex potential, so $e^{-V(\bphi)}$
satisfies a log-Sobolev inequality on $\R^n$ with log-Sobolev constant depending
only on $\eps$. Then the argument of \cite{bauerschmidt2019very} verifies that
$\mu_\Gibbs$ also satisfies a log-Sobolev inequality on $\R^n$, with
log-Sobolev constant depending only on $\omega,\eps,U(\cdot)$.
This implies by a $T_2$-transportation inequality, contraction of relative
entropy, and log-Harnack/entropy-cost inequality (as in the proof of Theorem \ref{thm:replica}) that for any
deterministic vector $\btheta^0 \in \R^n$ and $t \geq 1$,
\begin{align*}
W_2(\Law(\btheta^t \mid \btheta^0),\,\mu_\Gibbs)^2
&\leq C\DKL(\Law(\btheta^t \mid \btheta^0)\|\,\mu_\Gibbs)\\
&\leq Ce^{-c(t-1)}\DKL(\Law(\btheta^1 \mid \btheta^0)\|\,\mu_\Gibbs)\\
&\leq C'e^{-ct}W_2(\delta_{\btheta^0},\,\mu_\Gibbs)^2\\
&\leq 2C'e^{-ct}(\|\btheta^0\|_2^2+\langle \|\btheta\|_2^2 \rangle).
\end{align*}

Recall from \eqref{eq:thetanormbound} that $\frac{1}{n}\langle\|\btheta\|_2^2\rangle \leq C$, where this requires only the assumptions $\|\X\|_\op \leq \Lambda_{\max}<\alpha$. Then,
applying the above bound with the specified initial condition $\btheta^0$ having i.i.d.\ entries on $[-1,1]$, this implies for all $t \geq 1$ that
\[\frac{1}{n}\langle \|\btheta^t\|_2^2 \rangle \leq C.\]
This holds also for $t \in [0,1]$, by It\^o's lemma which shows $\frac{\d}{\d t}\frac{1}{n}\langle \|\btheta^t\|_2^2 \rangle \leq C(\frac{1}{n}\langle \|\btheta^t\|_2^2 \rangle +1)$ and Gr\"onwall's inequality. Then, applying the above bound with
$\btheta^s$ in place of $\btheta^0$
shows the first statement of Assumption \ref{assump:convergence} for $t \geq s+1$.
The bound for $t \in [s,s+1]$ again follows from It\^o's lemma for $\frac{\d}{\d t}\frac{1}{n}\langle \|\btheta^t-\btheta^s\|_2^2 \rangle$ and Gr\"onwall's inequality. The second statement of
Assumption \ref{assump:convergence} follows similarly by applying the above with $\tilde\btheta^0 \sim \mu_\Gibbs$ in place of $\btheta^0$.
This checks all needed conditions for Theorem \ref{thm:replica}(a).

Note that if \eqref{eq:sufficientalphacondition}
holds, then it also holds for $(\beta\Lambda_{\max},\beta^{p+1}\kappa_{p+1})$ in
place of $(\Lambda_{\max},\kappa_{p+1})$ for every $\beta \in [0,1]$, because
the left side of each condition in \eqref{eq:sufficientalphacondition}
for $(\beta\Lambda_{\max},\beta^{p+1}\kappa_{p+1})$ is increasing
in $\beta$. Thus, all of the above statements hold also for $\beta\X$ in place
of $\X$, verifying the needed condition for Theorem \ref{thm:replica}(b).
\end{proof}

\begin{proof}[Proof of Corollary \ref{cor:ising}]
Define the discrete law on $\{\pm 1\}$
\[\nu=\frac{e^h}{e^{-h}+e^h}\,\delta_{+1}+\frac{e^{-h}}{e^{-h}+e^h}\,\delta_{-1},
\qquad M=1.\]
Since $\Lambda_{\max}<1/2$
and $\limsup_{p \to \infty} |\kappa_p|^{1/p}<1$ by Proposition
\ref{prop:freecumulants},
the condition \eqref{eq:sufficientalphacondition} holds for all sufficiently
large $\alpha$. Thus, for all large $\alpha$,
Corollary \ref{cor:hightemp} holds with the potential
\[e^{-U_\alpha(\theta)}=\frac{e^h}{e^{-h}+e^h}
\sqrt{\frac{\alpha}{2\pi}}e^{-\frac{\alpha}{2}(\theta-1)^2}
+\frac{e^{-h}}{e^{-h}+e^h}
\sqrt{\frac{\alpha}{2\pi}}e^{-\frac{\alpha}{2}(\theta+1)^2}\]
Let us define $p_\alpha(\d\theta)=e^{-U_\alpha(\theta)}\d\theta$ for
$\alpha<\infty$, and $p_\alpha(\d\theta)=\nu(\d\theta)$ for $\alpha=\infty$.
We denote also
\[\mu_\alpha(\d\btheta)=\frac{1}{\cZ_\alpha}
\exp\left(\frac{1}{2}\btheta^\top \X\btheta\right)
\prod_{i=1}^n p_\alpha(\d\theta_i),
\qquad \cZ_\alpha= \int \exp\left(\frac{1}{2}\btheta^\top \X\btheta\right)
\prod_{i=1}^n p_\alpha(\d\theta_i)\]
and write $\langle \cdot \rangle_\alpha$ for the associated Gibbs average.
Note then that $\mu_\infty \equiv \mu_{\text{Ising}}$ is precisely the Ising
Gibbs measure \eqref{eq:ising}, and
\begin{equation}\label{eq:ZinftyIsing}
\cZ_\infty=\cZ_{\text{Ising}} \prod_{i=1}^n \frac{1}{e^{-h}+e^h}.
\end{equation}

Let $v_*(\alpha),q_*(\alpha)$ be the values corresponding to $\mu_\alpha$,
for all large enough $\alpha<\infty$ such that
Corollary \ref{cor:hightemp} holds.
Recall the explicit forms \eqref{eq:gaussianmixtureZ}
for $\cZ_\alpha$, $\frac{1}{n}\langle \btheta^\top \btheta' \rangle_\alpha$,
and $\frac{1}{n}\langle \|\btheta\|_2^2 \rangle_\alpha$, where the expectation
$\E$ is over $\y \in \{\pm 1\}^n$ with entries $y_i \overset{iid}{\sim} \nu$.
Write $\cZ_\alpha(\X,\y)
\equiv \cZ(\X,\y)$ to now make explicit the dependence on $\alpha$. Note that
for any $\y \in \{\pm 1\}^n$, $\X \in \R^{n \times n}$
with $\|\X\|_\op<\Lambda_{\max}$, and
$\alpha>\Lambda_{\max}$, we have
\[\frac{1}{n}\left|\alpha^{-1}\Tr (\Id-\alpha^{-1}\X)^{-1}
+\y^\top(\Id-\alpha^{-1}\X)^{-2}\y-1\right|
\leq \frac{1}{\alpha-\Lambda_{\max}}+\frac{\alpha^2}{(\alpha-\Lambda_{\max})^2}-1.\]
Then the third identity of \eqref{eq:gaussianmixtureZ} implies 
\[|v_*(\alpha)-1| \leq \limsup_{n \to \infty}
\left|\frac{1}{n}\langle \|\btheta\|_2^2 \rangle_\alpha-1\right| \leq
\frac{1}{\alpha-\Lambda_{\max}}+\frac{\alpha^2}{(\alpha-\Lambda_{\max})^2}-1,\]
so
\begin{equation}\label{eq:qstardef}
1=\lim_{\alpha \to \infty} v_*(\alpha)
\end{equation}

For $q_*(\alpha)$, observe first
that
\[\begin{aligned}
\frac{1}{\sqrt{n}}\|(\Id-\alpha^{-1}\X)^{-1}\y-\y\|_2
\leq \frac{\Lambda_{\max}}{\alpha-\Lambda_{\max}},
\qquad
\frac{1}{\sqrt{n}}\|(\Id-\alpha^{-1}\X)^{-1}\y+\y\|_2
\leq 2+\frac{\Lambda_{\max}}{\alpha-\Lambda_{\max}}.
\end{aligned}\]
Hence, by the second identity of \eqref{eq:gaussianmixtureZ} and Cauchy-Schwarz,
\begin{align}
&\left|\frac{1}{n}\langle \btheta^\top\btheta'\rangle_\alpha
-\frac{1}{n}\frac{\|\E[\cZ_\alpha(\X,\y)\y \mid \X]\|_2^2}
{\E[\cZ_\alpha(\X,\y) \mid \X]^2}\right|\notag\\
&\leq \frac{1}{n}\frac{\|\E[\cZ_\alpha(\X,\y)((\Id-\alpha^{-1}\X)^{-1}\y-\y) \mid
\X]\|_2 \cdot \|\E[\cZ_\alpha(\X,\y)((\Id-\alpha^{-1}\X)^{-1}\y+\y) \mid
\X]\|_2}{\E[\cZ_\alpha(\X,\y) \mid \X]^2}\notag\\
&\leq \frac{\Lambda_{\max}}{\alpha-\Lambda_{\max}}
\left(2+\frac{\Lambda_{\max}}{\alpha-\Lambda_{\max}}\right)\label{eq:overlapcompare}
\end{align}
Momentarily denoting
\[[f(\y)]_\omega
=\frac{\E[\cZ_{\omega^{-1}}(\X,\y)f(\y) \mid \X]}{\E[\cZ_{\omega^{-1}}(\X,\y) \mid \X]}
=\frac{\E[\exp(\frac{1}{2}\y^\top\X(\Id-\omega \X)^{-1}\y) \cdot f(\y) \mid
\X]}{\E[\exp(\frac{1}{2}\y^\top\X(\Id-\omega \X)^{-1}\y) \mid \X]}\]
and writing $\Var_{[\cdot]_\omega}$ and
$\Cov_{[\cdot]_\omega}$ for the associated variance and covariance, observe that
\[\begin{aligned}
\frac{\d}{\d\omega}[y_i]_\omega
&=\Cov_{[\cdot]_\omega}\left[y_i,\,\frac{\d}{\d\omega}\frac{1}{2}\y^\top\X(\Id-\omega\X)^{-1}\y\right]\\
&=\Cov_{[\cdot]_\omega}\left[y_i,\,\frac{1}{2}\y^\top\X^2(\Id-\omega\X)^{-2}\y\right].
\end{aligned}\]
Then
\[\begin{aligned}
\left\|\frac{\d}{\d \omega}[\y]_\omega\right\|_2
&=\sup_{\u:\|\u\|_2=1} \Cov_{[\cdot]_\omega}\left[\u^\top\y,
\,\frac{1}{2}\y^\top\X^2(\Id-\omega\X)^{-2}\y\right]
\\
&\leq C\sqrt{n}\sup_{\u:\|\u\|_2=1} \Var_{[\cdot]_\omega}[\u^\top\y]^{1/2}
\end{aligned}\]
for a constant $C>0$ and all small $\omega$. Since $\|\X\|_\op \leq
\Lambda_{\max}<1/2$, also
$\|\X(\Id-\omega \X)^{-1}\|_\op<c_0<1/2$ for some constant
$c_0 \in (\Lambda_{\max},1/2)$ and all small enough $\omega$. Then by the result
of \cite{bauerschmidt2019very}, the Gibbs measure on $\{\pm 1\}^n$ defining
$[\cdot]_\omega$ (having Hamiltonian $\frac{1}{2}\y^\top \X(\Id-\omega
\X)^{-1}\y+h\sum_{i=1}^n y_i$ with respect to the uniform measure over
$\y \in \{\pm 1\}^n$)
satisfies a Poincar\'e inequality with a dimension-free
Poincar\'e constant depending only on $c_0,h$. Thus,
$\Var_{[\cdot]_\omega}[\u^\top\y] \leq C\|\u\|_2^2 \leq C$ for a constant $C>0$, any unit vector $\u$,
and all small $\omega$. Applying this above shows that for all
small enough $\omega$, $\|[\y]_\omega-[\y]_0\|_2 \leq C'\omega\sqrt{n}$.
Thus, there is a constant $C>0$ such that for all large enough
$\alpha<\infty$,
\[\begin{aligned}
&\left|\frac{1}{n}\langle \btheta^\top \btheta' \rangle_\infty
-\frac{1}{n}\frac{\|\E[\cZ_\alpha(\X,\y)\y \mid \X]\|_2^2}
{\E[\cZ_\alpha(\X,\y) \mid \X]^2}\right|
\\
&=\left|\frac{1}{n}\frac{\|\E[\cZ_\infty(\X,\y)\y \mid \X]\|_2^2}
{\E[\cZ_\infty(\X,\y) \mid \X]^2}
-\frac{1}{n}\frac{\|\E[\cZ_\alpha(\X,\y)\y \mid \X]\|_2^2}
{\E[\cZ_\alpha(\X,\y) \mid \X]^2}\right|
\leq \frac{C}{\alpha}.
\end{aligned}\]
Combining this with \eqref{eq:overlapcompare} shows
\[\lim_{\alpha \to \infty}
\sup_{n \geq 1}\left|\frac{1}{n}\langle \btheta^\top \btheta' \rangle_\alpha
-\frac{1}{n}\langle \btheta^\top \btheta' \rangle_\infty\right|=0.\]
Then, since the limit
$q_*(\alpha)=\lim_{n \to \infty} \frac{1}{n}\langle \btheta^\top
\btheta' \rangle_\alpha$ exists a.s.\ for all large $\alpha$,
this implies that there exists a limit
\begin{equation}\label{eq:vstardef-alpha}
q_*:=\lim_{\alpha \to \infty} q_*(\alpha)
\end{equation}
for which also
\[q_*=\lim_{n \to \infty} \frac{1}{n}\langle \btheta^\top \btheta'
\rangle_\infty
=\lim_{n \to \infty}
\frac{1}{n}\langle \btheta^\top \btheta'
\rangle_\text{Ising}\text{ a.s.}\]

For all large $\alpha$, by Theorem \ref{thm:replica},
$v_*(\alpha),q_*(\alpha)$ satisfy the fixed point equation
\begin{align*}
q_*(\alpha)&=\E_{\xi \sim \cN(0,1)}
\left(\frac{\int \theta
\exp(\frac{1}{2}\cR_\Lambda(v_*(\alpha)-q_*(\alpha))\theta^2
+\sqrt{q_*(\alpha)\cR_\Lambda'(v_*(\alpha)-q_*(\alpha))}\xi\theta)p_\alpha(\d\theta)}
{\int
\exp(\frac{1}{2}\cR_\Lambda(v_*(\alpha)-q_*(\alpha))\theta^2
+\sqrt{q_*(\alpha)\cR_\Lambda'(v_*(\alpha)-q_*(\alpha))}\xi\theta)p_\alpha(\d\theta)}\right)^2.
\end{align*}
Then taking $\alpha \to \infty$ and applying
 \eqref{eq:qstardef}, \eqref{eq:vstardef-alpha}, and the weak convergence of 
$p_\alpha$ to $p_\infty=\nu$,
\begin{align*}
q_*&=\E_{\xi \sim \cN(0,1)}
\left(\frac{\int \theta
\exp(\frac{1}{2}\cR_\Lambda(1-q_*)\theta^2+\sqrt{q_*\cR_\Lambda'(1-q_*)}\xi\theta)
\nu(\d\theta)}{\int
\exp(\frac{1}{2}\cR_\Lambda(1-q_*)\theta^2
+\sqrt{q_*\cR_\Lambda'(1-q_*)}\xi\theta)\nu(\d\theta)}\right)^2\\
&=\E_{\xi \sim \cN(0,1)}
\tanh\Big(h+\sqrt{q_*\cR_\Lambda'(1-q_*)}\xi\Big)^2,
\end{align*}
the second equality following from definition of $\nu(\d\theta)$ which is
supported on $\{\pm 1\}$.

We have also
\begin{align*}
|\log \cZ_\alpha(\X,\y)-\log \cZ_\infty(\X,\y)|
&\leq \left|\frac{1}{2}\y^\top \X(\Id-\alpha^{-1}\X)^{-1}\y
-\frac{1}{2}\y^\top \X\y\right|
+\left|\frac{1}{2}\log \det(\Id-\alpha^{-1}\X)\right|\\
&\leq \frac{n}{2} \cdot \frac{\Lambda_{\max}^2}{\alpha-\Lambda_{\max}}
+\frac{n}{2}\left|\log\left(1-\frac{\Lambda_{\max}}{\alpha}\right)\right|
\end{align*}
Thus, applying the first identity of \eqref{eq:gaussianmixtureZ} and
taking the expectation over $\y$,
\[\lim_{\alpha \to \infty} \sup_{n \geq 1}
\left|\frac{1}{n}\log \cZ_\alpha-\frac{1}{n}\log \cZ_\infty\right|=0.\]
Since $f(\alpha):=\lim_{n \to \infty} \frac{1}{n}\log \cZ_\alpha$ exists
a.s.\ for all large $\alpha$, this implies that almost surely,
\begin{align*}
\lim_{n \to \infty} \frac{1}{n}\log \cZ_\infty
&=\lim_{\alpha \to \infty} f(\alpha)\\
&=\E_{\xi \sim \cN(0,1)}
\log \int \exp\left(\frac{1}{2}\cR_\Lambda(1-q_*)\theta^2
+\sqrt{q_*\cR_\Lambda'(1-q_*)}\xi\theta\right)\nu(\d\theta)\\
&\hspace{0.5in}
+\frac{1}{2}\int_0^{1-q_*} \cR_\Lambda(s)\d s
-\frac{1}{2}(1-q_*)\cR_\Lambda(1-q_*)
-\frac{1}{2}q_*(1-q_*)\cR_\Lambda'(1-q_*)\\
&=\E_{\xi \sim \cN(0,1)} \log 2\cosh
\left(h+\sqrt{q_*\cR_\Lambda'(1-q_*)}\xi\right)
-\log 2\cosh(h)\\
&\hspace{0.5in}+\frac{1}{2}\int_0^{1-q_*} \cR_\Lambda(s)\d s
+\frac{1}{2}q_*\cR_\Lambda(1-q_*)\\
&\hspace{0.5in}-\frac{1}{2}q_*(1-q_*)\cR_\Lambda'(1-q_*).
\end{align*}
From \eqref{eq:ZinftyIsing} we have
$\frac{1}{n}\log \cZ_\infty=\frac{1}{n}\log \cZ_\text{Ising}
-\log 2\cosh(h)$, proving all statements of the corollary.
\end{proof}

\section{Universality over the disorder}\label{sec:universality}

We prove Corollary \ref{cor:universality} on universality over the
disorder $\X$, using the results of \cite{wang2024universality}.

\begin{proof}[Proof of Corollary \ref{cor:universality}]
Let $\tilde \X \in \R^{n \times n}$ satisfy Assumption \ref{assump:X-universal}.
Let $\X \in \R^{n \times n}$ be an orthogonally-invariant matrix
satisfying Assumption \ref{assump:X} with the same bound $\Lambda_{\max}$ and
asymptotic eigenvalue distribution $\Lambda$ as $\tilde\X$.
Let $\S \in \R^{n \times n}$ be
a random diagonal sign matrix as in Corollary \ref{cor:universality}(d),
and let $\bPi \in \R^{n \times n}$ be a uniformly random permutation matrix,
where $(\S,\bPi)$ are independent of each other and of $\btheta^0,\{\b^t\}_{t
\geq 0}$. Then \cite[Proposition 2.7, Theorem 2.8]{wang2024universality}
implies that the AMP algorithm \eqref{eq:discreteAMP} applied with $\tilde\X$ has the
same dynamical mean-field limit \eqref{eq:AMPconvergence} (in the sense of almost-sure
Wasserstein-2 convergence of the coordinatewise empirical distribution of its
iterates) as applied with the matrix
\[\W=\bPi\S\tilde\X\S\bPi^\top.\]
As \eqref{eq:AMPconvergence} and the bound $\|\X\|_\op \leq \Lambda_{\max}$ are the only
properties of $\X$ used in all preceding proofs, this implies that
Theorems \ref{thm:dmft-approx}, \ref{thm:replica} and Corollaries
\ref{cor:hightemp}, \ref{cor:ising} hold equally with $\W$ in place of
$\X$.

The dynamics \eqref{eq:dynamics} driven by $\W$ have the form
\[\d\btheta^t=[f(\btheta^t)+\bPi\S\tilde\X\S\bPi^\top\btheta^t]\d t
+\sqrt{2\gamma}\,\d\b^t.\]
Define $\g^t=\W\btheta^t-\int_0^t r_g(t-s)\btheta^s \d s$ as in Theorem
\ref{thm:dmft-approx}, and let
\[\tilde\btheta^t=\bPi^\top \btheta^t,
\qquad \tilde\g^t=\bPi^\top \g^t,
\qquad \tilde\b^t=\bPi^\top \b^t.\]
The conclusion of Theorem \ref{thm:dmft-approx} 
for $\{(\btheta^t,\g^t,\b^t)\}_{t \geq 0}$ implies the same convergence for
$\{(\tilde\btheta^t,\tilde\g^t,\tilde\b^t)\}_{t \geq 0}$, since this has the
same empirical distribution of sample paths.
Multiplying the above dynamics by $\bPi^\top$ and noting that
$\bPi^\top f(\btheta^t)=f(\tilde\btheta^t)$ since $f:\R \to \R$ is applied
entrywise, the dynamics of $\{(\tilde\btheta^t,\tilde\g^t,\tilde\b^t)\}_{t \geq
0}$ are given by
\begin{equation}\label{eq:permutedprocess}
\d\tilde\btheta^t=[f(\tilde\btheta^t)+\S\tilde\X\S\tilde\btheta^t]\d t
+\sqrt{2\gamma}\,\d\tilde \b^t,
\qquad \tilde \g^t=\S\tilde\X\S\tilde\btheta^t-\int_0^t
r_g(t-s)\tilde\btheta^s\d s,
\end{equation}
coinciding with the dynamics \eqref{eq:dynamics} driven by $\S\tilde\X\S$.
This shows Corollary \ref{cor:universality}(d) for the dynamics.
Since the Gibbs measures $\mu_\Gibbs$ and $\mu_{\text{Ising}}$ defined by
$\W$ are also identical to those defined by $\S\tilde\X\S$,
Corollary \ref{cor:universality}(d) holds also for the Gibbs measures.

If $f(\cdot)$ is odd and $\theta^0$ is sign-invariant in law, then 
the joint law of $\{(\theta^t,g^t,b^t)\}_{t \geq 0}$ in the dynamical
mean-field limit of Definition \ref{def:DMFT} is also invariant under
multiplication by a single sign $s \in \{\pm 1\}$. Setting
\[\tilde\btheta^t=\S\bPi^\top \btheta^t,
\qquad \tilde\g^t=\S\bPi^\top \g^t,
\qquad \tilde\b^t=\S\bPi^\top \b^t,\]
the conclusion of Theorem \ref{thm:dmft-approx} 
for $\{(\btheta^t,\g^t,\b^t)\}_{t \geq 0}$ implies that the empirical
distribution of sample paths of
$\{(\tilde\btheta^t,\tilde\g^t,\tilde\b^t)\}_{t \geq 0}$ converges to
the law of $\{(s\theta^t,sg^t,sb^t)\}_{t \geq 0}$ where $s \in \{\pm 1\}$ is
an independent and uniformly random sign. This is the same law as
that of $\{(\theta^t,g^t,b^t)\}_{t \geq 0}$, so Theorem \ref{thm:dmft-approx} 
holds also for $\{(\tilde\btheta^t,\tilde\g^t,\tilde\b^t)\}_{t \geq 0}$.
Since $\S\bPi^\top f(\btheta^t)=f(\S\bPi^\top\btheta^t)=f(\tilde\btheta^t)$ when
$f(\cdot)$ is odd, we have
\[\d\tilde\btheta^t=[f(\tilde\btheta^t)+\tilde\X\tilde\btheta^t]\d t
+\sqrt{2\gamma}\,\d\tilde \b^t,
\qquad \tilde \g^t=\tilde\X\tilde\btheta^t-\int_0^t
r_g(t-s)\tilde\btheta^s\d s\]
coinciding with the dynamics \eqref{eq:dynamics} driven by $\tilde\X$.
This shows
Corollary \ref{cor:universality}(a). If $U(\cdot)$ is even and $h=0$ in the
Ising model, then the Gibbs measures $\mu_\Gibbs$ and $\mu_\text{Ising}$ defined
by $\W$ are also identical to those defined by $\tilde \X$, implying
Corollary \ref{cor:universality}(b--c).
\end{proof}

\appendix

\section{Free cumulant bound}\label{app:free-cumulant-bound}

\begin{proposition}\label{prop:freecumulants}
Suppose $\Lambda$ is a random variable with $|\Lambda| \leq K$
for a constant $K>0$. Let $\{\kappa_p\}_{p \geq 1}$ be the free
cumulants of $\Lambda$. Then
\[\limsup_{p \to \infty} |\kappa_p|^{1/p} \leq 2K.\]
\end{proposition}
\begin{proof}
By rescaling $\Lambda$, assume without loss of generality $K=1$.
For $z \in \CC$ outside the support of $\Lambda$, define the Cauchy transform
$\cG_\Lambda(z)=\E[1/(z-\Lambda)]$. It is shown in \cite[Theorem
17]{mingo2017free} that $\cG_\Lambda$ is univalent over $\{z:|z|>4\}$, that
$\{w:|w| \in (0,1/6)\} \subset \cG_\Lambda(\{z:|z|>4\})$, and that
$\cR_\Lambda(w)$ defined by \eqref{eq:Rtransform} is absolutely
convergent for $|w|<1/6$ and given by
\begin{equation}\label{eq:cauchyR}
\cR_\Lambda(w)=\cG_\Lambda^{-1}(w)-\frac{1}{w} \text{ for } |w| \in (0,1/6)
\end{equation}
where $\cG_\Lambda^{-1}(w)$ denotes the unique inverse in $\{z:|z|>4\}$.

We check that $\cR_\Lambda$ extends analytically to $|w|<1/2$:
Define, on the open disk
$\DD=\{z:|z|<1\}$, the analytic function
\[H(z)=\cG_\Lambda(1/z)=\E[z/(1-z\Lambda)].\]
For any $z \in \DD$ and $|\Lambda| \leq 1$, note that
$|1-z\Lambda|^2 \leq 1+\Lambda^2-2\Re[z]\Lambda \leq (2+\eps)(1-\Re[z]\Lambda)$.
Then $\Re[1/(1-z\Lambda)]=(1-\Re[z]\Lambda)/|1-z\Lambda|^2 \geq 1/(2+\eps)$.
This implies $\E[1/(1-z\Lambda)] \neq 0$,
so $z=0$ is the unique simple root of $H(z)$ on $\DD$. Furthermore,
\[|H(z)| \geq |z| \cdot \Re\,\E[1/(1-z\Lambda)] \geq \frac{1-\eps}{2+\eps}
\text{ for } |z|=1-\eps.\]
Then by Rouch\'e's theorem, for each $w \in \CC$ with $|w|<(1-\eps)/(2+\eps)$,
$H(z)=w$ also has a unique simple root $z$ in the disk $\{z:|z|<1-\eps\}$.
Since $H'(z) \neq 0$ at this root,
the inverse function theorem implies that there exists an analytic
inverse $H^{-1}$ on $\{w:|w|<(1-\eps)/(2+\eps)\}$.
Since $H^{-1}(0)=0$, we have $|H^{-1}(w)|<1/4$ on some neighborhood $\cO$
of $w=0$, and $\cG_\Lambda(1/H^{-1}(w))=H(H^{-1}(w))=w$. Then the
above univalence implies $1/H^{-1}(w)=\cG_\Lambda^{-1}(w)$ for all $w \in \cO$.
Then the right side of \eqref{eq:cauchyR} has an analytic extension
to $\{w:|w| \in (0,(1-\eps)/(2+\eps))\}$. Since $\eps>0$ is arbitrary, this
implies $\cR_\Lambda(w)$ extends analytically to $|w|<1/2$, as claimed. Then
$\cR_\Lambda(w)$ is given by the convergent series \eqref{eq:Rtransform}
for all $|w|<1/2$, so the root test implies that $\limsup_{p \to \infty}
|\kappa_p|^{1/p} \leq 2$.
\end{proof}

\subsection*{Acknowledgments}
Z.F.\ would like to thank Rishabh Dudeja for helpful conversations about the dynamical approach to free energy universality. This research was supported in part by NSF DMS-2142476 and a Sloan Research Fellowship.

\subsection*{Statement on AI use}

Several supporting proof ideas were obtained in dialogue with ChatGPT (GPT-5.5) over the course of this research: the Gaussian coupling construction of Corollary \ref{cor:semigroupapprox}, the Gaussian mixture interpolation of Corollary \ref{cor:ising}, the extension from H\"older-continuity to continuous-differentiability concluding the proof of Theorem
\ref{thm:replica}(b), and the free cumulant bound of Proposition \ref{prop:freecumulants}. The authors assume responsibility for the correctness of all stated results.

\bibliographystyle{plain}
\bibliography{main}

@book{vershynin2018high,
  title={High-dimensional probability: An introduction with applications in data science},
  author={Vershynin, Roman},
  volume={47},
  year={2018},
  publisher={Cambridge university press}
}

@article{wang2017stochastic,
  title={Stochastic Volterra integral equations with a parameter},
  author={Wang, Yanqing},
  journal={Advances in Difference Equations},
  volume={2017},
  number={1},
  pages={333},
  year={2017},
  publisher={Springer}
}

@book{engel2000one,
  title={One-parameter semigroups for linear evolution equations},
  author={Engel, Klaus-Jochen and Nagel, Rainer},
  year={2000},
  publisher={Springer}
}

@article{fan2025dynamicalI,
  title={Dynamical mean-field analysis of adaptive {L}angevin diffusions: {P}ropagation-of-chaos and convergence of the linear response},
  author={Fan, Zhou and Ko, Justin and Loureiro, Bruno and Lu, Yue M and Shen, Yandi},
  journal={arXiv preprint arXiv:2504.15556},
  year={2025}
}

@article{fan2025dynamicalII,
  title={Dynamical mean-field analysis of adaptive {L}angevin diffusions: {R}eplica-symmetric fixed point and empirical {B}ayes},
  author={Fan, Zhou and Ko, Justin and Loureiro, Bruno and Lu, Yue M and Shen, Yandi},
  journal={arXiv preprint arXiv:2504.15558},
  year={2025}
}

@article{eberle2019sticky,
  title={Sticky couplings of multidimensional diffusions with different drifts},
  author={Eberle, Andreas and Zimmer, Raphael},
  journal={Annales de l'Institut Henri Poincar{\'e}-Probabilit{\'e}s et Statistiques},
  volume={55},
  number={4},
  pages={2370--2394},
  year={2019}
}

@book{bakry2014analysis,
  title={Analysis and geometry of {M}arkov diffusion operators},
  author={Bakry, Dominique and Gentil, Ivan and Ledoux, Michel},
  volume={103},
  year={2014},
  publisher={Springer}
}

@incollection{kunita2006stochastic,
  title={Stochastic differential equations and stochastic flows of diffeomorphisms},
  author={Kunita, Hiroshi},
  booktitle={Ecole d'{\'e}t{\'e} de probabilit{\'e}s de Saint-Flour XII-1982},
  pages={143--303},
  year={2006},
  publisher={Springer}
}

@article{fan2022approximate,
  title={Approximate message passing algorithms for rotationally invariant matrices},
  author={Fan, Zhou},
  journal={The Annals of Statistics},
  volume={50},
  number={1},
  pages={197--224},
  year={2022},
  publisher={JSTOR}
}

@article{wang2024universality,
  title={Universality of approximate message passing algorithms and tensor networks},
  author={Wang, Tianhao and Zhong, Xinyi and Fan, Zhou},
  journal={The Annals of Applied Probability},
  volume={34},
  number={4},
  pages={3943--3994},
  year={2024},
  publisher={Institute of Mathematical Statistics}
}

@article{dudeja2023universality,
  title={Universality of approximate message passing with semirandom matrices},
  author={Dudeja, Rishabh and M. Lu, Yue and Sen, Subhabrata},
  journal={The Annals of Probability},
  volume={51},
  number={5},
  pages={1616--1683},
  year={2023},
  publisher={Institute of Mathematical Statistics}
}

@article{gorini2026universality,
  title={Universality of first-order methods on random and deterministic matrices},
  author={Gorini, Nicola and Jones, Chris and Kunisky, Dmitriy and Pesenti, Lucas},
  journal={arXiv preprint arXiv:2604.11729},
  year={2026}
}

@article{bauerschmidt2019very,
  title={A very simple proof of the {LSI} for high temperature spin systems},
  author={Bauerschmidt, Roland and Bodineau, Thierry},
  journal={Journal of Functional Analysis},
  volume={276},
  number={8},
  pages={2582--2588},
  year={2019},
  publisher={Elsevier}
}

@article{marinari1994replica,
  title={Replica field theory for deterministic models. {II. A} non-random spin glass with glassy behaviour},
  author={Marinari, Enzo and Parisi, Giorgio and Ritort, Felix},
  journal={Journal of Physics A: Mathematical and General},
  volume={27},
  number={23},
  pages={7647--7668},
  year={1994}
}

@article{otto2000generalization,
  title={Generalization of an inequality by {T}alagrand and links with the logarithmic {S}obolev inequality},
  author={Otto, Felix and Villani, C{\'e}dric},
  journal={Journal of Functional Analysis},
  volume={173},
  number={2},
  pages={361--400},
  year={2000},
  publisher={Elsevier}
}

@article{ito1998trotter,
  title={The {T}rotter-{K}ato theorem and approximation of {PDEs}},
  author={Ito, Kazufumi and Kappel, Franz},
  journal={Mathematics of computation},
  volume={67},
  number={221},
  pages={21--44},
  year={1998}
}

@article{rockner2010log,
  title={Log-{H}arnack inequality for stochastic differential equations in {H}ilbert spaces and its consequences},
  author={R{\"o}ckner, Michael and Wang, Feng-Yu},
  journal={Infinite Dimensional Analysis, Quantum Probability and Related Topics},
  volume={13},
  number={01},
  pages={27--37},
  year={2010},
  publisher={World Scientific}
}

@article{dembo2010markovian,
  title={Markovian perturbation, response and fluctuation dissipation theorem},
  author={Dembo, Amir and Deuschel, Jean-Dominique},
  journal={Annales de l'IHP Probabilit{\'e}s et statistiques},
  volume={46},
  number={3},
  pages={822--852},
  year={2010}
}

@article{eberle2016reflection,
  title={Reflection couplings and contraction rates for diffusions},
  author={Eberle, Andreas},
  journal={Probability Theory and Related Fields},
  volume={166},
  number={3},
  pages={851--886},
  year={2016},
  publisher={Springer}
}

@book{mingo2017free,
  title={Free probability and random matrices},
  author={Mingo, James A and Speicher, Roland},
  volume={35},
  year={2017},
  publisher={Springer}
}

@article{sompolinsky1982relaxational,
  title={Relaxational dynamics of the {E}dwards-{A}nderson model and the mean-field theory of spin-glasses},
  author={Sompolinsky, Haim and Zippelius, Annette},
  journal={Physical Review B},
  volume={25},
  number={11},
  pages={6860},
  year={1982},
  publisher={APS}
}

@article{arous1997symmetric,
  title={Symmetric {L}angevin spin glass dynamics},
  author={Arous, G Ben and Guionnet, Alice},
  journal={The Annals of Probability},
  volume={25},
  number={3},
  pages={1367--1422},
  year={1997},
  publisher={Institute of Mathematical Statistics}
}

@article{guionnet1997averaged,
  title={Averaged and quenched propagation of chaos for spin glass dynamics},
  author={Guionnet, Alice},
  journal={Probability Theory and Related Fields},
  volume={109},
  number={2},
  pages={183--215},
  year={1997},
  publisher={Springer}
}

@article{sherrington1975solvable,
  title={Solvable model of a spin-glass},
  author={Sherrington, David and Kirkpatrick, Scott},
  journal={Physical Review Letters},
  volume={35},
  number={26},
  pages={1792},
  year={1975},
  publisher={APS}
}

@article{parisi1979infinite,
  title={Infinite number of order parameters for spin-glasses},
  author={Parisi, Giorgio},
  journal={Physical Review Letters},
  volume={43},
  number={23},
  pages={1754},
  year={1979},
  publisher={APS}
}

@article{talagrand2006parisi,
  title={The parisi formula},
  author={Talagrand, Michel},
  journal={Annals of Mathematics},
  pages={221--263},
  year={2006},
  publisher={JSTOR}
}

@article{guerra2003broken,
  title={Broken replica symmetry bounds in the mean field spin glass model},
  author={Guerra, Francesco},
  journal={Communications in Mathematical Physics},
  volume={233},
  number={1},
  pages={1--12},
  year={2003},
  publisher={Springer}
}

@book{panchenko2013sherrington,
  title={The {S}herrington-{K}irkpatrick model},
  author={Panchenko, Dmitry},
  year={2013},
  publisher={Springer Science \& Business Media}
}

@book{talagrand2010mean,
  title={Mean field models for spin glasses: {V}olume {I}: {B}asic examples},
  author={Talagrand, Michel},
  volume={54},
  year={2010},
  publisher={Springer Science \& Business Media}
}

@book{talagrand2011mean,
  title={Mean field models for spin glasses: {V}olume {II}: {A}dvanced replica-symmetry and low temperature},
  author={Talagrand, Michel},
  volume={55},
  year={2011},
  publisher={Springer Science \& Business Media}
}

@article{carmona2006universality,
  title={Universality in {S}herrington-{K}irkpatrick's spin glass model},
  author={Carmona, Philippe and Hu, Yueyun},
  journal={Annales de l'IHP Probabilit{\'e}s et statistiques},
  volume={42},
  number={2},
  pages={215--222},
  year={2006}
}

@article{dembo2021universality,
  title={Universality for {L}angevin-like spin glass dynamics},
  author={Dembo, Amir and Lubetzky, Eyal and Zeitouni, Ofer},
  journal={The Annals of Applied Probability},
  volume={31},
  number={6},
  pages={2864--2880},
  year={2021},
  publisher={Institute of Mathematical Statistics}
}

@article{dembo2021diffusions,
  title={Diffusions interacting through a random matrix: {U}niversality via stochastic Taylor expansion},
  author={Dembo, Amir and Gheissari, Reza},
  journal={Probability Theory and Related Fields},
  volume={180},
  number={3},
  pages={1057--1097},
  year={2021},
  publisher={Springer}
}

@article{opper2016theory,
  title={A theory of solving {TAP} equations for {I}sing models with general invariant random matrices},
  author={Opper, Manfred and Cakmak, Burak and Winther, Ole},
  journal={Journal of Physics A: Mathematical and Theoretical},
  volume={49},
  number={11},
  pages={114002},
  year={2016},
  publisher={IOP Publishing}
}

@article{donoho2009message,
  title={Message-passing algorithms for compressed sensing},
  author={Donoho, David L and Maleki, Arian and Montanari, Andrea},
  journal={Proceedings of the National Academy of Sciences},
  volume={106},
  number={45},
  pages={18914--18919},
  year={2009},
  publisher={National Academy of Sciences}
}

@article{kabashima2003cdma,
  title={A {CDMA} multiuser detection algorithm on the basis of belief propagation},
  author={Kabashima, Yoshiyuki},
  journal={Journal of Physics A: Mathematical and General},
  volume={36},
  number={43},
  pages={11111--11121},
  year={2003}
}

@article{bolthausen2014iterative,
  title={An iterative construction of solutions of the {TAP} equations for the {S}herrington--{K}irkpatrick model},
  author={Bolthausen, Erwin},
  journal={Communications in Mathematical Physics},
  volume={325},
  number={1},
  pages={333--366},
  year={2014},
  publisher={Springer}
}

@article{rangan2019vector,
  title={Vector approximate message passing},
  author={Rangan, Sundeep and Schniter, Philip and Fletcher, Alyson K},
  journal={IEEE Transactions on Information Theory},
  volume={65},
  number={10},
  pages={6664--6684},
  year={2019},
  publisher={IEEE}
}

@article{ma2017orthogonal,
  title={Orthogonal {AMP}},
  author={Ma, Junjie and Ping, Li},
  journal={IEEE Access},
  volume={5},
  pages={2020--2033},
  year={2017},
  publisher={IEEE}
}

@article{takeuchi2019rigorous,
  title={Rigorous dynamics of expectation-propagation-based signal recovery from unitarily invariant measurements},
  author={Takeuchi, Keigo},
  journal={IEEE Transactions on Information Theory},
  volume={66},
  number={1},
  pages={368--386},
  year={2019},
  publisher={IEEE}
}

@article{liu2024unifying,
  title={Unifying {AMP} algorithms for rotationally-invariant models},
  author={Liu, Songbin and Ma, Junjie},
  journal={arXiv preprint arXiv:2412.01574},
  year={2024}
}

@article{celentano2021high,
  title={The high-dimensional asymptotics of first order methods with random data},
  author={Celentano, Michael and Cheng, Chen and Montanari, Andrea},
  journal={arXiv preprint arXiv:2112.07572},
  year={2021}
}

@article{gerbelot2024rigorous,
  title={Rigorous dynamical mean-field theory for stochastic gradient descent methods},
  author={Gerbelot, Cedric and Troiani, Emanuele and Mignacco, Francesca and Krzakala, Florent and Zdeborova, Lenka},
  journal={SIAM Journal on Mathematics of Data Science},
  volume={6},
  number={2},
  pages={400--427},
  year={2024},
  publisher={SIAM}
}

@article{han2025entrywise,
  title={Entrywise dynamics and universality of general first order methods},
  author={Han, Qiyang},
  journal={The Annals of Statistics},
  volume={53},
  number={4},
  pages={1783--1807},
  year={2025},
  publisher={Institute of Mathematical Statistics}
}

@article{celentano2025state,
  title={State evolution beyond first-order methods {I}: {R}igorous predictions and finite-sample guarantees},
  author={Celentano, Michael and Cheng, Chen and Pananjady, Ashwin and Verchand, Kabir Aladin},
  journal={arXiv preprint arXiv:2507.19611},
  year={2025}
}

@article{fan2026high,
  title={High-dimensional learning dynamics of multi-pass Stochastic Gradient Descent in multi-index models},
  author={Fan, Zhou and Wang, Leda},
  journal={arXiv preprint arXiv:2601.21093},
  year={2026}
}

@article{nishiyama2026high,
  title={High-Dimensional Limit of Stochastic Gradient Flow via Dynamical Mean-Field Theory},
  author={Nishiyama, Sota and Imaizumi, Masaaki},
  journal={arXiv preprint arXiv:2602.06320},
  year={2026}
}

@article{dandi2025sequential,
  title={Sequential dynamics in {I}sing spin glasses},
  author={Dandi, Yatin and Gamarnik, David and Pernice, Francisco and Zdeborov{\'a}, Lenka},
  journal={arXiv preprint arXiv:2506.09877},
  year={2025}
}

@article{arous2001aging,
  title={Aging of spherical spin glasses},
  author={Arous, G Ben and Dembo, Amir and Guionnet, Alice},
  journal={Probability Theory and Related Fields},
  volume={120},
  number={1},
  pages={1--67},
  year={2001},
  publisher={Springer}
}

@article{dembo2007limiting,
  title={Limiting dynamics for spherical models of spin glasses at high temperature},
  author={Dembo, Amir and Guionnet, Alice and Mazza, Christian},
  journal={Journal of Statistical Physics},
  volume={126},
  number={4},
  pages={781--815},
  year={2007},
  publisher={Springer}
}

@article{dembo2019dynamics,
  title={Dynamics for spherical spin glasses: disorder dependent initial conditions},
  author={Dembo, Amir and Subag, Eliran},
  journal={arXiv preprint arXiv:1908.01126},
  year={2019}
}

@article{dembo2025dynamics,
  title={Dynamics for spherical spin glasses: {G}ibbs distributed initial conditions},
  author={Dembo, Amir and Subag, Eliran},
  journal={arXiv preprint arXiv:2503.23342},
  year={2025}
}

@article{cherrier2003role,
  title={Role of the interaction matrix in mean-field spin glass models},
  author={Cherrier, Rapha{\"e}l and Dean, David S and Lef{\`e}vre, Alexandre},
  journal={Physical Review E},
  volume={67},
  number={4},
  pages={046112},
  year={2003},
  publisher={APS}
}

@article{parisi1995mean,
  title={Mean-field equations for spin models with orthogonal interaction matrices},
  author={Parisi, Giorgio and Potters, Marc},
  journal={Journal of Physics A: Mathematical and General},
  volume={28},
  number={18},
  pages={5267--5285},
  year={1995}
}

@article{opper2001adaptive,
  title={Adaptive and self-averaging {T}houless-{A}nderson-{P}almer mean-field theory for probabilistic modeling},
  author={Opper, Manfred and Winther, Ole},
  journal={Physical Review E},
  volume={64},
  number={5},
  pages={056131},
  year={2001},
  publisher={APS}
}

@article{maillard2019high,
  title={High-temperature expansions and message passing algorithms},
  author={Maillard, Antoine and Foini, Laura and Castellanos, Alejandro Lage and Krzakala, Florent and M{\'e}zard, Marc and Zdeborov{\'a}, Lenka},
  journal={Journal of Statistical Mechanics: Theory and Experiment},
  volume={2019},
  number={11},
  pages={113301},
  year={2019},
  publisher={IOP Publishing and SISSA}
}

@article{bhattacharya2016high,
  title={High temperature asymptotics of orthogonal mean-field spin glasses},
  author={Bhattacharya, Bhaswar B and Sen, Subhabrata},
  journal={Journal of Statistical Physics},
  volume={162},
  number={1},
  pages={63--80},
  year={2016},
  publisher={Springer}
}

@article{fan2024replica,
  title={The replica-symmetric free energy for {I}sing spin glasses with orthogonally invariant couplings},
  author={Fan, Zhou and Wu, Yihong},
  journal={Probability Theory and Related Fields},
  volume={190},
  number={1-2},
  pages={1--77},
  year={2024},
  publisher={Springer Nature BV}
}

@inproceedings{bolthausen2018morita,
  title={A {M}orita type proof of the replica-symmetric formula for {SK}},
  author={Bolthausen, Erwin},
  booktitle={International Conference on Statistical Mechanics of Classical and Disordered Systems},
  pages={63--93},
  year={2018},
  organization={Springer}
}

@inproceedings{ding2019capacity,
  title={Capacity lower bound for the {I}sing perceptron},
  author={Ding, Jian and Sun, Nike},
  booktitle={Proceedings of the 51st Annual ACM SIGACT Symposium on Theory of Computing},
  pages={816--827},
  year={2019}
}

@inproceedings{barbier2018mutual,
  title={The mutual information in random linear estimation beyond iid matrices},
  author={Barbier, Jean and Macris, Nicolas and Maillard, Antoine and Krzakala, Florent},
  booktitle={2018 IEEE International Symposium on Information Theory (ISIT)},
  pages={1390--1394},
  year={2018},
  organization={IEEE}
}

@article{gerbelot2022asymptotic,
  title={Asymptotic errors for teacher-student convex generalized linear models (or: {H}ow to prove {K}abashima's replica formula)},
  author={Gerbelot, Cedric and Abbara, Alia and Krzakala, Florent},
  journal={IEEE Transactions on Information Theory},
  volume={69},
  number={3},
  pages={1824--1852},
  year={2022},
  publisher={IEEE}
}

@article{ccakmak2019memory,
  title={Memory-free dynamics for the {T}houless-{A}nderson-{P}almer equations of {I}sing models with arbitrary rotation-invariant ensembles of random coupling matrices},
  author={{\c{C}}akmak, Burak and Opper, Manfred},
  journal={Physical Review E},
  volume={99},
  number={6},
  pages={062140},
  year={2019},
  publisher={APS}
}

@article{fan2022tap,
  title={{TAP} equations for orthogonally invariant spin glasses at high temperature},
  author={Fan, Zhou and Li, Yufan and Sen, Subhabrata},
  journal={arXiv preprint arXiv:2202.09325},
  year={2022}
}

@article{li2023random,
  title={Random linear estimation with rotationally-invariant designs: {A}symptotics at high temperature},
  author={Li, Yufan and Fan, Zhou and Sen, Subhabrata and Wu, Yihong},
  journal={IEEE Transactions on Information Theory},
  volume={70},
  number={3},
  pages={2118--2153},
  year={2023},
  publisher={IEEE}
}

@article{bayati2011lasso,
  title={The {LASSO} risk for {G}aussian matrices},
  author={Bayati, Mohsen and Montanari, Andrea},
  journal={IEEE Transactions on Information Theory},
  volume={58},
  number={4},
  pages={1997--2017},
  year={2011},
  publisher={IEEE}
}

@article{donoho2016high,
  title={High dimensional robust m-estimation: {A}symptotic variance via approximate message passing},
  author={Donoho, David and Montanari, Andrea},
  journal={Probability Theory and Related Fields},
  volume={166},
  number={3},
  pages={935--969},
  year={2016},
  publisher={Springer}
}

@article{lopatto2026replica,
  title={Replica symmetry up to the de {A}lmeida-{T}houless line in the {S}herrington-{K}irkpatrick model},
  author={Lopatto, Patrick},
  journal={arXiv preprint arXiv:2604.11921},
  year={2026}
}

@article{KurdykaMostowskiParusinski2000,
  author  = {Kurdyka, Krzysztof and Mostowski, Tadeusz
             and Parusi{\'n}ski, Adam},
  title   = {Proof of the Gradient Conjecture of {R. Thom}},
  journal = {Annals of Mathematics},
  volume  = {152},
  number  = {3},
  pages   = {763--792},
  year    = {2000},
  doi     = {10.2307/2661354}
}

@article{altieri2020dynamical,
  title={Dynamical mean-field theory and aging dynamics},
  author={Altieri, Ada and Biroli, Giulio and Cammarota, Chiara},
  journal={Journal of Physics A: Mathematical and Theoretical},
  volume={53},
  number={37},
  pages={375006},
  year={2020},
  publisher={IOP Publishing}
}

@article{kabashima2008inference,
  title={Inference from correlated patterns: a unified theory for perceptron learning and linear vector channels},
  author={Kabashima, Yoshiyuki},
  journal={Journal of Physics: Conference Series},
  volume={95},
  number={1},
  pages={012001},
  year={2008}
}

@article{takahashi2022macroscopic,
  title={Macroscopic analysis of vector approximate message passing in a model-mismatched setting},
  author={Takahashi, Takashi and Kabashima, Yoshiyuki},
  journal={IEEE Transactions on Information Theory},
  volume={68},
  number={8},
  pages={5579--5600},
  year={2022},
  publisher={IEEE}
}

\end{document}